\documentclass[reqno,11pt,letterpaper]{amsart}

\usepackage{amsmath,amssymb,amsthm,mathrsfs,bbm}
\usepackage{accents}
\usepackage{arydshln}
\usepackage{upgreek}
\usepackage{slashed}
\usepackage{xifthen}
\usepackage{graphicx}
\usepackage{longtable}
\usepackage[inline]{enumitem}

\usepackage{xcolor}
\definecolor{winered}{rgb}{0.7,0,0}
\definecolor{lessblue}{rgb}{0,0,0.7}

\usepackage[pdftex,colorlinks=true,linkcolor=winered,citecolor=lessblue,breaklinks=true,bookmarksopen=true]{hyperref}

\makeatletter
\newcommand{\myitem}[3]{\item[#2]\def\@currentlabel{#3}\label{#1}}
\makeatother

\usepackage{titletoc}

\makeatletter

\def\@tocline#1#2#3#4#5#6#7{
\begingroup
  \par
    \parindent\z@ \leftskip#3 \relax \advance\leftskip\@tempdima\relax
                  \rightskip\@pnumwidth plus 4em \parfillskip-\@pnumwidth
    \ifcase #1 % sections
       \vskip 0.6em \hskip 0em % add a little vspace before
       \or
       \or \hskip 0em % subsections
       \or \hskip 1em % subsubsections
    \fi%
    #6
    \nobreak\relax{\leavevmode\leaders\hbox{\,.}\hfill}
    \hbox to\@pnumwidth {\@tocpagenum{#7}}
  \par
\endgroup
}

 \def\l@section{\@tocline{0}{0pt}{0pc}{}{}}

\renewcommand{\tocsection}[3]{%
  \indentlabel{\@ifnotempty{#2}{ % for numbered sections
    \ignorespaces\bfseries{#2. #3}}}
  \indentlabel{\@ifempty{#2}{\ignorespaces\bfseries{#3}}{}} % for unnumbered sections
    \vspace{1.5pt}}

\renewcommand{\tocsubsection}[3]{%
  \indentlabel{\@ifnotempty{#2}{
    \ignorespaces#2. #3}}
  \indentlabel{\@ifempty{#2}{\ignorespaces #3}{}}
    \vspace{1.5pt}}

\renewcommand{\tocsubsubsection}[3]{%
  \indentlabel{\@ifnotempty{#2}{
    \ignorespaces#2. #3}}
  \indentlabel{\@ifempty{#2}{\ignorespaces #3}{}}
    \vspace{1.5pt}}

\makeatother

\makeatletter
\def\@nomenstarted{0}
\newlength{\@nomenoldtabcolsep}

\newcommand{\nomenstart}
  {%
    \def\@nomenstarted{1}%
    \setlength{\@nomenoldtabcolsep}{\tabcolsep}%
    \setlength{\tabcolsep}{3.5pt}%
    \begin{longtable}{p{0.11\textwidth} p{0.86\textwidth}}%found by hand
  }

\newcommand{\nomenitem}[2]{%
    \ifcase\@nomenstarted%
      \or % if nomenstarted=1, do nothing
      \or \\ % if nomenstarted=2, add newline to previous one
    \fi%
    #1\,{\leavevmode\leaders\hbox{\,.}\hfill} & #2%
    \def\@nomenstarted{2}%
  }%
\newcommand{\nomenend}
  {\\%
      \end{longtable}%
      \setlength{\tabcolsep}{\@nomenoldtabcolsep}%
      \def\@nomenstarted{0}%
  }
\makeatother

\makeatletter
\newcommand{\vast}{\bBigg@{4}}
\newcommand{\Vast}{\bBigg@{5}}
\makeatother

\allowdisplaybreaks

\numberwithin{equation}{section}
\numberwithin{figure}{section}
\newtheorem{thm}{Theorem}[section]

\newtheorem{prop}[thm]{Proposition}
\newtheorem{lemma}[thm]{Lemma}
\newtheorem{cor}[thm]{Corollary}
\newtheorem*{thm*}{Theorem}
\newtheorem*{prop*}{Proposition}
\newtheorem*{cor*}{Corollary}
\newtheorem*{conj*}{Conjecture}

\theoremstyle{definition}
\newtheorem{definition}[thm]{Definition}

\theoremstyle{remark}
\newtheorem{rmk}[thm]{Remark}

\newcommand{\mc}{\mathcal}
\newcommand{\cA}{\mc A}

\newcommand{\cC}{\mc C}

\newcommand{\cE}{\mc E}
\newcommand{\cF}{\mc F}

\newcommand{\cH}{\mc H}

\newcommand{\cK}{\mc K}
\newcommand{\cL}{\mc L}
\newcommand{\cM}{\mc M}

\newcommand{\cO}{\mc O}
\newcommand{\cP}{\mc P}

\newcommand{\cR}{\mc R}

\newcommand{\cV}{\mc V}

\newcommand{\cX}{\mc X}

\newcommand{\ms}{\mathscr}

\newcommand{\sC}{\ms C}
\newcommand{\sD}{\ms D}

\newcommand{\scri}{\ms I}

\newcommand{\sR}{\ms R}

\newcommand{\C}{\mathbb{C}}
\newcommand{\N}{\mathbb{N}}
\newcommand{\R}{\mathbb{R}}
\newcommand{\Z}{\mathbb{Z}}

\newcommand{\Sph}{\mathbb{S}}

\newcommand{\sfr}{\mathsf{r}}

\newcommand{\sfG}{\mathsf{G}}
\newcommand{\sfH}{\mathsf{H}}

\newcommand{\bfa}{\mathbf{a}}

\newcommand{\bfX}{\mathbf{X}}

\newcommand{\fc}{\mathfrak{c}}

\newcommand{\fm}{\mathfrak{m}}

\newcommand{\ft}{\mathfrak{t}}
\newcommand{\fv}{\mathfrak{v}}

\newcommand{\sld}{\slashed{d}{}}
\newcommand{\slg}{\slashed{g}{}}

\newcommand{\slG}{\slashed{G}{}}
\newcommand{\slH}{\slashed{H}{}}

\newcommand{\sldelta}{\slashed{\delta}{}}
\newcommand{\slDelta}{\slashed{\Delta}{}}
\newcommand{\slnabla}{\slashed{\nabla}{}}
\newcommand{\slpi}{\slashed{\pi}{}}

\newcommand{\slstar}{\slashed{\star}}
\newcommand{\sltr}{\operatorname{\slashed\tr}}

\newcommand{\scal}{\mathbb{S}}
\newcommand{\vect}{\mathbb{V}}
\newcommand{\asph}{{\mathrm{AS}}}
\newcommand{\TAS}{T_\asph}
\newcommand{\TS}{T_{\mathrm{S}}}

\newcommand{\vol}{\operatorname{vol}}
\newcommand{\ran}{\operatorname{ran}}
\newcommand{\ann}{\operatorname{ann}}

\newcommand{\End}{\operatorname{End}}

\newcommand{\Hom}{\operatorname{Hom}}
\renewcommand{\Re}{\operatorname{Re}}
\renewcommand{\Im}{\operatorname{Im}}

\newcommand{\mathspan}{\operatorname{span}}
\newcommand{\supp}{\operatorname{supp}}

\newcommand{\tr}{\operatorname{tr}}

\newcommand{\diag}{\operatorname{diag}}

\newcommand{\Poly}{{\mathrm{Poly}}}

\newcommand{\cd}{\fc}

\newcommand{\Ups}{\Upsilon}

\newcommand{\eps}{\epsilon}

\newcommand{\ftrans}{\;\!\wh{\ }\;\!}

\newcommand{\hra}{\hookrightarrow}
\newcommand{\la}{\langle}

\newcommand{\ol}{\overline}
\newcommand{\pa}{\partial}
\newcommand{\ra}{\rangle}

\newcommand{\ul}[1]{\underline{#1}{}}

\newcommand{\weakto}{\rightharpoonup}
\newcommand{\wh}{\widehat}
\newcommand{\wt}{\widetilde}
\newcommand{\xra}{\xrightarrow}
\newcommand{\ubar}[1]{\underaccent{\bar}{#1}}
\newcommand{\pfstep}[1]{$\bullet$\ \underline{\textit{#1}}}
\newcommand{\pfsubstep}[1]{$-$\ \textit{#1}}

\newcommand{\bop}{{\mathrm{b}}}
\newcommand{\scop}{{\mathrm{sc}}}

\newcommand{\cp}{{\mathrm{c}}}

\newcommand{\scl}{{\mathrm{sc}}}

\newcommand{\semi}{\hbar}

\newcommand{\Diff}{\mathrm{Diff}}

\newcommand{\Vb}{\cV_\bop}
\newcommand{\Diffb}{\Diff_\bop}

\newcommand{\Diffbh}{\Diff_{\bop,\semi}}

\newcommand{\Vsc}{\cV_\scl}
\newcommand{\Diffsc}{\Diff_\scl}

\newcommand{\WF}{\mathrm{WF}}

\newcommand{\Omegasc}{{}^{\scop}\Omega}

\newcommand{\Lambdasc}{{}^{\scop}\Lambda}
\newcommand{\WFsc}{\WF_{\scop}}

\newcommand{\Tb}{{}^{\bop}T}
\newcommand{\Tsc}{{}^{\scl}T}

\newcommand{\half}{{\tfrac{1}{2}}}

\newcommand{\sigmasc}{{}^\scop\sigma}

\newcommand{\CI}{\cC^\infty}
\newcommand{\CIdot}{\dot\cC^\infty}
\newcommand{\CIdotc}{\CIdot_\cp}

\newcommand{\CIc}{\cC^\infty_\cp}

\newcommand{\Hb}{H_{\bop}}

\newcommand{\Hbext}{\bar H_{\bop}}

\newcommand{\Hbsupp}{\dot H_{\bop}}

\newcommand{\Hext}{\bar H}

\newcommand{\Hsupp}{\dot H}

\newcommand{\Hbh}{H_{\bop,h}}
\newcommand{\Hbhext}{\bar H_{\bop,h}}

\newcommand{\Hsc}{H_{\scop}}

\newcommand{\phg}{{\mathrm{phg}}}

\newcommand{\Ric}{\mathrm{Ric}}
\newcommand{\Ein}{\mathrm{Ein}}
\newcommand{\bhm}{\fm}
\newcommand{\bha}{\bfa}

\newcommand{\openbigpmatrix}[1]
  {%
    \def\@bigpmatrixsize{#1}%
    \addtolength{\arraycolsep}{-#1}%
    \begin{pmatrix}%
  }
\newcommand{\closebigpmatrix}
  {%
    \end{pmatrix}%
    \addtolength{\arraycolsep}{\@bigpmatrixsize}%
  }

\newlength{\enummargin}

\DeclareGraphicsExtensions{.mps}
\newcommand{\inclfig}[1]{\includegraphics{#1}}

\makeatletter
\newcommand*{\fwbw}[1]{\expandafter\@fwbw\csname c@#1\endcsname}
\newcommand*{\@fwbw}[1]{\ifcase #1 \or {\rm fw}\or {\rm bw}\fi}
\AddEnumerateCounter{\fwbw}{\@fwbw}
\makeatother

\begin{document}

%%%%%%%%%%%%%%%%%%%%%%%%%%%%%%%%%%%%%%%%%%%%%%%%%%%%%%%%%%%%%%%%%%%%%%
% title page
\title[Linear stability of Kerr black holes]{Linear stability of slowly rotating Kerr black holes}

\date{\today. Original version: June 3, 2019.}

% 83C05: Einstein's equations
% 58J50: spectral problems, spectral geometry, scattering theory
% 83C57: black holes
% 35B40: asymptotic behavior of solutions
% 83C35: gravitational waves

\author{Dietrich H\"afner}
\address{Universit\'e Grenoble Alpes, Institut Fourier, 100 rue des maths, 38402 Gi\`eres, France}
\email{dietrich.hafner@univ-grenoble-alpes.fr}

\author{Peter Hintz}
\address{Department of Mathematics, Massachusetts Institute of Technology, Cambridge, Massachusetts 02139-4307, USA}
\email{phintz@mit.edu}

\author{Andr\'as Vasy}
\address{Department of Mathematics, Stanford University, Stanford, California 94305-2125, USA}
\email{andras@math.stanford.edu}

\begin{abstract}
  We prove the linear stability of slowly rotating Kerr black holes as solutions of the Einstein vacuum equation: linearized perturbations of a Kerr metric decay at an inverse polynomial rate to a linearized Kerr metric plus a pure gauge term. We work in a natural wave map/DeTurck gauge and show that the pure gauge term can be taken to lie in a fixed $7$-dimensional space with a simple geometric interpretation. Our proof rests on a robust general framework, based on recent advances in microlocal analysis and non-elliptic Fredholm theory, for the analysis of resolvents of operators on asymptotically flat spaces. With the mode stability of the Schwarzschild metric as well as of certain scalar and 1-form wave operators on the Schwarzschild spacetime as an input, we establish the linear stability of slowly rotating Kerr black holes using perturbative arguments; in particular, our proof does not make any use of special algebraic properties of the Kerr metric. The heart of the paper is a detailed description of the resolvent of the linearization of a suitable hyperbolic gauge-fixed Einstein operator at low energies. As in previous work by the second and third authors on the nonlinear stability of cosmological black holes, constraint damping plays an important role. Here, it eliminates certain pathological generalized zero energy states; it also ensures that solutions of our hyperbolic formulation of the linearized Einstein equation have the stated asymptotics and decay for general initial data and forcing terms, which is a useful feature in nonlinear and numerical applications.
\end{abstract}

\maketitle

\setlength{\parskip}{0.00in}
\tableofcontents
\setlength{\parskip}{0.05in}

%%%%%%%%%%%%%%%%%%%%%%%%%%%%%%%%%%%%%%%%%%%%%%%%%%%%%%%%%%%%%%%%%%%%%%
\section{Introduction}
\label{SI}

We continue our investigation of stability problems in general relativity from a systematic microlocal and spectral theoretic point of view. In previous work \cite{HintzVasyKdSStability,HintzKNdSStability}, the second and third authors proved the full nonlinear stability of slowly rotating Kerr--de~Sitter (KdS), resp.\ Kerr--Newman--de~Sitter (KNdS) black holes as solutions of the Einstein vacuum equations, resp.\ Einstein--Maxwell equations, with positive cosmological constant $\Lambda>0$. The proofs of these results rest on the completion of two main tasks:
\begin{enumerate}
\item control of \emph{asymptotics and decay} of tensor-valued linear waves on \emph{exact} slowly rotating KdS spacetimes via spectral theory/resonance analysis---we were in fact able to deduce the structure of resonances as well as mode stability of slowly rotating KdS black holes from that of spherically symmetric Schwarzschild--de~Sitter (SdS) spacetimes;
\item robust control of the \emph{regularity} of linear waves on \emph{asymptotically} KdS spacetimes via microlocal analysis on the spacetime. (Combined with the spectral theoretic results on exact KdS spacetimes, this gives precise regularity and decay results for waves on asymptotically KdS spacetimes.)
\end{enumerate}
The present paper completes the first task on slowly rotating Kerr spacetimes: we show that solutions of the linearization of the Einstein vacuum equation around a slowly rotating Kerr solution decay at an inverse polynomial rate to a linearized Kerr metric, plus a pure gauge solution which, in a linearized wave map gauge, lies in an (almost) explicit 7-dimensional vector space.

More precisely, recall that the metric of a Schwarzschild black hole with mass $\bhm>0$ is given in static coordinates by
\[
  g_{(\bhm,0)} = \left(1-\frac{2\bhm}{r}\right)d t^2 - \left(1-\frac{2\bhm}{r}\right)^{-1}d r^2 - r^2\slg,\qquad t\in\R,\ r\in(2\bhm,\infty),
\]
where $\slg$ is the standard metric on $\Sph^2$ \cite{SchwarzschildPaper}. The more general Kerr family of metrics $g_{(\bhm,\bha)}$ \cite{KerrKerr} depends in addition on the angular momentum $\bha\in\R^3$. These metrics are solutions of the \emph{Einstein vacuum equation}
\begin{equation}
\label{EqIEin}
  \Ric(g) = 0.
\end{equation}

Fix a mass parameter $\bhm_0>0$ and set $b_0=(\bhm_0,0)\in\R^4$. Restricting to Kerr black hole parameters $b=(\bhm,\bha)\in\R^4$ close to $b_0$, we can regard $g_b$ as a smooth family of stationary (time-independent) Lorentzian metrics on a fixed 4-dimensional manifold
\[
  M^\circ := \R_\ft \times [r_-,\infty) \times \Sph^2,
\]
where $r_-<2\bhm_0$. The level sets of $\ft$ here are equal to those of $t$ in $r>4\bhm_0$, i.e.\ far away from the black hole, and are regular and transversal to the future event horizon $\cH^+$, which for $b=b_0$ is located at the Schwarzschild radius $r=2\bhm_0$. Linearizing equation~\eqref{EqIEin} for $g=g_b$ in the parameters $b$, we see that the \emph{linearized Kerr metrics}
\[
  \dot g_b(\dot b) := \frac{d}{d s}\bigg|_{s=0}g_{b+s\dot b},\quad \dot b\in\R^4,
\]
are solutions of the linear equation $D_{g_b}\Ric\bigl(\dot g_b(\dot b)\bigr) = 0$.

Our main result concerns the long-time behavior of general solutions of the \emph{linearized Einstein vacuum equation}
\begin{equation}
\label{EqIEinLin}
  D_{g_b}\Ric(h) = 0.
\end{equation}
To describe it, recall that the non-linear equation~\eqref{EqIEin} admits a formulation as a Cauchy problem \cite{ChoquetBruhatLocalEinstein,ChoquetBruhatGerochMGHD}: fix a Cauchy surface
\[
  \Sigma_0^\circ = \ft^{-1}(0) \subset M^\circ.
\]
Then the initial data are a Riemannian metric $\gamma$ and a symmetric 2-tensor $k$ on $\Sigma_0^\circ$, and one seeks a solution $g$ of~\eqref{EqIEin} such that $-\gamma$ and $k$ are, respectively, the induced metric and second fundamental form of $\Sigma_0^\circ$ with respect to $g$. A solution $g$ exists locally near $\Sigma_0^\circ$ if and only if $\gamma,k$ satisfy the \emph{constraint equations}, which are the Gauss--Codazzi equations, see~\eqref{EqPfConstraints}. The Cauchy problem for~\eqref{EqIEinLin} is the linearization of this initial value problem; its solutions exist globally and are unique modulo addition of a Lie derivative $\cL_V g_b$ of $g_b$ along any vector field $V$.

\begin{thm}
\label{ThmIBaby}
  Let $b=(\bhm,\bha)$ be close to $b_0=(\bhm_0,0)$; let $\alpha\in(0,1)$. Suppose $\dot\gamma,\dot k\in\CI(\Sigma_0^\circ;S^2\,T^*\Sigma_0^\circ)$ satisfy the linearized constraint equations, and decay according to
  \[
    |\dot\gamma(r,\omega)|\leq C r^{-1-\alpha},\quad
    |\dot k(r,\omega)|\leq C r^{-2-\alpha},
  \]
  together with their derivatives along $r\pa_r$ and $\pa_\omega$ (spherical derivatives) up to order $8$. Let $h$ denote a solution of the linearized Einstein vacuum equation~\eqref{EqIEinLin} on $M^\circ$ which attains the initial data $\dot\gamma,\dot k$ at $\Sigma_0^\circ$. Then there exist linearized black hole parameters $\dot b=(\dot\bhm,\dot\bha)\in\R\times\R^3$ and a vector field $V$ on $M^\circ$ such that
  \begin{equation}
  \label{EqIBaby}
    h = \dot g_b(\dot b) + \cL_V g_b + \tilde h,
  \end{equation}
  where for bounded $r$ the tail $\tilde h$ satisfies the bound $|\tilde h|\leq C_\eta\ft^{-1-\alpha+\eta}$ for all $\eta>0$.

  Upon imposing a suitable linearized generalized harmonic gauge condition on $h$, and replacing $\dot g_b(\dot b)$ by its gauge-fixed version, we can choose $V$ to lie in a 7-dimensional space (only depending on $b$) of smooth vector fields on $M^\circ$.
\end{thm}

The gauge-fixed version of $\dot g_b(\dot b)$ is a symmetric 2-tensor $\dot g_b'(\dot b)=\dot g_b(\dot b)+\cL_{V(\dot b)}g_b$, where $V(\dot b)$ is a suitable vector field chosen so that $\dot g_b'(\dot b)$ satisfies the chosen gauge condition. We refer the reader to Theorem~\ref{ThmPfKerr} for the precise result, which \begin{enumerate*} \item operates under precise regularity assumptions on $\dot\gamma,\dot k$ encoded by weighted Sobolev spaces, \item controls $\tilde h$ in a weighted spacetime Sobolev space, and \item gives uniform estimates on spacetime, namely pointwise bounds on $\tilde h$ by $t_*^{-1-\alpha+\eta}(\tfrac{\la r\ra+t_*}{t_*})^{-\alpha+\eta}$ (in $t_*\geq 1$), where $t_*$ is equal to $\ft$ near the black hole, and equal to $t-(r+2\bhm\log(r-2\bhm))$ (which is an affine time function along null infinity) for large $r$.\end{enumerate*} See Figure~\ref{FigISetup} for an illustration of the setup.

\begin{figure}[!ht]
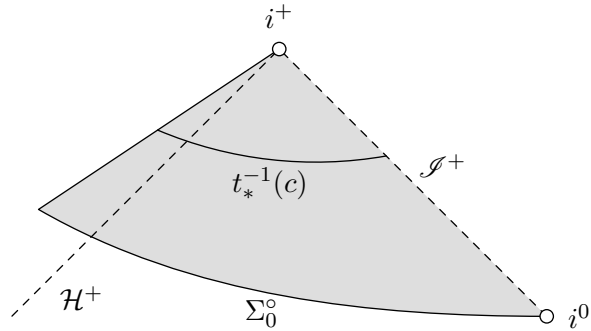

  \centering
  \inclfig{FigISetup}
  \caption{Part of the Penrose diagram of a slowly rotating Kerr spacetime, including the future event horizon $\cH^+$, null infinity $\scri^+$, future timelike infinity $i^+$ and spacelike infinity $i^0$. Shaded in gray is the domain $\{\ft\geq 0\}$ inside of $M^\circ$ where we solve the linearized Einstein equation. Also shown are the Cauchy surface $\Sigma_0^\circ=\ft^{-1}(0)$, and a level set of the function $t_*$ with respect to which we measured decay in Theorem~\ref{ThmIBaby}.}
\label{FigISetup}
\end{figure}

The gauge in which we work is (the linearization of) the natural wave map gauge for studying perturbations of a given spacetime $(M^\circ,g_b)$. In this gauge, the vector field $V$ in~\eqref{EqIBaby} is then asymptotic (as $r\to\infty$) to a linear combination of translations and boosts of Minkowski space (and an additional non-geometric vector field, which can be eliminated by a small, only $b$-dependent, modification of the gauge); asymptotic rotations either do not appear (when $g_b$ is a Schwarzschild metric, which is spherically symmetric) or can be subsumed in the infinitesimal change $\dot\bha$ of the rotation axis. Thus, we can read off the change $\dot b$ of black hole parameters (mass and angular momentum) as well as the shift $V$ of the rest frame of the black hole (translation and boost). In a nonlinear iteration, one thus expects to be able to change the gauge condition at each step to `re-center' the black hole; a (less explicit, in the sense that it is more analytic than geometric) version of this was a key ingredient in \cite{HintzVasyKdSStability}.

We use the DeTurck trick \cite{DeTurckPrescribedRicci} to relate equation~\eqref{EqIEin} to a quasilinear wave equation $P(g)=0$ for the Lorentzian metric $g$, which correspondingly relates equation~\eqref{EqIEinLin} to a linear wave equation $L_{g_b}h:=D_{g_b}P(h)=0$ for the symmetric 2-tensor $h$. After reduction to a forcing problem for $L_{g_b}$, with forcing supported in the future of a hypersurface transversal both to $\cH^+$ and future null infinity $\scri^+$, we immediately pass to the Fourier transform in $t_*$. The main part of the proof then takes place on the spectral/resolvent side; see~\S\ref{SsII} for a description of the key ingredients.

Importantly, our hyperbolic formulation $L_{g_b}h=0$ of the linearized Einstein equation is significantly better than stated in Theorem~\ref{ThmIBaby}: \emph{any} pair of Cauchy data on $\Sigma_0^\circ$ (i.e.\ a pair of smooth sections of the \emph{spacetime} symmetric 2-tensor bundle $S^2 T^*M^\circ$ over $\Sigma_0^\circ$) with $r^{-1-\alpha}$, resp.\ $r^{-2-\alpha}$ decay produces a solution $h$ of the form~\eqref{EqIBaby}. (In this sense, the geometric origin of initial data in Theorem~\ref{ThmIBaby} is irrelevant; their only use is to eventually ensure that $h$ not only solves $L_{g_b}h=0$, but also the linearized Einstein equation~\eqref{EqIEinLin}.) We expect such a strong stability statement to be useful in nonlinear applications, as explained at the end of \cite[\S1.1]{HintzVasyKdSStability} in the KdS setting; it also relates to numerical (in)stabilities when solving Einstein's equation, cf.\ \cite{PretoriusBinaryBlackHole}.

Beyond our straightforward choice of gauge condition, our construction of $L_{g_b}$ with these properties involves the implementation of \emph{constraint damping} (CD), which was first discussed in \cite{GundlachCalabreseHinderMartinConstraintDamping} and played a central role both in numerical work \cite{PretoriusBinaryBlackHole} and the nonlinear stability proofs \cite{HintzVasyKdSStability,HintzKNdSStability,HintzVasyMink4}. In fact, CD is a crucial input already in the proof of Theorem~\ref{ThmIBaby}, as it makes the low energy behavior of the spectral family of $L_{g_b}$ non-degenerate and thus perturbation-stable, allowing us to deduce Theorem~\ref{ThmIBaby} perturbatively from the statement for Schwarzschild parameters $b=b_0$; we discuss this in~\S\S\ref{SssII0}--\ref{SssIIK}.

Indeed, a key feature of our analysis is that we (prove and) use a suitable version of mode stability \emph{only of the Schwarzschild metric}, as proved by Regge--Wheeler \cite{ReggeWheelerSchwarzschild}, Vishveshwara \cite{VishveshwaraSchwarzschild}, and Zerilli \cite{ZerilliPotential}; we work with the formalism of Kodama--Ishibashi \cite{KodamaIshibashiMaster}. The only structure of Kerr metrics which we use, beyond the fact that they satisfy the Einstein equation, concerns their asymptotic behavior as $r\to\infty$. A central aim of this paper is thus to show how the (in principle straightforward) computations on Schwarzschild spacetimes, when combined with robust, perturbation-stable tools from modern microlocal analysis, imply the linear stability of \emph{slowly rotating} Kerr spacetimes, without the need for delicate separation of variable techniques. (See~\S\ref{SsIF} for comments about the full subextremal range $|\bha|<\bhm$ of Kerr parameters.) The relevant recent advances in microlocal analysis include (see~\S\ref{SsII} for more details):

\begin{enumerate}
\item non-elliptic Fredholm theory, as introduced by Vasy \cite{VasyMicroKerrdS} (see also~\cite[\S2]{HintzVasySemilinear}); for a quick guide, see Zworski \cite{ZworskiRevisitVasy};
\item a robust conceptual understanding of normally hyperbolic trapping, and the microlocal estimates which are consequences of its structural properties, as discussed in the context of Kerr and KdS black holes by Wunsch--Zworski \cite{WunschZworskiNormHypResolvent}, Dyatlov \cite{DyatlovSpectralGaps}, and Hintz \cite{HintzPsdoInner} (see also \cite{HintzVasyNormHyp,HintzPolyTrap,HintzVasyKdSStability});
\item specific to the asymptotically flat setting: scattering theory at the large end of Euclidean, or asymptotically conic, spaces, in the form pioneered by Melrose \cite{MelroseEuclideanSpectralTheory} and extended by Vasy--Zworski \cite{VasyZworskiScl} and Vasy \cite{VasyLAPLag,VasyLowEnergyLag};
\item elliptic b-analysis (analysis on manifolds with cylindrical ends), introduced by Melrose~\cite{MelroseAPS}, and used here for the study of stationary solutions (bound states, half-bound states, and generalizations).
\end{enumerate}

(We refer the reader to Dyatlov--Zworski \cite{DyatlovZworskiBook} for an introduction to some of these themes in scattering theory.) A commonality of these tools is that they rely only on structural properties of the null-bicharacteristic flow of the (wave) operator in question, rather than any special algebraic structures; we note here that our hyperbolic formulation $L_{g_b}$ of the linearized Einstein operator is a \emph{principally scalar} wave operator, to which these microlocal tools readily apply.

Nonlinear stability problems for solutions of~\eqref{EqIEin} have attracted a large amount of interest, see Friedrich \cite{FriedrichStability}, Christodoulou--Klainerman \cite{ChristodoulouKlainermanStability}, Lindblad--Rodnianski \cite{LindbladRodnianskiGlobalStability} for de~Sitter and Minkowski spacetimes, the aforementioned \cite{HintzVasyKdSStability,HintzKNdSStability} on cosmological black holes spacetimes, and the recent proof, by Klainerman--Szeftel \cite{KlainermanSzeftelPolarized}, of the nonlinear stability of the Schwarzschild metric under axially symmetric and polarized perturbations; see also Remark~\ref{RmkIPolarized}. We also mention the \emph{backwards} construction of black hole spacetimes settling down at an exponential rate to a Kerr solution \cite{DafermosHolzegelRodnianskiKerrBw}.

In very recent work, Andersson, B\"ackdahl, Blue, and Ma \cite{AnderssonBackdahlBlueMaKerr} proved Theorem~\ref{ThmIBaby} for initial data with strong decay (roughly pointwise $r^{-7/2}$ decay of $\dot\gamma$), proving $\ft^{-3/2+\eps}$ decay for the metric coefficients using energy methods; the strong decay ensures that solutions are purely radiative, i.e.\ decay to \emph{zero} (so $\dot b=0$) modulo pure gauge solutions, see also \cite{AksteinerAnderssonCharges}; we recover this by using the structure of the zero energy dual bound states, see Remark~\ref{RmkDecTo0}. They work in an outgoing radiation gauge, available on algebraically special spacetimes \cite[Remark~3.3]{AnderssonBackdahlBlueMaKerr}. Their argument uses the Newman--Penrose \cite{NewmanPenroseSpin} and Geroch--Held--Penrose \cite{GerochHeldPenrose} spin formalism, and in fact conditionally proves the linear stability of the Kerr metric in the full subextremal range, \emph{assuming} integrated energy decay holds for the Teukolsky equation; in the slowly rotating case, the latter was proved by Ma \cite{MaGravityKerr} and Dafermos--Holzegel--Rodnianski \cite{DafermosHolzegelRodnianskiTeukolsky}. See Finster--Smoller \cite{FinsterSmollerKerrStability} for results in the general case. Mode stability for curvature perturbations of Schwarzschild and Kerr spacetimes was proved by Bardeen--Press \cite{BardeenPressSchwarzschild}, Teukolsky \cite{TeukolskyKerr}, Whiting~\cite{WhitingKerrModeStability}, and Andersson--Ma--Paganini--Whiting \cite{AnderssonMaPaganiniWhitingModeStab}; see Chandrasekhar's book on the subject \cite{ChandrasekharBlackHoles} for an extensive literature review.

Previously, Dafermos--Holzegel--Rodnianski \cite{DafermosHolzegelRodnianskiSchwarzschildStability} proved the linear stability of the Schwarzschild metric in a double null gauge by reconstructing a perturbation from a certain decoupled quantity; they prove $\ft^{-1/2+\eps}$ decay of metric coefficients, and stronger decay for certain geometric quantities. Giorgi \cite{GiorgiRNTeukolsky,GiorgiRNThesis} establishes the linear stability of weakly charged Reissner--Nordstr\"om black holes using similar techniques. (See \cite{GiorgiRNPolarizedI,GiorgiRNPolarizedII} for progress in the nonlinear, axially symmetric, polarized setting.) Hung--Keller--Wang \cite{HungKellerWangSchwarzschild} proved $\ft^{-1/2}$ decay in the generalized harmonic gauge also used in the present paper, extending work by Hung--Keller \cite{HungKellerAxial}; Hung \cite{HungSchwarzschildOdd} proves $\ft^{-1+\eps}$ decay for the odd part of the linear perturbation (called `vector part' on the level of individual modes in~\S\ref{SMS}), and in upcoming work \cite{HungSchwarzschildEven} proves up to $\ft^{-2+\eps}$ decay (for $\alpha$ close to $1$) for the even part. Similarly, Johnson \cite{JohnsonSchwarzschild} obtained $\ft^{-1/2}$ decay using a modification of this gauge by suitable gauge source functions. We furthermore mention the work by Hung--Keller--Wang \cite{HungKellerWangHigher} on the decay of master quantities for the perturbation of higher-dimensional Schwarzschild black holes. We will discuss further related work in~\S\ref{SsIW}.

%%%%%%%%%%%%%%%%%%%%%%%%%%%%%%%%%%%%%%%%%%%%%%%%%%
\subsection{Ingredients of the proof}
\label{SsII}

We define the gauge 1-form
\[
  \Ups(g;g^0) := g(g^0)^{-1}\delta_g\sfG_g g^0,
\]
where $g^0$ is an arbitrary `background metric' on $M^\circ$; this vanishes iff the pointwise identity map $(M^\circ,g)\to(M^\circ,g^0)$ is a wave map. Here, $\delta_g$ is the negative divergence (thus, the adjoint of the symmetric gradient $\delta_g^*$), and $\sfG_g=1-\half g\tr_g$ is the trace reversal operator in $4$ spacetime dimensions. Following DeTurck \cite{DeTurckPrescribedRicci}, one then considers the nonlinear operator
\begin{equation}
\label{EqIIOp}
  P(g;g^0) := \Ric(g) - \delta_g^*\Ups(g;g^0).
\end{equation}
Solving the initial value problem for $\Ric(g)=0$ in the gauge $\Ups(g;g^0)=0$ is then equivalent to solving the quasilinear wave equation $P(g)=0$ with suitable Cauchy data constructed from the geometric initial data, namely, the Cauchy data induce the given geometric data at $\Sigma_0^\circ$, and the gauge condition $\Ups(g;g^0)=0$ holds there.

Let now $g_b$, $b=(\bhm,\bha)$, denote a fixed Kerr metric. It is then natural to study perturbations of $g_b$ in the gauge $\Ups(-):=\Ups(-;g_b)=0$. (Note that $g_b$ itself satisfies this gauge condition.) The linearization of $P(g;g_b)$ around $g=g_b$ is
\begin{equation}
\label{EqIILin}
  L_b := D_{g_b}P(-;g_b) = D_{g_b}\Ric - \delta_{g_b}^*\circ D_{g_b}\Ups.
\end{equation}
We solve $D_{g_b}\Ric(h)=0$ in the gauge $D_{g_b}\Ups(h)=0$ by solving $L_b h=0$ with suitable initial data. Explicitly, $L_b=\half\Box_{g_b}+\text{lower order terms}$; this is thus a linear wave operator acting on symmetric 2-tensors. Simple linear theory (using the framework of \cite{HintzVasyMink4}) allows us to solve the initial value problem for $L_b h=0$ up to a hypersurface which is transversal to future null infinity $\scri^+$ and the future event horizon $\cH^+$, see Figure~\ref{FigISetup}.

Concretely, fix a function $t_*$ which equals $t_*=t+r_*$ near $\cH^+$ and $t_*=t-r_*$ near $\scri^+$, where $r_*=r+2\bhm\log(r-2\bhm)$ is the Regge--Wheeler tortoise coordinate. It then remains to solve a forcing problem
\[
  L_b h = f,\qquad t_*\geq 0\ \ \text{on}\ \supp h,\ \supp f,
\]
where $f$ has compact support in $t_*$ and suitable decay (roughly, $r^{-2-\alpha}$) as $r\to\infty$. Our approach is to take the Fourier transform in $t_*$, giving the representation
\begin{equation}
\label{EqIIFT}
  h(t_*) = \frac{1}{2\pi}\int_{\Im\sigma=C} e^{-i\sigma t_*} \wh{L_b}(\sigma)^{-1}\hat f(\sigma)\,d\sigma,
\end{equation}
initially for $C\gg 1$ (which gives exponential bounds for $h$). We point out that typically one takes the Fourier transform in $t$ rather than $t_*$; the advantage of the latter is that precise mapping properties of $\wh{L_b}(\sigma)$ are easier to read off, and, more importantly, the analysis near $\sigma=0$ is simplified.

The strategy of our proof of Theorem~\ref{ThmIBaby} is to shift the contour of integration in~\eqref{EqIIFT} to $C=0$, which requires a detailed analysis of $\wh{L_b}(\sigma)$. A simple combination of microlocal tools already gives a large amount of information on $\wh{L_b}(\sigma)$:

\begin{enumerate}
\item\label{ItII1} the operator $\wh{L_b}(\sigma)$ is Fredholm (of index $0$) as an operator between suitable function spaces based on weighted Sobolev spaces. This uses the non-elliptic Fredholm framework of \cite{VasyMicroKerrdS}, scattering (radial point) estimates at infinity \cite{MelroseEuclideanSpectralTheory,VasyLAPLag} (for non-zero $\sigma$), radial point estimates at the horizons \cite{VasyMicroKerrdS}, real principal type propagation of regularity \cite{DuistermaatHormanderFIO2}, and (for $\sigma=0$) elliptic b-theory \cite{MelroseAPS};
\item\label{ItII2} $\wh{L_b}(\sigma)$ satisfies high energy estimates (in particular: is invertible) for $|\Re\sigma|\gg 1$ and bounded $\Im\sigma\geq 0$. This uses semiclassical estimates at the aforementioned places, together with estimates at normally hyperbolic trapping \cite{WunschZworskiNormHypResolvent,DyatlovResonanceProjectors,DyatlovSpectralGaps,HintzPsdoInner} which originate with \cite{NonnenmacherZworskiDecay}; see also \cite{GerardSjostrandHyperbolic,ChristiansonNonconc}. High energy estimates at infinity are due to Vasy--Zworski \cite{VasyZworskiScl} and Vasy~\cite{VasyLAPLag};
\item\label{ItII3} uniform Fredholm estimates for $\wh{L_b}(\sigma)$ down to $\sigma=0$ \cite{VasyLowEnergy,VasyLowEnergyLag}; see also \cite{BonyHaefnerResolvent,VasyWunschMorawetz}. (See also \cite{GuillarmouHassellResI,GuillarmouHassellResII,GuillarmouHassellSikoraResIII} for an explicit construction of the resolvent, in the $t$-Fourier transform picture, of Schr\"odinger operators on asymptotically conic manifolds.)
\end{enumerate}
We discuss these results in more detail in~\S\ref{SOp}. We only need to apply them once in order to obtain the uniform Fredholm statements for $\wh{L_b}(\sigma)$; the rest of the paper, starting with~\S\ref{SY}, contains no further microlocal analysis.

There are only two remaining ingredients, the proofs of which occupy \S\S\ref{S0}--\ref{SReg}:
\begin{enumerate}
\setcounter{enumi}{3}
\item\label{ItIII1} mode stability of $L_b$, that is, the invertibility of $\wh{L_b}(\sigma)$ for $\Im\sigma\geq 0$, $\sigma\neq 0$;
\item\label{ItIII2} the regularity of the resolvent $\wh{L_b}(\sigma)^{-1}$ near $\sigma=0$.
\end{enumerate}
(The regularity at low frequency determines the decay rate in Theorem~\ref{ThmIBaby}, as we explain in detail in~\S\S\ref{SD}--\ref{SPf}.) We first sketch our arguments on Schwarzschild spacetimes in~\S\S\ref{SssIIS}--\ref{SssII0}. In~\S\ref{SssIIK}, we explain the perturbative arguments which give~\eqref{ItIII1}--\eqref{ItIII2} on slowly rotating Kerr spacetimes.

We stress that the linearized Einstein operator itself is analytically very ill-behaved (infinite-dimensional kernel and cokernel, no control of regularity of solutions, etc.), which precludes the study e.g.\ of mode stability of slowly rotating Kerr black holes by perturbative arguments starting with the mode stability of the Schwarzschild metric. On the other hand, the gauge-fixed operator is well-behaved, in the sense of points~\eqref{ItII1}--\eqref{ItII3} above, and has strong stability properties under perturbations. Thus, a general theme underlying \S\S\ref{SK}--\ref{SR} is the exploitation of the exploitation relationship between $L_b$ and $D_{g_b}\Ric$.

%%%%%%%%%%%%%%%%%%%%%%%%%%%%%%%%%%%%%%%%
\subsubsection{Mode stability}
\label{SssIIS}

We work with a fixed \emph{Schwarzschild} metric $g_{b_0}$, and study mode solutions, with frequency $\sigma\in\C$, of the operator $L_{b_0}$,
\begin{equation}
\label{EqIIS}
  L_{b_0} h=D_{g_{b_0}}\Ric(h)-\delta_{g_{b_0}}^*D_{g_{b_0}}\Ups(h) = 0,\quad h=e^{-i\sigma t_*}h_0,
\end{equation}
where $h_0$ is stationary, i.e.\ only depends on the spatial coordinates $(r,\omega)$, and satisfies an outgoing radiation condition, which in particular entails the smoothness of $h$ across $\cH^+$.

\begin{prop}
\label{PropIIS}
  (See Proposition~\ref{PropL0}.) Mode stability holds for $L_{b_0}$: there are no non-trivial mode solutions of~\eqref{EqIIS} with $\sigma\neq 0$, $\Im\sigma\geq 0$.
\end{prop}

The linearized \emph{gauge-fixed} Einstein operator $L_{b_0}$ and the linearized Einstein operator $D_{g_{b_0}}\Ric$ are distinct, but closely related, allowing for a conceptually straightforward proof of this proposition which relies on \begin{enumerate*} \item mode stability for the wave equation on 1-forms and \item mode stability for the linearized Einstein equation on a Schwarzschild background.\end{enumerate*}

Indeed, suppose first that in addition to~\eqref{EqIIS}, $h$ also satisfies the linearized gauge condition
\begin{equation}
\label{EqIISGauge}
  D_{g_{b_0}}\Ups(h) = 0,
\end{equation}
then we also have $D_{g_{b_0}}\Ric(h)=0$. Mode stability of the Schwarzschild metric \cite{ReggeWheelerSchwarzschild,VishveshwaraSchwarzschild,ZerilliPotential} implies that $h$ is \emph{pure gauge}, that is, of the form $h=\delta_{g_{b_0}}^*\omega$ (which equals the Lie derivative $\cL_V g_{b_0}$, $V=\half\omega^\sharp$), where the gauge potential $\omega$ is outgoing and has time frequency $\sigma$ as well. Our gauge condition~\eqref{EqIISGauge} further restricts $\omega$, to wit
\begin{equation}
\label{EqIISGaugeWave}
  (D_{g_{b_0}}\Ups\circ\delta_{g_{b_0}}^*)\omega = 0.
\end{equation}
(We refer to this as the \emph{gauge potential wave equation}.) But this is one half times the wave equation on 1-forms on the Schwarzschild spacetime, for which we prove mode stability in~\S\ref{S1} by adapting the arguments from~\cite{HintzVasyKdsFormResonances}.\footnote{Mode stability for~\eqref{EqIISGaugeWave} for KdS metrics is likely \emph{false}: it is known to be false for de~Sitter metrics \cite[Appendix~C]{HintzVasyKdSStability}. In this sense, the spectral theory for the linearized gauge-fixed Einstein operator on KdS spacetimes is more complicated than on Kerr or Schwarzschild spacetimes. Since the behavior of $\wh{L_{b_0}}(\sigma)$ for real $\sigma$ is more delicate on Kerr spacetimes, we gladly use the extra information here.} Therefore, $\omega=0$ and so $h=0$.

In order to prove~\eqref{EqIISGauge}, we note that the linearization of the second Bianchi identity, $\delta_{g_{b_0}}\sfG_{g_{b_0}}\circ D_{g_{b_0}}\Ric\equiv 0$, applied to~\eqref{EqIIS} gives the equation (which we refer to as the \emph{constraint propagation wave equation})
\begin{equation}
\label{EqIISGaugeProp}
  \cP_{b_0}(D_{g_{b_0}}\Ups(h))=0,\quad \cP_{b_0}=2\delta_{g_{b_0}}\sfG_{g_{b_0}}\circ\delta_{g_{b_0}}^*.
\end{equation}
Therefore, mode stability for $\cP_{b_0}$, which is the 1-form wave operator as well, implies~\eqref{EqIISGauge}.

%%%%%%%%%%%%%%%%%%%%%%%%%%%%%%%%%%%%%%%%
\subsubsection{Zero energy modes; resolvent near zero}
\label{SssII0}

We need to analyze
\begin{enumerate}[label=(\alph*),ref=\alph*]
\item\label{ItII0Ker} the space of bound and half-bound states, and more generally the space of (generalized) zero energy modes of the operator $L_{b_0}$ in~\eqref{EqIIS}, and
\item\label{ItII0Reg} the regularity of the resolvent $\wh{L_{b_0}}(\sigma)^{-1}$ near $\sigma=0$.
\end{enumerate}
This is markedly different from the analysis of spacetimes with positive cosmological constant $\Lambda>0$: the asymptotic flatness of the spacetime causes the (regular part of) the resolvent to only have finite regularity at $\sigma=0$. Moreover, the operator $\wh{L_{b_0}}(\sigma)$ satisfies uniform estimates as $\sigma\to 0$ only on function spaces with a restricted range of allowed decay rates as $r\to\infty$, roughly, requiring the decay rate to be between $r^{-1}$ and $r^0=1$. (This is closely related to the fact that the Euclidean Laplacian on $\R^3$ is invertible on suitable weighted (b-)Sobolev spaces only when the weight of the domain allows for $r^{-1}$ asymptotics but disallows $r^0$ asymptotics as $r\to\infty$; this is in turn is linked to the off-diagonal $r^{-1}$ decay of the Green's function of the Euclidean Laplacian in $3$ dimensions.) In particular, the zero modes of interest are of size $o(1)$ as $r\to\infty$, and smooth across $\cH^+$.

\begin{prop}
\label{PropII0}
  (See Propositions~\ref{PropL0} and \ref{PropL0Lin}.) The space $\cK_{b_0}$ of zero energy modes of $L_{b_0}$ is $7$-dimensional; it is the sum of a 3-dimensional space of linearized Kerr metrics $\dot g_{b_0}(0,\dot\bha)$ corrected by addition of a pure gauge solution to arrange the gauge condition~\eqref{EqIISGauge}, a 3-dimensional space of pure gauge solutions $\delta_{g_{b_0}}^*\omega$ with $\omega$ asymptotic to a translation, and another 1-dimensional space of (spherically symmetric) pure gauge solutions.\footnote{The gauge potential of the latter is given in Proposition~\ref{PropL0}, but has no geometric significance. See Remark~\ref{RmkCDUGauge} for an indication of how to eliminate it.}

  The space of generalized zero energy modes (with $o(1)$ decay as $r\to\infty$ for fixed $t_*$, but allowing for polynomial growth in $t_*$) of $L_{b_0}$ contains the space $\wh\cK_{b_0}$ which is the sum of $\cK_{b_0}$, a 1-dimensional space of linearized Schwarzschild metrics $\dot g_{b_0}(\dot\bhm,0)$ corrected by a pure gauge solution, and a 3-dimensional space of pure gauge solutions $\delta_{g_{b_0}}^*\omega$ with $\omega$ asymptotic to a Lorentz boost.
\end{prop}

The proof of the first part is similar to that of Proposition~\ref{PropIIS}, albeit more subtle. The constraint propagation operator $\cP_{b_0}$ in~\eqref{EqIISGaugeProp} does have a zero mode, which however has exactly $r^{-1}$ decay, and thus decays more slowly than $D_{g_{b_0}}\Ups(h)=o(r^{-1})$ when $h$ is a zero mode ($\wh{L_{b_0}}(h)=0$ with $h=o(1)$). Thus, $h$ in fact solves $D_{g_{b_0}}\Ric(h)=0$, and is therefore a sum of a pure gauge solution and a linearized Kerr metric, in the precise sense stated in Theorem~\ref{ThmMS}. An analysis of the gauge potential equation~\eqref{EqIISGaugeWave} then restricts the possibilities to those stated above. Note here that in general a linearized Kerr metric $\dot g_{b_0}(\dot b)$ does not lie in $\ker L_{b_0}$; rather, $\dot g_{b_0}(\dot b)+\delta_{g_{b_0}}^*\omega$ does, where $\omega$ needs to solve equation~\eqref{EqIISGaugeWave} with non-trivial right hand side $-D_{g_{b_0}}\Ups(\dot g_{b_0}(\dot b))$, which is not always possible with stationary $\omega$. The fact that the linearization of the Schwarzschild family in the mass parameter does not give rise to a zero mode is, in this sense, due to our choice of gauge; see~\S\ref{SsL0}. The second part of Proposition~\ref{PropII0} lists all \emph{linearly} growing generalized zero energy modes, as can be shown by similar arguments; see~\S\ref{SsL0Lin}.

It turns out that there do exist generalized zero modes of $L_{b_0}$ which are (at least) \emph{quadratically growing} in $t_*$; see~\S\ref{SsL0Qu}. \emph{However}, these are pathological in that they are \emph{not} solutions of the linearized (not gauge-fixed) Einstein equation, and do not satisfy the linearized gauge condition~\eqref{EqIISGauge}. Thus, their existence is due to the failure of equation~\eqref{EqIISGaugeProp} to enforce $D_{g_{b_0}}\Ups(h)=0$ for $h$ which grow quadratically in $t_*$.

To remedy this, we thus implement \emph{constraint damping}, which means replacing $\delta_{g_{b_0}}^*$ in the definition~\eqref{EqIIOp} of the gauge-fixed Einstein operator by a zeroth order modification; concretely, we shall take
\[
  \wt\delta^*_{g,\gamma}\omega := \delta_g^*\omega + \gamma\bigl(2\cd\otimes_s\omega-g\la\cd,\omega\ra_G\bigr),\quad G=g^{-1},
\]
for a suitable (future timelike) 1-form $\cd$ with compact support near $\cH^+$ and small non-zero $\gamma$. The linearized modified gauge-fixed Einstein operator $L_{b_0,\gamma}$ is then given by~\eqref{EqIILin} with $\wt\delta_{g_{b_0},\gamma}^*$ in place of $\delta_{g_{b_0}}^*$, and correspondingly the modified gauge propagation operator is
\[
  \cP_{b_0,\gamma} = 2\delta_{g_{b_0}}\sfG_{g_{b_0}}\circ\wt\delta_{g_{b_0},\gamma}^*.
\]
\begin{prop}
  (See Proposition~\ref{PropCDU}.) For a suitable choice of $\cd$ and $\gamma$, $\cP_{b_0,\gamma}$ has no modes $\sigma\in\C$ with $\Im\sigma\geq 0$.
\end{prop}

For this modified version $L_{b_0,\gamma}$ of the linearized gauge-fixed Einstein equation, we can then show that quadratically or faster growing generalized zero modes do not exist, and thus Proposition~\ref{PropII0} captures the \emph{full} space of generalized zero modes of $L_{b_0,\gamma}$, accomplishing (the constraint damping modification of) part~\eqref{ItII0Ker}:

\begin{prop}
\label{PropII0Mod}
  (See Theorem~\ref{ThmCD0Modes}.) The space of generalized zero energy modes of $L_{b_0,\gamma}$ is equal to $\wh\cK_{b_0}$.
\end{prop}

\begin{rmk}
\label{RmkIPolarized}
  For comparison of our linear result with the nonlinear analysis of Klainerman--Szeftel \cite{KlainermanSzeftelPolarized} for axially symmetric and polarized perturbations of a Schwarzschild metric, we observe that the subspace of $\wh\cK_{b_0}$ consisting of those elements which verify (the linearized version of) these symmetry conditions is $4$-dimensional, spanned by infinitesimal changes of the Schwarzschild black hole mass ($1$ dimension), the Lie derivative of $g_{b_0}$ along the asymptotic translations and asymptotic boosts in the direction of the axis of rotational symmetry ($2$ dimensions), and the spherically symmetric pure gauge solution of Proposition~\ref{PropII0} ($1$ dimension).
\end{rmk}

Part~\eqref{ItII0Reg}, or rather the precise regularity of the resolvent of $L_{b_0,\gamma}$ near $\sigma=0$, is the most technical part of the argument; it relies on a careful analysis of the formal resolvent identity (dropping $b_0,\gamma$ from the notation for brevity) $\hat L(\sigma)^{-1}-\hat L(0)^{-1}=-\hat L(\sigma)^{-1}(\hat L(\sigma)-\hat L(0))\hat L(0)^{-1}$: when does it hold and how often can it be applied, restrictions coming from the limited range (as far as weights at $r=\infty$ are concerned) of spaces on which $\hat L(\sigma)^{-1}$ acts in a uniform manner near $\sigma=0$ and the mapping properties of $\hat L(\sigma)-\hat L(0)$ on such spaces. This is discussed in general in~\cite[\S7]{VasyLowEnergy} and \cite[\S6]{VasyLowEnergyLag}, and executed in detail in the setting of current interest in~\S\S\ref{SsRS}--\ref{SReg}.

%%%%%%%%%%%%%%%%%%%%%%%%%%%%%%%%%%%%%%%%%%%%%%%%%%
\subsubsection{Perturbation to Kerr metrics}
\label{SssIIK}

The main work is the extension of Proposition~\ref{PropII0Mod} to the operators $L_{b,\gamma}$ for $b$ near $b_0$. We accomplish this constructively by exhibiting an $11$-dimensional space $\wh\cK_b$ of generalized zero energy modes of $L_{b,\gamma}$. This can be done with robust arguments which only use the asymptotic behavior of the Kerr metric $g_{(\bhm,\bha)}$: the model case to keep in mind is that for scalar wave operators, the operator $\wh{\Box_{g_{(\bhm,\bha)}}}(0)$ equals $\wh{\Box_{g_{(\bhm,0)}}}(0)$ modulo two orders (in the sense of decay of coefficients) lower, and equals $\wh{\Box_{\ubar g}}(0)$ (with $\ubar g$ the Minkowski metric) modulo one order down; the latter operator is the Euclidean Laplacian on $\R^3$. (We discuss the precise sense in which these statements hold in~\S\ref{SK}.

We thus use \emph{normal operator arguments} familiar from b-analysis \cite{MelroseAPS}, or more simply from the analysis of ODEs with regular-singular points. Namely, an element of the nullspace of the asymptotic model (or \emph{normal operator}) $\wh{\Box_{\ubar g}}(0)$ can be corrected to an element of the nullspace of the actual operator of interest $\wh{\Box_{g_b}}(0)$; see Proposition~\ref{Prop00Grow} and its proof for the simplest instance of this. In this fashion, we can extend asymptotic symmetries of Minkowski spacetimes, namely translations, boosts, and rotations, to gauge potentials on Kerr spacetimes whose symmetric gradients already span most of $\wh\cK_b$. (The non-geometric gauge potential mentioned in Proposition~\ref{PropII0} can easily be extended to Kerr spacetimes, too.) The rest of $\wh\cK_b$ is constructed by adding to linearized Kerr metrics suitable pure gauge solutions in order to ensure the linearized gauge condition.

Throughout \S\S\ref{S0}--\ref{SL} (with the exception of \S\ref{SMS}), in which we study the (generalized) zero modes of various wave operators of interest, we shall construct those for Kerr black holes at the same time as those for Schwarzschild black holes using such normal operator arguments.

Finally, Proposition~\ref{PropIIS} for the modified operator $L_{b_0,\gamma}$ holds for $L_{b,\gamma}$ by simple perturbative arguments which exploit the non-degeneracy of $\wh{L_{b_0,\gamma}}(\sigma)$ near $\sigma=0$. The basic structure of the argument is illustrated by a simple linear algebra example: suppose $\hat L(\sigma)$ is holomorphic with values in $N\times N$ matrices; suppose $\ker\hat L(0)=\C h$ and $\ker\hat L(0)^*=\C h^*$. Then if the pairing $\la\pa_\sigma\hat L(0)h,h^*\ra$ is non-degenerate (i.e.\ non-zero), then $\hat L(\sigma)^{-1}$ has a simple pole at $\sigma=0$. If now $\hat L(\sigma)=\wh{L_0}(\sigma)$ is a member of a continuous family $\wh{L_a}(\sigma)$, $a\in\R^p$, of holomorphic operators, and $\ker\wh{L_a}(0)=\C h_a$, $\ker\wh{L_a}(0)^*=\C h_a^*$, with $h_a,h_a^*$ continuously depending on $a$, then $\wh{L_a}(\sigma)^{-1}$ has a simple pole at $\sigma=0$ as well, since the pairing $\la\pa_\sigma\wh{L_a}(0)h_a,h_a^*\ra$ is non-degenerate for small $a$ by continuity. Thus, invertibility of $\wh{L_{b,\gamma}}(\sigma)^{-1}$ in a uniform punctured neighborhood of $\sigma=0$ follows from such arguments. On other hand, invertibility for $\sigma\in\C$, $\Im\sigma\geq 0$, $|\sigma|\gtrsim 1$, follows from that of $\wh{L_{b_0,\gamma}}(\sigma)$ by standard (Fredholm) perturbation theory. See~\S\ref{SsRE}.

%%%%%%%%%%%%%%%%%%%%%%%%%%%%%%%%%%%%%%%%%%%%%%%%%%
\subsection{Further related work}
\label{SsIW}

Decay of solutions of Maxwell's equation to stationary states was proved on Schwarzschild spacetimes by Sterbenz--Tataru \cite{SterbenzTataruMaxwellSchwarzschild} and Blue~\cite{BlueMaxwellSchwarzschild}, and on slowly rotating Kerr spacetimes by Andersson--Blue \cite{AnderssonBlueMaxwellKerr}. Pasqualotto~\cite{PasqualottoMaxwell} proved decay for the Teukolsky equation for Maxwell fields on Schwarzschild spacetimes. Finster--Kamran--Smoller--Yau considered Dirac waves on Kerr spacetimes \cite{FinsterKamranSmollerYauDiracKerrNewman}.

There is a vast literature on the scalar wave equation on Kerr spacetimes, starting with the work by Wald and Kay--Wald \cite{WaldSchwarzschild,KayWaldSchwarzschild}. Sharp pointwise decay (Price's law \cite{PriceLawI,PriceLawII}) is now known in the full subextremal range ($|\bha|<\bhm$) by work of Tataru~\cite{TataruDecayAsympFlat} (see also the subsequent work by Metcalfe--Tataru--To\-ha\-ne\-a\-nu \cite{MetcalfeTataruTohaneanuPriceNonstationary}) and Dafermos--Rodnianski--Shlapentokh-Rothman \cite{DafermosRodnianskiShlapentokhRothmanDecay,ShlapentokhRothmanModeStability} (building on prior work \cite{DafermosRodnianskiKerrDecaySmall,DafermosRodnianskiKerrBoundedness} which followed $L^\infty$ estimates on Schwarzschild spacetimes by Donninger--Schlag--Soffer \cite{DonningerSchlagSofferSchwarzschild}). Finster--Kamran--Smoller--Yau \cite{FinsterKamranSmollerYauKerr} proved decay without quantitative rates. Earlier results include decay on slowly rotating Kerr black holes due to Andersson--Blue \cite{AnderssonBlueHiddenKerr} and Tataru--Tohaneanu \cite{TataruTohaneanuKerrLocalEnergy}. Marzuola--Metcalfe--Tataru--To\-ha\-ne\-a\-nu and Tohaneanu \cite{MarzuolaMetcalfeTataruTohaneanuStrichartz,TohaneanuKerrStrichartz} proved Strichartz estimates on Kerr spacetimes. See Luk \cite{LukKerrNonlinear} for the solution of a scalar semilinear equation with null form nonlinearity on Kerr black holes, and Ionescu--Klainerman and Stogin \cite{IonescuKlainermanWavemap,StoginKerrWaveMap} for a wave map equation related to the study of polarized perturbations. We remark that Theorems~\ref{Thm0} and \ref{Thm1} easily imply the decay of scalar waves (to zero) and 1-forms (to an element of a 1-dimensional space of stationary solutions) on slowly rotating Kerr spacetimes when combined with results on the regularity of the resolvent which follow by (a simpler version of) arguments in~\S\S\ref{SR}--\ref{SReg}.

Closely related to black hole stability problems is the black hole uniqueness problem; the stability of Kerr family would imply that it gives, locally, the full space of stationary solutions of the Einstein vacuum equation. See \cite{IonescuKlainermanUniqueness,AlexakisIonescuKlainermanUniqueness,RobinsonBlackHoleUniquenessReview,ChruscielCostaHeuslerStationaryBH} for results and further references, and \cite{HintzKNdSUniq} for results in the cosmological setting.

In the algebraically more complicated but analytically less degenerate context of cosmological black holes, we recall that S\'a Barreto--Zworski \cite{SaBarretoZworskiResonances} studied the distribution of resonances of SdS black holes; exponential decay of linear scalar waves to constants was proved by Bony--H\"afner \cite{BonyHaefnerDecay} and Melrose--S\`a Barreto--Vasy \cite{MelroseSaBarretoVasySdS} on SdS and by Dyatlov \cite{DyatlovQNM,DyatlovQNMExtended} on KdS spacetimes, and substantially refined by Dyatlov \cite{DyatlovAsymptoticDistribution} to a full resonance expansion. (See \cite{DafermosRodnianskiSdS} for a physical space approach giving superpolynomial energy decay.) Tensor-valued and nonlinear equations on KdS spacetimes were studied in a series of works by Hintz--Vasy \cite{HintzVasySemilinear,HintzVasyQuasilinearKdS,HintzVasyKdsFormResonances,HintzVasyKdSStability,HintzKNdSStability}. For a physical space approach to resonances, see Warnick \cite{WarnickQNMs}, and for the Maxwell equation on SdS spacetimes, see Keller~\cite{KellerMaxwellSdS}.

%%%%%%%%%%%%%%%%%%%%%%%%%%%%%%%%%%%%%%%%
\subsection{Future directions}
\label{SsIF}

A natural problem is the extension of Theorem~\ref{ThmIBaby} to the full subextremal range of Kerr spacetimes, as conditionally accomplished by \cite{AnderssonBackdahlBlueMaKerr}. In our framework, this requires:
\begin{enumerate}
\item\label{ItIF1} a suitable version of mode stability for metric perturbations of Kerr spacetimes, generalizing~\S\ref{SMS};
\item\label{ItIF2} a mode analysis of 1-form operator, and the implementation of constraint damping;
\item\label{ItIF3} non-degenerate control of generalized zero energy states.
\end{enumerate}
We stress that these are all ingredients on the level of \emph{individual modes}. It is natural to expect that \eqref{ItIF1} and \eqref{ItIF3} can be accomplished by following, \emph{on the level of modes}, the procedure in \cite[\S\S3, 8]{AnderssonBackdahlBlueMaKerr} for recovering metric perturbations from the Teukolsky scalar---for which mode stability is known \cite{WhitingKerrModeStability,AnderssonMaPaganiniWhitingModeStab}. Problem~\eqref{ItIF2} has, to the authors' knowledge, not yet been studied in the full subextremal range (though it is related to mode stability for the Maxwell equation on Kerr). The explicit nature of these ingredients suggests that the use of arguments specifically tailored to the special nature of the Kerr metric, such as separation of variables, are unavoidable. On the other hand, the general Fredholm framework discussed in points~\eqref{ItII1}--\eqref{ItII3} at the beginning of~\S\ref{SsII} applies in the full subextremal range, hence full linear stability would follow from the above mode stability inputs, as shown for slowly rotating Kerr black holes in the present paper. (See \cite{DyatlovWaveAsymptotics} for a discussion of trapping for scalar waves in this generality. Work by Marck \cite{MarckParallelNull} implies that tensor-valued waves on Kerr spacetimes can be treated as well using the techniques of \cite{HintzPsdoInner}; this will be taken up elsewhere.)

In another direction, we expect that the methods of the present paper can be used to give another proof of the results by Andersson--Blue and Sterbenz--Tataru \cite{AnderssonBlueMaxwellKerr,SterbenzTataruMaxwellSchwarzschild} on decay to the stationary Coulomb solution for the Maxwell equation on slowly rotating Kerr spacetimes, or more generally on stationary perturbations of such spacetimes. In fact, we expect that differential form-valued waves (of any form degree) on slowly rotating Kerr spacetimes decay to stationary solutions as in the Kerr--de~Sitter case studied in \cite{HintzVasyKdsFormResonances}; for differential 1-forms, this follows from Theorem~\ref{Thm1} and (a simpler version of) the arguments in~\S\S\ref{SR}--\ref{SD}. Coupling the Maxwell equation to the Einstein equation, we expect the full linear stability of slowly rotating Kerr--Newman spacetimes (with subextremal charge) under coupled gravitational and electromagnetic perturbations to follow by an (essentially only computational) extension of the methods of the present paper.

Moreover, we expect the linear stability of higher-dimensional Schwarzschild black holes (see \cite{HungKellerWangHigher} for a first step in this direction) and their perturbations, slowly rotating Myers--Perry black holes \cite{MyersPerryBlackHoles}, to follow by similar arguments, the main task again being computational, namely the detailed mode analysis; the general microlocal tools apply to such spacetimes as well. (See \cite{DiasHartnettSantosQNM} for linear instabilities of black holes in high dimensions with \emph{large} angular momenta.)

%%%%%%%%%%%%%%%%%%%%%%%%%%%%%%%%%%%%%%%%%%%%%%%%%%
\subsection{Outline of the paper}

The paper is structured as follows:
\begin{itemize}
\item in~\S\ref{SB}, we introduce notions pertaining to geometry and analysis on compactifications of non-compact spaces (such as spatial slices $\Sigma_0^\circ$ of $M^\circ$), namely b-analysis and scattering analysis and associated bundles and function spaces, following Melrose \cite{MelroseAPS,MelroseEuclideanSpectralTheory};
\item in~\S\ref{SK}, we describe the Kerr family of metrics as a smooth family of metrics on a fixed spacetime manifold;
\item in~\S\ref{SOp}, we define the gauge-fixed Einstein operator, prove the general properties listed in~\S\ref{SsII}, and introduce general constraint damping modifications;
\item in~\S\ref{SY}, we introduce useful terminology for the description of scalar, 1-form, and 2-tensor perturbations of spherically symmetric spacetimes;
\item in~\S\ref{S0}, we study modes for the scalar wave equation on Schwarzschild and slowly rotating Kerr spacetimes;
\item in~\S\ref{S1}, we do the same for 1-forms, and also construct the gauge potentials for asymptotic translations, boosts, and rotations;
\item in~\S\ref{SMS}, we prove the version of mode stability of the Schwarzschild metric used in the sequel;
\item in~\S\ref{SL}, we combine the previous results, analyze the mode stability of the linearized gauge-fixed Einstein operator, and motivate the need for constraint damping;
\item in~\S\ref{SCD}, we implement constraint damping and discuss the consequences for the linearized modified gauge-fixed Einstein operator $L_{b,\gamma}$;
\item in~\S\ref{SR}, we prove mode stability for $L_{b,\gamma}$ and determine the structure of its resolvent near zero frequency;
\item in~\S\ref{SReg}, we prove higher regularity of the regular part of the resolvent of $L_{b,\gamma}$ near zero frequency as well as for large frequencies;
\item in~\S\ref{SD}, we combine the previous sections to establish the precise asymptotic behavior of solutions of the linearized modified gauge-fixed Einstein operator;
\item in~\S\ref{SPf} finally, we reduce the initial value problem for the Einstein equation to the general decay result of the previous section.
\end{itemize}

For the reader interested in getting an impression of the flavor of our arguments, we refer to the construction of (generalized) zero energy modes in~\S\ref{Ss00Grow} and \S\ref{Ss10Gen}; the proof of Proposition~\ref{PropL0} can be read early on as well, and serves as motivation for most of the preceding constructions. The key idea/calculation behind the perturbation theory in the context of constraint damping is explained in~\S\ref{SsCD0}. The perturbative arguments for the existence of the resolvent on slowly rotating Kerr spacetimes are given in~\S\ref{SsRE}.

%%%%%%%%%%%%%%%%%%%%%%%%%%%%%%%%%%%%%%%%%%%%%%%%%%
\subsection*{Acknowledgments}

D.H.\ acknowledges support from the ANR funding ANR-16-CE40-0012-01. Part of this research was conducted during the period P.H.\ served as a Clay Research Fellow. P.H.\ would also like to thank the Miller Institute at the University of California, Berkeley, for support during the early stages of this project. A.V.\ gratefully acknowledges partial support from the NSF under grant number DMS-1664683 and from a Simons Fellowship.

%%%%%%%%%%%%%%%%%%%%%%%%%%%%%%%%%%%%%%%%%%%%%%%%%%
\section{b- and scattering structures}
\label{SB}

We first discuss geometric structures on manifolds with boundaries or corners, and corresponding function spaces. Thus, let $X$ be a compact $n$-dimensional manifold with boundary $\pa X\neq\emptyset$, and let $\rho\in\CI(X)$ denote a boundary defining function: $\pa X=\rho^{-1}(0)$, $d\rho\neq 0$ on $\pa X$. We then define the Lie algebras of \emph{b-vector fields} and \emph{scattering vector fields} by
\begin{equation}
\label{EqBVf}
  \Vb(X) = \{ V\in\cV(X) \colon V\ \text{is tangent to}\ \pa X \}, \quad
  \Vsc(X) = \rho\Vb(X).
\end{equation}
In local \emph{adapted coordinates} $x\geq 0$, $y\in\R^{n-1}$ on $X$, with $x=0$ locally defining the boundary of $X$ (thus $\rho=a(x,y)x$ for some smooth $a>0$), elements of $\Vb(X)$ are of the form $a(x,y)x\pa_x+\sum_{i=1}^{n-1}b^i(x,y)\pa_{y^i}$, with $a,b^i\in\CI(X)$, while elements of $\Vsc(X)$ are of the form $a(x,y)x^2\pa_x+\sum_{i=1}^{n-1}b^i(x,y) x\pa_{y^i}$. Thus, there are natural vector bundles
\[
  \Tb X\to X, \quad \Tsc X \to X,
\]
with local frames given by $\{x\pa_x,\pa_{y^i}\}$, resp.\ $\{x^2\pa_x,x\pa_{y^i}\}$, such that $\Vb(X)=\CI(X;\Tb X)$ and $\Vsc(X)=\CI(X;\Tsc X)$; thus, for example, $x\pa_x$ is a \emph{smooth, non-vanishing} section of $\Tb X$ down to $\pa X$. Over the interior $X^\circ$, these bundles are naturally isomorphic to $T X^\circ$, but the maps $\Tb X\to T X$ and $\Tsc X\to T X$ fail to be injective over $\pa X$. We denote by $\Diffb^m(X)$, resp.\ $\Diffsc^m(X)$ the space of $m$-th order b-, resp.\ scattering differential operators, consisting of linear combinations of up to $m$-fold products of elements of $\Vb(X)$, resp.\ $\Vsc(X)$.

The dual bundles $\Tb^*X\to X$ (b-cotangent bundle), resp.\ $\Tsc^*X\to X$ (scattering cotangent bundle) have local frames
\[
  \frac{d x}{x},\ d y^i,\quad \text{resp.}\quad \frac{d x}{x^2},\ \frac{d y^i}{x},
\]
which are \emph{smooth} down to $\pa X$ as sections of these bundles (despite their being singular as standard covectors, i.e.\ elements of $T^*X$). A \emph{scattering metric} is then a section $g\in\CI(X;S^2\,\Tsc^*X)$ which is a non-degenerate quadratic form on each scattering tangent space $\Tsc_p X$, $p\in X$; b-metrics are defined analogously.

These structures arise naturally on compactifications of non-compact manifolds, the simplest example being the \emph{radial compactification of $\R^n$}, defined by
\begin{equation}
\label{EqBRadCp}
  \ol{\R^n} := \bigl(\R^n \sqcup ([0,1)_\rho\times\Sph^{n-1})\bigr) / \sim
\end{equation}
where the relation $\sim$ identifies a point in $\R^n\setminus\{0\}$, expressed in polar coordinates as $r\omega$, $r>0$, $\omega\in\Sph^{n-1}$, with the point $(\rho,\omega)$ where
\[
  \rho=r^{-1};
\]
this has a natural smooth structure, with smoothness near $\pa\ol{\R^n}=\rho^{-1}(0)$ meaning smoothness in $(\rho,\omega)$. In polar coordinates in $r>1$, the space of b-vector fields is then locally spanned over $\CI(\ol{\R^n})$ by $\rho\pa_\rho=-r\pa_r$ and $\cV(\Sph^{n-1})$; scattering vector fields are spanned by $\rho^2\pa_\rho=-\pa_r$ and $\rho\cV(\Sph^{n-1})$. Using standard coordinates $x^1,\ldots,x^n$ on $\R^n$, scattering vector fields on $\ol{\R^n}$ are precisely those of the form
\[
  \sum_{i=1}^n a^i\pa_{x^i}, \quad a^i\in\CI(\ol{\R^n});
\]
this entails the statement that $\{\pa_{x^1},\ldots,\pa_{x^n}\}$, which is a frame of $T^*\R^n$, extends by continuity to a smooth frame of $\Tsc^*\ol{\R^n}$ down to $\pa\ol{\R^n}$. Thus, the space of scattering vector fields on $\ol{\R^n}$ is generated over $\CI(\ol{\R^n})$ by constant coefficient (translation-invariant) vector fields on $\R^n$. On the other hand, $\Vb(\ol{\R^n})$ is spanned over $\CI(X)$ by vector fields on $\R^n$ with coefficients which are \emph{linear} functions, i.e.\ by $\pa_{x^1},\ldots,\pa_{x^n}$, and $x^i\pa_{x^j}$, $1\leq i,j\leq n$.

On the dual side, $\Tsc^*\ol{\R^n}$ is spanned by $d x^i$, $1\leq i\leq n$, \emph{down to} $\pa\ol{\R^n}$. Therefore, a scattering metric $g\in\CI(\ol{\R^n},S^2\,\Tsc^*\ol{\R^n})$ is a non-degenerate linear combination of $d x^i\otimes_s d x^j=\half(d x^i\otimes d x^j+d x^j\otimes d x^i)$ with $\CI(\ol{\R^n})$ coefficients. In particular, the Euclidean metric
\[
  (d x^1)^2+\cdots+(d x^n)^2 \in \CI(\ol{\R^n};S^2\,\Tsc^*\ol{\R^n})
\]
is a Riemannian scattering metric.

By extension from $T^*X^\circ$, one can define Hamilton vector fields $H_p$ of smooth functions $p\in\CI(\Tsc^*X)$. In fact $H_p\in\Vsc(\Tsc^*X)$ is a scattering vector field on $\Tsc^*X$, which is a manifold with boundary $\Tsc^*_{\pa X}X$. (Likewise, if $p\in\CI(\Tb^*X)$, then $H_p\in\Vb(\Tb^*X)$.) For us, the main example will be the Hamilton vector field $H_G$ where $G(z,\zeta):=|\zeta|_{g_z^{-1}}^2$ is the dual metric function of a scattering metric $g\in\CI(X;S^2\,\Tsc^*X)$.

We next introduce Sobolev spaces corresponding to b- and scattering structures. As an integration measure on $X$, let us fix a \emph{scattering density}, i.e.\ a positive section of $\Omegasc^1 X=|\Lambda^n\,\Tsc^*X|$, which in local adapted coordinates takes the form $a(x,y)|\tfrac{d x}{x^2}\frac{d y}{x^{n-1}}|$ with $0<a\in\CI(X)$. (On $\ol{\R^n}$, one can take $|d x^1\cdots d x^n|$.) This provides us with a space $L^2(X)$; the norm depends on the choice of density, but all choices lead to equivalent norms. Working with a b-density on the other hand would give a different space, namely a weighted version of $L^2(X)$; we therefore stress that even for b-Sobolev spaces, we work with a scattering density. Thus, for $k\in\N_0$, we define
\[
  H_\bullet^k(X) := \{ u\in L^2(X) \colon V_1\cdots V_j u\in L^2(X)\ \forall\,V_1,\ldots,V_j\in\cV_\bullet(X),\ 0\leq j\leq k \},\quad
  \bullet=\bop,\scop,
\]
called \emph{b-} or \emph{scattering Sobolev space}. Using a finite spanning set in $\cV_\bullet(X)$, one can give this the structure of a Hilbert space; $H_\bullet^s(X)$ for general $s\in\R$ is then defined by duality and interpolation. If $\ell\in\R$, we denote weighted Sobolev spaces by
\[
  H_\bullet^{s,\ell}(X) = \rho^\ell H_\bullet^s(X) = \{ \rho^\ell u \colon u\in H_\bullet^s(X) \}.
\]
For example, $\Hsc^{s,\ell}(\ol{\R^n})\cong\la x\ra^{-\ell}H^s(\R^n)$ is the standard weighted Sobolev space on $\R^n$. The space of weighted ($L^2$-)conormal functions on $X$ is
\[
  \Hb^{\infty,\ell}(X) = \bigcap_{s\in\R}\Hb^{s,\ell}(X).
\]
Dually, we define
\[
  \Hb^{-\infty,\ell}(X) = \bigcup_{s\in\R}\Hb^{s,\ell}(X).
\]
Note that $\Hb^{s,\ell}(X)\subset\cC^{-\infty}(X):=\CIdot(X)^*$ (where $\CIdot(X)\subset\CI(X)$ is the subspace of functions vanishing to infinite order at $\pa X$) consists of tempered distributions. (In particular, they are extendible distributions at $\pa X$ in the sense of \cite[Appendix~B]{HormanderAnalysisPDE3}.) We furthermore introduce the notation
\begin{equation}
\label{EqBHsPM}
  H_\bop^{s,\ell+} := \bigcup_{\eps>0} H_\bop^{s,\ell+\eps},\quad
  H_\bop^{s,\ell-} := \bigcap_{\eps>0} H_\bop^{s,\ell-\eps}
\end{equation}
for $s\in\R\cup\{\pm\infty\}$. A space closely related to $\Hb^{\infty,\ell}(X)$ is
\[
  \cA^\ell(X) := \{ u\in\rho^\ell L^\infty(X) \colon \Diffb(X)u\subset\rho^\ell L^\infty(X) \},
\]
consisting of \emph{weighted $L^\infty$-conormal functions}. For $X=\ol{\R^3}$, we have the inclusions
\[
  \Hb^{\infty,\ell}(\ol{\R^3}) \subset \cA^{\ell+3/2}(\ol{\R^3}), \quad
  \cA^\ell(\ol{\R^3}) \subset \Hb^{\infty,\ell-3/2-}(\ol{\R^3}),
\]
by Sobolev embedding. (The shift $\tfrac32$ in the weight is due to our defining b-Sobolev spaces with respect to scattering densities; indeed, for $s>\tfrac32$,
\begin{equation}
\label{EqBHbSobEmb}
  \Hb^{s,\ell}(\ol{\R^3};|d x^1\,d x^2\,d x^3|)=\Hb^{s,\ell+3/2}(\ol{\R^3};\la r\ra^{-3}|d x^1\,d x^2\,d x^3|) \hra \la r\ra^{-\ell-3/2}L^\infty(\ol{\R^3}),
\end{equation}
with the second density here being a b-density on $\ol{\R^3}$.) We define $\cA^{\ell+}$ and $\cA^{\ell-}$ analogously to~\eqref{EqBHsPM}. These notions extend readily to sections of rank $k$ vector bundles $E\to X$: for instance, in a local trivialization of $E$, an element of $H_\bullet^{s,\ell}(X,E)$ is simply a $k$-tuple of elements of $H_\bullet^{s,\ell}(X)$.

We next turn to the notion of \emph{$\cE$-smoothness}, where $\cE\subset\C\times\N_0$ is an \emph{index set}; the latter means that $(z,k)\in\cE$ implies $(z,j)\in\cE$ for $0\leq j\leq k$ and $(z-i,k)\in\cE$, and that any sequence $(z_j,k_j)\in\cE$ with $|z_j|+|k_j|\to\infty$ satisfies $\Im z_j\to-\infty$. The space
\[
  \cA^\cE_\phg(X) \subset \cA^{-\infty}(X) = \bigcup_{\ell\in\R} \cA^\ell(X)
\]
then consists of all $u$ for which
\[
  \prod_{\genfrac{}{}{0pt}{}{(z,j)\in\cE}{\Im z\geq -N}} (\rho D_\rho-z)u \in \cA^N(X) \quad \forall\,N\in\R,
\]
where we define $D_\rho=i^{-1}\pa_\rho$ with respect to any fixed collar neighborhood of $\pa X$. Equivalently, there exist $a_{(z,j)}\in\CI(\pa X)$, $(z,j)\in\cE$, such that
\begin{equation}
\label{EqBExp}
  u - \sum_{\genfrac{}{}{0pt}{}{(z,j)\in\cE}{\Im z\geq -N}} \rho^{i z}|\log\rho|^j a_{(z,j)} \in \cA^N(X) \quad\forall\,N\in\R.
\end{equation}
We say that $u\in\cA^{-\infty}(X)$ is \emph{polyhomogeneous} if it is $\cE$-smooth for some index set $\cE$.

Suppose now $X'$ is a compact manifold with boundary, and let $X\subset X'$ be a submanifold with boundary. Suppose that its boundary decomposes into two non-empty sets
\begin{equation}
\label{EqBInt}
  \pa X = \pa_-X \sqcup \pa_+X,\qquad \pa_-X=\pa X\setminus\pa X',\quad \pa_+X=\pa X';
\end{equation}
we consider $\pa_+X$ to be a boundary `at infinity', while $\pa_-X$ is an interior, `artificial' boundary. Concretely, this means that we define (by a slight abuse of notation)
\[
  \Vb(X) := \{ V|_X \colon V\in\Vb(X') \}, \quad
  \Vsc(X) := \{ V|_X \colon V\in\Vsc(X') \};
\]
these vector fields are b or scattering at infinity, but are unrestricted at $\pa_-X$. A typical example is $X'=\ol{\R^n}$ and  $X=\ol{\{r\geq 1\}}\subset X'$, in which case $\pa_-X=\{r=1\}$, while $\pa_+X=\pa X'$ is the boundary (at infinity) of $\ol{\R^n}$. See Figure~\ref{FigBInt}.

\begin{figure}[!ht]
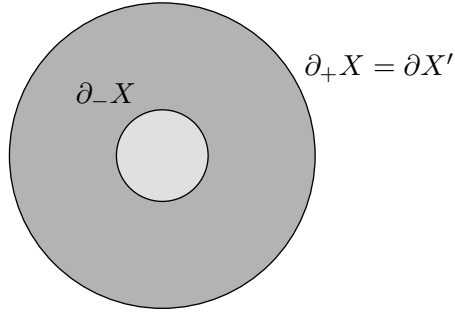

  \centering
  \inclfig{FigBInt}
  \caption{A typical example of the setting~\eqref{EqBInt}: $X$ (dark gray) is a submanifold of $X'$ (the union of the dark and light gray regions) with two boundary components $\pa_+X=\pa X'$ and $\pa_-X\subset(X')^\circ$. We then consider function spaces such as $\Hbext^{s,\ell}(X)$, which measures b-regularity of degree $s$ at $\pa_+X$ (with decay rate $\ell$), and standard regularity (regularity with respect to incomplete vector fields) at $\pa_-X$.}
\label{FigBInt}
\end{figure}

There are now two natural classes of Sobolev spaces: those consisting of \emph{extendible distributions},
\begin{equation}
\label{EqBExt}
  \Hext_\bullet^{s,\ell}(X) := \{ u|_{X^\circ} \colon u\in H_\bullet^{s,\ell}(X') \},\quad \bullet=\bop,\scop,
\end{equation}
and those consisting of \emph{supported distributions},
\begin{equation}
\label{EqBSupp}
  \Hsupp_\bullet^{s,\ell}(X) := \{ u \colon u\in H_\bullet^{s,\ell}(X'),\ \supp u\subset X \}.
\end{equation}
Away from $\pa_-X$, these are the same as the standard spaces $H_\bullet^{s,\ell}(X)$; thus, the subspaces of $\Hext_\bullet^{s,\ell}(X)$ or $\Hsupp_\bullet^{s,\ell}(X)$ consisting of those elements which are polyhomogeneous (in particular automatically conormal) at $\pa_+X$ are well-defined.

If $X$ is the `spatial part' of a stationary spacetime $M=\R_\ft\times X$ with projection $\pi_X\colon M\to X$, and $\cF(X)$ denotes a space of distributions on $X$ such as $\cF(X)=\Hbext^{\infty,\ell}(X)$ or $\Hbsupp^{-\infty,\ell}(X)$ in the setting~\eqref{EqBInt}, we will be interested not only in \emph{zero modes} $\cF(X)\cong\pi_X^*\cF(X)$, i.e.\ $\ft$-independent distributions, but also \emph{generalized zero modes},
\begin{equation}
\label{EqBFnSpacetime}
  \Poly(\ft)\cF(X) := \bigcup_{k\in\N_0}\Poly^k(\ft)\cF(X), \quad
  \Poly^k(\ft)\cF(X) := \left\{ \sum_{j=0}^k \ft^j a_j \colon a_j\in\cF(X) \right\}.
\end{equation}
(Here, we do not require $X$ to be spacelike or $d\ft$ to be timelike. Note that the definition~\eqref{EqBFnSpacetime} is independent of a choice of metric on $M$.)

For high energy estimates of the resolvent on Kerr spacetimes, we will work with \emph{semiclassical b-Sobolev spaces}. Thus, if $X$ is a manifold with boundary, we define $\Hbh^{s,\ell}(X)=\Hb^{s,\ell}(X)$ as a set, but with norm depending on the semiclassical parameter $h\in(0,1]$: if $V_1,\ldots,V_N\in\Vb(X)$ spans $\Vb(X)$ over $\CI(X)$, we let
\[
  \|u\|_{\Hbh^k(X)}^2 := \sum_{\genfrac{}{}{0pt}{}{1\leq i_1,\ldots,i_j\leq N}{0\leq j\leq k}} \|(h V_{i_1})\cdots(h V_{i_j})u\|_{L^2(X)}^2,\quad
  \|u\|_{\Hbh^{k,\ell}(X)} := \|\rho^{-\ell}u\|_{\Hbh^k(X)},
\]
for $k\in\N_0$; for $k\in\R$, we take the dual and interpolated norms. (Alternatively, one can define $\Hbh^{s,\ell}(X)$ using semiclassical b-pseudodifferential operators, see~\cite[Appendix~A]{HintzVasyKdSStability}.) On manifolds $X$ as in~\eqref{EqBInt}, one can then define semiclassical spaces $\bar H_{\bop,h}^{s,\ell}(X)$ and $\dot H_{\bop,h}^{s,\ell}(X)$ of extendible and supported distributions analogously to~\eqref{EqBExt}--\eqref{EqBSupp}.

%%%%%%%%%%%%%%%%%%%%%%%%%%%%%%%%%%%%%%%%%%%%%%%%%%%%%%%%%%%%%%%%%%%%%%
\section{The spacetime manifold and the Kerr family of metrics}
\label{SK}

In~\S\ref{SsK0}, we define a manifold $M^\circ$, equipped with the metric of a Schwarzschild black hole with mass
\[
  \bhm_0>0,
\]
which is a (small) extension of the domain of outer communications across the future event horizon; the purpose of such an extension is that it allows the immediate application of by now standard microlocal tools at the event horizon, as we will discuss in~\S\ref{SOp}. In~\S\ref{SsKa}, we define the Kerr family with black hole parameters $(\bhm,\bha)\in\R\times\R^3$ close to $(\bhm_0,0)$ as a smooth family of metrics on $M^\circ$. In~\S\ref{SsKSt}, we elucidate the structure of stationary differential operators on $M^\circ$ near spatial infinity. In~\S\ref{SsKFl} finally, we describe the full null-geodesic dynamics of slowly rotating Kerr spacetimes.

%%%%%%%%%%%%%%%%%%%%%%%%%%%%%%%%%%%%%%%%%%%%%%%%%%
\subsection{The Schwarzschild metric}
\label{SsK0}

Fix a mass parameter $\bhm_0>0$. We define the \emph{static patch} of the mass $\bhm_0$ Schwarzschild spacetime to be the manifold
\begin{equation}
\label{EqK0Static}
  \cM = \R_t \times \cX,\quad \cX=(2\bhm_0,\infty)_r\times\Sph^2,
\end{equation}
with $t$ called the \emph{static time function}. We equip $\cM$ with the metric
\begin{equation}
\label{EqK0Metric}
  g_{(\bhm_0,0)} := \mu(r)\,d t^2 - \mu(r)^{-1}\,d r^2 - r^2\slg,
\end{equation}
with $\slg$ denoting the standard metric on $\Sph^2$, and where
\begin{equation}
\label{EqK0mu}
  \mu(r) = 1-\frac{2\bhm_0}{r}.
\end{equation}
This is the unique family (depending on the real parameter $\bhm_0$) of spherically symmetric solutions of the Einstein vacuum equation in $3+1$ dimensions:
\[
  \Ric(g_{(\bhm_0,0)}) = 0.
\]

We denote the dual metric by $G_{(\bhm_0,0)}=g_{(\bhm_0,0)}^{-1}$. The form~\eqref{EqK0Metric} of the metric is singular at the Schwarzschild radius $r=r_{(\bhm_0,0)}:=2\bhm_0$. This is merely a coordinate singularity: switching to the null coordinate
\begin{equation}
\label{EqK0t0}
  t_0 := t + r_*,\ \ r_*:=r+2\bhm_0\log(r-2\bhm_0),\quad
  |d t_0|_{G_{(\bhm_0,0)}}^2 = 0,
\end{equation}
so $d r_*=\mu^{-1}d r$, the Schwarzschild metric and its dual take the form
\begin{equation}
\label{EqK0Metrict0}
  g_{(\bhm_0,0)} = \mu\,d t_0^2 - 2\,d t_0\,d r - r^2\slg, \quad
  G_{(\bhm_0,0)} = -2\pa_{t_0}\pa_r-\mu\pa_r^2 - r^2\slG.
\end{equation}
This is now smooth and non-degenerate on the extended manifold
\begin{equation}
\label{EqK0Ext}
  M^\circ = \R_{t_0}\times X^\circ \supset \cM,\quad X^\circ=[r_-,\infty)_r\times\Sph^2 \supset \cX,
\end{equation}
where the endpoint $r_-\in(0,2\bhm_0)$ is an arbitrary fixed number.

On the other hand, the metric $g_{(\bhm_0,0)}$ is a warped product in static coordinates, which is a useful structure at infinity; we thus introduce another coordinate,
\begin{equation}
\label{EqK0TimeFnChi}
  t_{\chi_0} := t + \int \frac{1-\chi_0(r)}{\mu(r)}\,d r,
\end{equation}
where $\chi_0$ is smooth, vanishes near $r\leq 3\bhm_0$, and is identically $1$ for $r\geq 4\bhm_0$; thus, $t_{\chi_0}-t_0$ is smooth and bounded in $r\leq 4\bhm_0$, while $t_{\chi_0}=t$ for $r\geq 4\bhm_0$ provided we choose the constant of integration suitably.

We compactify $X^\circ$ as follows: recalling the definition of $\ol{\R^3}$ from~\eqref{EqBRadCp}, we set
\[
  X := \ol{X^\circ} \subset \ol{\R^3}, \quad \rho:=r^{-1},
\]
adding the boundary $\{\rho=0\}\cong\Sph^2$ at infinity, with $\rho$ a boundary defining function of infinity. Thus, $X=\ol{\{r\geq r_-\}}$, and we let $\pa_-X=r^{-1}(r_-)$, $\pa_+X=\pa\ol{\R^3}\subset X$. Within $X$, the topological boundary of $\cX$ has two components,
\begin{equation}
\label{EqK0StaticBdy}
  \pa\cX = \pa_-\cX \sqcup \pa_+\cX,\qquad
  \pa_-\cX := r^{-1}(2\bhm_0), \quad
  \pa_+\cX := \rho^{-1}(0).
\end{equation}
We shall call (somewhat imprecisely) $\pa_-\cX$ the event horizon.

The level sets of $t_0$ are smooth submanifolds of $M^\circ$ (unlike those of $t$ which are singular at $r=2\bhm_0$) which are transversal to the future event horizon $\R_{t_0}\times\pa_-\cX$. However, a sequence of points with $t_0$ bounded and $r\to\infty$ tends to \emph{past} null infinity. Thus, for the description of waves near the future event horizon and \emph{future} null infinity (and in between), we introduce another function
\begin{equation}
\label{EqK0TimeFn}
  t_* := t + (r+2\bhm_0\log(r-2\bhm_0))\chi(r) - (r+2\bhm_0\log(r-2\bhm_0))(1-\chi(r)),
\end{equation}
where $\chi(r)\equiv 1$ for $r\leq 3\bhm_0$ and $\chi(r)\equiv 0$ for $r\geq 4\bhm_0$; it smoothly interpolates between $t+r_*$ near the event horizon and $t-r_*$ near null infinity. In the bulk of this paper, we will study forcing problems for wave equations of the type $\Box_{g_{(\bhm_0,0)}}u=f$, where $f$ is supported in $t_*\geq 0$. (Choosing $t_*$ more carefully so as to make $d t_*$ future causal would ensure that $u$ is supported in $t_*\geq 0$ as well; since we are not arranging this, we will have $t_*\geq -C$ on $\supp u$ for some constant $C\geq 0$ depending only on our choice of $t_*$.) Note that $\R_{t_*}\times\pa X$ has two components,
\begin{equation}
\label{EqK0SurfFin}
  \Sigma_{\rm fin} := \R_{t_*}\times\pa_- X
\end{equation}
(which is a spacelike hypersurface inside of the black hole) and $\R_{t_*}\times\pa_+ X$ (which is future null infinity, typically denoted $\scri^+$); moreover, the future event horizon is $\R_{t_*}\times\pa_-\cX$.

The more common setting for the Einstein equation is to place asymptotically flat initial data on a Cauchy surface
\begin{equation}
\label{EqK0Sigma0}
  \Sigma_0^\circ \subset M^\circ
\end{equation}
which we can choose to be a smooth and spacelike transversal to $t_0$-translations, and equal to $t^{-1}(0)$ in $r\geq 3\bhm_0$. See Figure~\ref{FigK0Time}.

\begin{figure}[!ht]
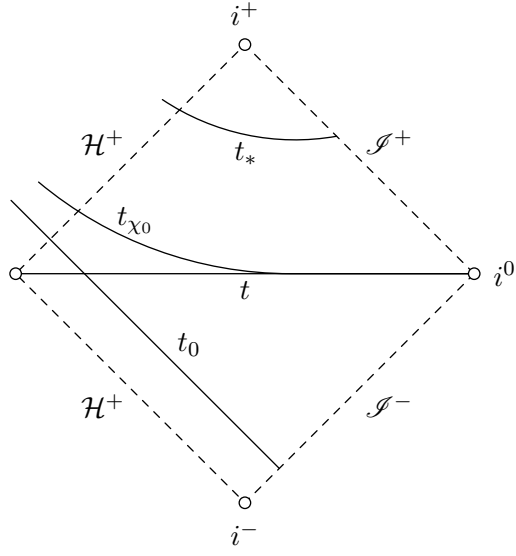

  \centering
  \inclfig{FigK0Time}
  \caption{Illustration of time functions on $M^\circ$ in the Penrose diagram of the Schwarzschild metric (including future/past null infinity $\scri^\pm$, the future/past event horizon $\cH^\pm$, spacelike infinity $i^0$, and future/past timelike infinity $i^\pm$). Shown are level sets of the static time function $t$, of its modification near the event horizon $t_{\chi_0}$, of the null coordinate $t_0$, and of the function $t_*$.}
\label{FigK0Time}
\end{figure}

%%%%%%%%%%%%%%%%%%%%%%%%%%%%%%%%%%%%%%%%%%%%%%%%%%
\subsection{The Kerr family}
\label{SsKa}

Write $b_0:=(\bhm_0,0)$. Consider black hole parameters $b=(\bhm,\bha)\in\R\times\R^3$, with $\bha\in\R^3$ denoting the angular momentum. If $a=|\bha|\neq 0$, choose \emph{adapted polar coordinates} $(\theta,\varphi)$ on $\Sph^2$, meaning that $\hat\bha=\bha/|\bha|$ is the north pole $\theta=0$; for $a=0$, adapted polar coordinates are simply any polar coordinates. The Kerr metric in Boyer--Lindquist coordinates $(t,r,\theta,\varphi)$ is then
\begin{equation}
\label{EqKaMetric}
\begin{split}
  g_b^{\rm BL} &= \frac{\Delta_b}{\varrho_b^2}(d t-a\sin^2\theta\,d\varphi)^2 - \varrho_b^2\Bigl(\frac{d r^2}{\Delta_b}+d\theta^2\Bigr) - \frac{\sin^2\theta}{\varrho_b^2}\bigl(a\,d t-(r^2+a^2)d\varphi\bigr)^2, \\
  G_b^{\rm BL} &= \frac{1}{\Delta_b\varrho_b^2}\bigl((r^2+a^2)\pa_t+a\pa_\varphi\bigr)^2 - \frac{\Delta_b}{\varrho_b^2}\pa_r^2 - \frac{1}{\varrho_b^2}\pa_\theta^2 - \frac{1}{\varrho_b^2\sin^2\theta}(\pa_\varphi+a\sin^2\theta\,\pa_t)^2, \\
  &\quad \Delta_{(\bhm,\bha)} = r^2-2\bhm r+a^2, \ \ 
  \varrho_{(\bhm,\bha)}^2 = r^2+a^2\cos^2\theta.
\end{split}
\end{equation}
This is a solution of the Einstein vacuum equation:
\begin{equation}
\label{EqKaEinstein}
  \Ric\bigl(g_{(\bhm,\bha)}^{\rm BL}\bigr)=0.
\end{equation}
Here, we will focus on parameters $(\bhm,\bha)$ close to $(\bhm_0,0)$; in particular, we are looking at slowly rotating ($a\ll\bhm$) Kerr black holes. The form~\eqref{EqKaMetric} of the metric breaks down at the event horizon
\begin{equation}
\label{EqKaRadius}
  r=r_{(\bhm,\bha)}:=\bhm+\sqrt{\bhm^2-\bha^2}.
\end{equation}
This is again merely a coordinate singularity: for $\chi\in\CI(\R)$, vanishing in $r\leq 2\bhm$, put
\begin{equation}
\label{EqKaNull}
  t_{b,\chi} = t + \int\frac{r^2+a^2}{\Delta_b}(1-\chi(r))\,d r, \quad
  \varphi_{b,\chi} = \varphi + \int\frac{a}{\Delta_b}(1-\chi(r))\,d r.
\end{equation}
The metric $g_b^{\rm BL}$ then takes the form
\begin{equation}
\label{EqKaMetric2}
\begin{split}
  g_{b,\chi} &= \frac{\Delta_b}{\varrho_b^2}(d t_{b,\chi}-a\sin^2\theta\,d\varphi_{b,\chi})^2-2(1-\chi)(d t_{b,\chi}-a\sin^2\theta\,d\varphi_{b,\chi})d r \\
    &\hspace{6em} - \frac{\chi(2-\chi)\varrho_b^2}{\Delta_b}d r^2 - \varrho_b^2\,d\theta^2-\frac{\sin^2\theta}{\varrho_b^2}\bigl(a\,d t_{b,\chi}-(r^2+a^2)d\varphi_{b,\chi}\bigr)^2,
\end{split}
\end{equation}
which is smooth and non-degenerate on $M^\circ_b=\R_{t_{b,\chi}}\times[r_-,\infty)_r\times\Sph^2_{\theta,\varphi_{b,\chi}}$. Taking $\chi=\chi_0$ as in~\eqref{EqK0TimeFnChi}, and choosing suitable constants of integration, we have $t_{b,\chi_0}\equiv t$, $\varphi_{b,\chi_0}\equiv\varphi$ for $r\geq 4\bhm_0$; defining the diffeomorphism
\[
  \Phi_b \colon M^\circ_b \to M^\circ,
  \quad (t_{\chi_0},r,\theta,\varphi)(\Phi_b(p))=(t_{b,\chi_0},r,\theta,\varphi_{b,\chi_0})(p),
\]
where $t_{\chi_0}$ is defined in~\eqref{EqK0TimeFnChi} and $(\theta,\varphi)$ denotes polar coordinates on $\Sph^2$ adapted to $\bha$,
\begin{equation}
\label{EqKaMetricEmbed}
  g_b = (\Phi_b)_*(g_b^{\rm BL}) \in \CI(M^\circ;S^2 T^*M^\circ)
\end{equation}
is a stationary metric on $M^\circ$. For $b=b_0=(\bhm_0,0)$, this produces the Schwarzschild metric $g_{b_0}$, thus the notation is unambiguous. As in \cite[Proposition~3.5]{HintzVasyKdSStability}, one can prove that $g_b$ is a smooth family of metrics on $M^\circ$.

Furthermore, an inspection of~\eqref{EqKaMetric} shows that the mass parameter $\bhm$ contributes $\rho\CI(X)$ terms to the metric, while the angular momentum $\bha$ only contributes $\rho^2\CI(X)$ terms; see also~\eqref{EqKStMetrics}--\eqref{EqKStMetrics2} below.

The choice~\eqref{EqKaMetricEmbed} of defining the Kerr family as a smooth family of metrics on the fixed manifold $M^\circ$ is not unique, and in fact another presentation is more convenient for calculations later on. Namely, we also consider coordinates $t_{b,0}$, $\varphi_{b,0}$ (that is, $t_{b,\chi}$ and $\varphi_{b,\chi}$ for the function $\chi\equiv 0$) and use the embedding
\[
  \Phi_b^0 \colon M_b^\circ \to M^\circ,\quad
  (t_0,r,\theta,\varphi)(\Phi_b^0(p))=(t_{b,0},r,\theta,\varphi_{b,0})(p).
\]
Denote the resulting presentation of the Kerr family on $M^\circ$ by
\[
  g_b^0 = (\Phi_b^0)_*(g_b^{\rm BL}) \in \CI(M^\circ;S^2 T^*M^\circ).
\]
We have $\Phi_{b_0}^0=\Phi_{b_0}$, hence $g_{b_0}^0=g_{b_0}$; see also~\eqref{EqKaLie} below. The full Schwarzschild family becomes
\begin{equation}
\label{EqKaMetric2Schw}
  g_{(\bhm,0)}^0 = \mu_\bhm\,d t_0^2 - 2 d t_0\,d r-r^2\slg,\quad
  \mu_\bhm = 1-\frac{2\bhm}{r}.
\end{equation}

Linearizing the families $g_b$ and $g_b^0$ in the parameter $b$ yields \emph{linearized Kerr metrics},
\begin{equation}
\label{EqKaLin}
  \dot g_{(\bhm,\bha)}^{(0)}(\dot\bhm,\dot\bha) = \bigl(\tfrac{d}{d s}g^{(0)}_{(\bhm,\bha)+s(\dot\bhm,\dot\bha)}\bigr)\big|_{s=0},
\end{equation}
with linear dependence on $\dot b=(\dot\bhm,\dot\bha)$. We record the particular cases
\begin{equation}
\label{EqKaLin2}
  \dot g_{(\bhm_0,0)}^0(1,0) = -\tfrac{2\bhm_0}{r}\,d t_0^2, \quad
  \dot g_{(\bhm_0,0)}^0(0,\dot\bha) = (\tfrac{4\bhm_0}{r} d t_0+2 d r)\sin^2\theta\,d\varphi,
\end{equation}
where in the second line $|\dot\bha|=1$, and $(\theta,\varphi)$ are spherical coordinates adapted to $\dot\bha$. By linearizing~\eqref{EqKaEinstein}, we find $(D_{g_{(\bhm,\bha)}^{(0)}}\Ric)\bigl(\dot g^{(0)}_{(\bhm,\bha)}(\dot\bhm,\dot\bha)\bigr) = 0$.

\begin{rmk}
\label{RmkKaLie}
  Since $g_b=(\Phi_b^0\circ\Phi_b^{-1})^*g_b^0$, the linearized Kerr metric $\dot g_{(\bhm,\bha)}$ can be obtained from $\dot g^0_{(\bhm,\bha)}$ by pullback along a diffeomorphism and addition of a Lie derivative of $g_{(\bhm,\bha)}$ along a suitable vector field. More precisely, in $(t_*,r,\theta,\varphi)$ coordinates on $M^\circ$, we have
  \begin{subequations}
  \begin{equation}
  \label{EqKaLie}
    \Phi_b^0\circ\Phi_b^{-1} \colon (t_*,\,r,\,\theta,\,\varphi) \mapsto \left(t_*+\int\left(\frac{r^2+a^2}{\Delta_b}-\frac{r^2}{\Delta_{b_0}}\right)\chi_0\,d r,\,r,\,\theta,\,\varphi+\int\frac{a}{\Delta_b}\chi_0\,d r\right).
  \end{equation}
  In particular, for $\dot b=(\dot\bhm,\dot\bha)$, the two versions of the linearization of the Kerr metric at $g_{b_0}$ are related by
  \begin{equation}
  \label{EqKaLie2}
  \begin{split}
    \dot g_{b_0}(\dot b) = \dot g_{b_0}^0(\dot b) +{}& \cL_{V(\dot b)} g_{b_0},\\
    V(\dot b)&=\frac{d}{d s}\bigg|_{s=0}(\Phi_{b_0+s\dot b}^0\circ\Phi_{b_0+s\dot b}^{-1}) \\
      &= \dot\bhm\biggl(\int_{r_0}^r \frac{2 r^3}{\Delta_{b_0}^2}\chi_0\,d r\biggr)\pa_{t_*} + \dot\bha\biggl(\int_{r_0}^r\frac{\chi_0}{\Delta_{b_0}}d r\biggr)\pa_\varphi.
  \end{split}
  \end{equation}
  \end{subequations}
  (Since $\pa_{t_*}$ and $\pa_\varphi$ are Killing vector fields for $g_{b_0}$, the constants of integration, and thus $r_0$ in the definition of $V(\dot b)$, can be chosen arbitrarily.)
\end{rmk}

Finally, we note that spherically symmetric outgoing light cones for $g_{(\bhm,0)}$ depend on $\bhm$ via a logarithmic (in $r$) correction. Thus, we introduce the function
\begin{equation}
\label{EqKTimeFn}
  t_{\bhm,*} := t_* - \bigl(2\bhm\log(r-2\bhm)-2\bhm_0\log(r-2\bhm_0)\bigr) (1-\chi(r)),
\end{equation}
generalizing~\eqref{EqK0TimeFn}; it is smooth on $M^\circ$, and equals $t-(r+2\bhm\log(r-2\bhm))$ in $r\geq 4\bhm_0$. In particular, in $r\geq 4\bhm_0$, we have
\[
  g_{(\bhm,0)}=\mu_\bhm\,d t^2-\mu_\bhm^{-1}\,d r^2-r^2\slg=\mu_\bhm\,d t_{\bhm,*}^2+2 d t_{\bhm,*}\,d r-r^2\slg,
\]
which thus has the same form as $g_{(\bhm_0,0)}$ with respect to $t_*=t_{\bhm_0,*}$; in particular, $d t_{\bhm,*}$ is null for large $r$. (One can also construct $t_{(\bhm,\bha),*}$, a lower order ($\cO(r^{-1})$) correction of $t_{\bhm,*}$, which takes the angular momentum $\bha$ into account and has the property that $d t_{(\bhm,\bha),*}$ is null with respect to $g_{(\bhm,\bha)}$ for large $r$; see \cite{PretoriusIsraelKerr}.)

%%%%%%%%%%%%%%%%%%%%%%%%%%%%%%%%%%%%%%%%%%%%%%%%%%
\subsection{Stationarity, vector bundles, and geometric operators}
\label{SsKSt}

In the notation~\eqref{EqK0Ext}, denote the projection to the spatial manifold by
\[
  \pi_X \colon M^\circ \to X^\circ;
\]
this is independent of the choice of time function. Suppose $E_1\to X^\circ$ is a vector bundle; then differentiation along $\pa_t=\pa_{t_0}=\pa_{t_*}$ is a well-defined operation on sections of the pullback bundle $\pi_X^*E_1$. The tangent bundle of $M^\circ$ is an important example of such a pullback bundle, as
\[
  T M^\circ \cong \pi_X^*(T_{t_0^{-1}(0)}M^\circ) \cong \pi_X^*(T_{t_*^{-1}(0)}M^\circ),
\]
likewise for the cotangent bundle and other tensor bundles.

Let $E_2\to X^\circ$ be another vector bundle, and suppose $\wh L(0)\in\Diff(X^\circ;E_1,E_2)$ is a differential operator; fixing $\ft=t_*+F$, $F\in\CI(X^\circ)$, with $d\ft\neq 0$ everywhere, we can then define its \emph{stationary extension} by assigning to $u\in\CI(M^\circ;\pi_X^*E_1)$ the section $(L u)(\ft,-):=\wh L(0)(u(\ft,-))$ of $\pi_X^*E_2$; this extension \emph{does} depend on the choice of $\ft$. The action of $L$ on stationary functions on the other hand is independent of the choice of $\ft$ since
\begin{equation}
\label{EqKStStat}
  L\pi_X^*=\pi_X^*\wh L(0).
\end{equation}
Via stationary extension, one can consider $\Diffb(X;E)$ (for $E\to X$ a smooth vector bundle down to $\pa_+ X$) to be a subalgebra of $\Diff(M^\circ;\pi_X^*E)$; likewise $\Diffsc(X;E)\hra\Diff(M^\circ;\pi_X^*E)$.

Conversely, if $L\in\Diff(M^\circ;\pi_X^*E_1,\pi_X^*E_2)$ is \emph{stationary}, i.e.\ commutes with $\pa_t$, there exists a unique (independent of the choice of $\ft$) operator $\wh L(0)\in\Diff(X^\circ;E_1,E_2)$ such that the relation~\eqref{EqKStStat} holds. More generally, we can consider the formal conjugation of $L$ by the Fourier transform in $\ft$,
\[
  \wh L(\sigma) := e^{i\sigma\ft}L e^{-i\sigma\ft} \in \Diff(X^\circ;E_1,E_2),
\]
where we identify the stationary operator $e^{i\sigma\ft}L e^{-i\sigma\ft}$ with an operator on $X^\circ$. Switching from $\ft$ to another time function, $\ft+F'$, $F'\in\CI(X^\circ)$, amounts to conjugating $\wh L(\sigma)$ by $e^{i\sigma F'}$.

In order to describe the uniform behavior of geometric operators at $\pa_+X$ concisely, we need to define a suitable extension of $T^*M^\circ$ to `infinity'. To accomplish this, note that the product decomposition~\eqref{EqK0Ext} induces a splitting
\[
  T^*M^\circ \cong T^*\R_{t_0} \boxplus T^*X^\circ = \pi_T^*(T^*\R_{t_0}) \oplus \pi_X^*(T^*X^\circ),
\]
where $\pi_T\colon M_0\to\R_{t_0}$ is the projection. We therefore define the \emph{extended scattering cotangent bundle} of $X$ by
\begin{equation}
\label{EqKStExtTsc}
  \wt\Tsc{}^*X := \R d t_0 \oplus \Tsc^*X.
\end{equation}
At this point, $d t_0$ is merely a name for the basis of a trivial real rank $1$ line bundle over $X$; considering the pullback bundle $\pi_X^*\wt\Tsc{}^*X\to M^\circ$, we identify it with the differential of $t_0\in\CI(M^\circ)$, giving an isomorphism
\begin{equation}
\label{EqKStExtTscIso}
  (\pi_X^*\wt\Tsc{}^*X)|_{M^\circ} \cong T^*M^\circ.
\end{equation}
Smooth sections of $\wt{\Tsc^*}X\to X$ are linear combinations, with $\CI(X)$ coefficients, of $d t_0$ and the 1-forms $d x^i$, where $(x^1,x^2,x^3)$ are standard coordinates on $X^\circ\subset\R^3$. One can switch to another time function in~\eqref{EqKStExtTsc}, say, $t_*$, by writing $d t_*=d t_0+(d t_*-d t_0)$, with the second term being a smooth scattering 1-form on $X$; likewise for $t_{\chi_0}$ and also for the static time $t$ in $r>2\bhm_0$.

For a stationary metric $g$ on $M^\circ$, there exists a unique $g'\in\CI(X^\circ;S^2\,\wt\Tsc{}^*X)$ such that $\pi_X^*g'=g$, namely $g'$ is the restriction (as a section of $S^2\,\wt\Tsc{}^*X$) of $g$ to any transversal of $\pi_X$, such as level sets of $t_0,t_{\chi_0},t_*$. Identifying $g$ with $g'$ and applying this to the Kerr family, we then have 
\begin{equation}
\label{EqKStMetrics}
  g_b,\ g_b^0 \in \CI(X;S^2\,\wt\Tsc{}^*X);
\end{equation}
they are non-degenerate down to $\pa_+ X$. Moreover, we have
\begin{equation}
\label{EqKStMetrics2}
\begin{split}
  &g_b - \ubar g \in \rho\CI, \quad \ubar g:=d t_{\chi_0}^2-d r^2-r^2\slg, \\
  &g_{(\bhm,\bha)}-g_{(\bhm,0)} \in \rho^2\CI,
\end{split}
\end{equation}
i.e.\ a Kerr metric equals the Minkowski metric $\ubar g$ to leading order, and is a $\cO(\rho^2)$ perturbation of the Schwarzschild metric of the same mass.

We proceed to discuss basic geometric operators on Kerr spacetimes. We write
\begin{equation}
\label{EqKStDel}
  (\delta_g^*\omega)_{\mu\nu}=\half(\omega_{\mu;\nu}+\omega_{\nu;\mu}), \quad
  (\delta_g h)_\mu=-h_{\mu\nu;}{}^\nu, \quad
  \sfG_g = 1-\half g\,\mathrm{tr}_g,
\end{equation}
and furthermore denote by
\begin{equation}
\label{EqKStWave}
  \Box_{g,0},\ \Box_{g,1},\ \Box_{g,2},
\end{equation}
the wave operator $-\!\tr_g\nabla^2$ on scalars, 1-forms, and symmetric 2-tensors, respectively. When the bundle is clear from the context, we shall simply write $\Box_g$.
\begin{prop}
\label{PropKStBox2}
  Writing the operator $\Box_{g_b,2}$ as
  \begin{equation}
  \label{EqKStBox2}
    \Box_{g_b,2} = |d t_{\chi_0}|_{G_b}^2 D_{t_{\chi_0}}^2 + \wh{\Box_{g_b,2}}(0) + Q_b D_{t_{\chi_0}},
  \end{equation}
  we have $\wh{\Box_{g_b,2}}(0)\in\rho^2\Diffb^2(X;S^2\,\wt\Tsc{}^*X)$ and $Q_b\in\rho^2\Diffsc^1(X;S^2\,\wt\Tsc{}^*X)$.
\end{prop}

Away from $\pa_+X$, this merely states that $\Box_{g_b,2}$ is a second order differential operator with smooth coefficients, with principal symbol given by the dual metric function. It thus suffices to analyze $\Box_{g,2}$ near spatial infinity where $t_{\chi_0}\equiv t$ is the static or Boyer--Lindquist time coordinate; there, Proposition~\ref{PropKStBox2} is a consequence of:

\begin{lemma}
\label{LemmaKStBox2}
  Suppose $g$ is a stationary Lorentzian metric on $M^\circ$ for which
  \begin{equation}
  \label{EqKStBox2Struct}
  \begin{split}
    g(\pa_t,\pa_t) &\in 1+\rho\CI(X), \\
    g(\pa_t,-) &\in \rho^2\CI(X;\Tsc^*X), \\
    g|_{\Tsc X\times\Tsc X} &\in -h + \rho\CI(X;S^2\,\Tsc^*X),
  \end{split}
  \end{equation}
  where $h\in\CI(X;S^2\,\Tsc^*X)$ is Riemannian. Then the operator $\Box_g=\Box_{g,2}$ takes the form
  \[
    \Box_g = |d t|_G^2 D_t^2 + \wh{\Box_g}(0) + Q D_t,
  \]
  with $\wh{\Box_g}(0)\in\rho^2\Diffb^2(X;S^2\,\wt{\Tsc^*}X)$ and $Q\in\rho^2\Diffsc^1(X;S^2\,\wt{\Tsc^*}X)$. Moreover, $\wh{\Box_g}(0)\bmod\rho^3\Diffb^2(X;S^2\,\wt{\Tsc^*}X)$ only depends on $h$.
\end{lemma}
\begin{proof}
  We use the splittings
  \begin{subequations}
  \begin{align}
  \label{EqKStBox2Spl1}
    \wt\Tsc{}^*X &= \la d t\ra \oplus \Tsc^*X, \\
  \label{EqKStBox2Spl2}
    S^2\,\wt\Tsc{}^*X &= \la d t^2\ra \oplus (2 d t\otimes_s \Tsc^*X) \oplus S^2\,\Tsc^*X, \\
  \label{EqKStBox2Spl3}
    \wt\Tsc{}^*X \otimes \wt\Tsc{}^*X &= \la d t^2\ra \oplus (d t\otimes \Tsc^*X) \oplus (\Tsc^*X\otimes d t) \oplus (\Tsc^*X\otimes\Tsc^*X), \\
  \label{EqKStBox2Spl4}
    \wt\Tsc{}^*X \otimes S^2\,\wt\Tsc{}^*X &= (d t \otimes S^2\,\wt\Tsc{}^*X) \oplus (\Tsc^*X\otimes S^2\,\wt\Tsc{}^*X), \\
  \label{EqKStBox2Spl5}
  \begin{split}
    S^2\,\wt\Tsc{}^*X\otimes S^2\,\wt\Tsc{}^*X &= (d t^2 \otimes S^2\,\wt\Tsc{}^*X) \oplus \bigl((2 d t\otimes_s\Tsc^*X) \otimes S^2\,\wt\Tsc{}^*X\bigr) \\
      &\hspace{14em} \oplus (S^2\,\Tsc^*X \otimes S^2\,\wt\Tsc{}^*X),
  \end{split}
  \end{align}
  \end{subequations}
  with the $S^2\,\wt\Tsc{}^*X$ factors in~\eqref{EqKStBox2Spl4}--\eqref{EqKStBox2Spl5} further split according to~\eqref{EqKStBox2Spl2}. Thus, writing $\cO^k:=\rho^k\CI(X)$ (and writing $f=\cO^k$ in the spirit of $\cO$-notation instead of $f\in\cO^k$ when $f$ is a smooth function), we have
  \begin{equation}
  \label{EqKStTrace}
    g = (1+\cO^1,\,\cO^2,\,-h+\cO^1 )^T, \quad
    \tr_g = (1+\cO^1,\,\cO^2,\,-\tr_h+\cO^1 )
  \end{equation}
  in the splitting~\eqref{EqKStBox2Spl2} and its dual. The dual metric takes the form $G:=g^{-1}=(1+\cO^1,\cO^2,-h^{-1}+\cO^1)$.

  For subsequent calculations, let us introduce coordinates
  \[
    z=(z^0,z^1,z^2,z^3),
  \]
  where $z^0=t$, and $z^1,z^2,z^3$ are standard coordinates on $\R^3$. We use Latin letters $i,j,k$ for indices from $1$ to $3$, and Greek letters $\mu,\nu,\lambda$ for indices from $0$ to $3$. We first compute the Levi-Civita connection of $g$; for example, $2\Gamma_{0 i j}=\pa_i g_{0 j}+\pa_j g_{0 i}=\cO^3$ since $g$ is stationary, and since $\pa_i\in\rho\Vb(X)$ maps $g_{0 j}\in\rho^2\CI(X)$ into $\rho^3\CI(X)$. The same reasoning shows that all Christoffel symbols of $g$ lie in $\cO^2$ except for
  \begin{equation}
  \label{EqKStGamma}
    \Gamma^k_{i j}=\Gamma^k_{i j}(h)+\cO^2;
  \end{equation}
  by an explicit calculation, some Christoffel symbols have faster decay. We collect all of them by stating the form of $\nabla\in\Diff^1(M^\circ;T^*M^\circ,T^*M^\circ\otimes T^*M^\circ)$, defined by $u\mapsto\nabla u$, $(\nabla u)(V,W)=(\nabla_V u)(W)$: the $d z^\mu\otimes d z^\nu$ coefficient of $\nabla(u\,d z^\lambda)$, with $u$ a scalar function, is given by
  \[
    (\nabla_\mu(u\,d z^\lambda))_\nu = \bigl(\pa_\mu\delta_\nu^\lambda - \Gamma_{\mu\nu}^\lambda\bigr)u.
  \]
  In the splittings~\eqref{EqKStBox2Spl1} and \eqref{EqKStBox2Spl3}, we then have
  \begin{equation}
  \label{EqKStNablaT}
    \nabla^{\wt\Tsc{}^*X}=
    \begin{pmatrix}
      \pa_t+\cO^4 & \cO^2 \\
      \cO^2 & \pa_t+\cO^3 \\
      d_X+\cO^2 & \cO^3 \\
      \cO^3 & \nabla^h+\cO^2
    \end{pmatrix},
  \end{equation}
  where $d_X$ is the exterior differential on $X$, and $\nabla^h$ is the Levi-Civita connection of $h$. Note that $d_X\in\rho\Diffb^1(X;\ul\C,\Tsc^*X)$ (with $\ul\C:=X\times\C\to X$ the trivial bundle) and $\nabla^h\in\rho\Diffb^1(X;\Tsc^*X,\Tsc^*X\otimes\Tsc^*X)$.

  Using this, one computes $\nabla$ acting on symmetric 2-tensors, expressed in the splittings~\eqref{EqKStBox2Spl2} and \eqref{EqKStBox2Spl4} (the latter refined by~\eqref{EqKStBox2Spl2} as explained before), to be
  \begin{equation}
  \label{EqKStNablaS2T}
    \nabla^{S^2\,\wt\Tsc{}^*X}=
    \begin{pmatrix}
      \pa_t+\cO^4 & \cO^2 & 0 \\
      \cO^2 & \pa_t+\cO^3 & \cO^2 \\
      0 & \cO^2 & \pa_t+\cO^3 \\
      d_X+\cO^2 & \cO^3 & 0 \\
      \cO^3 & \nabla^h+\cO^2 & \cO^3 \\
      0 & \cO^3 & \nabla^h+\cO^2
    \end{pmatrix}.
  \end{equation}
  Now $\delta_g=-\tr_g^{1 3}\nabla$ on symmetric 2-tensors, where $\tr_g^{1 3}$ is the contraction of the first and third slot. In the splittings~\eqref{EqKStBox2Spl4} and \eqref{EqKStBox2Spl1}, we have
  \[
    \tr_g^{1 3}=
    \begin{pmatrix}
      1+\cO^1 & \cO^2 & 0 & \cO^2 & -\mathrm{tr}_h+\cO^1 & 0 \\
      0 & 1+\cO^1 & \cO^2 & 0 & \cO^2 & -\mathrm{tr}_h^{1 3}+\cO^1
    \end{pmatrix},
  \]
  hence
  \begin{equation}
  \label{EqVKDiv}
    \delta_g=
    \begin{pmatrix}
      -\pa_t+\rho^3\Diffb^1+\cO^1\pa_t & -\delta_h+\rho^2\Diffb^1+\cO^2\pa_t & \cO^3 \\
      \cO^2 & -\pa_t+\rho^3\Diffb^1+\cO^1\pa_t & -\delta_h+\rho^2\Diffb^1+\cO^2\pa_t
    \end{pmatrix}.
  \end{equation}

  Symmetrizing $\nabla\colon\CI(M^\circ;T^*M^\circ\otimes S^2 T^*M^\circ)\to\CI(M^\circ;T^*M^\circ\otimes T^*M^\circ\otimes S^2 T^*M^\circ)$ in the first two factors on the right and expressing the resulting symmetrized gradient in the splittings~\eqref{EqKStBox2Spl4} and \eqref{EqKStBox2Spl5} (both refined by~\eqref{EqKStBox2Spl2}), one furthermore computes
  \begin{equation}
  \label{EqVKNablaTS2T}
  \begin{split}
    &\mathrm{Sym}\nabla^{\wt\Tsc{}^*X\otimes S^2\wt\Tsc{}^*X} \\
    &\quad=
    \openbigpmatrix{2pt}
      {\pa_t+}\cO^4 & \cO^2 & 0 & \cO^2 & 0 & 0 \\
      \cO^2 & {\pa_t+}\cO^3 & \cO^2 & 0 & \cO^2 & 0 \\
      0 & \cO^2 & {\pa_t+}\cO^3 & 0 & 0 & \cO^2 \\
      {\half d_X+}\cO^2 & \cO^3 & 0 & {\half\pa_t+}\cO^3 & \cO^2 & 0 \\
      \cO^3 & {\half\nabla^h+}\cO^2 & \cO^3 & \cO^2 & {\half\pa_t+}\cO^3 & \cO^2 \\
      0 & \cO^3 & {\half\nabla^h+}\cO^2 & 0 & \cO^2 & {\half\pa_t+}\cO^3 \\
      \cO^3 & 0 & 0 & {\nabla^h+}\cO^2 & \cO^3 & 0 \\
      0 & \cO^3 & 0 & \cO^3 & {\nabla^h+}\cO^2 & \cO^3 \\
      0 & 0 & \cO^3 & 0 & \cO^3 & {\nabla^h+}\cO^2
    \closebigpmatrix,
  \end{split}
  \end{equation}
  where $\nabla^h$ in the $(8,5)$ entry acts on $\Tsc^*X\otimes\Tsc^*X$, which is isomorphic to the $5$-th summand of~\eqref{EqKStBox2Spl4} when refined by~\eqref{EqKStBox2Spl2}, i.e.\ $\Tsc^*X\otimes(2 d t\otimes_s\Tsc^*X)$; the meaning of the other $\nabla^h$ is analogous.
  
  The trace in the first factor on the left of~\eqref{EqKStBox2Spl5}, expressed in the splittings~\eqref{EqKStBox2Spl5} and \eqref{EqKStBox2Spl2}, is given by
  \[
    \tr_g^{1 2}=
      \openbigpmatrix{4pt}
        {1+}\cO^1 & 0 & 0 & \cO^2 & 0 & 0 & \!\!{-}{\mathrm{tr}_h+}\cO^1\!\! & 0 & 0 \\
        0 & {1+}\cO^1 & 0 & 0 & \cO^2 & 0 & 0 & \!\!{{-}\mathrm{tr}_h^{1 2}+}\cO^1 & 0 \\
        0 & 0 & {1+}\cO^1 & 0 & 0 & \cO^2 & 0 & 0 & \!\!{{-}\mathrm{tr}_h^{1 2}+}\cO^1
      \closebigpmatrix.
  \]
  Combining this with~\eqref{EqKStNablaS2T} and \eqref{EqVKNablaTS2T}, one finds that on symmetric 2-tensors,
  \begin{equation}
  \label{EqKStBox}
  \begin{split}
    \Box_g = -\tr_g^{1 2}\nabla\nabla
    &=(1+\cO^1)D_t^2
     +
     \begin{pmatrix}
       \tr_h\nabla^h d_X & 0 & 0 \\
       0 & \tr_h^{1 2}\nabla^h\nabla^h & 0 \\
       0 & 0 & \tr_h^{1 2}\nabla^h\nabla^h
     \end{pmatrix} \\
    &\qquad + \begin{pmatrix} 0 & \cO^2 D_t & 0 \\ \cO^2 D_t & 0 & \cO^2 D_t \\ 0 & \cO^2 D_t & 0 \end{pmatrix} + \rho^3\Diffb^1 D_t + \rho^3\Diffb^2;
  \end{split}
  \end{equation}
  The third and fourth summand lie in $\rho^2\Diffsc^1 D_t$. Lastly, the coefficient of $D_t^2$ can be computed from the principal symbol of $\Box_g$, hence is $|d t|_G^2$ as stated. The proof is complete.
\end{proof}

By~\eqref{EqKStGamma} and the discussion preceding it, the Riemann curvature tensor
\begin{equation}
\label{EqKStRiemann}
  R_g\in\rho^2\CI\bigl(X;\wt\Tsc{}X \otimes (\wt\Tsc{}^*X)^3\bigr)
\end{equation}
is determined, modulo $\rho^3\CI$, by $h$; likewise for the Ricci curvature $\Ric(g)\in\rho^2\CI$. We similarly obtain that the scalar and 1-form wave operators for metrics of the form~\eqref{EqKStBox2Struct} are given by 
\begin{equation}
\label{EqVKOps}
  \Box_{g,j} = |d t_{\chi_0}|_G^2 D_{t_{\chi_0}}^2 + \wh{\Box_{g,j}}(0) + Q_j D_{t_{\chi_0}},\quad j=0,1,
\end{equation}
where $Q_0\in\rho^3\Diffb^1(X)$ and $Q_1\in\rho^2\Diffsc^1(X;\wt\Tsc{}^*X)$, and where $\wh{\Box_{g,j}}(0)\bmod\rho^3\Diffb^2(X)$ only depends on $h$. Indeed, for $j=0$, \eqref{EqKStTrace} and \eqref{EqKStNablaT} imply
\[
  \delta_g=-\tr_g\nabla=(-\pa_t+\rho^3\Diffb^1+\cO^1\pa_t,\,-\delta_h+\rho^2\Diffb^1+\cO^2\pa_t),
\]
which together with $d=(\pa_t,\,d_X)^T$ gives $\Box_{g,0}=\delta_g d=G^{0 0}D_t^2-\Delta_h + \rho^3\Diffb^2+\rho^3\Diffb^1 D_t$. The proof for $j=1$ follows from a simple variant of the calculations of Lemma~\ref{LemmaKStBox2}.

We furthermore note that $\sfG_g\in\CI(X;\End(S^2\,\wt\Tsc{}^*X))$, with $\sfG_g\bmod\rho\CI$ depending only on $h$. Moreover,
\begin{equation}
\label{EqKStDivSymmGrad}
\begin{aligned}
  \delta_g &= -\iota_{d t_{\chi_0}^\sharp}D_{t_{\chi_0}} + \wh{\delta_g}(0), &\quad&
    \wh{\delta_g}(0) \in \rho\Diffb^1(X;S^2\,\wt\Tsc{}^*X,\wt\Tsc{}^*X), \\
  \delta_g^* &= d t_{\chi_0} \otimes_s D_{t_{\chi_0}} + \wh{\delta_g^*}(0), &&
    \wh{\delta_g^*}(0) \in \rho\Diffb^1(X;\wt\Tsc{}^*X,S^2\,\wt\Tsc{}^*X),
\end{aligned}
\end{equation}
with $\wh{\delta_g}(0),\wh{\delta_g^*}(0)\bmod\rho^2\Diffb^1$ only depending on $h$.

In the context of~\eqref{EqKStMetrics2}, it is useful to record the following strengthening of the leading order control: if $g_1,g_2$ are two metrics of the form~\eqref{EqKStBox2Struct} and so that in addition $g_1-g_2\in\rho^2\CI(X;S^2\,\wt\Tsc{}^*X)$, then
\begin{equation}
\label{EqKStSymmGradLeading}
  \wh{\delta_{g_1}^*}(0)-\wh{\delta_{g_2}^*}(0) \in \rho^3\CI(X;\Hom(\wt\Tsc{}^*X,S^2\,\wt\Tsc{}^*X)),
\end{equation}
similarly for other operators, including
\begin{equation}
\label{EqKStBoxLeading}
  \wh{\Box_{g_1,j}}(0)-\wh{\Box_{g_2,j}}(0) \in \rho^4\Diffb^2(X;\wt\Tsc{}^*X),\quad j=0,1,2,
\end{equation}
and $R_{g_1}-R_{g_2}\in\rho^4\CI$.

When $g$ is the Kerr metric, then $g|_{\Tsc X\times\Tsc X}=-h+\rho\CI$ is to leading order equal to the Euclidean metric $h=(d x^1)^2+(d x^2)^2+(d x^3)^2$ on $\R^3$ (equipped with standard coordinates $(x^1,x^2,x^3)$ on $\R^3\setminus B(0,3\bhm)\cong X^\circ\setminus\{r<3\bhm\}$). Thus, the leading order terms at $\rho=0$ are simply those of the corresponding operators on Minkowski space $\R^4=\R_t\times\R^3_x$ with metric
\begin{equation}
\label{EqKStMink}
  \ubar g = d t^2-d x^2.
\end{equation}
But the latter take a very simple form in the standard coordinate trivialization of $\wt\Tsc{}^*X$ by $d t$, $d x^i$, $i=1,2,3$:

\begin{lemma}
\label{LemmaKStNormal}
  Let $N_0=1$, $N_1=4$, $N_2=10$. For $g=g_{(\bhm,\bha)}$, we have
  \[
    \wh{\Box_{g,j}}(0) - \wh{\Box_{\ubar g,j}}(0) \in \rho^3\Diffb^2,
  \]
  where $\Box_{\ubar g}$ is the scalar wave operator on Minkowski space, given by $\Box_{\ubar g,j}=\Box_{\ubar g}\otimes 1_{N_j\times N_j}$ in the standard coordinate basis. Likewise,
  \[
    \delta_g^* - \delta_{\ubar g}^* \in \rho^2\Diffb^1, \quad
    \delta_{g_{(\bhm,\bha)}}^*-\delta_{g_{(\bhm,0)}}^* \in \rho^3\Diffb^1.
  \]
\end{lemma}
\begin{proof}
  By Lemma~\ref{LemmaKStBox2}, the only structure of $g_{(\bhm,\bha)}$ relevant for this calculation is~\eqref{EqKStBox2Struct}, with $h$ the Euclidean metric on $X=\ol{\R^3}$. That is, it suffices to compute $\Box_{\ubar g,j}$, where $\ubar g=d t^2-h$ is the Minkowski metric. The claim is then immediate since $d t$, $d x^i$ are parallel for $\ubar g$. The final statement follows from~\eqref{EqKStSymmGradLeading} combined with~\eqref{EqKStMetrics2}.
\end{proof}

In the language of~\cite{MelroseAPS}, the \emph{normal operators} of $\wh{\Box_{g,j}}(0)$ and $\wh{\Box_{\ubar g,j}}(0)$ at $\pa_+ X$ are the same; this will allow us to deduce precise asymptotic expansions of zero energy modes for waves on Kerr spacetimes from simple calculations on Minkowski in~\S\S\ref{S0}--\ref{SL}.

%%%%%%%%%%%%%%%%%%%%%%%%%%%%%%%%%%%%%%%%%%%%%%%%%%
\subsection{Properties of the null-geodesic flow}
\label{SsKFl}

Fix black hole parameters $b=(\bhm,\bha)$ close to $b_0$, and let
\[
  g=g_{(\bhm,\bha)},\quad
  G=g^{-1}.
\]
We proceed to describe the null-bicharacteristic flow of the spectral family of the wave operator $\Box_g$ on Schwarzschild or slowly rotating Kerr spacetimes. (Since this concerns only the principal symbol of $\Box_g$, which is the dual metric function, this discussion applies to the scalar, 1-form, and symmetric 2-tensor wave operators, as well as to the linearized gauge-fixed Einstein operator $L_g$ in~\eqref{EqOpLinOp} below.) Concretely, recalling the function $t_{\bhm,*}$ from~\eqref{EqKTimeFn}, we are interested in the null-bicharacteristic flow of
\begin{equation}
\label{EqKFl}
  \wh{\Box_g}(\sigma) := e^{i t_{\bhm,*}\sigma}\Box_g e^{-i t_{\bhm,*}\sigma} \in \Diffsc^2(X),
\end{equation}
both for finite $\sigma$ as well as in the semiclassical regime, see~\eqref{EqKFlScl}.

The principal symbol of $-\wh{\Box_g}(\sigma)$ as a large parameter (in $\sigma$) differential operator is
\[
  p(\sigma;\xi) := -G(-\sigma\,d t_{\bhm,*}+\xi) \in \CI(\Tsc^*X\times\C_\sigma),\quad \xi\in \pa_{t_{\bhm,*}}^\perp = \Tsc^*X;
\]
the overall minus sign ensures that $p(\sigma;-)$ has positive principal symbol for $\sigma\in\R$ and large $r$.

\begin{rmk}
  One typically considers the spectral family of $\Box_g$ with respect to another `time' function such as $t_{\chi_0}$, which equals $t$ for large $r$. (Formally taking $b=(0,0)$, so $g_b=d t^2-d r^2-r^2\slg=\ubar g$ is the Minkowski metric, the spectral family with respect to $t$ is $-\Delta+\sigma^2$, with $\Delta\geq 0$ the Euclidean Laplacian.) For future reference, let us thus define
  \[
    \wt{\Box_g}(\sigma) := e^{i t_{\chi_0}\sigma}\Box_g e^{-i t_{\chi_0}\sigma};
  \]
  this can be obtained from $\wh{\Box_g}(\sigma)$ by conjugation by the stationary function $e^{i(t_{\bhm,*}-t_{\chi_0})\sigma}$.
\end{rmk}

The semiclassical rescaling of $-\wh{\Box_g}(\sigma)$ is $-h^2\wh{\Box_g}(h^{-1}z)$ where $h=|\sigma|^{-1}$, $z=\sigma/|\sigma|$, and its semiclassical principal symbol is
\begin{equation}
\label{EqKFlScl}
  p_\semi(\xi) := -G(-z\,d t_{\bhm,*}+\xi) \in \CI(\Tsc^*X).
\end{equation}
For the sake of definiteness, we consider the case $\sigma>0$, so $z=1$. We then define
\[
  \Sigma_\semi \subset \ol{\Tsc^*}X
\]
as the closure of $p_\semi^{-1}(0)$ in the (fiber-wise) radially compactified scattering cotangent bundle. Note that on the set where $\pa_{t_{\bhm,*}}=\pa_t$ is timelike, in particular for large $r$, $p_\semi(\xi)$ is classically elliptic since $p_\semi(\xi)\gtrsim|\xi|^2$ (with $|\cdot|$ denoting the Euclidean metric) for large $|\xi|$. We shall consider the rescaled Hamilton flow of $p_\semi$, namely, the flow of the vector field
\[
  \sfH := |\xi|^{-1}\rho^{-1}H_{p_\semi}\in\Vb(\ol{\Tsc^*}X).
\]

The structure of the $\sfH$-flow on subextremal Kerr spacetimes has been described in detail before: by the third author \cite[\S6]{VasyMicroKerrdS} on Kerr--de~Sitter spacetimes, which are very similar to Kerr spacetimes except for the presence of a cosmological horizon; by Dyatlov~\cite[\S\S3.1--3.3]{DyatlovWaveAsymptotics} on subextremal Kerr spacetimes, building on the work by Wunsch--Zworski~\cite{WunschZworskiNormHypResolvent}, and with refinements in the presence of bundles due to the Dyatlov~\cite{DyatlovSpectralGaps} and the second and third authors \cite{HintzPsdoInner,HintzVasyKdSStability,HintzKNdSStability,HintzPolyTrap}; see also \cite{DyatlovZworskiTrapping}. (We refer the reader to \cite{ZworskiResonanceReview} for a survey of trapping phenomena.) Vasy--Zworski \cite{VasyZworskiScl} analyzed the semiclassical scattering behavior near $\pa X$ for the null-bicharacteristic flow of the semiclassical principal symbol of $-\sigma^{-2}\wt{\Box_g}(\sigma)$, which is the same, up to a canonical transformation, as that of $p_\semi$. In the non-semiclassical setting, i.e.\ for fixed $\sigma$, the description of the Hamilton dynamics within the characteristic set of $\pa X$ is due to Melrose~\cite{MelroseEuclideanSpectralTheory}; the remaining part of the characteristic set in this setting lies over $r\leq 2\bhm$ on a mass $\bhm$ Schwarzschild spacetime, and in a neighborhood thereof for slowly rotating Kerr spacetimes.

Here, we shall thus merely list (without proof) the relevant properties of the flow. To begin, over $\pa_+X$, we have $G(-d t_{\bhm,*}+\xi)=|{-}d t+(d r+\xi)|_{\ubar G}^2=1-|d r+\xi|^2$, where $\ubar G=\ubar g^{-1}$ is the dual of the Minkowski metric, hence
\[
  \Sigma_{\semi,\pa} := \Sigma_\semi \cap \Tsc^*_{\pa X}X = \{ \zeta-d r\colon |\zeta|^2=1 \}.
\]
Note that the semiclassical characteristic set of $-\sigma^{-2}\wt{\Box_g}(\sigma)$ over $\pa X$ is the zero set of $G(-d t+\xi)$, i.e.\ equal to $\Sigma_{\semi,\pa}+d r$. More generally, conjugation by $e^{i(t_{\bhm,*}-t_{\chi_0})\sigma}$ (is a semiclassical scattering FIO which) shifts the characteristic set by $d(t_{\bhm,*}-t_{\chi_0})$, but preserves the qualitative properties of the null-bicharacteristic flow discussed here.

First of all, $\Sigma_{\semi,\pa}$ has two distinguished submanifolds (of radial points),
\begin{equation}
\label{EqKFlRadInfty}
  \cR_{\rm in} = \{ -2 d r \in \Tsc^*_p X,\ p\in\pa X \},\quad
  \cR_{\rm out} = o_{\pa X} \subset \Tsc^*_{\pa X}X,
\end{equation}
with $o_{\pa X}$ denoting the zero section; these are critical manifolds for the rescaled Hamilton vector field $\rho^{-1}H_{p_\semi}$ within $\Sigma_{\semi,\pa}$ (see \cite{MelroseEuclideanSpectralTheory}) and in fact within all of $\Sigma_\semi$ (see \cite{VasyZworskiScl}), with $\cR_{\rm in}$ being a source and $\cR_{\rm out}$ a sink. (Indeed, the linearization of $\rho^{-1}H_{p_\semi}$ at $\cR_{\rm out}$, resp.\ $\cR_{\rm in}$, is $-2(\rho\pa_\rho+\eta_\scop\pa_{\eta_\scop})$, resp.\ $2(\rho\pa_\rho+\eta_\scop\pa_{\eta_\scop})$, where we write scattering covectors as $\xi_r\,d r+r\eta_\scop$ with $\eta_\scop\in T^*\Sph^2$; this calculation uses that to leading order at $\pa X$ we have $p_\semi=(1+\xi_r)^2+|\eta_\scop|^2-1$.)

Next, globally, $\Sigma_\semi$ has two connected components,
\begin{equation}
\label{EqKFlSigmaComp}
  \Sigma_\semi = \Sigma_\semi^+ \cup \Sigma_\semi^-,
\end{equation}
with $\Sigma_{\semi,\pa}\subset\Sigma_\semi^+$; they are defined as the intersection of the future ($+$), resp.\ past ($-$) light cone in $\ol{\Tsc^*X}$ intersected with $\ol{-d t_{\bhm,*}+\Tsc^*X}$. (Over a point $p\in X$, $\Sigma_\semi^-$ is empty unless $\Tsc_p^*X$ is timelike or null, i.e.\ unless $\pa_t$ is spacelike or null, i.e.\ unless $p$ lies in the ergoregion, on the event horizon, or inside the black hole.)

We recall that fiber infinity of the conormal bundle of the event horizon $\pa_-\cX$, see~\eqref{EqK0StaticBdy}, has two components
\[
  \pa\bigl(\ol{N^*\pa_-\cX}\bigr) = \cR_\semi^+ \cup \cR_\semi^-,\quad \cR_\semi^\pm\subset\Sigma_\semi^\pm,
\]
which are invariant under the $\sfH$-flow.

Finally, recall that there is a \emph{trapped set} $\Gamma\subset\Sigma_\semi^+\cap T^*X^\circ$ consisting of all $\alpha\in T^*X^\circ$ such that $r$ remains in a compact subset of $(r_{(\bhm,\bha)},\infty)$ along the $\sfH$-integral curve with initial condition $\alpha$. The trapped set is $\sfr$-normally hyperbolic for every $\sfr\in\R$ \cite{HirschPughShubInvariantManifolds} as proved in~\cite{WunschZworskiNormHypResolvent,DyatlovWaveAsymptotics}.

Recalling the definition of the final hypersurface $\Sigma_{\rm fin}$ from \eqref{EqK0SurfFin}, the global structure of the $\sfH$-flow is then as follows:
\begin{prop}
\label{PropKFl}
  Let $s\mapsto\gamma(s)\subset\Sigma_\semi^\pm$ be a maximally extended integral curve of $\pm\sfH$ with domain of definition $I\subseteq\R$; let $s_-=\inf I$, $s_+=\sup I$.
  \begin{enumerate}
  \item If $\gamma\subset\Sigma_\semi^-$, then either $\gamma\subset\cR_\semi^-$; or $\gamma(s)\to\cR_\semi^-$ as $s\to s_-$, and $\gamma(s)$ crosses $\Sigma_{\rm fin}$ into the inward direction (decreasing $r$) in finite time $s_+<\infty$.
  \item If $\gamma\subset\Sigma_\semi^+$, then either:
    \begin{enumerate}
    \item $\gamma\subset\Gamma\cup\cR_\semi^+\cup\cR_{\rm in}\cup\cR_{\rm out}$; or
    \item as $s\to s_-$, $\gamma(s)$ tends to $\cR_\semi^+\cup\cR_{\rm in}\cup\Gamma$, and as $s\to s_+$, $\gamma(s)$ tends to $\cR_{\rm out}\cup\Gamma$ or crosses $\Sigma_{\rm fin}$ into the inward direction (decreasing $r$) in finite time $s_+<\infty$. Moreover, $\gamma(s)$ cannot tend to $\Gamma$ in both the forward and backward direction.
    \end{enumerate}
  \end{enumerate}
\end{prop}

Next, we discuss the properties of the null-bicharacteristic flow of $\wh{\Box_g}(\sigma)\in\Diffsc^2(X)$ when $\sigma\in\R$ is \emph{fixed}; concretely, we consider $\sfH:=|\xi|^{-1}\rho^{-1}H_{p(\sigma;-)}$. Since $\wh{\Box_g}(\sigma)$ is a scattering differential operator, its characteristic set is a subset of $\pa\ol{\Tsc^*X}=\ol{\Tsc^*_{\pa X}}X\sqcup S^*X^\circ$, and indeed is a disjoint union
\[
  \Sigma_\sigma = \Sigma_{\sigma,\pa} \cup \Sigma_{\sigma,\infty},
  \qquad \Sigma_{\sigma,\pa} = \{ \zeta-\sigma\,d r \colon |\zeta|^2=\sigma^2 \}, \quad
  \Sigma_{\sigma,\infty} = \Sigma_\sigma \cap S^*X^\circ.
\]
Note here that on $X^\circ$, $\wh{\Box_g}(\sigma)$ is elliptic where $\pa_{t_{\bhm,*}}=\pa_t$ is timelike, which is in particular true for large $r$; the component $\Sigma_{\sigma,\infty}$ lies over $r\leq 2\bhm$ for $g=g_{(\bhm,0)}$, and in a neighborhood thereof for small angular momenta. In fact, $\Sigma_{\sigma,\infty}$ is the boundary at fiber infinity of $\Sigma_\semi$, and as such has two connected components $\Sigma_{\sigma,\infty}^\pm$ as a consequence of~\eqref{EqKFlSigmaComp}. Fiber infinity of the conormal bundle of the event horizon,
\[
  \cR^\pm:=\cR_\semi^\pm\subset S^*X^\circ,
\]
is again an invariant submanifold of $\sfH$. The analogue of Proposition~\ref{PropKFl} is that maximally extended $\pm\sfH$-integral curves inside of $\Sigma_{\sigma,\infty}$ tend to $\cR^\pm$ in one direction and escape through $\Sigma_{\rm fin}$ in the other; integral curves in $\Sigma_{\sigma,\pa}$ on the other hand tend to
\[
  \cR_{\sigma,\rm in}=\{-2\sigma\,d r\},\quad \text{resp.}\quad \cR_{\sigma,\rm out}=o_{\pa X}
\]
in the backward, resp.\ forward direction.

In the case $\sigma=0$, $p(0,-)$ vanishes quadratically at $o_{\pa X}\subset\Tsc^*_{\pa X}X$, and in fact $\wh{\Box_g}(0)\in\rho^2\Diffb^2(X)$ by Proposition~\ref{PropKStBox2} and equation~\eqref{EqVKOps}. (The degeneracy of $\Sigma_{\sigma,\pa}$ as $\sigma\to 0$ can be resolved by working on a resolution of a parameterized version $[0,1)_\sigma\times\ol{\Tb^*X}$ of phase space, see \cite{VasyLowEnergyLag}.) Away from $\pa X$ on the other hand, thus in $\Sigma_{0,\infty}$, the characteristic set and null-bicharacteristic flow remain non-degenerate, i.e.\ have the same structure as for non-zero real $\sigma$.

%%%%%%%%%%%%%%%%%%%%%%%%%%%%%%%%%%%%%%%%%%%%%%%%%%%%%%%%%%%%%%%%%%%%%%
\section{The gauge-fixed Einstein operator}
\label{SOp}

We now commence the study of the Einstein equation in a wave map (or DeTurck \cite{DeTurckPrescribedRicci}, or generalized harmonic coordinate \cite{ChoquetBruhatLocalEinstein,FriedrichHyperbolicityEinstein}) gauge. We shall deduce a significant amount of information from the structural and dynamical properties of Kerr metrics discussed in~\S\S\ref{SsKSt}--\ref{SsKFl}; only once we turn to obtaining very precise (spectral) information in subsequent sections do we need to use their exact form.

%%%%%%%%%%%%%%%%%%%%%%%%%%%%%%%%%%%%%%%%%%%%%%%%%%
\subsection{The unmodified gauge-fixed Einstein operator}

We first study the gauge-fixed operator arising from a natural wave map gauge which we already used in the Kerr--de~Sitter setting in \cite{HintzVasyKdSStability}, following \cite{GrahamLeeConformalEinstein}:

\begin{definition}
\label{DefOpGauge}
  Given two pseudo-Riemannian metrics $g,g^0$ on $M^\circ$, we denote the \emph{gauge 1-form} by
  \[
    \Ups(g;g^0) := g(g^0)^{-1}\delta_g\sfG_g g^0.
  \]
  In local coordinates, $\Ups(g;g^0)_\mu=g_{\mu\nu}(g^{-1})^{\kappa\lambda}(\Gamma(g)^\nu_{\kappa\lambda}-\Gamma(g^0)^\nu_{\kappa\lambda})$.
\end{definition}

Fixing $g^0$, the \emph{(unmodified) gauge-fixed Einstein operator} is then the map
\begin{equation}
\label{EqOpEin}
  g \mapsto P(g) := \Ric(g) - \delta_g^*\Ups(g;g^0).
\end{equation}
(Modifications are discussed in~\S\ref{SsOpCD}.) For any choice of $g^0$, the equation $P(g)=0$ is a quasilinear wave equation for $g$ with Lorentzian signature. Fix Kerr black hole parameters
\[
  b=(\bhm,\bha),
\]
and let $g^0=g_b$. The linearization of $2 P$ around $g=g_b$ is then given by the operator
\begin{equation}
\label{EqOpLinOp}
\begin{split}
  &L_{g_b} := 2 D_{g_b} P = 2(D_{g_b}\Ric + \delta_{g_b}^*\delta_{g_b}\sfG_{g_b}) = \Box_{g_b,2} + 2\sR_{g_b}, \quad (\sR_{g_b} u)_{\mu\nu} := (R_{g_b})^\kappa{}_{\mu\nu\lambda}u_\kappa{}^\lambda,
\end{split}
\end{equation}
see \cite{GrahamLeeConformalEinstein}. We call $L_{g_b}$ the \emph{linearized (unmodified) gauge-fixed Einstein operator}. By Proposition~\ref{PropKStBox2} and the membership~\eqref{EqKStRiemann}, and writing $\rho=r^{-1}$, we have
\begin{equation}
\label{EqOpLinOpSt}
\begin{split}
  &L_{g_b} = |d t_{\chi_0}|_G^2 D_{t_{\chi_0}}^2 + \wh{L_{g_b}}(0) + Q D_{t_{\chi_0}}, \\
  &\qquad \wh{L_{g_b}}(0)\in\rho^2\Diffb^2(X;S^2\,\wt\Tsc{}^*X), \quad
  Q \in \rho^2\Diffsc^1(X;S^2\,\wt\Tsc{}^*X).
\end{split}
\end{equation}

As motivated in~\S\ref{SI}, we consider the spectral family of $L_{g_b}$ with respect to the function $t_{\bhm,*}$ from~\eqref{EqKTimeFn}; the level sets of $t_{\bhm,*}$ are approximately null in the sense that
\begin{equation}
\label{EqOpLinTimeFn}
  |d t_{\bhm,*}|^2_{G_b} \in \rho^2\CI(X).
\end{equation}
(On the other hand, we only have $|d t_{\bhm',*}|^2_{G_b}\in\rho\CI(X)$ when $\bhm'\neq\bhm$.) Thus, let
\begin{equation}
\label{EqOpLinFT}
  \wh{L_{g_b}}(\sigma) := e^{i\sigma t_{\bhm,*}}L_{g_b}e^{-i\sigma t_{\bhm,*}} \in \Diff^2(X;S^2\,\wt{\Tsc^*}X).
\end{equation}
We prove that this fits into the framework of~\cite{VasyLAPLag,VasyLowEnergyLag}. We work in the collar neighborhood $[0,(3\bhm)^{-1})_\rho\times\Sph^2$ of $\pa X$.

\begin{lemma}
\label{LemmaOpLinFT}
  Let $\rho=r^{-1}$. The operator $\wh{L_{g_b}}(\sigma)$ has the form
  \begin{equation}
  \label{EqOpLinFT2}
    \wh{L_{g_b}}(\sigma) = 2\sigma\rho(\rho D_\rho+i) + \wh{L_{g_b}}(0) + \sigma Q' + \alpha\sigma^2,
  \end{equation}
  where $Q'\in\rho^2\Diffsc^1(X;S^2\,\wt{\Tsc^*}X)$, $\alpha\in\rho^2\CI(X)$, and $\wh{L_{g_b}}(0)\in\rho^2\Diffb^2(X;S^2\,\wt{\Tsc^*}X)$.
\end{lemma}

As a concrete illustration, we note that on Minkowski space with metric $\ubar g=d t^2-d r^2-r^2\slg$, and taking the Fourier transform in $t-r$, we have
\[
  \wh{L_{\ubar g}}(\sigma) = \wh{\Box_{\ubar g,2}}(\sigma)=2\sigma\rho(\rho D_\rho+i) + \bigl( -(\rho^2 D_{\rho})^2+2 i\rho^3 D_\rho-\rho^2\slDelta \bigr).
\]
We also remark that the form of the first term in~\eqref{EqOpLinFT2} is consistent with~\cite[Lemma~3.8]{HintzVasyMink4} (taking $\gamma=h=0$ in the reference); see also the proof of Lemma~\ref{LemmaPfIVP}.

\begin{proof}[Proof of Lemma~\ref{LemmaOpLinFT}]
  Changing from $t_{\chi_0},r$ coordinates in~\eqref{EqOpLinOpSt} to $t_{\bhm,*},r$ coordinates, with $t_{\bhm,*}\equiv t_{\chi_0}-(r+2\bhm\log(r-2\bhm))$ modulo bounded smooth functions in $r\geq 3\bhm$, transforms $\pa_{t_{\chi_0}}$, $\pa_r$ into $\pa_{t_{\bhm,*}}$, $\pa_r-((1-\tfrac{2\bhm}{r})^{-1}+\rho^2\CI)\pa_{t_{\bhm,*}}$.
  
  Thus, since $Q$ is a sum of terms of the form $\rho^2\CI$, $\rho^2 D_r$, $\rho^3\cV(\Sph^2)$, the spectral family of $Q D_{t_{\chi_0}}$ is of the form $\sigma\rho^2\Diffsc^1+\sigma^2\rho^2\CI$. In a similar vein, the $t_{\bhm,*}$-spectral family of an operator in $\rho^4\Diffb^2(X;S^2\,\wt{\Tsc^*}X)$, extended by stationarity using the time function $t_{\chi_0}$, lies in the space $\rho^4\Diffb^2+\sigma\rho^2\Diffsc^1+\sigma^2\rho^2\CI$. Terms in $\rho^2\Diff^2(\Sph^2)$ in $\wh{L_{g_b}}(0)$ in~\eqref{EqOpLinOpSt} are unaffected upon changing coordinates. Cross terms involving $\pa_r$ and $\cV(\Sph^2)$ only arise for $b=(\bhm,\bha)$ with $\bha\neq 0$ and hence contribute at the level of, schematically, $r^{-2}\pa_r\circ r^{-1}\cV(\Sph^2)=\rho^4(r\pa_r)\circ\cV(\Sph^2)\subset\rho^4\Diffb^2(X)$, see~\eqref{EqKStMetrics2}; hence upon changing coordinates, they can be subsumed in the final three terms in~\eqref{EqOpLinFT2}.

  It remains to consider the terms $(1-\frac{2\bhm}{r})^{-1}D_{t_{\chi_0}}^2+((1-\frac{2\bhm}{r})\pa_r^2+\frac{2}{r}\pa_r)$ coming from the first two summands in~\eqref{EqOpLinOpSt}; upon changing coordinates to $(r,t_{\bhm,*})$, the $D_{t_{\bhm,*}}^2$ terms cancel modulo $\rho^2\CI(M)D_{t_{\bhm,*}}^2$, which gives $\alpha$ upon passing to the spectral family. (The membership of $\alpha$ can also be directly deduced from principal symbol considerations: $\alpha\in\rho^2\CI(X)$ is equivalent to~\eqref{EqOpLinTimeFn}.) The only remaining unaccounted term is
  \[
    2(1-\tfrac{2\bhm}{r})\pa_r\circ(-(1-\tfrac{2\bhm}{r})\pa_{t_{\bhm,*}}) - 2 r^{-1}\pa_{t_*},
  \]
  whose spectral family is, modulo terms that can be subsumed in $Q',\alpha$, equal to $2\sigma\rho(\rho D_\rho+i)$. The proof is complete.
\end{proof}

We can now prove (omitting the vector bundle $S^2\,\wt{\Tsc^*}X$ from the notation for brevity):
\begin{thm}
\label{ThmOp}
   Let $b_0=(\bhm_0,0)$. There exists $\eps>0$ such that for $b\in\R^4$ with $|b-b_0|<\eps$, the following holds. Suppose that $s>\tfrac52$ and $\ell<-\half$ with $s+\ell>-\half$.
  \begin{enumerate}
  \item\label{ItOpLow} (Uniform estimates for finite $\sigma$.) For any fixed $C>1$, and $s_0<s$, $\ell_0<\ell$, there exists a constant $C'>0$ (independent of $b$) such that
    \begin{equation}
    \label{EqOpFinite}
      \|u\|_{\Hbext^{s,\ell}} \leq C'\Bigl(\bigl\|\wh{L_{g_b}}(\sigma)u\bigr\|_{\Hbext^{s-1,\ell+2}} + \|u\|_{\Hbext^{s_0,\ell_0}} \Bigr)
    \end{equation}
    for all $\sigma\in\C$, $\Im\sigma\in[0,C]$, satisfying $C^{-1}\leq |\sigma|\leq C$. If $\ell\in(-\tfrac32,-\half)$, then this estimate holds uniformly down to $\sigma=0$, i.e.\ for $|\sigma|\leq C$.
  \item\label{ItOpHigh} (High energy estimates in strips.) For any fixed $C>0$, there exist $C_1>1$ and $C'>0$ (independent of $b$) such that for $\sigma\in\C$, $\Im\sigma\in[0,C]$, $|\Re\sigma|>C_1$, and $h:=|\sigma|^{-1}$, we have
    \begin{equation}
    \label{EqOpHigh}
      \|u\|_{\Hbhext^{s,\ell}} \leq C'\bigl\|\wh{L_{g_b}}(\sigma)u\bigr\|_{\Hbhext^{s,\ell+1}}.
    \end{equation}
  \end{enumerate}
  Moreover, the operators
  \begin{subequations}
  \begin{alignat}{3}
  \label{EqOpFredS}
    \wh{L_{g_b}}(\sigma)&\colon\{u\in\Hbext^{s,\ell}(X)\colon\wh{L_{g_b}}(\sigma)u\in\Hbext^{s,\ell+1}(X)\}&&\to\Hbext^{s,\ell+1}(X), \quad \Im\sigma\geq 0,\ \sigma\neq 0, \\
  \label{EqOpFred0}
    \wh{L_{g_b}}(0) &\colon \{u\in\Hbext^{s,\ell}(X)\colon\wh{L_{g_b}}(0)u\in\Hbext^{s-1,\ell+2}(X)\}&&\to \Hbext^{s-1,\ell+2}(X)
  \end{alignat}
  \end{subequations}
  are Fredholm operators of index $0$.
\end{thm}

The fixed frequency version of~\eqref{EqOpHigh} reads
\[
  \|u\|_{\Hbext^{s,\ell}} \leq C'\left(\|\wh{L_{g_b}}(\sigma)u\bigr\|_{\Hbext^{s,\ell+1}} + \|u\|_{\Hbext^{s_0,\ell_0}}\right).
\]
The relationship between this and~\eqref{EqOpFinite} is obscured by our usage of imprecise function spaces: we refer the reader to \cite{VasyLAPLag,VasyLowEnergyLag} to the precise statements in terms of second microlocal scattering-b-Sobolev spaces, and only remark here that the above estimates are optimal as far as the b- (as well as scattering) decay orders are concerned.

\begin{proof}[Proof of Theorem~\ref{ThmOp}]
  The main ingredient which did not arise in spectral theory on Kerr--\emph{de~Sitter} spacetimes \cite{DyatlovQNM,VasyMicroKerrdS,HintzVasyKdSStability} and which goes beyond the settings discussed in \cite{MelroseEuclideanSpectralTheory,GuillarmouHassellResI,GuillarmouHassellResII,GuillarmouHassellSikoraResIII} is the Fredholm analysis near zero energy. Furthermore, we are working with $t_{\bhm,*}$ here, rather than the more usual $t_{\chi_0}$, which allows for concise proofs of stronger (in terms of function spaces) results. Concretely, we use the main results of \cite{VasyLAPLag} for the estimates at fixed $\sigma\neq 0$, $\Im\sigma\geq 0$, as well as for high energy estimates when $|\Re\sigma|\to\infty$ with $\Im\sigma\geq 0$ bounded, while the uniform description in $\Im\sigma\geq 0$ near $\sigma=0$ is provided by \cite{VasyLowEnergyLag}.

  Radial point estimates at $\cR_\semi^\pm=\cR^\pm$ require the computation of threshold regularities, which was done for Schwarzschild--de~Sitter metrics in~\cite{HintzVasyKdSStability}; the calculations there apply also in the case of Schwarzschild metrics for which the cosmological constant $\Lambda$ vanishes. In short, the subprincipal operator (see \cite{HintzPsdoInner}) at $\cR_\semi$ is computed in~\S9.2; \cite[Equation~(9.9]{HintzVasyKdSStability} (with the bottom sign, $\kappa_-=(4\bhm_0)^{-1}$ being the surface gravity, and with $\gamma_1=\gamma_2=0$) gives as eigenvalues of its $0$-th order part $-4\kappa_-,-2\kappa_-,0,2\kappa_-,4\kappa_-$, and in the subsequent displayed equation, with $\beta_{-,0}=2\kappa_-$ by~\cite[Equations~(3.13), (3.27), (6.1), (6.15)]{HintzVasyKdSStability}, the bundle endomorphism $\hat\beta_-$ of $S^2 T^*X^\circ$ thus has eigenvalues $-2,-1,0,1,2$. As discussed in \cite[Theorem~5.4]{HintzVasyKdSStability}, the threshold regularity (for spectral parameters $\sigma\in\C$, $\Im\sigma\geq 0$, thus taking $C=0$ in the reference) is therefore $\half+2=\tfrac52$.\footnote{Dually, the cokernel may contain 2-tensors which barely fail to lie in $H^{-3/2}$ at the event horizon, thus permitting at most once differentiated $\delta$-distributions, which indeed arise, see Proposition~\ref{PropL0}. The threshold regularity is $\half$ for $\Box_{g_b}$ acting on functions and $\half+1=\tfrac32$ for $\Box_{g_b}$ acting on 1-forms, the latter being a consequence of \cite[Equation~(6.15)]{HintzVasyKdSStability}.} This implies that the threshold regularity for nearby Kerr metrics is close to $\tfrac52$; a calculation shows that it is in fact \emph{equal} to $\tfrac52$ for all $b$.\footnote{This calculation is not needed if one assumes that $s$ is some fixed amount larger than $\tfrac52$, say, $s>3$, which ensures that it exceeds the threshold regularity for $L_{g_b}$ for $b$ close to $b_0$ simply by continuity.}

  For $0\neq\sigma\in\R$, radial point estimates at $\Sigma_{\sigma,\pa}$ for $\wh{L_{g_b}}(\sigma)$ similarly require the computation of a threshold decay rate relative to $L^2(X)$. Concretely, the threshold $-\half$ from \cite[Propositions~9 and 10]{MelroseEuclideanSpectralTheory}, \cite{VasyZworskiScl}, \cite[Theorems~1.1 and 1.3]{VasyLAPLag} is modified by the subprincipal symbol $\sigmasc_1(\frac{1}{2 i\rho}(\wh{L_{g_b}}(\sigma)-\wh{L_{g_b}}(\sigma)^*))|_{\cR_{\sigma,\rm in/out}}$; we now argue that this symbol vanishes. Indeed, formally taking $b=(0,0)$, so $g_b=\ubar g$ is the Minkowski metric, and working in the trivialization of $S^2\,\wt{\Tsc^*}X$ given in terms of the differentials of standard coordinates $t,x^1,x^2,x^3$, the operator $L_{g_b}$ is the wave operator on Minkowski space acting on symmetric 2-tensors, hence a $10\times 10$ diagonal matrix of scalar wave operators, and therefore the subprincipal symbol vanishes when using the fiber inner product on $S^2\,\wt{\Tsc^*}X$ which makes $d t^2$, $2\,d t\,d x^i$, $d x^i\,d x^j$ orthonormal. Changing from the Minkowski metric to a Kerr metric does not affect the subprincipal symbol at $\cR_{\sigma,\rm in/out}$, as follows from a simple calculation using~\eqref{EqKStMetrics2} and Proposition~\ref{PropKStBox2} together with (the proof of) Lemma~\ref{LemmaOpLinFT}.
  
  Combining the radial point estimates at infinity from \cite{VasyLAPLag} with those at the event horizon from \cite{VasyMicroKerrdS} (see also \cite[Proposition~2.1]{HintzVasySemilinear}), gives the stated uniform estimates for $\Im\sigma\in[0,C]$, $C^{-1}\leq|\sigma|\leq C$ for any fixed $C>1$. (We also point the reader to \cite[\S6]{VasyLowEnergy} for a discussion of the low energy Fredholm analysis for the $t_{\chi_0}$-spectral family of the scalar wave equation on Kerr spacetimes.) The uniformity of the stated estimate down to $\sigma=0$ is proved in \cite[Proposition~5.3]{VasyLowEnergyLag}; this uses the invertibility of a model operator, see \cite[\S5]{VasyLowEnergyLag}, which in the current setting and in the standard coordinate trivialization of $S^2\,\wt{\Tsc^*}X$ is the $10\times 10$ identity matrix tensored with the scalar model operator discussed (and proved to be invertible) in \cite[Proposition~5.4]{VasyLowEnergyLag}.

  The high energy estimates in strips of bounded $\Im\sigma\geq 0$ use the phase space dynamics of the Hamilton vector field of the semiclassical principal symbol, as described in Proposition~\ref{PropKFl}. The main new ingredient concerns the trapped set which was discussed, for Schwarzschild--de~Sitter spacetimes, in \cite[\S10.1]{HintzVasyKdSStability}; the calculations there apply directly in the Schwarzschild setting ($\Lambda=0$) as well, and imply (as discussed in \cite[\S4]{HintzPsdoInner} and \cite[\S5.1]{HintzVasyKdSStability}) that the semiclassical estimates at $\Gamma$ proved in \cite{DyatlovSpectralGaps} (see \cite[\S4.4]{HintzVasyQuasilinearKdS} for the microlocalized version) apply. (See also the discussion prior to, and the proof of \cite[Theorem~5.4]{HintzVasyKdSStability} for further details.)

  It remains to prove that $\wh{L_{g_b}}(\sigma)$ has index $0$ as stated in~\eqref{EqOpFredS}--\eqref{EqOpFred0}. This is clear when $|\sigma|$ is large since $\wh{L_{g_b}}(\sigma)$ is then invertible; hence we only need to consider bounded $\sigma$. One approach is to prove the continuity of the index in $\sigma$ by exploiting uniform Fredholm estimates; we present the details of such an argument in the proof of Theorem~\ref{Thm0}. Here, we instead use a deformation argument, which reduces the index $0$ property of $\wh{L_{g_b}}(\sigma)$ to that of the Fourier-transformed \emph{scalar} wave operator (which is established, by means of direct, non-perturbative arguments for real $\sigma$, in the proof of Theorem~\ref{Thm0}).
  
  We first treat the case $\sigma=0$: choose a global trivialization of $S^2\,\wt{\Tsc^*}X$, then $\wh{L_{g_b}}(0)$ is a $10\times 10$ matrix of scalar operators in $\rho^2\Diffb^2(X)$, with the off-diagonal operators lying in $\rho^2\Diffb^1(X)$. Since adding an element of $\rho^2\Diffb^1$ to $\wh{L_{g_b}}(0)$ does not change the domain in~\eqref{EqOpFred0}, we can continuously deform $\wh{L_{g_b}}(0)$ within the class of Fredholm operators on the spaces in~\eqref{EqOpFred0} to a diagonal $10\times 10$ matrix with all diagonal entries equal to the scalar wave operator at zero energy, $\wh{\Box_{g_b}}(0)$; the latter operator is well known to be invertible for $b=b_0$ (we recall the argument in the proof of Theorem~\ref{Thm0} below), and thus for $|b-b_0|$ small enough by a perturbative argument as in \cite[\S2.7]{VasyMicroKerrdS}, which we recall in the proof of Theorem~\ref{Thm0} below; in particular, it has index $0$. Thus, $\wh{L_{g_b}}(0)$ has index $0$ as well.

  Moreover, as shown in~\cite[\S5]{VasyLowEnergyLag}, the invertibility of $\wh{\Box_{g_b}}(0)$ for $b$ near $b_0$ implies that of $\wh{\Box_{g_b}}(\sigma)$ on the spaces~\eqref{EqOpFredS} for $(b,\sigma)$ (with $\Im\sigma\geq 0$) near $(b_0,0)$. For these $(b,\sigma)$, we can use a completely analogous deformation argument to deform $\wh{L_{g_b}}(\sigma)$ to a $10\times 10$ matrix of scalar operators $\wh{\Box_{g_b}}(\sigma)$, hence $\wh{L_{g_b}}(\sigma)$ is Fredholm of index $0$ indeed.

  For $\sigma\in\C$, $\Im\sigma\geq 0$, bounded away from $0$ and $\infty$, say $C^{-1}\leq|\sigma|\leq C$, the index $0$ claim follows again by a deformation argument together with the invertibility of $\wh{\Box_{g_b}}(\sigma)$. The latter invertibility is straightforward to prove for $b=b_0$ using boundary pairing/integration by parts arguments; we recall the argument in the present conjugated setting in the proof of Theorem~\ref{Thm0}. Thus, invertibility holds for $b$ with $|b-b_0|$ sufficiently small (depending on $C$) by a perturbative argument as in~\cite[\S2.7]{VasyMicroKerrdS}.
\end{proof}

Moreover, we can describe putative elements of the nullspace of $\wh{L_{g_b}}(\sigma)$ rather precisely:

\begin{prop}
\label{PropOpNull}
  Let $s>\tfrac52$, and suppose $b=(\bhm,\bha)$ is close to $b_0$.
  \begin{enumerate}
  \item\label{ItOpNull0} Suppose $u\in\Hbext^{s,\ell}(X)$, $\ell\in(-\tfrac32,-\half)$. If $\wh{L_{g_b}}(0)u=0$, then $u\in\cA^{1-}(X)$. More precisely, there exists $u_0\in\CI(\pa X;S^2\,\wt{\Tsc^*_{\pa X}}X)$ such that $u-r^{-1}u_0\in\cA^{2-}(X)$. More generally, for arbitrary $\ell\in\R$, every $u\in\ker\wh{L_{g_b}}(0)\cap\Hbext^{s,\ell}$ is polyhomogeneous with index set contained in $\{(z,k)\colon z\in i\N,\,k\in\N_0\}$.
  \item\label{ItOpNull1} If $u\in\ker\wh{L_{g_b}}(0)^*\cap\Hbsupp^{-\infty,\ell}(X)$, then near $\pa X$, $u\in\Hb^{\infty,\ell}(X)$, and it has expansions as in part~\eqref{ItOpNull0}. Moreover, $u\in\Hbsupp^{1-s,\ell}(X)$ for all $s>\tfrac52$.
  \item\label{ItOpNullSigma} If $\sigma\neq 0$ and $u\in\ker\wh{L_{g_b}}(\sigma)\cap\Hbext^{s,\ell}(X)$ for some $\ell\in\R$ with $s+\ell>-\half$, then $u\in\rho\CI(X)+\cA^{2-}(X)$.
  \end{enumerate}
\end{prop}
\begin{proof}
  For part~\eqref{ItOpNull0}, note that $u\in\Hbext^{\infty,\ell}(X)$ by elliptic regularity, the propagation of regularity at the radial sets $\cR^\pm$ at the horizons, and real principal type propagation. The polyhomogeneity of $u$ is then a consequence of the fact that $\wh{L_{g_b}}(0)\in\rho^2\Diffb^2$ is a (weighted) elliptic b-operator near $\pa X$, see \cite[\S\S4--5]{MelroseAPS}, with boundary spectrum contained in $i\Z\times\{0\}$.
  
  In more detail, the normal operator of $\wh{L_{g_b}}(0)$ is the negative Euclidean Laplacian $-\Delta=\rho^2(\rho\pa_\rho(\rho\pa_\rho-1)-\slDelta)$ tensored with the $10\times 10$ identity matrix when working in the standard coordinate trivialization of $S^2\,\wt{\Tsc^*}X$ (meaning: the difference of the two operators lies in $\rho^3\Diffb^2$). Thus, the asymptotic behavior of $u$ can be found by writing
  \[
    \rho^{-2}\Delta(\chi u) = \rho^{-2}(\wh{L_{g_b}}(0)-(-\Delta))(\chi u) - \rho^{-2}[\wh{L_{g_b}}(0),\chi]u \in \Hbext^{\infty,\ell+1},
  \]
  where $\chi$ is a cutoff, identically $1$ near $\pa_+X$ and vanishing for $r\leq 3\bhm_0$; one then takes the Mellin transform in $\rho$ and uses the properties of the meromorphic (in $\lambda\in\C$) inverse of the operator $\wh{\rho^{-2}\Delta}(\lambda)=i\lambda(i\lambda+1)+\slDelta$ on $\CI(\pa_+X;S^2\,\wt{\Tsc^*_{\pa X}}X)$ to deduce a partial expansion of $\chi u$, plus a remainder term in $\Hbext^{\infty,\ell+1}$ near $\pa_+X$. An iterative argument gives a full polyhomogeneous expansion.
  
  The boundary spectrum of the scalar Euclidean Laplacian is, by definition, the divisor of $\wh{\rho^{-2}\Delta}(\lambda)^{-1}$. Decomposing functions on $\pa X$ into spherical harmonics, and denoting by $\scal_l$ a degree $l\in\N_0$ spherical harmonic, we have
  \[
    \rho^{-i\lambda}\Delta(\rho^{i\lambda}\scal_l) = \bigl(i\lambda(i\lambda+1) + l(l+1)\bigr)\scal_l,
  \]
  which vanishes for $\lambda=i l,-i(l+1)$. Thus, the boundary spectrum of $\Delta$ is equal to $i\Z$, with space of resonant states at $i l$ given by $r^l\scal_l$ for $l\geq 0$ and $r^l\scal_{-l-1}$ for $l\leq -1$. (The need to allow for logarithmic powers of $\rho$ in the expansion of $u$ arises as usual from the presence of integer coincidences in the boundary spectrum.)

  Part~\eqref{ItOpNull1} is proved similarly; the regularity statement follows from the fact that we have $u\equiv 0$ in the interior $r<r_b$ of the black hole, together with a radial sink estimate at $\cR^\pm$.

  In part~\eqref{ItOpNullSigma}, smoothness of $u$ away from $\pa X$ follows as above, while the radial point estimates at $\pa X$ in \cite{VasyLAPLag} imply that $u$ is conormal at $\pa X$. But then note that the normal operator of $\wh{L_{g_b}}(\sigma)$ (which for $\sigma\neq 0$ merely lies in $\rho\Diffb^2$) is $2\sigma\rho(\rho D_\rho+i)$ by Lemma~\ref{LemmaOpLinFT}, whose boundary spectrum consists of the single point $\{(-i,0)\}$. This implies $u\in\cA_\phg^\cE(X)$ where $\cE\subset\{(-i,0)\}\cup\{(-i j,k)\colon 2\leq j\in\N,\,k\in\N_0\}$.
\end{proof}

We will make abundant use of such the normal operator arguments. Note that part~\eqref{ItOpNull0} holds under much weaker assumptions, namely $\wh{L_{g_b}}(0)u\in\CIc(X^\circ)$ or just $\wh{L_{g_b}}(0)u\in\Hbext^{\infty,5/2-}$ (except for the last statement).

\begin{rmk}
\label{RmkOpOutgoing}
  The relationship of $\ker\wh{L_{g_b}}(\sigma)$ to the usual outgoing condition on Schwarzschild spacetimes is as follows: any $u\in\ker\wh{L_{g_b}}(\sigma)$ as in part~\eqref{ItOpNullSigma} solves
  \[
    L_{g_b}\bigl(e^{-i\sigma t_{\chi_0}} u'\bigr) = 0,\quad u'=e^{i\sigma(t_{\chi_0}-t_{\bhm,*})}u;
  \]
  but $t_{\chi_0}-t_{\bhm,*}\equiv r_*=r+2\bhm\log r$ up to addition of a smooth bounded function, hence $u'\sim r^{-1}e^{i\sigma r_*}$ for large $r$.
\end{rmk}

%%%%%%%%%%%%%%%%%%%%%%%%%%%%%%%%%%%%%%%%%%%%%%%%%%
\subsection{Constraint damping and the modified gauge-fixed Einstein operator}
\label{SsOpCD}

We will show in~\S\ref{SCD} and exploit in~\S\ref{SR} and subsequent sections that the properties of the low energy resolvent of the linearized gauge-fixed Einstein operator can be crucially improved by modifying the way the gauge 1-form is combined with the Ricci tensor in~\eqref{EqOpEin}. Concretely:

\begin{definition}
\label{DefOpCD}
  Let $E\in\CI(M^\circ;\Hom(T^*M^\circ,S^2 T^*M^\circ))$, and let $g$ denote a pseudo-Riemannian metric. Then the \emph{modified symmetric gradient} is
  \[
    \wt\delta_{g,E}^* := \delta_g^* + E.
  \]
\end{definition}

In this paper, we shall use $E$ of the form
\begin{equation}
\label{EqOpCD}
  E = E(g;\cd,\gamma_1,\gamma_2) := 2\gamma_1\cd\otimes_s(-)-\gamma_2 g^{-1}(\cd,-)g,
\end{equation}
where $\gamma_1,\gamma_2\in\R$, and $\cd$ is a stationary 1-form on $M^\circ$ with compact spatial support, i.e.\ $\cd\in\CIc(X^\circ;\wt{\Tsc^*}X)$. The \emph{modified gauge-fixed Einstein operator} is then the map
\begin{equation}
\label{EqOpCDEin}
  g \mapsto P_E(g) := \Ric(g) - \wt\delta_{g,E}^*\Ups(g;g^0).
\end{equation}
Fixing $g^0=g_{(\bhm,\bha)}$ to be a subextremal Kerr metric, and linearizing around $g=g_{(\bhm,\bha)}$, we then have
\begin{equation}
\label{EqOpCDEinLin}
  L_{g,E} := 2 D_g P_E = 2(D_g\Ric + \wt\delta_{g,E}^*\delta_g\sfG_g) = \Box_{g,2} + 2 E\delta_g\sfG_g + 2\sR_g.
\end{equation}

Here, it will suffice to use \emph{small} $\gamma_1,\gamma_2$ and perturbative arguments in order to reap the benefits of constraint damping, as outlined in~\S\ref{SI} and explained in detail in~\S\S\ref{SsL0Qu} and \ref{SCD}. Thus, we record here that Theorem~\ref{ThmOp} and Proposition~\ref{PropOpNull} remain valid for $\wh{L_{g_b,E}}(\sigma)$, $E=E(g_b;\cd,\gamma_1,\gamma_2)$ for some fixed $\cd\in\CIc(X^\circ;\wt{\Tsc^*}X)$, provided $b$ is sufficiently close to $b_0$ and $|\gamma_1|,|\gamma_2|$ are sufficiently small (depending on the regularity parameter $s$), with the estimates in Theorem~\ref{ThmOp} being \emph{uniform} for such $b,\gamma_1,\gamma_2$. (For $s>3$, say, the point being that it is a \emph{fixed} amount about $\tfrac52$, the theorem and the proposition hold for $|b-b_0|+|\gamma_1|+|\gamma_2|<\eps$ with $\eps$ independent of $s$.)

We note that for $g$ satisfying $\Ric(g)=0$, we have $L_{g,E}(\delta_g^*\omega)=0$ for a 1-form $\omega$ provided that $\delta_g\sfG_g\delta_g^*\omega=0$, which is the tensor wave equation on 1-forms. In this way, suitable zero energy states of the 1-form wave equation give rise to pure gauge bound states of $L_{g,E}$. Dually, we have
\begin{equation}
\label{EqOpCDAdjoint}
  L_{g,E}^* = 2\bigl(\sfG_g(D_g\Ric)\sfG_g + \sfG_g\delta_g^*\wt\delta_{g,E}\bigr),\quad \wt\delta_{g,E}=(\delta_{g,E}^*)^*,
\end{equation}
which satisfies $L_{g,E}^*(\sfG_g\delta_g^*\omega^*)=0$ provided $\wt\delta_{g,E}\sfG_g\delta_g^*\omega^*=0$. (For $E=0$, this is the same equation as for $\omega$, though we need to solve it on different function spaces.) Such `dual-pure-gauge' 2-tensors $\sfG_g\delta_g^*\omega^*$ are thus, for suitable $\omega^*$, bound states of $L_{g,E}^*$.

%%%%%%%%%%%%%%%%%%%%%%%%%%%%%%%%%%%%%%%%%%%%%%%%%%%%%%%%%%%%%%%%%%%%%%
\section{Spherical harmonic decompositions}
\label{SY}

We introduce the terminology which will be used in the subsequent precise (generalized) mode analysis, borrowing from~\cite{KodamaIshibashiMaster}, and taking some of the notation from~\cite[\S5]{HintzKNdSStability}.

%%%%%%%%%%%%%%%%%%%%%%%%%%%%%%%%%%%%%%%%%%%%%%%%%%
\subsection{Spherical harmonics on the sphere}
\label{SsYS}

Recall that $\slg$ denotes the standard metric on $\Sph^2$, and denote geometric operators on $\Sph^2$ using a slash, thus $\sltr=\tr_\slg$, $\sldelta=\delta_\slg$, etc. We denote by $Y_{l m}$, $l\in\N_0$, $m\in\Z$, $|m|\leq l$, the usual spherical harmonics on $\Sph^2$ satisfying $\slDelta Y_{l m}=l(l+1)Y_{l m}$. Define the space 
\begin{equation}
\label{EqYS0}
  \scal_l := \mathspan\{Y_{l m} \colon |m|\leq l\}
\end{equation}
of degree $l$ spherical harmonics. Thus, $L^2(\Sph^2)=\bigoplus_{j\in\N_0}\scal_j$ is an orthogonal decomposition.

Consider next 1-forms on $\Sph^2$. Denote the Hodge Laplacian by $\slDelta_H=(\sld+\sldelta)^2$; the tensor Laplacian $\slDelta_{\slg,1}=-\sltr\slnabla^2$ (also denoted $\slDelta$ for brevity) satisfies $\slDelta=\slDelta_H-\Ric(\slg)=\slDelta_H-1$. Therefore, a spectral decomposition of $\slDelta$ on $L^2(\Sph^2;T^*\Sph^2)$ is provided by the scalar/vector decomposition
\begin{equation}
\label{EqYS1}
  \sld\scal_l,\ 
  \vect_l := \slstar\sld\scal_l \subset \ker\bigl(\slDelta-(l(l+1)-1)\bigr)\qquad (l\geq 1);
\end{equation}
note that $\sldelta\vect_l=0$, and that the two spaces in~\eqref{EqYS1} are trivial for $l=0$.

For symmetric 2-tensors finally, we have an analogous orthogonal decomposition into scalar and vector type symmetric 2-tensors: the scalar part consists of a pure trace and a trace-free part, the latter defined using the trace-free symmetric gradient $\sldelta_0^*:=\sldelta^*+\half\slg\sldelta$:
\begin{subequations}
\begin{equation}
\label{EqYS2scal}
  \scal_l\slg\ \ (l\geq 0),\qquad
  \sldelta_0^*\sld\scal_l\ \ (l\geq 2).
\end{equation}
(Note here that for $\scal\in\scal_0\oplus\scal_1$, we have $\sldelta_0^*\sld\scal=0$, hence the restriction to $l\geq 2$.) The vector part consists only of trace-free tensors with $l\geq 2$ (since the 1-forms $\vect_1$ are Killing),
\begin{equation}
\label{EqYS2vect}
  \sldelta^*\vect_l\ \ (l\geq 2).
\end{equation}
\end{subequations}

The geometric operators on $\Sph^2$ which we will encounter here preserve scalar and vector type spherical harmonics; indeed, this holds in the strong sense that a scalar type function/1-form/symmetric 2-tensor built out of a particular $\scal\in\scal_l$ is mapped into another scalar type tensor \emph{with the same $\scal$}, likewise for vector type tensors; this is clear for $\sld$ on functions, $\sldelta$ on 1-forms ($\sldelta(\sld\scal)=l(l+1)\scal$). Furthermore, for $\scal\in\scal_l$ and $\vect\in\vect_l$,
\begin{gather*}
  \sldelta^*(\sld\scal) = -\tfrac{l(l+1)}{2}\scal\slg + \sldelta_0^*\sld\scal, \quad
  \sldelta(\scal\slg)=-\sld\scal,\quad
  \sldelta(\sldelta_0^*\sld\scal)=\tfrac{l(l+1)-2}{2}\sld\scal, \\
  \sldelta\sldelta^*\vect = \tfrac{l(l+1)-2}{2}\vect, \quad
  \slDelta(\sldelta_0^*\sld\scal) = (l(l+1)-4)\sldelta_0^*\sld\scal, \quad
  \slDelta(\sldelta^*\vect) = (l(l+1)-4)\sldelta^*\vect.
\end{gather*}

%%%%%%%%%%%%%%%%%%%%%%%%%%%%%%%%%%%%%%%%%%%%%%%%%%
\subsection{Decompositions on spacetime}
\label{SsYM}

Rather than working in a splitting into temporal and spatial parts, we split the spacetime $M^\circ$ in~\eqref{EqK0Ext} into an aspherical and spherical part,
\begin{equation}
\label{EqYMProd}
  M^\circ = \wh X \times \Sph^2,\quad
  \wh\pi\colon M^\circ \to \wh X,\ \ 
  \slpi\colon M^\circ \to \Sph^2,
\end{equation}
where $\wh X=\R_{t_*}\times[r_-,\infty)_r$. Via pullback by $\wh\pi^*$, we can identify $\CI(\wh X)$ with the subspace $\wh\pi^*\CI(\wh X)\subset\CI(M^\circ)$ of \emph{aspherical functions}. Similarly, $\CI(\Sph^2)\subset\CI(M^\circ)$ via pullback by $\slpi$, and $\CI(\wh X;T^*\wh X)\subset\CI(M^\circ;T^*M^\circ)$, likewise for other tensor bundles on $\wh X$ and $\Sph^2$. Functions on $M^\circ$ can be decomposed into spherical harmonics in the $\Sph^2$ factor; restricting to degree $l_0$ harmonics gives the space of \emph{scalar $l=l_0$ functions} $\{u\scal_{l_0}\colon u\in\CI(\wh X)\}$.

We can split the cotangent bundle of $M^\circ$ into aspherical and spherical parts,
\begin{equation}
\label{EqYMSplit}
  T^*M^\circ = \TAS^*\oplus\TS^*,\quad \TAS^*=\wh\pi^* T^*\wh X,\ \ \TS^*=\slpi^* T^*\Sph^2;
\end{equation}
this induces the splitting
\begin{equation}
\label{EqYMSplit2}
  S^2 T^* M \cong S^2\TAS^* \oplus (\TAS^*\otimes\TS^*) \oplus S^2\TS^*
\end{equation}
of the second symmetric tensor power into the aspherical, mixed, and spherical subbundles; here, the second summand is a subbundle of $S^2 T^*M$ via $a\otimes s\mapsto 2 a\otimes_s s$.

Corresponding to the scalar/vector decomposition~\eqref{EqYS1}, there are two classes of 1-forms of fixed spherical harmonic degree on $M^\circ$, which we write in the splitting~\eqref{EqYMSplit} and using $T\in\CI(\wh X;T^*\wh X)$, $L\in\CI(\wh X)$:
\begin{equation}
\label{EqYM1forms}
\begin{aligned}
  &\text{scalar $l=l_0$ ($l_0\geq 1$):} \ \ && (T\scal,\,L\sld\scal), &&\scal\in\scal_{l_0}, \\
  &\text{scalar $l=0$:} && (T,\,0), \\
  &\text{vector $l=l_0$ ($l_0\geq 1$):} && (0,\,L\vect), &&\vect\in\vect_{l_0}.
\end{aligned}
\end{equation}

Similarly, symmetric 2-tensors on $M^\circ$ come in two classes, with low spherical harmonic degrees requiring separate treatment. Below, $H_L,H_T\in\CI(\wh X)$, $f\in\CI(\wh X;T^*\wh X)$, and $\wt f\in\CI(\wh X;S^2 T^*\wh X)$:
\begin{equation}
\label{EqYMSym2}
\begin{aligned}
  &\text{scalar $l=l_0$ ($l_0\geq 2$):} \ \ && (\wt f\scal,\,f\otimes\sld\scal,\,H_L\scal\slg+H_T\sldelta_0^*\sld\scal), &&\scal\in\scal_{l_0}, \\
  &\text{scalar $l=1$:} && (\wt f\scal,\,f\otimes\sld\scal,\,H_L\scal\slg), && \scal\in\scal_1, \\
  &\text{scalar $l=0$:} && (\wt f,\,0,\,H_L\slg), \\
  &\text{vector $l=l_0$ ($l_0\geq 2$):} && (0,\,f\otimes\vect,\,H_T\sldelta^*\vect), && \vect\in\vect_{l_0}. \\
  &\text{vector $l=1$:} && (0,\,f\otimes\vect,\,0), && \vect\in\vect_1.
\end{aligned}
\end{equation}
We call tensors of this form \emph{scalar type $\scal$}, resp.\ \emph{vector type $\vect$} if $\scal$, resp.\ $\vect$ are fixed.

%%%%%%%%%%%%%%%%%%%%%%%%%%%%%%%%%%%%%%%%%%%%%%%%%%%%%%%%%%%%%%%%%%%%%%
\section{Mode analysis of the scalar wave operator}
\label{S0}

As a preparation for the precise spectral analysis of wave-type operators on tensor bundles, we briefly discuss the properties of the Fourier-transformed \emph{scalar} wave operator on slowly rotating Kerr spacetimes. As in~\eqref{EqOpLinFT}, we define the spectral family of a stationary operator $L$ on a Kerr spacetime with parameters $b=(\bhm,\bha)$ by
\begin{equation}
\label{Eq0Spectral}
  \wh L(\sigma) := e^{i\sigma t_{\bhm,*}}L e^{-i\sigma t_{\bhm,*}}.
\end{equation}

\begin{thm}
\label{Thm0}
  Let $g=g_{(\bhm_0,0)}$. For $\Im\sigma\geq 0$, $\sigma\neq 0$, the operator
  \[
    \wh{\Box_g}(\sigma)\colon\{u\in\Hbext^{s,\ell}(X)\colon\wh{\Box_g}(\sigma)u\in\Hbext^{s,\ell+1}(X)\} \to\Hbext^{s,\ell+1}(X)
  \]
  is invertible when $s>\half$, $\ell<-\half$, and $s+\ell>-\half$. The stationary operator
  \[
    \wh{\Box_g}(0)\colon\{u\in\Hbext^{s,\ell}(X)\colon\wh{\Box_g}(0)u\in\Hbext^{s-1,\ell+2}(X)\} \to\Hbext^{s-1,\ell+2}(X)
  \]
  is invertible for all $s>\half$ and $\ell\in(-\tfrac32,-\half)$. Both statements continue to hold for $g=g_{(\bhm,\bha)}$ with $(\bhm,\bha)$ near $(\bhm_0,0)$.
\end{thm}
\begin{proof}
  We analyze the case of the Schwarzschild metric first, beginning with $\sigma=0$. Thus, suppose $\wh{\Box_g}(0)u=-(r^{-2}D_r\mu r^2 D_r+r^{-2}\slDelta)u=0$, where $\mu=1-\frac{2\bhm_0}{r}$, and $u\in\Hbext^{s,\ell}(X)$. Then $u\in\Hbext^{\infty,\ell}$, and in fact $|u|,|r\pa_r u|\lesssim r^{-1}$ by Proposition~\ref{PropOpNull}. Therefore,
  \begin{align*}
    0 &= -\int_{\Sph^2}\int_{2\bhm_0}^R (\pa_r\mu r^2\pa_r-\slDelta)u\cdot\bar{u}\,d r|d\slg|\\
      &= \int_{\Sph^2}\int_{2\bhm_0}^R \bigl(\mu|r\pa_r u|^2+|\slnabla u|^2\bigr)\,d r|d\slg|
        - \int_{\{r=R\}} \mu r^2(\pa_r u)\cdot\bar u\,|d\slg|.
  \end{align*}
  The second term is $\lesssim R^{-1}$. Taking the limit $R\to\infty$, we thus conclude that $u$ is constant and hence vanishes in $r\geq 2\bhm_0$; the infinite order vanishing at $r=2\bhm_0$ and smoothness in $r\leq 2\bhm_0$ then imply $u\equiv 0$ in $r\leq 2\bhm_0$ as well, see \cite[Lemma~1]{ZworskiRevisitVasy}. Thus, $\wh{\Box_g}(0)$ is injective.

  To prove surjectivity, we can either use an abstract deformation or perturbation argument to show that $\wh{\Box_g}(0)$ has index $0$; or we can proceed directly, and show that $v\in\Hbsupp^{1-s,-\ell-2}(X)$ with $0=\wh{\Box_g}(0)^*v=\wh{\Box_g}(0)v$ vanishes. We do the latter: energy estimates in $r<2\bhm_0$ for $\wh{\Box_g}(0)^*$ (which is a hyperbolic operator there, with $r$ a timelike function) show that $v$ vanishes there; furthermore, $v$ is smooth in $r>2\bhm_0$, and $|v|,|r\pa_r v|\lesssim r^{-1}$ in $r>2\bhm_0$. Moreover, radial point estimates at the event horizon imply that $v\in H^{1/2-}$ there. We claim that $v|_{r>2\bhm_0}$ is smooth down to $r=2\bhm_0$; the arguments above then imply $v=0$ in $r>2\bhm_0$, hence $v\equiv 0$ since $v$ cannot be a (differentiated) $\delta$-distribution at $r=2\bhm_0$ since it lies in $L^2$.\footnote{One can alternatively check by an explicit calculation that sums of differentiated $\delta$-distributions at $r=2\bhm_0$ do not lie in $\ker\wh{\Box_g}(0)$.} To prove the smoothness, recall that $v$ is conormal at $r=2\bhm_0$ by \cite{HaberVasyPropagation}, hence we can obtain the asymptotic behavior of $v$ there by writing
  \begin{equation}
  \label{Eq0Fuchs}
    (\mu r^2 D_r)^2 v+\mu r^4\slDelta v=0,
  \end{equation}
  where the crucial point is that $\mu=0$ at $r=2\bhm_0$. Now $(\mu r^2 D_r)^2$ is a Fuchsian operator with a double indicial root at $0$: it annihilates $1$ and $\log\mu$ in $\mu>0$; we thus have $v=H(r-2\bhm_0)(v_0+v_1\log\mu) + v'$ where $v_0,v_1\in\CI(\Sph^2)$ and $v'\in\cA^{1-}(\{r\geq 2\bhm_0\})$ is conormal at $r=2\bhm_0$ and bounded by $\mu^{1-}$. Now $\wh{\Box_g}(0)H(r-2\bhm_0)=0$, so
  \[
    0 = \wh{\Box_g}(0)v = \wh{\Box_g}(0)(v_1 H(r-2\bhm_0)\log\mu) + \wh{\Box_g}(0)v' = \frac{v_1}{2\bhm_0}\delta(r-2\bhm_0) + \wh{\Box_g}(0)v'.
  \]
  The second term lies in $\cA^{0-}$, whereas the first term does not lie in $\cA^{0-}$ unless $v_1=0$. Using this information in the asymptotic analysis of~\eqref{Eq0Fuchs} at $\mu=0$ then implies $v\in\CI(\{r\geq 2\bhm_0\})$, as desired.

  We shall sketch the proof of invertibility of $\wh{\Box_g}(\sigma)$ for non-zero $\sigma$ as well to illustrate the relationship between outgoing solutions in the `conjugated perspective' (i.e.\ using $t_{\bhm_0,*}$ to define spectral families) and in the `standard perspective', cf.\ Remark~\ref{RmkOpOutgoing}. We start with $0\neq\sigma\in\R$. In this case, $u\in\ker\wh{\Box_g}(\sigma)$ gives rise to an outgoing solution
  \begin{equation}
  \label{Eq0Spec}
    \wt{\Box_g}(\sigma)u'=0,\qquad \wt{\Box_g}(\sigma):=e^{i\sigma t}\Box_g e^{-i\sigma t},\quad u':=e^{-i\sigma(t-t_{\bhm_0,*})}u.
  \end{equation}
  A boundary pairing argument, see \cite[\S2.3]{MelroseGeometricScattering} and also the proof of Proposition~\ref{Prop1Re} (starting at~\eqref{Eq1ReCutoff}) below, then shows that the leading order terms of $u'$ at the event horizon and at infinity must vanish, thus $u(2\bhm_0)=0$ and $u\in\cA^{2-}(X)$. An indicial root argument then implies that $u$ vanishes to infinite order at $r=2\bhm_0$ and $\pa X$. A unique continuation argument\footnote{One can use either unique continuation at infinity for $\wt{\Box_g}(\sigma)u'=0$ as in \cite[Theorem~17.2.8]{HormanderAnalysisPDE2}, or unique continuation at $r=2\bhm_0$ using \cite{MazzeoUniqueContinuation}.} then gives $u\equiv 0$ in $r\geq 2\bhm_0$, hence also in $r<2\bhm_0$ using energy estimates as in \cite{ZworskiRevisitVasy}. This proves the injectivity of $\wh{\Box_g}(\sigma)$ for $0\neq\sigma\in\R$.
  
  The proof of surjectivity is not quite symmetric from the conjugated perspective, hence we sketch the direct argument; for an abstract perturbative argument, see the discussion of the $\Im\sigma>0$ case below. First, one notes that $v\in\ker\wh{\Box_g}(\sigma)^*\cap\Hbsupp^{-s,-\ell-1}$, where now $-s<-\half$ and $-\ell-1>-\half$, is an element of a second microlocal scattering-b space as defined in \cite[\S2]{VasyLAPLag}, to wit, $v\in\dot H_{\scop,\bop}^{-s,-s-\ell-1,-\ell-1}$, where the orders denote the scattering regularity, scattering decay, and b-decay (which can be thought of as a very precise form of `scattering decay order at $0$ scattering frequency'). Note that $-s-\ell-1<-\half$. Let us work in $r>3\bhm_0$; then elliptic regularity implies $v\in H_{\scop,\bop}^{\infty,-s-\ell-1,-\ell-1}$, and a normal operator argument at $\pa X$ as in the proof of Proposition~\ref{PropOpNull}\eqref{ItOpNullSigma} improves the b-decay order to $+\infty$, so $v\in H_{\scop,\bop}^{\infty,-s-\ell-1,\infty}$ in $r>3\bhm_0$. Radial point estimates at the (lift of the) outgoing radial set $\cR_{\rm out}$, see~\eqref{EqKFlRadInfty}, improve the scattering decay order to $+\infty$ away from $\cR_{\rm in}$. This now implies that (restricting to $r>3\bhm_0$ still) $\WFsc(u')\subset\cR_{\rm in}$. At this point, we can again consider~\eqref{Eq0Spec}; thus $u'$ is now an \emph{incoming} mode solution, and can easily be shown to be equal to\footnote{One can follow the arguments of \cite[\S12]{MelroseEuclideanSpectralTheory}. An approach closer in spirit to the conjugated perspective passes to the spectral family relative to the time variable $2 t-t_{\bhm_0,*}$, which is the time-reversed analogue of $t_{\bhm_0,*}$: its level sets are transversal to the past event horizon and past null infinity. Indeed, what we have just proved is that an incoming mode solution for $\wh{\Box_g}(\sigma)$ is an outgoing mode solution for this new spectral family.} $e^{-i\sigma r_*}(c r^{-1}+\cA^{2-})$ for some $c\in\C$. The same boundary pairing and unique continuation arguments as for the direct problem prove that $v=0$.

  For $\Im\sigma>0$, the proof of injectivity uses that $u'$ in~\eqref{Eq0Spec} now decays exponentially fast as $|r_*|\to\infty$ where $r_*=r+2\bhm_0\log(r-2\bhm_0)$ is the tortoise coordinate; taking imaginary parts of $0=\la\wt{\Box_g}(\sigma)u',u'\ra$ then implies $u'=0$. Surjectivity is proved most easily by establishing that $\wh{\Box_g}(\sigma)$ has index $0$. This holds for fixed $\Im\sigma=:C>0$ when $|\Re\sigma|\gg 1$, since then $\wh{\Box_g}(\sigma)$ is in fact invertible by high energy estimates, cf.\ Theorem~\ref{ThmOp}. We claim that the index of $\wh{\Box_g}(\sigma)$ is constant on $\Im\sigma=C$, and thus equal to $0$ for all such $\sigma$; it suffices to show that it is locally constant. Let
  \[
    \cX(\sigma) := \{ u\in\Hbext^{s,\ell}(X) \colon \wh{\Box_g}(\sigma)u\in\Hbext^{s,\ell+1}(X) \}.
  \]
  Then if $\wh{\Box_g}(\sigma_0)$, $\Im\sigma_0=C$, has kernel and cokernel of dimension $k_1$ and $k_2$, respectively, we define an operator
  \[
    L(\sigma) = \begin{pmatrix} \wh{\Box_g}(\sigma) & L_1 \\ L_2 & 0 \end{pmatrix} \colon\cX(\sigma)\oplus\C^{k_2}\to\Hbext^{s,\ell+1}(X)\oplus\C^{k_1},
  \]
  where $L_1\colon\C^{k_2}\to\CIc(X^\circ)$ maps into a complement of $\ran\wh{\Box_g}(\sigma_0)$, and $L_2\colon\sD'(X^\circ)\to\C^{k_1}$ restricts to an isomorphism $\ker\wh{\Box_g}(\sigma_0)\to\C^{k_1}$; thus, $L(\sigma_0)$ is invertible. Uniform Fredholm estimates for $\wh{\Box_g}(\sigma)$ imply such for $L(\sigma)$; but $L(\sigma_0)$ is invertible, hence so is $L(\sigma)$ for $\sigma$ near $\sigma_0$ by the perturbation arguments in~\cite[\S2.7]{VasyMicroKerrdS}. Therefore, the index of $\wh{\Box_g}(\sigma)$ is constant (namely, equal to $k_1-k_2$), as claimed.\footnote{One can give a more direct argument, in which one directly realizes $\wh{\Box_g}(\sigma)$ as a Fredholm family of operators acting between \emph{$\sigma$-independent} function spaces for $\Im\sigma>0$, by using the more precise microlocal point of view sketched in \cite[Remark~4.17]{VasyLAPLag}.}

  That the same results hold for slowly rotating Kerr metrics $g_b$, with $b$ close to $b_0$, was already discussed in the proof of Theorem~\ref{ThmOp}. Here, we flesh out the argument near zero energy, as we shall need more general versions of this later on. The key is that we have \emph{uniform} estimates for $s_0<s$, $\ell_0<\ell$,
  \[
    \|u\|_{\Hbext^{s,\ell}} \leq C\bigl(\bigl\|\wh{\Box_{g_b}}(0)u\bigr\|_{\Hbext^{s-1,\ell+2}} + \|u\|_{\Hbext^{s_0,\ell_0}}\bigr),
  \]
  for $b$ close to $b_0$. Thus, if $\ker\wh{\Box_{g_b}}(0)$ were non-trivial for a sequence $b_j\to b_0$, $j\to\infty$, we could find $u_j\in\Hbext^{s,\ell}\cap\ker\wh{\Box_{g_b}}(0)$, $\|u_j\|_{\Hbext^{s,\ell}}=1$; this estimate gives a positive lower bound $\|u_j\|_{\Hbext^{s_0,\ell_0}}\geq(2 C)^{-1}>0$ for large $j$, hence we can extract a weakly convergent subsequence $u_j\weakto u\in\Hbext^{s,\ell}$ (thus $u_j\to u$ in $\Hbext^{s_0,\ell_0}$) with the limit $u$ necessarily non-zero and satisfying
  \[
    \wh{\Box_{g_{b_0}}}(0)u = \lim_{j\to\infty} \Bigl(\wh{\Box_{g_{b_0}}}(0)(u-u_j) + \bigl(\wh{\Box_{g_{b_0}}}(0)-\wh{\Box_{g_{b_j}}}(0)\bigr)u_j\Bigr) = 0.
  \]
  But this contradicts $\ker\wh{\Box_{g_{b_0}}}(0)=0$. Surjectivity, which is equivalent to the injectivity of the adjoint, is shown similarly.

  A minor modification of this argument (now using the uniformity of the estimate~\eqref{EqOpFinite} \emph{near} rather than merely \emph{at} zero energy) applies more generally to putative sequences of normalized elements $u_j\in\ker\wh{\Box_{g_b}}(\sigma_j)\cap\Hbext^{s,\ell}$ where $\sigma_j\to 0$, $\Im\sigma_j\geq 0$. Thus, $\wh{\Box_{g_b}}(\sigma)$ is injective for $(b,\sigma)$ near $(b_0,0)$; surjectivity follows from the index $0$ property.
\end{proof}

%%%%%%%%%%%%%%%%%%%%%%%%%%%%%%%%%%%%%%%%%%%%%%%%%%
\subsection{Growing zero modes}
\label{Ss00Grow}

For later use, we record the explicit form of scalar functions in $\ker\wh{\Box_g}(0)$ which are allowed to have more growth at infinity. Their differentials are 1-forms, some of which are gauge potentials for (pure gauge) metric perturbations arising in the spectral analysis of $\wh{L_{g_b}}(0)$ in~\S\ref{SL}.

\begin{prop}
\label{Prop00Grow}
  For $b=(\bhm,\bha)$ near $b_0=(\bhm_0,0)$, we have
  \begin{subequations}
  \begin{alignat}{3}
  \label{Eq00Grows0Ker}
    &\ker\wh{\Box_{g_b}}(0) \cap \Hbext^{\infty,-3/2-}&=&\ \la u_{b,s 0}\ra, \\
  \label{Eq00Grows0KerAdj}
    &\ker\wh{\Box_{g_b}}(0)^* \cap \Hbsupp^{-\infty,-3/2-}&=&\ \la u_{b,s 0}^*\ra,
  \end{alignat}
  \end{subequations}
  where, in the notation~\eqref{EqKaRadius},
  \begin{equation}
  \label{Eq00Grows0}
    u_{b,s 0}=1,\quad u_{b,s 0}^*=H(r-r_{(\bhm,\bha)}).
  \end{equation}
  Furthermore, the spaces
  \begin{subequations}
  \begin{alignat}{4}
  \label{Eq00Grows1Ker}
    &\ker\wh{\Box_{g_b}}(0) \cap \Hbext^{\infty,-5/2-}& &=\ &\la u_{b,s 0}\ra & \oplus \{ u_{b,s 1}(\scal) \colon \scal\in\scal_1 \}, \\
  \label{Eq00Grows1KerAdj}
    &\ker\wh{\Box_{g_b}}(0)^* \cap \Hbsupp^{-\infty,-5/2-}& &=\ &\la u_{b,s 0}^*\ra & \oplus \{ u_{b,s 1}^*(\scal) \colon \scal\in\scal_1 \}
  \end{alignat}
  \end{subequations}
  are $4$-dimensional; the maps $b\mapsto u_{b,s 1}(\scal)$ and $b\mapsto u_{b,s 1}^*(\scal)$ can be chosen to be continuous (with values in the respective spaces), and to take the values
  \begin{equation}
  \label{Eq00Grows1}
    u_{(\bhm,0),s 1}(\scal) = (r-\bhm)\scal, \quad
    u_{(\bhm,0),s 1}^*(\scal) = (r-\bhm)H(r-2\bhm)\scal.
  \end{equation}
\end{prop}

\begin{rmk}
  We keep the notation $u_{b,s 0}$ etc.\ even though the spherical harmonic and scalar/vector decompositions cease to be globally well-defined on Kerr spacetimes with non-zero angular momentum. In later sections, proving the existence of continuous families of zero energy solutions will in fact require mixing continuous (in $b$) extensions of 1-forms or symmetric 2-tensors which at $b=b_0$ are of distinct types, see e.g.\ \eqref{EqL0Ansatz} and \eqref{EqL0v1Ansatz}.
\end{rmk}

\begin{rmk}
\label{Rmk00Reg}
  The dual states automatically vanish in $r<r_b$, are smooth in $r>r_b$, and are conormal at $\pa_+X$. They are only singular at the event horizon (and in fact only microlocally at its conormal bundle), where they lie in $H^{1/2-}$ in the present, scalar setting. Since we are interested in the construction of modes with controlled \emph{decay or growth}, we shall typically not state the precise regularity of dual states; the center of attention is the weight at $\pa_+X$.
\end{rmk}

\begin{proof}[Proof of Proposition~\ref{Prop00Grow}]
  Consider first an element $u\in\ker\wh{\Box_{g_b}}(0)\cap\Hbext^{\infty,-3/2-}$. By normal operator arguments as in Proposition~\ref{PropOpNull},\footnote{Note that passing from the weight $-\tfrac32+0$ to $-\tfrac32-0$, we cross the point $0$ in the boundary spectrum, corresponding to $\lambda=0$ and $l=0$ and the solution $1$ of the normal operator, $\wh{\Box_{\ubar g}}(0)1=0$.} $u$ has an expansion at infinity, $u=u_0+\wt u$, where $u_0$ is constant and $\wt u\in\Hbext^{\infty,-1/2-}$; but then $0=\wh{\Box_{g_b}}(0)u_0+\wh{\Box_{g_b}}(0)\wt u=\wh{\Box_{g_b}}(0)\wt u$ and Theorem~\ref{Thm0} imply $\wt u=0$, hence $u=u_0$ is constant. Conversely, constants do lie in $\ker\wh{\Box_{g_b}}(0)$, proving~\eqref{Eq00Grows0Ker}. The proof of~\eqref{Eq00Grows0KerAdj} is analogous: by normal operator arguments as in Proposition~\ref{PropOpNull}, the space in~\eqref{Eq00Grows0KerAdj} is at most 1-dimensional, and indeed $u_{b,s 0}^*\in\ker\wh{\Box_{g_b}}(0)^*$.

  Passing to the weights~\eqref{Eq00Grows1Ker}, one crosses the point $\lambda=i$ in the boundary spectrum, corresponding to $l=1$ and asymptotics $r\scal$, $\scal\in\scal_1$; since $\dim\scal_1=3$, the space in~\eqref{Eq00Grows1Ker} is at most $4$-dimensional, with its elements equal to $r\scal$ plus functions with faster decay. To prove that it is $4$-dimensional indeed, let $v:=r\scal\in\Hbext^{\infty,-5/2-}$ and fix a cutoff $\chi\in\CI(\R)$ with $\chi\equiv 0$ for $r\leq 3\bhm$, $\chi\equiv 1$ for $r\geq 4\bhm$; then
  \begin{align*}
    e := \wh{\Box_{g_b}}(0)(\chi v) &= \chi\wh{\Box_{\ubar g}}(0)v + [\wh{\Box_{\ubar g}}(0),\chi]v + \bigl(\wh{\Box_{g_b}}(0)-\wh{\Box_{\ubar g}}(0)\bigr)(\chi v) \\
     &\in 0 + \Hbext^{\infty,\infty} + \Hbext^{\infty,1/2-} = \Hbext^{\infty,1/2-},
  \end{align*}
  where we used Lemma~\ref{LemmaKStNormal} for the third summand. Now, $\wh{\Box_{g_b}}(0)w=-e$ can be solved with $w\in\Hbext^{\infty,-3/2-}$; indeed, $e$ is $L^2$-orthogonal to the kernel of $\wh{\Box_{g_b}}(0)^*$ on $\Hbsupp^{-\infty,-1/2+}$, which is trivial by Theorem~\ref{Thm0}. Therefore, $\chi v+w$ furnishes an element in~\eqref{Eq00Grows1Ker} with leading term $r\scal$. The continuous dependence on $b$ is a consequence of this construction.

  A similar argument gives~\eqref{Eq00Grows1KerAdj}, using in the final step now that $\ker\wh{\Box_{g_b}}(0) \cap \Hbext^{\infty,-1/2+}$ is trivial. The explicit expressions~\eqref{Eq00Grows1} on the Schwarzschild spacetime $(M^\circ,g_{b_0})$ are found by solving the radial ODE $\wh{\Box_{g_{b_0}}}(0)\bigl(p(r)\scal\bigr)=0$ for $p(r)$ with $p(r)-r=o(r)$. (Concerning the computation of these expressions for general Schwarzschild metrics with parameters $(\bhm,0)$, note that the stationary operator $\wh{\Box_{g_{(\bhm,0)}}}(0)$ is the same for any two choices of presentations of $g_{(\bhm,0)}$ as a metric on $M^\circ$ that are related by pullback along a diffeomorphism of $M^\circ$ of the form $(t_{\chi_0},r,\theta,\phi)\mapsto(t_{\chi_0}+F(r),r,\theta,\phi)$ for any smooth $F$; therefore, calculations for $g_{(\bhm,0)}$ are the same as those for $g_{(\bhm_0,0)}$ upon replacing $\bhm_0$ by $\bhm$.)
\end{proof}

%%%%%%%%%%%%%%%%%%%%%%%%%%%%%%%%%%%%%%%%%%%%%%%%%%%%%%%%%%%%%%%%%%%%%%
\section{Mode analysis of the 1-form wave operator}
\label{S1}

While we are ultimately interested in the linearized gauge-fixed Einstein operator acting on symmetric 2-tensors, the wave operator on 1-forms, $\Box_g:=\Box_{g,1}$, appears in two different functions when studying the unmodified operator: once as the (unmodified) constraint propagation operator acting on the gauge 1-form, and once as an operator acting on gauge potentials; see the discussion in~\S\ref{SssIIS}. We thus study it here in detail, largely following the notation, and parts of the presentation, of~\cite{HintzVasyKdsFormResonances}.

\begin{thm}
\label{Thm1}
  Let $g=g_{(\bhm_0,0)}$, and consider $\Box_g$ acting on 1-forms.
  \begin{enumerate}
  \item For $\Im\sigma\geq 0$, $\sigma\neq 0$, the operator
    \[
      \wh{\Box_g}(\sigma)\colon\left\{\omega\in \Hbext^{s,\ell}(X;\wt\Tsc{}^*X)\colon\wh{\Box_g}(\sigma)\omega\in\Hbext^{s,\ell+1}(X;\wt\Tsc{}^*X)\right\} \to\Hbext^{s,\ell+1}(X;\wt\Tsc{}^*X)
    \]
    is invertible when $s>\tfrac32$, $\ell<-\half$, $s+\ell>-\half$.
  \item For $s>\tfrac32$ and $\ell\in(-\tfrac32,-\half)$, the stationary operator
    \begin{equation}
    \label{Eq1StatOp}
    \begin{split}
      \wh{\Box_g}(0)\colon&\left\{\omega\in\Hbext^{s,\ell}(X;\wt\Tsc{}^*X)\colon\wh{\Box_g}(0)\omega\in\Hbext^{s-1,\ell+2}(X;\wt\Tsc{}^*X)\right\} \\
        &\qquad\to\Hbext^{s-1,\ell+2}(X;\wt\Tsc{}^*X)
    \end{split}
    \end{equation}
    has 1-dimensional kernel and cokernel.
  \end{enumerate}
  Both statements continue to hold for $g=g_{(\bhm,\bha)}$ with $b=(\bhm,\bha)$ near $b_0=(\bhm_0,0)$. Concretely, there exist $\omega_{b,s 0}\in\Hbext^{\infty,-1/2-}$ and $\omega_{b,s 0}\in\Hbsupp^{-\infty,-1/2-}$, depending continuously on $b$ near $b_0$, such that
  \begin{subequations}
  \begin{alignat}{4}
  \label{Eq1s0Ker}
    &\ker\wh{\Box_{g_b}}(0)\, &\cap\ & \Hbext^{\infty,-1/2-}& &=&\ \la \omega_{b,s 0}\ra, \\
  \label{Eq1s0KerAdj}
    &\ker\wh{\Box_{g_b}}(0)^*\, &\cap\ & \Hbsupp^{-\infty,-1/2-}& &=& \la \omega_{b,s 0}^*\ra.
  \end{alignat}
  \end{subequations}
  Explicitly, using the notation of Proposition~\ref{Prop00Grow}, we can take
  \begin{equation}
  \label{Eq1s0}
  \begin{aligned}
    \omega_{b_0,s 0} = r^{-1}(d t_0-d r), \quad
    &\omega_{b_0,s 0}^*&\!\!\!=\ & d u_{b_0,s 0}^*&\!\!\!=\ &\delta(r-2\bhm_0)d r, \\
    &\omega_{b,s 0}^*&\!\!\!=\ &d u_{b,s 0}^*&\!\!\!=\ &\delta(r-r_{(\bhm,\bha)})d r.
  \end{aligned}
  \end{equation}
\end{thm}

In the spirit of Remark~\ref{Rmk00Reg}, we note that the precise regularity for dual states here is smoothness away from the event horizon, conormality at infinity, and conormality relative to $H^{-1/2-}$ at the event horizon.

\begin{rmk}
\label{Rmk10KerrExpl}
  We can give an explicit expression for $\omega_{b,s 0}$ by noting that the elements $\frac{1}{r^2-2\bhm_0 r}d r$, $r^{-1}d t\in\ker\Box_{g_{b_0},1}$ (in the static patch $r>2\bhm_0$) have analogues on Kerr spacetimes. Namely, the 1-forms
  \[
    \omega^0_{b,1}:=\frac{1}{\Delta_b}\,d r,\quad
    \omega^0_{b,2}:=\frac{r}{\rho_b^2}(d t-\bha\sin^2\theta\,d\phi) = \frac{r}{\rho_b^2}(d t_{b,0}-a\sin^2\theta\,d\varphi_{b,0}) - \frac{r}{\Delta_b}d r
  \]
  lie in $\ker\Box_{g_b^0,1}$ in $r>r_{(\bhm,\bha)}$, and hence so does their linear combination $\omega^0_{b,s 0}=\omega^0_{b,2}+r_b\omega^0_{b,1}$ which depends continuously on $b$ and is smooth across the event horizon. We can then define $\omega_{b,s 0}$ as the pullback $(\Phi_b^0\circ\Phi_b^{-1})^*\omega^0_{b,s 0}$, cf.\ Remark~\ref{RmkKaLie}.
\end{rmk}

\begin{rmk}
\label{Rmk10AdjDecay}
  Regarding the nullspace of the adjoint in~\eqref{Eq1s0KerAdj}, we note that the dual of the target space of $\wh{\Box_g}(0)$ in~\eqref{Eq1StatOp} is $\Hbsupp^{-s+1,-\ell-2}$, whose weight lies in $(-\tfrac32,-\half)$; elements of the kernel of $\wh{\Box_g}(0)^*$ thus automatically have the decay rate $-\half-$ stated in~\eqref{Eq1s0KerAdj} by indicial root (or normal operator) arguments as in Proposition~\ref{PropOpNull}.
\end{rmk}

The arguments in the proof of Theorem~\ref{ThmOp} apply also to $\wh{\Box_{g_b}}(\sigma)$, thus it is Fredholm of index $0$, and the determination of its kernel suffices to determine the dimension of the kernel of the adjoint; we find the latter, when non-trivial, by a simple observation, see~\eqref{Eq10DualKer}. In the Schwarzschild case $b=b_0$, one can prove Theorem~\ref{Thm1} using separation of variables into radial and spherical variables, and expanding further into 1-form spherical harmonics; this is the approach used in the proof of mode stability for the linearized Einstein metric in~\S\ref{SMS}. Here, we instead opt for a more conceptual proof, adapting the arguments of~\cite{HintzVasyKdsFormResonances} to the present asymptotically flat setting.

The main part of the proof is the analysis of the Schwarzschild case; let us thus put $\bhm=\bhm_0$, $g=g_{(\bhm,0)}$, until specified otherwise. Since $g$ is Ricci-flat, the tensor wave operator equals the Hodge d'A\-lem\-bert\-ian,
\[
  \Box_g = (d+\delta_g)^2 = d\delta_g+\delta_g d.
\]
We are specifically interested in its action on 1-forms, though by virtue of $\Box_g$ being a square of $d+\delta_g$, which mixes form degrees, most of our analysis will apply more generally to $\Box_g$ acting on forms without a restriction on their form degree. We shall first prove the absence of mode solutions in $\Im\sigma\geq 0$ in \S\S\ref{Ss1Pos}--\ref{Ss1Re}, and proceed to compute the space of stationary solutions on Schwarzschild spacetimes in~\S\ref{Ss10}; simple perturbative arguments then finish the proof of Theorem~\ref{Thm1} for slowly rotating Kerr metrics.

We write $g=g_{(\bhm,0)}$ in the static patch $\cM$ (see~\eqref{EqK0Static}) as
\begin{equation}
\label{Eq1Metric}
  g = \alpha^2\,d t^2-h,\quad
  h = \alpha^{-2}\,d r^2+r^2\slg,\qquad
  0<\alpha = \sqrt{1-\frac{2\bhm}{r}},
\end{equation}
We use $\alpha$ as a (radial) coordinate; note $\,d\alpha=\frac{\bhm}{r^2}\alpha^{-1}\,d r$. Denoting points on $\Sph^2$ by $y$,
\[
  h = \wt\beta^2\,d\alpha^2 + k(\alpha^2,y,d y),\quad
  \wt\beta=\frac{r^2}{\bhm}=\frac{(2\bhm/(1-\alpha^2))^2}{\bhm},\ \ 
  k=r^2\slg.
\]
The (non-zero) quantity
\[
  \beta:=\wt\beta|_{\alpha=0}=4\bhm,
\]
which is the reciprocal of the surface gravity of the event horizon, will play an important role in describing outgoing asymptotics below.

We set up our analysis by splitting the (full) form bundle on $\cM$:
\begin{equation}
\label{Eq1Split}
  \Lambda\cM = \Lambda\cX \oplus (\alpha\,d t\wedge\Lambda\cX).
\end{equation}
Correspondingly, we write differential forms on $\cM$ as
\[
  \omega(t,x)=\omega_T + \alpha\,d t\wedge\omega_N =: \begin{pmatrix} \omega_T \\ \omega_N \end{pmatrix},
\]
where $\omega_T,\omega_N$ are $t$-dependent forms on $\cX$. In this splitting, we have
\begin{equation}
\label{Eq1Split2}
  d=\begin{pmatrix} d_\cX & 0 \\ \alpha^{-1}\pa_t & -\alpha^{-1}d_\cX\alpha \end{pmatrix},\quad
  G_k = \begin{pmatrix} (-1)^k H_k & 0 \\ 0 & (-1)^{k-1}H_{k-1} \end{pmatrix},
\end{equation}
where $G_k$ is the fiber inner product on $\Lambda^k\cM$ induced by $g$ (thus $G_1=g^{-1}$ is the dual metric), and $H_k$ is the fiber inner product on $\Lambda^k\cX$ induced by $h$. The adjoint of $d$ is thus
\begin{equation}
\label{Eq1delta}
  \delta\equiv\delta_g = \begin{pmatrix} -\alpha^{-1}\delta_\cX\alpha & -\alpha^{-1}\pa_t \\ 0 & \delta_\cX \end{pmatrix},
\end{equation}
where $\delta_\cX$ is the adjoint of $d_\cX$ with respect to $h$ (giving both the volume density on $\cX$ and the fiber inner product).

Recall the two components $\pa_-\cX$ (event horizon) and $\pa_+\cX$ (infinity) of the boundary of the closure $\ol\cX\subset X$ from~\eqref{EqK0StaticBdy}; both are diffeomorphic to $Y:=\Sph^2$.\footnote{In~\cite{HintzVasyKdsFormResonances}, we used the notation $\ol X_{\rm even}$ for the part of (in present notation) $X$ near the event horizon.} Near $\pa_-\cX$, we refine the splitting~\eqref{Eq1Split} by writing
\begin{equation}
\label{Eq1SplitFull}
  \Lambda\cX = \Lambda Y \oplus (d\alpha\wedge\Lambda Y), \quad
  \text{i.e.}\ \ \omega=\omega_{T T}+d\alpha\wedge\omega_{T N} + \alpha\,d t\wedge\omega_{N T}+\alpha\,d t\wedge d\alpha\wedge\omega_{N N}.
\end{equation}
(Note that this is not smooth down to $\pa_-\cX$, as $d\alpha$ is not smooth at $\pa_-\cX$.) In this splitting, we also record, using $[\sldelta,\wt\beta]=0$:
\begin{equation}
\label{Eq1Spatiald}
  d_\cX = \begin{pmatrix} \sld & 0 \\ \pa_\alpha & -\sld \end{pmatrix}, \quad
  \delta_\cX = \begin{pmatrix} \sldelta & \pa_\alpha^* \\ 0 & -\sldelta \end{pmatrix}.
\end{equation}
Here, $\pa_\alpha^*$ is the adjoint of $\pa_\alpha\colon L^2(\cX;|d h|,\Lambda Y,K)\to L^2(\cX;|d h|,\Lambda Y,\wt\beta^{-2}K)$ (with $K$ denoting the metric on forms on $Y$ induced by $k$), so
\[
  \pa_\alpha^* = -\beta^{-2}\pa_\alpha+\alpha^2 p_1\pa_\alpha + \alpha p_2,\quad
  p_1,p_2\in\CI(\ol{\cX}).
\]

Outgoing modes are of the form $e^{-i\sigma t_0}\omega'(x)$ near $\pa_-\cX$, where $t_0=t+r_*$, and with $\omega'$ smooth down to $\pa_-\cX$. Such a mode is a linear combination, with $\CI(\ol{\cX};\Lambda Y)$ coefficients, of $1$, $d(\alpha^2)$, $d t_*$, and $d t_*\wedge d(\alpha^2)$. Using that $r_*=r+2\bhm\log(r-2\bhm)\in\beta\log\alpha+\CI([2\bhm,\infty))$, we find that such forms $e^{-i\sigma t_0}\omega'=:e^{-i\sigma t}\omega$ are precisely those for which $\omega$ has the form
\begin{equation}
\label{Eq1CIsigma}
  \begin{pmatrix}
    \omega_{T T} \\ \omega_{T N} \\ \omega_{N T} \\ \omega_{N N}
  \end{pmatrix}
  =\sC \alpha^{-i\beta\sigma}\begin{pmatrix} \wt\omega_{T T} \\ \wt\omega_{T N} \\ \wt\omega_{N T} \\ \wt\omega_{N N} \end{pmatrix},
  \quad
  \sC:=\begin{pmatrix} 1 & 0 & 0 & 0 \\ 0 & \alpha & \beta\alpha^{-1} & 0 \\ 0 & 0 & \alpha^{-1} & 0 \\ 0 & 0 & 0 & 1 \end{pmatrix},\ \ 
  \wt\omega_{**} \in \CI(\ol\cX;\Lambda Y).
\end{equation}
We denote the space of stationary differential forms $\omega\in\CI(\cX;\Lambda\cX)$ with this structure near $\pa_-\cX$, and which for $\sigma\neq 0$ lie in $e^{i\sigma r_*}\Hb^{\infty,\ell}$ near $\pa_+\cX$ for some $\ell\in\R$, and for $\sigma=0$ lie in  $\Hb^{\infty,\ell}$ near $\pa_+\cX$ for some $\ell\in(-\tfrac32,-\half)$, by
\[
  \CI_{(\sigma)}.
\]
Phrased differently, the membership $\omega\in\CI_{(\sigma)}$ amounts to the smoothness of $e^{-i\sigma t}\omega$ on $M^\circ$ down to $\R_{t_*}\times\pa_-\cX$, together with an outgoing condition at $\R_{t_*}\times\pa_+X$.

We denote the conjugation of $d$ by the $t$-Fourier transform by
\[
  \wh d(\sigma)=e^{i\sigma t}d e^{-i\sigma t},
\]
likewise $\wh\delta(\sigma)=e^{i\sigma t}\delta e^{-i\sigma t}$, acting on $t$-independent forms on $\cM$.\footnote{This is different from the normalization used in~\eqref{Eq0Spectral}.} Note that $\wh d(\sigma)$ and $\wh\delta(\sigma)$ preserve $\CI_{(\sigma)}$, hence $\sC^{-1}\wh d(\sigma)\sC$ and $\sC^{-1}\wh\delta(\sigma)\sC$ preserve $\alpha^{-i\beta\sigma}\CI(\ol{\cX}\setminus\pa_+\cX;(\Lambda Y)^4)$. For later use, we note that the components of $\omega\in\CI_{(\sigma)}$ in~\eqref{Eq1CIsigma} satisfy
\begin{equation}
\label{Eq1CIsigma2}
  \omega_{T T},\,\omega_{N N}\in\alpha^{-i\beta\sigma}\CI,\quad \omega_{T N},\omega_{N T}\in\alpha^{-i\beta\sigma-1}\CI.
\end{equation}
For $\omega$ satisfying only~\eqref{Eq1CIsigma2} (rather than having the precise structure~\eqref{Eq1CIsigma}), it only follows from~\eqref{Eq1Split2}--\eqref{Eq1delta} and \eqref{Eq1Spatiald} that the components of $v=\wh d(\sigma)\omega$ and $w=\wh\delta(\sigma)\omega$ lie in the spaces~\eqref{Eq1CIsigma2} with the exceptions of $v_{N N},w_{T T}\in\alpha^{-i\beta\sigma-2}\CI$. (This is discussed in~\cite[Equations~(3.13)--(3.14)]{HintzVasyKdsFormResonances}.)

Near $\pa_+\cX$, we record that
\begin{equation}
\label{Eq1ddelSc}
  \wh d(\sigma),\ \wh\delta(\sigma) \in \Diffsc^1(\cX;\Lambda\cX\oplus\Lambda\cX),
\end{equation}
using the splitting~\eqref{Eq1Split}; note here that $\alpha\in 1+\rho\CI(\cX)$. More precisely, we can split $\Lambda\cX$ near $\pa_+\cX$, and in fact globally on $\cX$, into
\begin{equation}
\label{Eq1SplitInfty}
  \Lambda\cX = \Lambda(r\,T^*Y) \oplus (d r_*\wedge \Lambda(r\,T^*Y));
\end{equation}
we thus identify $\Lambda\cX\cong\Lambda Y\oplus\Lambda Y$ via $(\eta,\zeta)\mapsto r^{\deg(\eta)}\eta + d r_*\wedge r^{\deg(\zeta)}\zeta$ on the pure form summands of $\Lambda Y$, with $\deg$ denoting the form degree. In view of \eqref{Eq1Metric}, the inner product on $k$-forms is $H_k=\slG_k\oplus\alpha^{-2}\slG_{k-1}$. Correspondingly, we express forms on $\cM$ as finite sums of forms of the type
\begin{equation}
\label{Eq1SplitInfty2}
  \omega=r^{d_{T T}}\omega_{T T}+d r_*\wedge r^{d_{T N}}\omega_{T N}+\alpha\,d t\wedge r^{d_{N T}}\omega_{N T}+\alpha\,d t\wedge d r_*\wedge r^{d_{N N}}\omega_{N N}
\end{equation}
where the $\omega_{\bullet\bullet}$ are $\Lambda^{d_{\bullet\bullet}}Y$-valued forms on $\cX$. The point of rescaling spherical forms according to their degree in~\eqref{Eq1SplitInfty} is that the size of the coefficients of $\omega$ in the basis $d t$, $d x^i$, with $x^1,x^2,x^3$ standard coordinates on $\R^3$, is comparable to the size of $\omega_{\bullet\bullet}$ as forms on $Y=\Sph^2$ with respect to the standard metric. In the splitting~\eqref{Eq1SplitInfty}, one computes that
\begin{equation}
\label{Eq1ddelInfty}
  d_\cX=
  \begin{pmatrix}
    r^{-1}\sld & 0 \\
    \pa_{r_*}+\alpha^2 r^{-1}\deg & -r^{-1}\sld
  \end{pmatrix},\quad
  \delta_\cX=
  \begin{pmatrix}
    r^{-1}\sldelta & -\alpha^{-1}r^{-2}\pa_{r_*} r^2\alpha^{-1}+r^{-1}\deg \\
    0 & -r^{-1}\sldelta
  \end{pmatrix}.
\end{equation}

%%%%%%%%%%%%%%%%%%%%%%%%%%%%%%%%%%%%%%%%%%%%%%%%%%
\subsection{Absence of modes in \texorpdfstring{$\Im\sigma>0$}{the upper half plane}}
\label{Ss1Pos}

We are interested in mode solutions
\[
  \Box_g(e^{-i\sigma t}\omega)=0,\quad \omega\in\CI_{(\sigma)},\ \Im\sigma>0.
\]
In particular, $\omega$ is exponentially decaying at $\pa_+\cX$. This equation is equivalent to
\begin{equation}
\label{Eq1Pos}
  \wh{\Box_g}(\sigma)\omega=\bigl(\wh d(\sigma)+\wh\delta(\sigma)\bigr)^2\omega=0
\end{equation}

\begin{prop}
\label{Prop1Pos}
  Any solution $\omega$ of~\eqref{Eq1Pos} for $\sigma\in\C$, $\Im\sigma>0$, vanishes identically.
\end{prop}
\begin{proof}
  It suffices to show that if $\omega$ is outgoing with
  \begin{equation}
  \label{Eq1Pos2}
    (\wh d(\sigma)+\wh\delta(\sigma))\omega=0,
  \end{equation}
  then $\omega\equiv 0$: applying this first to $(\wh d(\sigma)+\wh\delta(\sigma))\omega$ in place of $\omega$ and subsequently to $\omega$ itself then proves the proposition. Assuming~\eqref{Eq1Pos2}, we obtain $\omega=0$ using an integration by parts argument by following~\cite[\S3.1]{HintzVasyKdsFormResonances} verbatim. Due to the exponential decay of $\omega$ at $\pa_+\cX$, integration by parts is immediately justified there.
\end{proof}

%%%%%%%%%%%%%%%%%%%%%%%%%%%%%%%%%%%%%%%%%%%%%%%%%%
\subsection{Absence of non-zero real modes}
\label{Ss1Re}

Here, the argument differs slightly from the Schwarzschild--de~Sitter case discussed in~\cite{HintzVasyKdsFormResonances} in that the outgoing condition at $\pa_+\cX$ enters; this boundary at infinity is not present for de~Sitter black holes. Concretely, fixing $\sigma\in\R$, $\sigma\neq 0$, the outgoing condition $\omega\in\CI_{(\sigma)}$ implies, using the identification $\Lambda\cM\cong\Lambda\cX\oplus\Lambda\cX$ via~\eqref{Eq1Split}, that in $r\geq r_0\gg 1$, we have (following Proposition~\ref{PropOpNull})
\begin{equation}
\begin{split}
\label{Eq1Re}
  &\omega(r,y) = e^{i\sigma r_*}\bigl(r^{-1}\omega_+(y) + \wt\omega(r^{-1},y)\bigr), \\
  &\qquad\quad \omega_+\in\CI(\pa_+X;\Lambdasc_{\pa_+X}X\oplus\Lambdasc_{\pa_+X}X), \\
  &\qquad\quad \wt\omega(\rho,y)\in \cA^{2-}([0,r_0^{-1})_\rho\times\Sph^2_y;\Lambdasc X\oplus\Lambdasc X).
\end{split}
\end{equation}

\begin{prop}
\label{Prop1Re}
  If $\sigma\in\R$, $\sigma\neq 0$, and $\omega\in\CI_{(\sigma)}$ solves $\wh{\Box_g}(\sigma)\omega=0$, then $\omega\equiv 0$.
\end{prop}
\begin{proof}
  We give the full proof, which is similar to that in~\cite[\S3.2]{HintzVasyKdsFormResonances}, in order to point out where the outgoing condition at infinity comes in. It suffices to show that any outgoing solution of the \emph{first order} equation $(\wh d(\sigma)+\wh\delta(\sigma))\omega=0$ vanishes identically. Expanding this equation into its tangential and normal components, this means
  \begin{equation}
  \label{Eq1ReSys}
    (\alpha d_\cX-\delta_\cX\alpha)\omega_T + i\sigma\omega_N = 0, \quad
    -i\sigma\omega_T + (-d_\cX\alpha+\alpha\delta_\cX)\omega_N = 0.
  \end{equation}
  We apply $(-d_\cX\alpha+\alpha\delta_\cX)$ to the first equation and use the second to obtain the decoupled equation
  \begin{equation}
  \label{Eq1ReDecoupled}
    (d_\cX\alpha\delta_\cX\alpha+\alpha\delta_\cX\alpha d_\cX-d_\cX\alpha^2 d_\cX-\sigma^2)\omega_T = 0.
  \end{equation}
  Applying $d_\cX$ from the left, the outgoing form $v_T:=d_\cX\omega_T$ satisfies the simpler equation
  \begin{equation}
  \label{Eq1ReSimple}
    (d_\cX\alpha\delta_\cX\alpha-\sigma^2)v_T = 0.
  \end{equation}
  It suffices to show that all outgoing solutions of this equation vanish identically; indeed, this would first give $v_T=0$, i.e.\ $d_\cX\omega_T=0$, which by~\eqref{Eq1ReDecoupled} implies $(d_\cX\alpha\delta_\cX\alpha-\sigma^2)\omega_T=0$, hence $\omega_T=0$; using~\eqref{Eq1ReSys} and $\sigma\neq 0$, this implies $\omega_N=0$.
  
  Note that the operator $d_\cX\alpha\delta_\cX\alpha$ in~\eqref{Eq1ReSimple} is formally self-adjoint in $L^2(\cX;\alpha|d h|,\Lambda\cX,H)$. Let now $f\in\CI((2m,\infty))$ denote a positive function with $f(r)=\alpha(r)=(1-2\bhm/r)^{1/2}$ for $2\bhm<r<3\bhm$ and $f(r)=r^{-1}$ for $r>4\bhm$. Fix moreover a cutoff $\chi\in\CI([0,\infty))$ which is identically $0$ on $[0,\half]$ and identically $1$ on $[1,\infty)$. For small $\eps>0$, define then
  \begin{equation}
  \label{Eq1ReCutoff}
    \chi_\eps := \chi(f/\eps) \in \CIc(\cX).
  \end{equation}
  Thus, $\chi_\eps\equiv 1$ when $\alpha\geq\eps$, $r<\eps^{-1}$, while $\alpha\geq\half\eps$, $r<2\eps^{-1}$ on $\supp\chi$. We then evaluate
  \begin{equation}
  \label{Eq1RePair}
  \begin{split}
    0 &= \lim_{\eps\to 0} \la (d_\cX\alpha\delta_\cX\alpha-\sigma^2)v_T,\chi_\eps v_T\ra - \la v_T,\chi_\eps(d_\cX\alpha\delta_\cX\alpha-\sigma^2)v_T\ra \\
    &= \lim_{\eps\to 0} \la v_T,[d_\cX\alpha\delta_\cX\alpha,\chi_\eps]v_T\ra.
  \end{split}
  \end{equation}
  (The localizer $\chi_\eps$ is necessary to make sense of the pairings since $v_T$ fails, just barely, to lie in $L^2$.) The commutator has two pieces: one near $\pa_-\cX$, and one near $\pa_+\cX$. The former piece was analyzed in~\cite[Proof of Proposition~3.6]{HintzVasyKdsFormResonances}: writing $v_T=(\alpha^{-i\beta\sigma}v'_{T T}, \alpha^{-i\beta\sigma-1}v'_{T N})$, its contribution to the above limit is
  \[
    -2 i\beta^{-2}\sigma\|v'_{T N}|_{\pa_-\cX}\|_{L^2(\pa_-\cX;|d k|,K)}^2.
  \]
  The latter piece, at $\pa_+\cX$, can be evaluated using~\eqref{Eq1ddelInfty}: put $\chi_\eps(r)=\chi(1/(\eps r))=:\psi(\eps r)$, with $\psi(r')\equiv 1$, resp.\ $0$ for $r'\leq 1$, resp.\ $r'\geq 2$; then $\pa_r\chi_\eps=\eps\psi_{1,\eps}$, $\psi_{1,\eps}(r):=\psi'(\eps r)$. Thus,
  \[
    [d_\cX\alpha\delta_\cX\alpha,\chi_\eps]
    =
    \eps
    \begin{pmatrix}
      0 & -\alpha^2 r^{-1}\psi_{1,\eps}\sld \\
      \alpha^4 r^{-1}\psi_{1,\eps}\sldelta & \psi_{1,\eps}\bigl(-\alpha^2\pa_{r_*}+\cA^1\bigr)-\pa_{r_*}(1+\cA^1)\psi_{1,\eps}
    \end{pmatrix}.
  \]
  Inserted in the pairing~\eqref{Eq1RePair}, the off-diagonal terms here give terms of the form
  \[
    \int (e^{i\sigma r_*}r^{-1})\cdot \eps r^{-1}\psi_{1,\eps}\cO(1)\overline{e^{i\sigma r_*}r^{-1}} v'' r^2\,d r\,|d\slg|
  \]
  with $v''$ bounded on $\cX$; as $\eps\to 0$ and changing variables to $r'=\eps r$, this is $\eps\int\psi'(r')\cO(1)\,d r'=\cO(\eps)\to 0$. Likewise, the two $\cA^1$ terms on the diagonal give vanishing contributions in the limit $\eps\to 0$. Next, writing $v=r^{d_{T T}}v_{T T}+d r_*\wedge r^{d_{T N}}v_{T N}$ as in~\eqref{Eq1SplitInfty2}, and further writing
  \[
    v_{T N} = e^{i\sigma r_*}(r^{-1}v_{T N,+}+\wt v_{T N}),\quad \wt v_{T N}\in\cA^{2-},
  \]
  we note that $\pa_{r_*}v_{T N}=e^{i\sigma r_*}(i\sigma r^{-1}v_{T N,+}+\wt v_{T N}')$, $\wt v_{T N}'\in\cA^{2-}$, hence in view of $\alpha-1\in r^{-1}\CI(\ol{\cX}\setminus\pa_-\cX)$:
  \[
    \la v_{T N},-\eps\psi_{1,\eps}\alpha^2\pa_{r_*}v_{T N}\ra = i\sigma \iint \eps\psi'(\eps r)|v_{T N,+}|^2\,d r\,|d\slg|+ o(1) \to -i\sigma\|v_{T N,+}\|_{L^2(\Sph^2;\Lambda\Sph^2)}.
  \]
  Similarly, we have $\eps\la v_{T N},-\pa_{r_*}\psi_{1,\eps}v_{T N}\ra\to -i\sigma\|v_{T N,+}\|_{L^2}$ as $\eps\to 0$, as can be seen by first integrating by parts and then taking the limit. In summary,~\eqref{Eq1RePair} reads
  \begin{equation}
  \label{Eq1ReVanish}
    -2 i\sigma\bigl(\beta^{-2}\|v'_{T N}|_{\pa_-\cX}\|_{L^2}^2 + \|v_{T N,+}\|_{L^2}^2\bigr) = 0,
  \end{equation}
  which implies $v'_{T N}|_{\pa_-\cX}\equiv 0$ and $v_{T N,+}\equiv 0$.
  
  A simple indicial root argument then shows that $v_T$ in fact vanishes to infinite order at $\pa_-\cX$, which by Mazzeo's unique continuation result~\cite{MazzeoUniqueContinuation} implies $v_T\equiv 0$ near $\pa_-\cX$, which by standard unique continuation implies $v_T\equiv 0$ globally, finishing the proof. Alternatively, one can study $v_T$ near $\pa_+\cX$ by calculating
  \[
    e^{-i\sigma r_*}(d_\cX\alpha\delta_\cX\alpha-\sigma^2)e^{i\sigma r_*} = \begin{pmatrix} -\sigma^2 & -i\sigma r^{-1}\sld \\ i\sigma r^{-1}\sldelta & -2 i\sigma r^{-1}(r\pa_r+1) \end{pmatrix} + \cA^{2-}\Diffb^2.
  \]
  Applying this to $e^{-i\sigma r_*}v_T=:(\hat v_{T T},\hat v_{T N})$, which is conormal with $\hat v_{T N}\in\cA^{2-}$ in view of~\eqref{Eq1ReVanish}, the first line shows that $\hat v_{T T}\in\cA^{3-}$; the second line then implies, by integration of $r\pa_r+1$ and using $\hat v_{T N}\in\cA^{1+}$, that $\hat v_{T N}\in\cA^{3-}$ as well. Proceeding iteratively, this gives $e^{-i\sigma r_*}v_T\in\cA^\infty$, hence $v_T$ vanishes to infinite order at $\pa_+\cX$. Unique continuation at infinity then implies that $v_T\equiv 0$ near $\pa_+\cX$, and thus globally as before.
\end{proof}

%%%%%%%%%%%%%%%%%%%%%%%%%%%%%%%%%%%%%%%%%%%%%%%%%%
\subsection{Description of zero modes}
\label{Ss10}

Membership in $\omega\in\CI_{(0)}$ together with $\wh{\Box_g}(0)\omega=0$ implies, by normal operator arguments as in Proposition~\ref{PropOpNull}, that
\[
  \omega(r,y) = r^{-1}\omega_+(y) + \wt\omega(r^{-1},y),
\]
with $\omega_+$ and $\wt\omega\in\cA^{2-}$ as in~\eqref{Eq1Re}. We now restrict to \emph{1-forms} on $M^\circ$. Define
\[
  \cK_1=\ker\wh{\Box_g}(0)\cap\CI_{(0)}.
\]
We aim to show that $\cK_1$ is 1-dimensional and spanned by $\omega_{b_0,s 0}$, see~\eqref{Eq1s0}. The proof proceeds in several steps, making use of various calculations of \cite[\S3.3--3.4]{HintzVasyKdsFormResonances}, but simplified and made concrete in the present Schwarzschild 1-form setting. First, writing $\omega\in\CI_{(0)}$ near $\pa_-\cX$ as $\omega=\sC\wt\omega$ as in~\eqref{Eq1CIsigma} (with $\sigma=0$), we define the restriction map
\[
  R \colon \omega \mapsto \wt\omega_{N T}|_{\pa_-\cX}.
\]
Define also the space
\[
  \cH_1 := \{ \omega\in\CI_{(0)} \colon \wh d(0)\omega=\wh\delta(0)\omega=0,\ R\omega=0 \}.
\]
\begin{lemma}
\label{Lemma10Seq}
  If $\omega\in\cK_1$, then $R\omega$ is a constant function. Let $\iota\colon\cH_1\hra\cK_1$ denote the inclusion map; then the following sequence is exact: $\cH_1 \xra{\iota} \cK_1 \xra{R} \C$.
\end{lemma}
\begin{lemma}
\label{Lemma10H1}
  $\cH_1=\{0\}$.
\end{lemma}

These two lemmas together imply that $\dim\cK_1\leq 1$. The statement~\eqref{Eq1s0Ker} of Theorem~\ref{Thm1} thus follows from the fact that $\omega_{b_0,s 0}\in\CI_{(0)}$ and $\wh\Box(0)\omega_{b_0,s 0}=0$, which is a direct calculation using \cite[Lemma~4.3]{HintzVasyKdsFormResonances}. \emph{We now drop `\;\!$(0)$\!' from the notation and write $\wh d\equiv\wh d(0)$, etc.}

\begin{proof}[Proof of Lemma~\ref{Lemma10Seq}]
  Near $\pa_-\cX$, we compute
  \[
    \wh d\sC=
      \begin{pmatrix}
        \sld & 0 & 0 & 0 \\
        \pa_\alpha & -\alpha\sld & -\beta\alpha^{-1}\sld & 0 \\
        0 & 0 & -\alpha^{-1}\sld & 0 \\
        0 & 0 & -\alpha^{-1}\pa_\alpha & \sld
      \end{pmatrix}, \ \ 
    \wh\delta\sC=
      \begin{pmatrix}
        -\sldelta & -\alpha^{-1}\pa_\alpha^*\alpha^2 & -\beta\alpha^{-1}\pa_\alpha^* & 0 \\
        0 & \alpha\sldelta & \beta\alpha^{-1}\sldelta & 0 \\
        0 & 0 & \alpha^{-1}\sldelta & \pa_\alpha^* \\
        0 & 0 & 0 & -\sldelta
      \end{pmatrix},
  \]
  and
  \[
    \sC^{-1}\wh d\sC=
      \openbigpmatrix{3pt}
        \sld & 0 & 0 & 0 \\
        \alpha^{-1}\pa_\alpha & -\sld & 0 & 0 \\
        0 & 0 & -\sld & 0 \\
        0 & 0 & -\alpha^{-1}\pa_\alpha & \sld
      \closebigpmatrix, \ \ 
    \sC^{-1}\wh\delta\sC=
      \openbigpmatrix{3pt}
        -\sldelta & -\alpha^{-1}\pa_\alpha^*\alpha^2 & -\beta\alpha^{-1}\pa_\alpha^* & 0 \\
        0 & \sldelta & 0 & -\beta\alpha^{-1}\pa_\alpha^* \\
        0 & 0 & \sldelta & \alpha\pa_\alpha^* \\
        0 & 0 & 0 & -\sldelta
      \closebigpmatrix.
  \]
  Since $\alpha\pa_\alpha^*\colon\CI(\ol{\cX})\to\alpha^2\CI(\ol{\cX})$, the $(N T)$ coefficient of $\sC^{-1}\wh\Box\sC\wt\omega$ is therefore equal to $-\slDelta\wt\omega_{N T}+\alpha^2\CI$; since this vanishes when $\omega=\sC\wt\omega\in\cK_1$, we find that $\wt\omega_{N T}|_{\pa_-\cX}$ is a harmonic 0-form on $\pa_-\cX\cong\Sph^2$, hence constant. This proves the first claim.

  Now if $\omega\in\cK_1$, $\omega=\sC\wt\omega$, has $R\omega=0$, i.e.\ $\wt\omega_{N T}|_{\pa_-\cX}=0$, we proceed as follows. First, given $\wh\Box\omega=0$, we can apply either $\wh d$ or $\wh\delta$, obtaining
  \begin{equation}
  \label{Eq10Seqddd}
    \wh d\wh\delta\wh d\omega=0,\quad
    \wh\delta\wh d\wh\delta\omega=0.
  \end{equation}

  Consider the first equation, and write $\omega'=\wh d\omega$, which is outgoing and decays faster, by a factor $r^{-1}$, than $\omega$ itself. Writing $\omega'=\sC\wt\omega'$, we have $\wt\omega'_{N T}|_{\pa_-\cX}=0$. Note that $\wh d$ and $-\wh\delta$ are adjoints of each other with respect to the $L^2$ pairing with volume density $\alpha|d h|$ and fiber inner product $H\oplus H$. \emph{We work with this inner product from now on, unless otherwise specified}. With $\chi,\chi_\eps$ as in~\eqref{Eq1ReCutoff}, we then have
  \begin{equation}
  \label{Eq10Seqddel}
    0 = -\lim_{\eps\to 0} \la \wh d\wh\delta\omega',\chi_\eps\omega'\ra = \lim_{\eps\to 0} \la \wh\delta\omega',\wh\delta\chi_\eps\omega'\ra = \lim_{\eps\to 0} \|\chi_\eps^{1/2}\wh\delta\omega'\|^2 + \la \wh\delta\omega',[\wh\delta,\chi_\eps]\omega'\ra.
  \end{equation}
  Let us evaluate the second term on the right. It has two pieces; using the $\cO(r^{-1})$, resp.\ $\cO(r^{-2})$ decay of $\omega'$, resp.\ $\wh\delta\omega'$, the piece near $\pa_+\cX$ is easily shown to vanish in the limit. Near $\pa_-\cX$, we write $\chi_{1,\eps}:=\chi'(f/\eps)=\chi'(\alpha/\eps)$ and compute
  \[
    [\wh\delta\sC,\chi_\eps]\in
      \eps^{-1}\chi_{1,\eps}
      \begin{pmatrix}
        0 & \beta^{-2}\alpha+\alpha^3\CI & \beta^{-1}\alpha^{-1}+\alpha\CI & 0 \\
        0 & 0 & 0 & 0 \\
        0 & 0 & 0 & -\beta^{-2}+\alpha^2\CI \\
        0 & 0 & 0 & 0
      \end{pmatrix}.
  \]
  Thus, using the vanishing of $\wt\omega'_{N T}$ at $\pa_-\cX$, we have
  \[
    [\wh\delta\sC,\chi_\eps]\wt\omega' \in \eps^{-1}\chi_{1,\eps}\begin{pmatrix} \alpha\CI \\ 0 \\ -\beta^{-2}\wt\omega'_{N N}+\alpha^2\CI \\ 0 \end{pmatrix}.
  \]
  We furthermore compute
  \begin{align*}
    (\wh\delta\sC\wt\omega')_{T T} &\in -\sldelta\wt\omega'_{T T} + 2\beta^{-2}\wt\omega'_{T N}-\beta\alpha^{-1}\pa_\alpha^*\wt\omega'_{N T} + \alpha^2\CI, \\
    (\wh\delta\sC\wt\omega')_{N T} &\in \alpha\CI.
  \end{align*}
  Since $H\oplus H=K\oplus\wt\beta^{-2}K\oplus K\oplus\wt\beta^{-2}K$ is block-diagonal, we thus see that the pointwise inner product of $\wh\delta\omega'$ and $[\wh\delta,\chi_\eps]\omega'$ is bounded by $\alpha\eps^{-1}\chi_{1,\eps}$; using $\alpha|d h|=\alpha\wt\beta\,d\alpha|d k|$, it follows easily that $\lim_{\eps\to 0}\la\wh\delta\omega',[\wh\delta,\chi_\eps]\omega'\ra = 0$, and therefore~\eqref{Eq10Seqddel} implies $\wh\delta\omega'=0$, so $\wh\delta\wh d\omega=0$.

  Consider similarly the second equation~\eqref{Eq10Seqddd}, and let $\omega''=\wh\delta\omega$: we now have
  \begin{equation}
  \label{Eq10Seqdeld}
    0 = \lim_{\eps\to 0} \|\chi_\eps^{1/2}\wh d\omega''\|^2 + \la\wh d\omega'',[\wh d,\chi_\eps]\omega''\ra.
  \end{equation}
  The contribution of the second term near $\pa_+\cX$ again vanishes in the limit, while the contribution near $\pa_-\cX$ can be evaluated using
  \[
    [\wh d\sC,\chi_\eps] = \eps^{-1}\chi_{1,\eps}
      \begin{pmatrix}
        0 & 0 & 0 & 0 \\
        1 & 0 & 0 & 0 \\
        0 & 0 & 0 & 0 \\
        0 & 0 & -\alpha^{-1} & 0
      \end{pmatrix},
  \]
  which for $\wt\omega''=\sC^{-1}\omega''$ with $\wt\omega''_{N T}|_{\pa_-\cX}=0$ gives
  \[
    [\wh d\sC,\chi_\eps]\wt\omega'' \in \eps^{-1}\chi_{1,\eps}\begin{pmatrix} 0 \\ \CI \\ 0 \\ \alpha\CI \end{pmatrix};
  \]
  on the other hand, $(\wh d\sC\wt\omega'')_{T N} \in \alpha\CI$ and $(\wh d\sC\wt\omega'')_{N N} \in \CI$. Thus, as in the previous calculation, the pairing is pointwise bounded by $\alpha\eps^{-1}\chi_{1,\eps}$, whose integral with respect to $\alpha\wt\beta\,d\alpha|d k|$ tends to $0$ as $\eps\to 0$. We conclude that $\wh d\omega''=0$, so $\wh d\wh\delta\omega=0$.

  We now repeat these arguments: replacing $\omega'$ by $\omega$ in~\eqref{Eq10Seqddel}, we find $\wh\delta\omega=0$; using~\eqref{Eq10Seqdeld} with $\omega$ in place of $\omega''$ gives $\wh d\omega=0$. Therefore, $\omega\in\cH_1$, finishing the proof.
\end{proof}

\begin{proof}[Proof of Lemma~\ref{Lemma10H1}]
  Let us write an outgoing 1-form as
  \[
    \omega=\omega_T+\omega_N\alpha\,d t.
  \]
  Then $\omega_N$ is a function which lies in $\alpha^{-1}\CI(\ol\cX\setminus\pa_-\cX)$ and takes the form $r^{-1}\omega_{N,+}+\cA^{2-}$ near $\pa_+\cX$; moreover,
  \begin{equation}
  \label{Eq10H1Split}
    \omega_T=\omega_{T T}+\omega_{T N}\,d r,
  \end{equation}
  where $\omega_{T T}\in\CI$ and $\omega_{T N}\in\alpha^{-2}\CI$ near $\pa_-\cX$; near $\pa_+X$ on the other hand, $\omega_{T N}=r^{-1}\omega_{T N,+}+\cA^{2-}$ and $\omega_{T T}=\omega_{T T,+}+\cA^{1-}$ with $\omega_{T N,+}\in\CI(\pa_+X)$, $\omega_{T T,+}\in\CI(\pa_+X;T^*\Sph^2)$.

  Since $\wh d$ and $\wh\delta$ are diagonal in the splitting~\eqref{Eq1Split}, $\omega_T$ and $\omega_N$ satisfy decoupled equations. For the function $\omega_N$, this gives $d_\cX\alpha\omega_N=0$, hence $\omega_N=c\alpha^{-1}=c(1+\cO(r^{-1}))$ for some $c\in\C$; for this to be $\cO(r^{-1})$ as $r\to\infty$, we need $c=0$, hence $\omega_N\equiv 0$.

  For the 1-form $\omega_T$, we first use $d_\cX\omega_T=0$, i.e.
  \[
    \sld\omega_{T T}=0,\quad
    \pa_r\omega_{T T}-\sld\omega_{T N}=0.
  \]
  Since $H^1(\Sph^2)=0$, we can write $\omega_{T T}=\sld\phi$ with $\phi$ smooth down to $\pa_-\cX$ and of the form
  \[
    \phi = \phi_+ +\cA^{1-}\ \ \text{near}\ \ \pa_+\cX,\quad \phi_+\in\CI(\pa_+\cX).
  \]
  The second equation then gives $\sld(\pa_r\phi-\omega_{T N})=0$, hence
  \begin{equation}
  \label{Eq10H1Potential}
    (\omega_{T T},\omega_{T N})=(\sld\phi,\pa_r\phi);
  \end{equation}
  there is no constant of integration since $\omega_{T N}$ is required to decay as $r\to\infty$.

  Finally, we use $\delta_\cX\omega_T=0$ and the expression $\delta_\cX=(r^{-2}\sldelta,-\alpha r^{-2}\pa_r\alpha r^2)$ in the splitting~\eqref{Eq10H1Split}. Expressing this equation in terms of $\phi$ gives
  \[
    \wh{\Box_{g,0}}(0)\phi = 0,\quad \wh{\Box_{g,0}}(0)=r^{-2}\pa_r\alpha^2 r^2\pa_r-r^{-2}\slDelta.
  \]
  By the normal operator calculations in the proof of Proposition~\ref{Prop00Grow}, $\phi_+$ must be constant and can thus be assumed to be zero by replacing $\phi$ by $\phi-\phi_+$ (which preserves~\eqref{Eq10H1Potential}); therefore $\phi\in\cA^{1-}=\Hbext^{\infty,-1/2-}$. Theorem~\ref{Thm0} gives $\phi=0$, and by~\eqref{Eq10H1Potential}, we get $\omega_T\equiv 0$.
\end{proof}

Since $\wh{\Box_{g_{b_0}}}(0)$ is Fredholm of index $0$, this proves the 1-dimensionality of~\eqref{Eq1s0KerAdj} for $b=b_0$. But in fact, we have, for all $b$,
\begin{equation}
\label{Eq10DualKer}
  \wh{\Box_{g_b,1}}(0)(d u_{b,s 0}^*) = d\bigl(\wh{\Box_{g_b,0}}(0)u_{b,s 0}^*\bigr) = d(0) = 0.
\end{equation}
This conversely implies that~\eqref{Eq1s0Ker} is at least $1$-dimensional for $b$ near $b_0$.

A slight extension of the argument in the proof of Theorem~\ref{Thm0} proves that $\cK_{b,1}:=\ker\wh{\Box_{g_b}}(0)\cap\Hbext^{\infty,-1/2-}$ is \emph{at most} 1-dimensional for $b$ near $b_0$. Indeed, fix $\eta\in\CIdotc(X^\circ;T^*X^\circ)$ such that $\la\omega_{b_0,s 0},\eta\ra=1$. Assuming that $\dim\ker\wh{\Box_{g_{b_j}}}(0)\geq 2$ for a sequence of parameters $b_j\to b_0$, $j\to\infty$, we can select $u_j$ in $\cK_{b,1}\cap\eta^\perp$ of norm $\|u_j\|_{\Hbext^{s,\ell}}=1$, where we fix $s>\tfrac32$, $\ell\in(-\tfrac32,-\half)$. As in the proof of Theorem~\ref{Thm0}, we can pass to a weakly convergent (in $\Hbext^{s,\ell}$) subsequence $u_j\weakto u\neq 0$; but $0=\la u_j,\eta\ra\to\la u,\eta\ra=1$, which is a contradiction.

This finishes the proof of Theorem~\ref{Thm1}.

%%%%%%%%%%%%%%%%%%%%%%%%%%%%%%%%%%%%%%%%%%%%%%%%%%
\subsection{Growing and generalized zero modes}
\label{Ss10Gen}

We shall later see that 1-form zero energy states with more decay at infinity may generate asymptotic symmetries which contribute to the kernel of the linearized gauge-fixed Einstein operator. This section can be skipped at first reading; its purpose as well as the motivation behind the constructions in its proof will become clear in~\S\ref{SL}.

\begin{prop}
\label{Prop10Grow}
  For $b=(\bhm,\bha)$ near $b_0=(\bhm_0,0)$, we have
  \begin{subequations}
  \begin{alignat}{3}
  \label{Eq10Grows0Ker}
    &\ker\wh{\Box_{g_b}}(0) \cap \Hbext^{\infty,-3/2-}& =&\ \la\omega_{b,s 0}\ra \oplus \la\omega_{b,s 0}^{(0)}\ra \oplus \{\omega_{b,s 1}(\scal) \colon \scal\in\scal_1\}, \\
  \label{Eq10Grows0KerAdj}
    &\ker\wh{\Box_{g_b}}(0)^* \cap \Hbsupp^{-\infty,-3/2-}& =&\ \la\omega_{b,s 0}^*\ra \oplus \{\omega_{b,s 1}^*(\scal) \colon \scal\in\scal_1\},
  \end{alignat}
  \end{subequations}
  where, with $\flat$ denoting the musical isomorphism $V^\flat:=g_b(V,-)$, and using~\eqref{Eq00Grows1Ker}--\eqref{Eq00Grows1KerAdj},
  \begin{alignat}{3}
  \label{Eq10Grows0}
    &\omega_{b,s 0}^{(0)}=\pa_t^\flat,&& \\
  \label{Eq10Grows1}
    &\omega_{b,s 1}(\scal)=d u_{b,s 1}(\scal),&\quad
    &\omega_{b,s 1}^*(\scal)=d u_{b,s 1}^*(\scal).
  \end{alignat}
\end{prop}
\begin{proof}
  This follows from a normal operator argument, using that the normal operator $\wh{\Box_{\ubar g,1}}(0)$ (see Lemma~\ref{LemmaKStNormal}) annihilates the 1-forms
  \begin{equation}
  \label{Eq10GrowsNormal}
    d t,\ d x^1,\ d x^2,\ d x^3;
  \end{equation}
  the first one is of scalar type $l=0$, while the latter three are of scalar type $l=1$.
  
  We first construct the space on the right in~\eqref{Eq10Grows0Ker}. Let $v$ be one of 1-forms in~\eqref{Eq10GrowsNormal}, in particular $v\in\Hbext^{\infty,-3/2-}$; then with a radial cutoff $\chi$, identically $1$ near infinity and vanishing for $r<3\bhm$, we have
  \[
    e:=\wh{\Box_{g_b}}(0)(\chi v) \in \Hbext^{\infty,3/2-},
  \]
  with the extra order of vanishing due to the normal operator annihilating $v$; moreover, $\supp e\cap\supp\omega_{b,s 0}^*=\emptyset$, the latter being a $\delta$-distribution at the event horizon. Therefore, we can find $w\in\Hbext^{\infty,-1/2-}$ with $\wh{\Box_{g_b}}(0)w=-e$, and $\chi v+w$ furnishes an element of~\eqref{Eq10Grows0Ker}. As $v$ varies over $\mathspan\{d t,d x^1,d x^2,d x^3\}$, we obtain the 4-dimensional supplement to $\la\omega_{b,s 0}\ra$ as in~\eqref{Eq10Grows0Ker}, with continuous dependence on $b$. The explicit expressions given in~\eqref{Eq10Grows0} and \eqref{Eq10Grows1} are of size $\cO(1)$ and thus lie in the desired space. (Note that $\pa_t$ is Killing, hence $\delta_{g_b}^*(\pa_t^\flat)=0$ and so $\Box_{g_b,1}(\pa_t^\flat)=2\delta_{g_b}\sfG_{g_b}\delta_{g_b}^*(\pa_t^\flat)=0$.)

  Next, note that the right hand side of~\eqref{Eq10Grows0KerAdj} indeed lies in the space on the left hand side. Arguing more robustly in the Schwarzschild case $b=b_0$, the 1-forms $v=\chi d x^i$, $i=1,2,3$, can be corrected similarly as above by decaying 1-forms with supported character at $\pa_-X$, giving the zero modes $\omega_{b_0,s 1}^*(\scal)$; this uses that $\ker\wh{\Box_{g_{b_0}}}(0)\cap\Hbext^{\infty,-3/2+}=\la\omega_{b_0,s 0}\ra$ (which is of scalar type $l=0$) is orthogonal to the error term $\wh{\Box_{g_{b_0}}}(0)^*(\chi v)\in\Hbsupp^{-\infty,3/2-}$ (which of scalar type $l=1$). On the other hand, $v=\chi\pa_t^\flat$ \emph{cannot} be corrected in this fashion since this orthogonality fails; indeed, we have
  \begin{equation}
  \label{Eq10NoDualdt}
    \big\la\wh{\Box_{g_{b_0}}}(0)^*(\chi\pa_t^\flat),\omega_{b_0,s 0}\big\ra=2,
  \end{equation}
  where we use the volume density and fiber inner product induced by $g_{b_0}$ in the pairing and in the definition of the adjoint. (The resulting $L^2$-type pairing is \emph{not} positive definite, but still non-degenerate, which is all that is needed here.)

  In order to prove `$\subseteq$' in~\eqref{Eq10Grows0Ker}, note that any $\omega\in\ker\wh{\Box_{g_b}}(0)\cap\Hbext^{\infty,-3/2-}$ is of the form $\omega=\chi\,v+\tilde\omega$ where $v$ is a linear combination (with constant coefficients) of the 1-forms~\eqref{Eq10GrowsNormal}, and $\tilde\omega\in\Hbext^{\infty,-1/2-}$; this follows from (the proof of) Proposition~\ref{PropOpNull}. Upon subtracting a linear combination of $\omega_{b,s 0}^{(0)}$ and $\omega_{b,s 1}(\scal)$ from $\omega$, we can thus assume $\omega=\tilde\omega$, which by Theorem~\ref{Thm1} is a scalar multiple of $\omega_{b,s 0}$.

  The argument for proving `$\subseteq$' in~\eqref{Eq10Grows0KerAdj} is slightly more subtle in view of the obstruction~\eqref{Eq10NoDualdt} to the existence of a mode with $\pa_t^\flat$ asymptotics. Now, $\omega^*\in\ker\wh{\Box_{g_b}}(0)^*\cap\Hbsupp^{-\infty,-3/2-}$ can be written as $\omega^*=\chi v+\tilde\omega^*$ where $v=v_0\,d t+v'$ with $v_0\in\C$ and $v'$ a linear combination of $d x^1,d x^2,d x^3$. Upon subtracting $\omega_{b,s 1}^*(\scal)$ for a suitable $\scal\in\scal_1$, we can assume $v'=0$. Therefore
  \[
    v_0\wh{\Box_{g_b}}(0)^*(\chi\,d t)=-\wh{\Box_{g_b}}(0)^*\tilde\omega^*
  \]
  is necessarily orthogonal to $\ker\wh{\Box_{g_b}}(0)\cap\Hbext^{\infty,-1/2-}=\la\omega_{b,s 0}\ra$, which in view of~\eqref{Eq10NoDualdt} (and continuity in $b$) implies $v_0=0$, thus $\omega^*=\tilde\omega^*$ is a scalar multiple of $\omega_{b,s 0}^*$ by Theorem~\ref{Thm1}.
\end{proof}

\begin{rmk}
  This is an instance of the relative index theorem \cite[Theorem~6.5]{MelroseAPS} (albeit in a non-elliptic setting): allowing more growth as in~\eqref{Eq10Grows0Ker} and thereby crossing the indicial root $0$, with $4$-dimensional space of resonant states, shifts the index by $4$; the cokernel, consisting of a $\delta$-distribution (which lies in every weighted Sobolev space with below-threshold regularity), remains unchanged, and therefore the dimension of the kernel must increase by $4$. On the dual side~\eqref{Eq10Grows0KerAdj} on the other hand, the index still shifts by $4$ when crossing the indicial root $0$, but now the cokernel (of the adjoint), consisting of $\omega_{b,s 0}\in\Hbext^{\infty,-1/2-}$, disappears as one imposes more decay. Therefore, the dimension of the kernel (of the adjoint) increases only by $4-1=3$.
\end{rmk}

Beyond the `asymptotic translations' $\omega_{b,s 1}$, we have `asymptotic rotations':
\begin{prop}
\label{Prop10Genv1}
  There exist continuous families (with $b$ near $b_0$)
  \[
    b\mapsto \omega_{b,v 1}(\vect) \in \ker\wh{\Box_{g_b}}(0)\cap\Hbext^{\infty,-5/2-}, \quad
    b\mapsto \omega_{b,v 1}^*(\vect) \in \ker\wh{\Box_{g_b}}(0)^*\cap\Hbsupp^{-\infty,-5/2-},
  \]
  linear in $\vect\in\vect_1$, which satisfy
  \begin{equation}
  \label{Eq10Genv1}
    \omega_{b_0,v 1}(\vect)=r^2\vect,\quad
    \omega_{b_0,v 1}^*(\vect)=r^2\vect H(r-2\bhm_0),
  \end{equation}
  and which are such that $\delta_{g_b}^*\omega_{b,v 1}(\vect)\in\Hbext^{\infty,1/2-}$ and $\delta_{g_b}^*\omega_{b,v 1}^*(\vect)\in\Hbsupp^{-\infty,-1/2-}$.
\end{prop}
\begin{proof}
  This follows again from a normal operator argument. Indeed, let $\vect\in\vect_1$ and put $v=r^2\vect\in\Hbext^{\infty,-5/2-}$, which on Schwarzschild spacetimes is a 1-form dual to a rotation vector field and thus lies in $\ker\delta_{g_{b_0}}^*\subset\ker\Box_{g_{b_0}}$. On Kerr spacetimes, we need to correct it: fixing a cutoff $\chi$ as in the previous proof, and writing $b=(\bhm,\bha)$, we have
  \begin{align*}
    e := \wh{\Box_{g_{(\bhm,\bha)}}}(0)(\chi v) &= \chi\wh{\Box_{g_{(\bhm,0)}}}(0)v + [\wh{\Box_{g_{(\bhm,0)}}}(0),\chi]v + \bigl(\wh{\Box_{g_{(\bhm,\bha)}}}(0)-\wh{\Box_{g_{(\bhm,0)}}}(0)\bigr)(\chi v) \\
      &\in 0 + \Hbext^{\infty,\infty} + \Hbext^{\infty,3/2-},
  \end{align*}
  where we use that the operator in the third summand lies in $\rho^4\Diffb^2$ by~\eqref{EqKStBoxLeading}. Since $\ker\wh{\Box_{g_b}}(0)^*\cap\Hbsupp^{-\infty,-3/2+}=\la\omega_{b,s 0}^*\ra$ consists of $\delta$-distributions with support disjoint from $e$, we can solve away the error, $-e=\wh{\Box_{g_b}}(0)w$, with $w\in\Hbext^{\infty,-1/2-}$; we then put $\omega_{b,v 1}(\vect):=\chi v+w$. Its symmetric gradient is
  \begin{align*}
    \delta_{g_b}^*\omega_{b,v 1}(\vect) &= [\delta_{g_b}^*,\chi]r^2\vect + \chi(\delta_{g_{(\bhm,\bha)}}^*-\delta_{g_{(\bhm,0)}}^*)r^2\vect + \delta_{g_b}^*w \\
      &\in \Hbext^{\infty,\infty} + \Hbext^{\infty,1/2-} + \Hbext^{\infty,1/2-},
  \end{align*}
  where for the second term we used~\eqref{EqKStSymmGradLeading}.

  For the dual state, we argue similarly; the error term $\wh{\Box_{g_b}}(0)^*(\chi v)\in\Hbsupp^{-\infty,3/2-}\subset\Hbsupp^{-\infty,1/2-}$ can now be solved away by $\wh{\Box_{g_b}}(0)^*w$, $w\in\Hbsupp^{-\infty,-3/2-}$, in view of the triviality of $\ker\wh{\Box_{g_b}}(0)\cap\Hbext^{\infty,-1/2+}$ proved in Theorem~\ref{Thm1}. (Note that we give up one order of decay here compared to the construction of $\omega_{b,v 1}(\vect)$.) The explicit form of $\omega_{b_0,v 1}^*$ in~\eqref{Eq10Genv1} can be verified by a direct calculation.
\end{proof}

Another useful family of 1-forms is the following; it plays a technical role in the sequel.
\begin{lemma}
\label{Lemma10Gens12}
  There exist continuous families
  \[
    b\mapsto\omega_{b,s 1}^{(1)}(\scal) \in \ker\wh{\Box_{g_b}}(0) \cap \Hbext^{\infty,-5/2-}, \quad
    b\mapsto\omega_{b,s 1}^{(1)*}(\scal) \in \ker\wh{\Box_{g_b}}(0)^* \cap \Hbsupp^{-\infty,-5/2-},
  \]
  depending linearly on $\scal\in\scal_1$, with $b$ near $b_0$, with specified leading order term:
  \[
    \omega_{b,s 1}^{(1)}(\scal)-r\scal\,d t_{\chi_0}\in\Hbext^{\infty,-3/2-},\quad
    \omega_{b,s 1}^{(1)*}(\scal)-\chi r\scal\,d t_{\chi_0}\in\Hbsupp^{-\infty,-3/2-},
  \]
  with $\chi\in\CI$, $\chi=0$ for $r\leq 3\bhm_0$, $\chi=1$ for $r\geq 4\bhm_0$. Moreover, $\omega_{b_0,s 1}^{(1)}=r(\mu\,d t_0-d r)\scal$.
\end{lemma}
\begin{proof}
  The key is that $\wh{\Box_{\ubar g}}(0)$ annihilates $x^i\,d t$. Since $r\scal\,d t_{\chi_0}\in\Hbext^{\infty,-5/2-}$, we thus have
  \[
    \wh{\Box_{g_b}}(0)(\chi r\scal\,d t_{\chi_0}) \in \Hbext^{\infty,1/2-},
  \]
  and this is orthogonal (by support considerations) to $\ker\wh{\Box_{g_b}}(0)^*\cap\Hbsupp^{-\infty,-1/2+}=\la\omega_{b,s 0}^*\ra$, hence can be written as $\wh{\Box_{g_b}}(0)w$ with $w\in\Hbext^{\infty,-3/2-}$. Then $\omega_{b,s 1}^{(1)}(\scal):=\chi r\scal\,d t_{\chi_0}+w$ is the desired 1-form. The explicit expression on Schwarzschild spacetimes is obtained by a direct calculation following these steps.

  For the dual states, the error term is now $\wh{\Box_{g_b}}(0)^*(\chi r\scal\,d t_{\chi_0})\in\Hbsupp^{-\infty,1/2-}$; but since $\ker\wh{\Box_{g_b}}(0)\cap\Hbext^{\infty,-1/2+}=0$, this error term can be solved away as claimed.
\end{proof}

For later use, we also record the leading order terms of the symmetric gradient,
\begin{equation}
\label{Eq10Gens12SymGrad}
  \delta_{g_b}^*\omega_{b,s 1}^{(1)}(\scal) = d t_{\chi_0} \otimes_s d(r\scal) + \Hbext^{\infty,-1/2-}.
\end{equation}
Indeed, we can replace $g_b$ by $\ubar g$ and $\omega_{b,s 1}^{(1)}(\scal)$ by $r\scal\,d t_{\chi_0}$ modulo error terms in $\Hbext^{\infty,-1/2-}$; the calculation then becomes straightforward.

Lastly, we discuss generalized zero modes. As a simple instance of degeneracy/non-degeneracy considerations in~\S\ref{SL}, we prove:
\begin{lemma}
\label{Lemma10NoLin}
  For $b$ close to $b_0$, there does not exist a 1-form $\omega=t_*\omega_1+\omega_0$ with $\omega_0,\omega_1\in\Hbext^{\infty,-3/2+}$ and $\omega_1\neq 0$ so that $\Box_{g_b}\omega=0$.
\end{lemma}
\begin{proof}
  Consider first the Schwarzschild case. Given $\omega$ of this type solving
  \[
    0=\Box_{g_{b_0}}\omega = t_*\wh{\Box_{g_{b_0}}}(0)\omega_1 + \bigl([\Box_{g_{b_0}},t_*]\omega_1+\wh{\Box_{g_{b_0}}}(0)\omega_0\bigr),
  \]
  we deduce that $\wh{\Box_{g_{b_0}}}(0)\omega_1=0$, hence (after rescaling by a non-zero constant) $\omega_1=\omega_{b_0,s 0}$. Since
  \begin{equation}
  \label{Eq10NoLinComm}
    [\Box_{g_{b_0}},t_*]\ftrans(0)=i\pa_\sigma\wh{\Box_{g_{b_0}}}(0)\in\rho^2\Diffb^1
  \end{equation}
  by Lemma~\ref{LemmaOpLinFT}, we conclude that
  \[
    \wh{\Box_{g_{b_0}}}(0)\omega_0 = -[\Box_{g_{b_0}},t_*]\omega_1 \in \Hbext^{\infty,1/2-}.
  \]
  The existence of a 1-form $\omega_0\in\Hbext^{\infty,\ell}$ for some $\ell\in\R$ satisfying this equation requires that the right hand side be orthogonal to $\ker\wh{\Box_{g_{b_0}}}(0)\cap\Hbsupp^{-\infty,-\ell-2}\ni\omega_{b_0,s 0}^*$. However,
  \begin{equation}
  \label{Eq10NoLinPair}
    \la[\Box_{g_{b_0}},t_*]\omega_{b_0,s 0},\omega_{b_0,s 0}^*\ra = 2 \neq 0;
  \end{equation}
  indeed, this holds for $t_0$ in place of $t_*$ (using $[\Box_{g_{b_0}},t_0]\omega_{b_0,s 0}=-2\bhm_0 r^{-3}d t_0$), and changing back from $t_0$ to $t_*$ gives a vanishing correction since for $f(r)=t_0-t_*=2 r+4\bhm_0\log(r)$, the 1-form $[\Box_{g_{b_0}},f(r)]\omega_{b_0,s 0} = \wh{\Box_{g_{b_0}}}(0)(f(r)\omega_{b_0,s 0})$ is orthogonal to $\omega_{b_0,s 0}^*$ (which annihilates the range of $\wh{\Box_{g_{b_0}}}(0)$ acting on $\Hbext^{\infty,-3/2-}$).
\end{proof}

By continuity in the parameter $b$, the pairing in~\eqref{Eq10NoLinPair} remains non-zero for $b$ near $b_0$:
\begin{equation}
\label{Eq10NoLinPairKerr}
  \la[\Box_{g_b},t_*]\omega_{b,s 0},\omega_{b,s 0}^*\ra \neq 0,\quad b\text{ near }b_0.
\end{equation}
The relationship between the (non)degeneracy of such pairings and the (non)existence of solutions which grow linearly in time will play a major role in~\S\ref{SL}.

On the other hand, there do exist linearly growing generalized modes with \emph{less} restrictive decay conditions at infinity, as well as linearly growing generalized \emph{dual} zero modes. This includes the `asymptotic Lorentz boosts' $\hat\omega_{b,s 1}(\scal)$ below:
\begin{prop}
\label{Prop10GenSc}
  There exist continuous families
  \begin{alignat*}{2}
    b &\mapsto \hat\omega_{b,s 0}^* &\,\in\,& \ker\Box_{g_b} \cap \Poly^1(t_*)\Hbsupp^{-\infty,-3/2-}, \\
    b &\mapsto \hat\omega_{b,s 1}(\scal) &\,\in\,& \ker\Box_{g_b} \cap \Poly^1(t_*)\Hbext^{\infty,-5/2-}, \\
    b &\mapsto \hat\omega_{b,s 1}^*(\scal) &\,\in\,& \ker\Box_{g_b} \cap \Poly^1(t_*)\Hbsupp^{-\infty,-5/2-},
  \end{alignat*}
  with linear dependence on $\scal\in\scal_1$, which in addition satisfy
  \begin{equation}
  \label{Eq10GenScGoodGrad}
  \begin{split}
    \delta_{g_b}^*\hat\omega_{b,s 1}(\scal)&\in\Poly^1(t_*)\Hbext^{\infty,-1/2-}, \\
    \sfG_{g_b}\delta_{g_b}^*\hat\omega_{b,s 0}^*,\ \sfG_{g_b}\delta_{g_b}^*\hat\omega_{b,s 1}^*(\scal) &\in \Poly^1(t_*)\Hbsupp^{-\infty,-1/2-}.
  \end{split}
  \end{equation}
  At $b=b_0$, $\hat\omega_{b,s 1}(\scal)$ takes the value
  \begin{align}
  \label{Eq10GenSc1}
    \hat\omega_{b_0,s 1}(\scal)&=t_0\omega_{b_0,s 1}(\scal) + \breve\omega_{b_0,s 1}(\scal), \\
    \begin{split}
      &\breve\omega_{b_0,s 1}(\scal)=-r\omega_{b_0,s 1}(\scal)-\omega_{b_0,s 1}^{(1)}(\scal)+\bhm_0\scal(d t_0+d r) \\
      &\hspace{9em} +d\Bigl(2\bhm_0(2\bhm_0+(\bhm_0-r)\log\bigl(\tfrac{r}{\bhm_0}\bigr)\bigr)\scal\Bigr).
    \end{split} \nonumber
  \end{align}
\end{prop}

In $r\gg 1$, our construction gives $\hat\omega_{b,s 1}(\scal)=t(\omega_{b,s 1}(\scal)+\cA^1)-\omega_{b,s 1}^{(1)}(\scal)+\cA^{0-}=t(d x^i+\cA^1)-x^i\,d t+\cA^{0-}$ when $\scal=r^{-1}x^i$, so $\hat\omega_{b,s 1}(\scal)$ is indeed asymptotic to a Lorentz boost.

\begin{proof}[Proof of Proposition~\ref{Prop10GenSc}]
  \pfstep{Scalar type $l=0$ generalized dual modes.} The ansatz $\hat\omega_{b,s 0}^*=t_*\omega_{b,s 0}^*+\breve\omega_{b,s 0}^*$ gives an element of $\ker\Box_{g_b}$ iff $\wh{\Box_{g_b}}(0)^*\breve\omega_{b,s 0}^*=-[\Box_{g_b},t_*]\omega_{b,s 0}^*$; note that the right hand side is a differentiated $\delta$-distribution at $r=r_b$. This equation can be solved for $\breve\omega_{b,s 0}^*\in\Hbsupp^{-\infty,-3/2-}$ since $\ker\wh{\Box_{g_b}}(0)\cap\Hbext^{\infty,-1/2+}=0$ for $b$ near $b_0$ by Theorem~\ref{Thm1}. While $\breve\omega_{b,s 0}^*$ is unique modulo $\Omega^*:=\la\omega_{b,s 0}^*\ra\oplus\{\omega_{b,s 1}^*(\scal)\colon\scal\in\scal_1\}$, one can force it to be unique, and thus automatically continuous in $b$, by requiring that it be orthogonal to a collection $\eta_1,\ldots,\eta_4\in\CIdot(X;\wt\Tsc{}^*X)$ of 1-forms which are linearly independent functionals on $\Omega^*$.

  The membership in~\eqref{Eq10GenScGoodGrad} holds since $\delta_{g_b}^*\hat\omega_{b,s 0}^* = t_*\delta_{g_b}^*\omega_{b,s 0}^* + [\delta_{g_b}^*,t_*]\omega_{b,s 0}^* + \delta_{g_b}^*\breve\omega_{b,s 0}^*$, with the first summand a linearly growing differentiated $\delta$-distribution at $r=r_b$, and the last two summands lying in $\Hbsupp^{-\infty,-1/2-}$ by definition and using~\eqref{EqKStDivSymmGrad}.

  \pfstep{Scalar type $l=1$ generalized modes.} For better readability, we put
  \[
    \ft=t_{\chi_0}.
  \]
  Let $h_{b,s 0}:=\delta_{g_b}^*\omega_{b,s 0}$ and $h_{b,s 1}(\scal)=\delta_{g_b}^*\omega_{b,s 1}(\scal)$. Note that
  \begin{equation}
  \label{Eq10SymGradDecay}
    h_{b,s 0},\ h_{b,s 1}(\scal) \in \Hbext^{\infty,1/2-}(X;S^2\,\wt\Tsc{}^*X);
  \end{equation}
  for $h_{b,s 1}(\scal)$, this is due to fact that the normal operator of $\wh{\delta_{g_b}^*}(0)$ (which is $\wh{\delta_{\ubar g}^*}(0)$) annihilates the leading order term $d(r\scal)$ (which is the differential of a linear function on $\R^3$) of $\omega_{b,s 1}(\scal)$.
  
  Fixing $\scal\in\scal_1$, the ansatz
  \begin{equation}
  \label{EqL0Ansatz}
    \hat\omega_{b,s 1} = \ft(\omega_{b,s 1}(\scal)+c_b\omega_{b,s 0}) + \breve\omega_{b,s 1},
  \end{equation}
  with $c_b\in\R$ and $\breve\omega_{b,s 1}$ to be determined, then gives
  \[
    \hat h := \delta_{g_b}^*\hat\omega_{b,s 1} = \ft(h_{b,s 1}(\scal)+c_b h_{b,s 0}) + \bigl([\delta_{g_b}^*,\ft](\omega_{b,s 1}(\scal)+c_b\omega_{b,s 0}) + \delta_{g_b}^*\breve\omega_{b,s 1}\bigr).
  \]
  In view of~\eqref{Eq10SymGradDecay} and $t_*-\ft\in\cA^{-1}$, we therefore have $\hat\omega_{b,s 1}\in\ker\Box_{g_b}$, $\hat h\in\Poly^1(t_*)\Hbext^{\infty,-1/2-}$ provided the two conditions
  \begin{subequations}
  \begin{gather}
  \label{EqL0Lins1ExistBox}
    \wh{\Box_{g_b}}(0)\breve\omega_{b,s 1}=-[\Box_{g_b},\ft](\omega_{b,s 1}(\scal)+c_b\omega_{b,s 0}), \\
  \label{EqL0Lins1ExistDel}
    [\delta_{g_b}^*,\ft]\omega_{b,s 1}(\scal) + \delta_{g_b}^*\breve\omega_{b,s 1} \in \Hbext^{\infty,-1/2-}
  \end{gather}
  \end{subequations}
  are satisfied. (Note that $[\delta_{g_b}^*,\ft]\in\cA^0$ maps $\omega_{b,s 0}$ into $\Hbext^{\infty,-1/2-}$, which thus automatically has the required decay). We first arrange~\eqref{EqL0Lins1ExistDel} using the refined ansatz
  \begin{equation}
  \label{EqL0Ansatz2}
    \breve\omega_{b,s 1} = (-1)\omega_{b,s 1}^{(1)}(\scal) + \breve\omega_{b,s 1}',
  \end{equation}
  with $\breve\omega_{b,s 1}'\in\Hbext^{\infty,-3/2-}$ to be determined; the prefactor $(-1)$ is explained below. Note that~\eqref{EqL0Lins1ExistDel} is insensitive to the choice of $\breve\omega_{b,s 1}'$. Moreover, since $[\delta_{g_b}^*,\ft]\omega_{b,s 1}(\scal)$, $\delta_{g_b}^*\omega_{b,s 1}^{(1)}(\scal)\in\Hbext^{\infty,-3/2-}$ change by elements of $\Hbext^{\infty,-1/2-}$ when replacing $b$ by any other Kerr parameters such as $b=(0,0)$ (formally), we conclude that~\eqref{EqL0Lins1ExistDel} is a condition solely involving the leading order parts of all appearing operators and 1-forms; we thus merely need to compute
  \[
    [\delta_{\ubar g}^*,t]\bigl(d(r\scal)\bigr) = d t\otimes_s d(r\scal),
  \]
  which indeed agrees modulo $\rho\CI\subset\Hbext^{\infty,-1/2-}$ with $\delta_{g_{b_0}}^*\omega_{b_0,s 1}^{(1)}(\scal)$, see~\eqref{Eq10Gens12SymGrad}. (This is not a coincidence, but merely the fact that on Minkowski space, $\delta_{\ubar g}^*(t\,d x^i)=\delta_{\ubar g}^*(x^i\,d t)$, which is precisely the statement that the Lorentz boost $t\,d x^i-x^i\,d t$ is Killing.) Thus, \eqref{EqL0Lins1ExistDel} holds, and \eqref{EqL0Ansatz} is asymptotic to a Lorentz boost.

  Turning to equation~\eqref{EqL0Lins1ExistBox}, we write it using~\eqref{EqL0Ansatz2}, and expanding the commutator as $[\Box_{g_b},\ft]=[2\delta_{g_b}\sfG_{g_b},\ft]\delta_{g_b}^*+2\delta_{g_b}\sfG_{g_b}[\delta_{g_b}^*,\ft]$, as
  \begin{align*}
    \wh{\Box_{g_b}(0)}\breve\omega_{b,s 1}' &= -c_b[\Box_{g_b},\ft]\omega_{b,s 0} - [2\delta_{g_b}\sfG_{g_b},\ft]h_{b,s 1}(\scal) \\
      &\qquad - 2\delta_{g_b}\sfG_{g_b}\bigl([\delta_{g_b}^*,\ft]\omega_{b,s 1}(\scal)-\delta_{g_b}^*\omega_{b,s 1}^{(1)}(\scal)\bigr) \in \Hbext^{\infty,1/2-},
  \end{align*}
  where we used~\eqref{Eq10SymGradDecay} and~\eqref{EqL0Lins1ExistDel} to get the improved decay of the second and third term, respectively. But this can be solved for $\breve\omega_{b,s 1}'\in\Hbext^{\infty,-3/2-}$ iff the right hand side integrates to $0$ against $\omega_{b,s 0}^*$ (which spans $\ker\wh{\Box_{g_b}}(0)^*\cap\Hbsupp^{-\infty,-1/2+}$). In view of the non-degeneracy~\eqref{Eq10NoLinPairKerr}, this can be accomplished by a suitable choice of $c_b$, with $c_b$ continuous in $b$ and linear in $\scal\in\scal_1$. (In the Schwarzschild case $b=b_0$, we have
  \begin{equation}
  \label{EqL0cb0}
    c_{b_0}=0
  \end{equation}
  since the second and third summands in this equation are of scalar type $l=1$, while $\omega_{b_0,s 0}^*$ is of scalar type $l=0$, and 1-forms of different pure types are orthogonal.)

  Making this explicit for Schwarzschild metrics, let us work for computational simplicity with the function $\ft=t_0-r$; we make the ansatz $\hat\omega_{b_0,s 1}=\ft\omega_{b_0,s 1}(\scal)-\omega_{b_0,s 1}^{(1)}(\scal)+\breve\omega'_{b_0,s 1}$. Thus, we seek $\breve\omega_{b_0,s 1}'$ such that
  \[
    \wh{\Box_{g_{b_0}}}(0)\breve\omega_{b_0,s 1}' = -[\Box_{g_{b_0}},\ft]\omega_{b_0,s 1}(\scal) = -\tfrac{2\bhm_0}{r^2}\scal\,d t_0-\tfrac{2\bhm_0(r+\bhm_0)}{r^2}\sld\scal =: e.
  \]
  We can solve this in two steps: firstly, one can check that $e-\wh{\Box_{g_{b_0},1}}(0)(\bhm_0(d t_0+d r)\scal)$ is exact, and indeed equals $d f$, $f=\tfrac{2\bhm_0(\bhm_0-3 r)}{r^2}\scal$; secondly, we have
  \[
    \wh{\Box_{g_{b_0},0}}(0)u = f,\quad u=2\bhm_0\bigl(2\bhm_0+(\bhm_0-r)\log\bigl(\tfrac{r}{\bhm_0}\bigr)\bigr)\scal.
  \]
  Combining these calculations gives the expression in~\eqref{Eq10GenSc1}.

  \pfstep{Scalar type $l=1$ generalized dual modes.} The arguments are completely analogous, with the role of $\omega_{b,s 1}^{(1)}(\scal)$ now being played by $\omega_{b,s 1}^{(1)*}(\scal)$.
\end{proof}

%%%%%%%%%%%%%%%%%%%%%%%%%%%%%%%%%%%%%%%%%%%%%%%%%%%%%%%%%%%%%%%%%%%%%%
\section{Mode stability of the Schwarzschild metric}
\label{SMS}

A crucial input for the spectral theory of the linearization $L_{g_{b_0}}$ of the gauge-fixed Einstein operator at the Schwarzschild metric, defined in equation~\eqref{EqOpLinOp}, is the mode analysis for the linearization of the Einstein equation itself. We carefully follow the arguments of Kodama--Ishibashi~\cite{KodamaIshibashiMaster} (which in turn build on \cite{KodamaIshibashiSetoBranes,KodamaSasakiPerturbation}) in the form presented in \cite[\S5]{HintzKNdSStability}. In particular, we will highlight the places where decay assumptions at infinity are used, and what asymptotics one obtains for the gauge potentials, i.e.\ the 1-forms whose symmetric gradients produce a given pure gauge perturbation.

Throughout this section, we take $\bhm=\bhm_0$ and
\[
  g = g_{(\bhm,0)},\quad t_*=t_{\bhm,*}.
\]
We work in the setting of~\S\ref{SsYM}, equipping $M^\circ$ with the Schwarzschild metric $g=g_{(\bhm,0)}$, written as $g=\wh g-r^2\slg$. Thus, the aspherical part $\wh X$ in~\eqref{EqYMProd} carries the Lorentzian metric $\wh g=\mu\,d t_0^2-2 d t_0\,d r$; in static coordinates on $\R_t\times(2\bhm,\infty)_r$, this means $\wh g=\mu\,d t^2-\mu^{-1}\,d r^2$. We shall moreover phrase the outgoing condition on modes from the conjugated perspective, cf.\ Remark~\ref{RmkOpOutgoing}.

\begin{thm}
\label{ThmMS}
  Let $\sigma\in\C$, $\Im\sigma\geq 0$, and suppose $\dot g$ is an outgoing mode solution of the linearized Einstein equation
  \begin{equation}
  \label{EqMSLinEq}
    D_g\Ric(\dot g)=0.
  \end{equation}
  Then there exist parameters $\dot\bhm\in\R$, $\dot\bha\in\R^3$, and an outgoing 1-form $\omega$ on $M$, such that
  \begin{equation}
  \label{EqMSLinDec}
    \dot g-\dot g_{(\bhm,0)}(\dot\bhm,\dot\bha)=\delta_g^*\omega.
  \end{equation}

  More precisely:
  \begin{enumerate}
  \item If $\sigma\neq 0$, suppose that $\dot g=e^{i\sigma t_*}\dot g_0$ with $\dot g_0\in\Hbext^{\infty,\ell}(X;S^2\,\wt{\Tsc^*}X)$ for some $\ell\in\R$. Then~\eqref{EqMSLinDec} holds with $(\dot\bhm,\dot\bha)=(0,0)$ and $\omega=e^{i\sigma t_*}\omega_0$, with $\omega_0\in\Hbext^{\infty,\ell'}(X;\wt{\Tsc^*}X)$ for some $\ell'\in\R$.
  \item If $\sigma=0$, and $\dot g\in\Hbext^{\infty,\ell}(X;S^2\,\wt\Tsc{}^*X)$, $\ell\in(-\tfrac32,-\half)$, is a stationary perturbation, we consider each part in the spherical harmonic decomposition of $\dot g$---which is of one of the types in~\eqref{EqYMSym2}---separately:
    \begin{enumerate}
    \item If $\dot g$ is a scalar perturbation with $l\geq 1$ or a vector perturbation with $l\geq 2$, then
    \[
      \dot g=\delta_g^*\omega,
    \]
    where $\omega\in\Hbext^{\infty,\ell-1}(X;\wt\Tsc{}^*X)$ is a 1-form of the same type as $\dot g$;
    \item\label{ItMSSc00} if $\dot g$ is a scalar perturbation with $l=0$, i.e.\ spherically symmetric, then
    \[
      \dot g-\dot g_{(\bhm,0)}(\dot\bhm,0)=\delta_g^*\omega,
    \]
    where $\omega\in\Hbext^{\infty,\ell-1}$ is a spherically symmetric 1-form;
    \item\label{ItMSV10} if $\dot g$ is a vector perturbation with $l=1$, then
    \[
      \dot g-\dot g_{(\bhm,0)}(0,\dot\bha)=\delta_g^*\omega,
    \]
    where $\omega\in\Hbext^{\infty,\ell-1}$ is a vector type $l=1$ 1-form.
    \end{enumerate}
  \end{enumerate}

  The zero energy scalar $l=0$ and $l=1$ statements can be strengthened as follows:
  \begin{enumerate}[label=\upshape{(\alph*)},ref=\alph*]
  \setcounter{enumi}{3}
  \item\label{ItMSSc1} If $\dot g$ is a stationary scalar $l=1$ perturbation with merely $\dot g\in\Hbext^{\infty,\ell}$ for some $\ell<-\half$, $\ell\notin-\half-\N$, then $\dot g$ is pure gauge,
  \[
    \dot g=\delta_g^*\omega
  \]
  with $\omega\in\Hbext^{\infty,\ell-1}(X;\wt\Tsc{}^*X)$ of scalar type $l=1$.
  \item\label{ItMSSc0} if $\dot g\in\Poly(t_*)^k\Hbext^{\infty,\ell}(X;S^2\,\wt\Tsc{}^*X)$ is of scalar type $l=0$ and satisfies $D_g\Ric(\dot g)=0$, then there exists $\dot\bhm\in\R$ such that
  \[
    \dot g=\dot g_{(\bhm,0)}(\dot\bhm,0) + \delta_g^*\omega,
  \]
  where $\omega\in\Poly(t_*)^{k+1}\Hbext^{\infty,\ell'}(X;\wt\Tsc{}^*X)$ (for some $\ell'\in\R$) is of scalar type $l=0$.
  \end{enumerate}
\end{thm}

\begin{rmk}
\label{RmkMSWhichLin}
  In parts~\eqref{ItMSSc00} and \eqref{ItMSV10}, we can replace $\dot g_{(\bhm,0)}$ by $\dot g^0_{(\bhm,0)}$ upon changing $\omega$ by an element of $\Hbext^{\infty,-3/2-}(X;\wt{\Tsc^*}X)$; likewise in part~\eqref{ItMSSc0}. Indeed, this follows from~\eqref{EqKaLie2}, which gives $V(\dot b)^\flat\in\cA^{0-}\subset\Hbext^{\infty,-3/2-}$.
\end{rmk}

We list a number of explicit expressions needed for the proof. Recall from~\cite{GrahamLeeConformalEinstein} that
\[
  D_g\Ric = \half\Box_g - \delta_g^*\delta_g\sfG_g + \sR_g,\quad
  (\sR_g\dot g)_{\mu\nu}=(R_g)_{\kappa\mu\nu\lambda}\dot g^{\kappa\lambda},
\]
where the expression for $\sR_g$ simplifies when $g$ is the Schwarzschild metric, since $\Ric(g)=0$. We adorn operators on $(\wh X,\wh g)$ with hats. For the calculations below, we recall from~\cite[\S5.1]{HintzKNdSStability} that in the splittings~\eqref{EqYMSplit} and \eqref{EqYMSplit2}, and writing $\varpi:=\wh d r$ and $\wh\Box:=\Box_{\wh g}$, we have
\begin{equation}
\label{EqMSExpr}
\begin{split}
  &\Box_g=\wh\Box-r^{-2}\slDelta + \diag\bigl[-2 r^{-1}\wh\nabla_\varpi+4 r^{-2}\varpi\otimes_s\iota_\varpi,\,4 r^{-2}\varpi\otimes_s\iota_\varpi+(-r^{-1}\wh\Box r+r^{-2}|\varpi|^2) \\
  &\hspace{7em} 2 r^{-1}\wh\nabla_\varpi-2 r^{-1}\wh\Box r\bigr]
    +
    \begin{pmatrix}
      0 & 4 r^{-3}\varpi\otimes_s\sldelta & 2 r^{-4}(\varpi\otimes\varpi)\sltr \\
      -2 r^{-1}\sld \iota_\varpi & 0 & 2 r^{-3}\varpi\otimes\sldelta \\
      2\slg\iota_\varpi\iota_\varpi & -4 r^{-1}\sldelta^*\iota_\varpi & 0
    \end{pmatrix}, \\
  &\delta_g^*=
    \begin{pmatrix}
      \wh\delta^* & 0 \\
      \half\sld & \half r^2\wh d r^{-2} \\
      -r\slg\iota_\varpi & \sldelta^*
    \end{pmatrix}, \quad
  \delta_g\sfG_g=
    \begin{pmatrix}
      r^{-2}\wh\delta r^2+\half\wh d\wh\tr & -r^{-2}\sldelta & -\half r^{-2}\wh d\sltr \\
      \half\sld\wh\tr & r^{-2}\wh\delta r^2 & -r^{-2}\sldelta-\half r^{-2}\sld\sltr
    \end{pmatrix}, \\
  &2\sR_g=
    \begin{pmatrix}
      2\mu''\sfG_{\wh g}+2 r^{-1}\mu' & 0 & r^{-3}\mu'\wh g\sltr \\
      0 & \half\mu''\wh g+r^{-2}(\mu-1)+3 r^{-1}\mu' & 0 \\
      r\mu'\slg\wh\tr & 0 & 4 r^{-2}(\mu-1)\sfG_\slg+2 r^{-1}\mu'
    \end{pmatrix}.
\end{split}
\end{equation}

On the static part $\R_t\times(2\bhm,\infty)_r$ of $\wh X$, we furthermore split
\begin{equation}
\label{EqMSExprHat}
  T^*\wh X=\la\wh d t\ra \oplus \la\wh d r\ra,\quad
  S^2 T^*\wh X=\la\wh d t^2\ra \oplus \la 2\wh dt\,\wh d r\ra \oplus \la\wh d r^2\ra.
\end{equation}
In the first splitting, $\varpi=(0,1)$. On functions,
\[
  \wh d=\begin{pmatrix}\pa_t \\ \pa_r\end{pmatrix}, \quad
  \wh\Box = -\mu^{-1}\pa_t^2 + \pa_r\mu\pa_r,
\]
while on 1-forms, $\iota_\varpi=(0,-\mu)$ and
\[
  \wh\delta=(-\mu^{-1}\pa_t,\,\pa_r\mu), \quad
  \wh\delta^*=\begin{pmatrix} \pa_t & -\half\mu\mu' \\ \half\mu\pa_r\mu^{-1} & \half\pa_t \\ 0 & \mu^{-1/2}\pa_r\mu^{1/2} \end{pmatrix}.
\]
On symmetric 2-tensors, $\wh\nabla_\varpi=\diag(-\mu^2\pa_r\mu^{-1},-\mu\pa_r,-\pa_r\mu)$ and
\begin{gather*}
  \wh\delta=
  \begin{pmatrix}
    -\mu^{-1}\pa_t & \pa_r\mu & 0 \\
    \half\mu'\mu^{-2} & -\mu^{-1}\pa_t & \mu^{-1/2}\pa_r\mu^{3/2}
  \end{pmatrix}, \quad
  \iota_\varpi=\begin{pmatrix} 0 & -\mu & 0 \\ 0 & 0 & -\mu \end{pmatrix}, \\
  \wh\Box=-\mu^{-1}\pa_t^2+\mu\pa_r^2
    +\begin{pmatrix}
       -\mu'\pa_r+\tfrac{\mu'^2}{2\mu}-\mu'' & 2\mu'\pa_t & -\half\mu\mu'^2 \\
       \tfrac{\mu'}{\mu^2}\pa_t & \mu'\pa_r-\tfrac{\mu'^2}{\mu} & \mu'\pa_t \\
       -\tfrac{\mu'^2}{2\mu^3} & \tfrac{2\mu'}{\mu^2}\pa_t & 3\mu'\pa_r+\tfrac{\mu'^2}{2\mu}+\mu''
     \end{pmatrix}.
\end{gather*}

Below, we shall write $\wt f\in\CI(\wh X;S^2 T^*\wh X)$, $f\in\CI(\wh X;T^*\wh X)$, $H_L,H_T\in\CI(\wh X)$, and $T\in\CI(\wh X;T^*\wh X)$, $L\in\CI(\wh X)$.

%%%%%%%%%%%%%%%%%%%%%%%%%%%%%%%%%%%%%%%%%%%%%%%%%%
\subsection{Scalar type perturbations}
\label{SsMSSc}

We discuss $l\geq 2$, $l=1$, and $l=0$ modes separately. We denote by $\scal\in\scal_l$ a spherical harmonic with eigenvalue $k^2$, where $k=\sqrt{l(l+1)}$. We also introduce a rescaled version of the trace-free part of the Hessian:
\[
  \slH_k=k^{-2}\sldelta_0^*\sld\quad (l\neq 0).
\]

%%%%%%%%%%%%%%%%%%%%%%%%%%%%%%
\subsubsection{Modes with \texorpdfstring{$l\geq 2$}{l at least 2}}
\label{SssMSSc2}

Consider a metric perturbation of the form
\begin{equation}
\label{EqMSScal2}
  \dot g= \begin{pmatrix} \wt f\scal \\ -\tfrac{r}{k}f\otimes\sld\scal \\ 2 r^2(H_L\scal\slg+H_T\slH_k\scal) \end{pmatrix}.
\end{equation}
Pure gauge solutions of the same type take the form
\begin{equation}
\label{EqMSSc2Pert}
  \updelta\dot g = \delta_g^*\omega
   =\begin{pmatrix}
      (2\wh\delta^* T)\scal \\
      -\frac{r}{k}(-\frac{k}{r}T + r\wh d r^{-1}L)\otimes\sld\scal \\
      2 r^2\bigl[(-r^{-1}\iota_\varpi T+\frac{k}{2 r}L)\scal\slg - \frac{k}{r}L\slH_k\scal\bigr]
    \end{pmatrix}, \qquad
  \omega=\begin{pmatrix}2 T\scal \\ -\frac{2 r}{k}L\sld\scal\end{pmatrix};
\end{equation}
upon adding this to $\dot g$, the quantities $\wt f$ etc.\ change by
\begin{equation}
\label{EqMSSc2Pert2}
  \updelta\wt f=2\wh\delta^*T, \quad
  \updelta f=-\tfrac{k}{r}T+r\wh d r^{-1}L, \quad
  \updelta H_L=-r^{-1}\iota_\varpi T+\tfrac{k}{2 r}L, \quad
  \updelta H_T=-\tfrac{k}{r}L.
\end{equation}
Defining the 1-form $\bfX:=\tfrac{r}{k}\bigl(f+\tfrac{r}{k}\wh d H_T\bigr)$, which satisfies $\updelta\bfX=-T$, the quantities
\begin{equation}
\label{EqMSSc2Inv}
  \wt F := \wt f+2\wh\delta^*\bfX, \quad
  J := H_L+\half H_T-r^{-1}\iota_\varpi\bfX
\end{equation}
are therefore gauge-invariant: $\updelta\wt F=\updelta J=0$. Conversely, if $\wt F=J=0$, then $\dot g$ is a pure gauge solution:
\begin{equation}
\label{EqMSSc2GaugePot}
  \delta_g^*\omega=\dot g,\quad \omega=\bigl(-2\bfX,\tfrac{2 r^2}{k^2}H_T\sld\scal\bigr).
\end{equation}
If $\dot g$ is an outgoing mode solution with frequency $\sigma\neq 0$, then so is $\omega$, with an extra factor of $r^2$ relative to $\dot g$. When $\sigma=0$ on the other hand, then $H_T$ is stationary; on time-independent functions, $r\wh d$ in the definition of $\bfX$ acts as an unweighted b-operator, and therefore $\omega$ grows at most by a factor of $r$ more than $\dot g$ in this case.

As explained after \cite[Equation~(5.27)]{HintzKNdSStability}, we can express the linearized Einstein equation in terms of $\wt F,J$ (by formally replacing $\wt f,f,H_L,H_T$ by $\wt F,0,J,0$). Expressing the scalar type symmetric 2-tensor equation $2 D_g\Ric(\dot g)=0$ analogously to~\eqref{EqMSScal2} in terms of quantities $\wt f^E,f^E,H_T^E,H_L^E$, one obtains
\begin{subequations}
\begin{align}
\label{EqMSSc2wtf}
  \begin{split}
  \wt f^E &= (\wh\Box-2\wh\delta^*\wh\delta-\wh\delta^*\wh d\wh\tr)\wt F+2 r^{-1}(2\wh\delta^*\iota_\varpi\wt F-\wh\nabla_\varpi\wt F) \\
    &\qquad + 4\wh\delta^*\wh d J + 8 r^{-1}\varpi\otimes_s\wh d J + (\mu''-k^2 r^{-2})\wt F-\mu''\wh g\wh\tr\wt F = 0,
  \end{split} \\
\label{EqMSSc2f}
  -\tfrac{r}{k}f^E &= -\wh\delta\wt F+2\wh d J-r\wh d r^{-1}\wh\tr\wt F = 0, \\
\label{EqMSSc2HL}
  2 r^2 H_L^E &= \wh\Box(2 r^2 J)+2 r\iota_\varpi\wh\delta\wt F-2\iota_\varpi\iota_\varpi\wt F+r\iota_\varpi\wh d\wh\tr\wt F + (r\mu'+\tfrac{k^2}{2})\wh\tr\wt F - 2 k^2 J = 0, \\
\label{EqMSSc2HT}
  2 r^2 H_T^E &= -k^2\wh\tr\wt F = 0.
\end{align}
\end{subequations}
Using~\eqref{EqMSSc2HT} and $k^2\neq 0$, we can eliminate all occurrences of $\wh\tr\wt F$. Plugging $\wh d J=\half\wh\delta\wt F$ from~\eqref{EqMSSc2f} into~\eqref{EqMSSc2wtf}, one obtains a wave equation for $\wt F$ which is (via subprincipal terms) coupled to the wave equation for $J$ resulting from~\eqref{EqMSSc2HL}; i.e.\ we obtain a principally scalar system of wave equations for $(\wt F,J)$. When $\dot g$ and thus $(\wt F,J)$ are smooth modes, then this wave equation becomes an ODE on the 1-dimensional space $t_*^{-1}(0)$ with a regular-singular point at $r=2\bhm$; the vanishing of $(\wt F,J)$ in $r>2\bhm$ thus implies its vanishing in $r\leq 2\bhm$ as well.

The goal is thus to prove $(\wt F,J)=0$ in $r>2\bhm$; there, we can use the static coordinates $(t,r)$. By the linearization of the second Bianchi identity, $\delta_g\sfG_g(D_g\Ric(\dot g))\equiv 0$ (for \emph{any} $\dot g$), the above equations are not independent: putting
\[
  \wt E:=\sfG_{\wh g}\wt f^E+2 H_L^E\wh g,
\]
we have
\[
  2 r^{-2}\wh\delta(r^2\wt E) + \wh d(\wh\tr\wt f^E)+\tfrac{2 k}{r}f^E = 0, \qquad
  \wh\tr\wt f^E-\bigl(\tfrac{2}{k} r^{-2}\wh\delta(r^3 f^E)+\tfrac{2(k^2-2)}{k^2}H_T^E\bigr) = 0.
\]
In particular, the vanishing of $f^E$ and $H_T^E$ implies that of $\wh\tr\wt f^E$ and $\wh\delta(r^2\wt E)$, and in this case, $\wt f^E$ and $H_L^E$ are the trace-free, resp.\ pure trace part of $\wt E$. By the calculations after~\eqref{EqMSExprHat}, the $d r$-component of $\wh\delta(r^2\wt E)=0$ reads
\[
  \tfrac{\mu'}{2\mu^2}\wt E_{t t}-r^2\mu^{-1}\pa_t\wt E_{t r}+\mu^{-1/2}\pa_r(r^2\mu^{3/2}\wt E_{r r})=0,
\]
so in view of $\mu'\neq 0$, the vanishing of $\wt E_{t r},\wt E_{r r}$ implies $\wt E_{t t}=0$. We have thus reduced the linearized Einstein equations to the system
\[
  (f^E,H_T^E,\wt E_{t r},\wt E_{r r}) = 0,
\]
with $H_T^E=0$ simply giving $\wh\tr\wt F=0$. Let us combine $\wt F$ (trace-free) and $J$ by writing
\begin{equation}
\label{EqMSSc2XYZ}
  \wt F+2 J\wh g = \begin{pmatrix} \mu X \\ -\mu^{-1}Z \\ -\mu^{-1}Y \end{pmatrix}
\end{equation}
in the splitting~\eqref{EqMSExprHat}; we have $\wh\delta(\wt F+2 J\wh g)=0$ by~\eqref{EqMSSc2f} and $4 J=\wh\tr(\wt F+2 J\wh g)=X+Y$. The equations for $f^E$, $\wt E_{t r}$, and $\mu\wt E_{r r}=\half(\mu^{-1}\wt f^E_{t t}+\mu\wt f^E_{r r})-2 H_L^E$ then read
\begin{subequations}
\begin{gather}
\label{EqMSSc2fCoord}
  \pa_t X+\pa_r Z=0, \qquad -\pa_r Y+\mu^{-2}\pa_t Z+\tfrac{\mu'}{2\mu}(X-Y)=0, \\
\label{EqMSSc2Etr}
  \pa_t\pa_r X-\tfrac{\mu'}{2\mu}\pa_t X+\pa_t\pa_r Y+\bigl(2 r^{-1}-\tfrac{\mu'}{2\mu}\bigr)\pa_t Y+\tfrac{k^2}{r^2\mu}Z=0, \\
\label{EqMSSc2Err}
  \mu^{-1}\pa_t^2 X-\tfrac{\mu'}{2}\pa_r X + \mu^{-1}\pa_t^2 Y-\bigl(\tfrac{\mu'}{2}+\tfrac{2\mu}{r}\bigr)\pa_r Y + \tfrac{k^2-2}{r^2}Y + \tfrac{4}{r\mu}\pa_t Z = 0.
\end{gather}
\end{subequations}

\emph{We first discuss the case that $\dot g$ is a mode with $\sigma\neq 0$.} We now have $\pa_t X=-i\sigma X$ etc., so writing $X'=\pa_r X$ etc., the equations~\eqref{EqMSSc2fCoord}--\eqref{EqMSSc2Etr} become
\[
  \begin{pmatrix} X' \\ Y' \\ \frac{Z'}{i\sigma} \end{pmatrix} = T \begin{pmatrix} X \\ Y \\ \frac{Z}{i\sigma} \end{pmatrix}, \qquad
  T=\begin{pmatrix}
      0 & \frac{\mu'}{\mu}-\frac{2}{r} & \frac{k^2}{r^2\mu}-\frac{\sigma^2}{\mu^2} \\
      \frac{\mu'}{2\mu} & -\frac{\mu'}{2\mu} & \frac{\sigma^2}{\mu^2} \\
      1 & 0 & 0
    \end{pmatrix},
\]
while~\eqref{EqMSSc2Err} gives the linear constraint
\begin{equation}
\label{EqMSSc2Constr}
  \gamma\begin{pmatrix} X \\ Y \\ \frac{Z}{i\sigma} \end{pmatrix} = 0, \qquad
  \gamma=\Bigl(-\frac{\sigma^2}{\mu}-\frac{\mu'^2}{4\mu}-\frac{\mu'}{r},\,-\frac{\sigma^2}{\mu}+\frac{k^2-2}{r^2}-\frac{\mu'^2}{4\mu}+\frac{2\mu'}{r},\,\frac{2\sigma^2}{r\mu}-\frac{k^2\mu'}{2 r^2\mu}\Bigr).
\end{equation}
Thus, a generic linear combination $\Phi$ of $X,Y,\frac{Z}{i\sigma}$ satisfies a second order ODE; choosing $\Phi$ carefully, one can make this ODE be of Schr\"odinger type
\begin{equation}
\label{EqMSSc2MasterEq}
  (\mu\pa_r)^2\Phi-(V-\sigma^2)\Phi=0.
\end{equation}
Concretely, this holds if we let
\begin{gather}
  m := k^2-2,\ \ 
  x := \frac{2\bhm}{r},\ \ 
  H := m+3 x, \nonumber\\
\label{EqMSSc2Master}
  \Phi := \frac{\frac{2 Z}{i\sigma}-r(X+Y)}{H},\quad
  V=\frac{\mu}{r^2 H^2}\bigl(9 x^3+9 m x^2+3 m^2 x+m^2(m+2)\bigr).
\end{gather}
Conversely, one can recover $X,Y,\frac{Z}{i\sigma}$ from $\Phi$ by means of
\begin{equation}
\label{EqMSSc2MasterRec}
  X = \bigl(\tfrac{\sigma^2 r}{\mu}-\tfrac{P_{X 0}}{2 r H^2}\bigr)\Phi + \tfrac{P_{X 1}}{2 H}\Phi', \ \ 
  Y = \bigl(-\tfrac{\sigma^2 r}{\mu}-\tfrac{P_{Y 0}}{2 r H^2}\bigr)\Phi + \tfrac{P_{Y 1}}{2 H}\Phi', \ \ 
  \tfrac{Z}{i\sigma} = \tfrac{P_Z}{2 H}\Phi-r\mu\Phi',
\end{equation}
where
\begin{alignat*}{3}
  P_{X 0} &= 27 x^3+24 m x^2+3 m(3 m+2)x+2 m^2(m+2), & \quad & P_{X 1} = 9 x^2+(5 m-6)x-4 m, \\
  P_{Y 0} &= 9 x^3+6 m x^2+3 m(m+2)x, && P_{Y 1} = 3 x^2-(m+6)x, \\
  P_Z &= 3 x^2+3 m x-2 m. &&
\end{alignat*}

It thus remains to show that $\Phi\equiv 0$. Now $m,x,H>0$ are bounded away from $0$ in $[2\bhm,\infty)_r$, hence $V>0$ is of size $\cO(\mu)$ near $r=2\bhm$, and of size $\cO(r^{-2})$ as $r\to\infty$; passing to the Regge--Wheeler coordinate $r_*=\int\mu^{-1}\,d r=r+2\bhm\log(r-2\bhm)$, this means exponential decay as $r_*\to-\infty$ and $\cO(r_*^{-2})$ decay as $r_*\to+\infty$.

Note then that $\wt F$ and $J$, as a 2-tensor, resp.\ function on $\wh X$ (which extends across $r=2\bhm$) are smooth. Thus, the contribution of $J=(X+Y)/4$ to $\Phi$ is smooth at the event horizon. Similarly, recall that in static coordinates, the vector fields
\begin{equation}
\label{EqMSSc2VF}
  \pa_t,\ \pa_r+\mu^{-1}\pa_t
\end{equation}
are smooth across the event horizon, hence $Z=-\mu\wt F_{t r}=-\mu\wt F(\pa_t,\pa_r+\mu^{-1}\pa_t)+\wt F_{t t}$ is smooth as well. Thus, $\Phi(t_*,r)=e^{-i\sigma t_*}\CI([2\bhm,\infty))$; writing this as
\[
  \Phi(t,r) = e^{-i\sigma t}\Psi(r),\quad \Psi(r)\in e^{-i\sigma r_*}\CI([2\bhm,\infty)),
\]
we see that $\Psi$ decays exponentially as $r_*\to-\infty$ when $\Im\sigma>0$.

Consider first the case that $\Im\sigma>0$. The outgoing condition on $\dot g$ implies the rapid vanishing of $\dot g$, hence of $(\wt F,J)$, thus of $(X,Y,Z/i\sigma)$ and finally of $\Psi$ as $r\to\infty$ (for constant $t$). Switching to the tortoise coordinate $r_*$ in~\eqref{EqMSSc2MasterEq}, i.e.\ $(D_{r_*}^2+V-\sigma^2)\Psi=0$, we can integrate against $\bar\Psi$ in $L^2(\R_{r_*};d r_*)$ and integrate by parts, obtaining
\[
  0 = \|V^{1/2}\Psi\|^2 - \sigma^2\|\Psi\|^2 + \|D_{r_*}\Psi\|^2.
\]
When $\sigma\notin i(0,\infty)$, taking the imaginary part gives $\Psi=0$, while for $\sigma\in i(0,\infty)$, this is a sum of squares, hence again $\Psi=0$; in both cases, we deduce $\Phi=0$.

When $\sigma\in\R$, $\sigma\neq 0$, note that $\Psi=\Psi(r)$ inherits the radiation condition from $\dot g$. Since it satisfies the ODE~\eqref{EqMSSc2MasterEq}, this excludes the `incoming' asymptotics $e^{-i\sigma r_*}$ as $r_*\to\infty$ and thus forces $\Psi=e^{i\sigma r_*}(\Psi_+ +\cA^{1-})$, $\Psi_+\in\C$; here, $\cA^\ell$ denotes the space of functions on $\ol{\R}$ which are conormal at $\infty$ relative to $r^{-\ell}L^\infty$, i.e.\ remain in this space upon application of $(r\pa_r)^j$, $j\in\N_0$. A standard boundary pairing, or in the present ODE setting Wronskian, argument then implies that $\lim_{r_*\to\pm\infty}e^{\mp i\sigma r_*}\Psi=0$ (the upper sign corresponding to $\Psi_+=0$), which then implies $\Psi\equiv 0$, finishing the argument in the case $\sigma\neq 0$.

\emph{Next, we describe the modifications for the case $\sigma=0$.} Now, $\dot g$ and thus $\wt F,J$ and $X,Y,Z$ are stationary, but $Z/i\sigma$ is no longer well-defined; hence we first rewrite the treatment of the case $\sigma\neq 0$: using~\eqref{EqMSSc2Constr}, we can express $Z/i\sigma$ as a linear combination of $X,Y$. Plugging the resulting expression into~\eqref{EqMSSc2Master} and making the $\sigma$-dependence explicit, we get
\begin{equation}
\label{EqMSSc2Master0}
\begin{split}
  \Phi(\sigma) &= C_X(\sigma)X+C_Y(\sigma)Y, \\
  &\qquad C_{X/Y}(\sigma) = \tfrac{P_{X/Y}}{3\wt H(\sigma)},\quad \wt H(\sigma)=(k^2\mu'-4 r\sigma^2)H, \\
  &\qquad P_X=(9 x-3(6+m))x,\ \ 
          P_Y=-3(9 x+5 m-6)x+12 m.
\end{split}
\end{equation}
But this means that $\Phi(\sigma)$ exists down to $\sigma=0$ (note that $\mu'\neq 0$ in $r>2\bhm$); we thus define the master variable at zero frequency to be $\Phi:=\Phi(0)$. One can then check that the linearized Einstein equations~\eqref{EqMSSc2fCoord}--\eqref{EqMSSc2Err} imply the master equation
\begin{equation}
\label{EqMSSc2MasterEq0}
  (\mu\pa_r)^2\Phi-V\Phi=0,
\end{equation}
with $V$ as before. Using the expressions~\eqref{EqMSSc2MasterRec} with $\sigma=0$, one can recover $X,Y$ from $\Phi$.

Since $\dot g\in\Hbext^{\infty,\ell}\subset\cA^{\ell+3/2}$, we also have $X,Y\in\cA^{\ell+3/2}$. Now $H\in m+\rho\CI([2\bhm,\infty])$ and thus $\wt H(0)\in-2 k^2 m\bhm/r^2+\rho^3\CI$, while $P_X,P_Y\in\CI$; therefore, $C_X(0),C_Y(0)\in\rho^{-2}\CI$, implying that $\Phi\in\cA^{\ell-1/2}\subset\cA^{-2+}$ a priori. Consider now the asymptotic behavior of solutions of~\eqref{EqMSSc2MasterEq0}: note that
\[
  V(r) = \frac{1}{r^2 m^2}m^2(m+2) + \rho^3\CI = \frac{k^2}{r^2} + \rho^3\CI,
\]
hence $V=\frac{k^2}{r_*^2}+\cA^{3-}$ as a function of $r_*$. Thus, the leading order part (as a weighted b-differential operator on $[0,1)_x$, $x=r_*^{-1}$) of $(\mu\pa_r)^2-V$ is $r_*^{-2}(r_*^2\pa_{r_*}^2-k^2)$, which has kernel $r_*^{\lambda_\pm}$, $\lambda_\pm=\half(1\pm\sqrt{4 k^2+1})$. Since $k^2\geq 6$, we have $\lambda_+\geq 3$; since therefore $\Phi=o(r^{\lambda_+})$, the leading order term of $\Phi$ at infinity must be $r_*^{\lambda_-}=\cO(r^{-2})$. This suffices to justify pairing~\eqref{EqMSSc2MasterEq0} against $\Phi$ (on $t=0$) and integrating by parts, giving $\Phi\equiv 0$ since $V\geq 0$. Thus, $X=Y=0$ in $r>2\bhm$; by~\eqref{EqMSSc2Etr}, we then also have $Z=0$, hence $(\wt F,J)=0$, proving that $\dot g$ is pure gauge.

%%%%%%%%%%%%%%%%%%%%%%%%%%%%%%
\subsubsection{Modes with \texorpdfstring{$l=1$}{l equal to 1}}
\label{SssMSSc1}

We now have $k^2=l(l+1)=2$. We shall again show that an outgoing mode solution $\dot g$ is pure gauge, with the gauge potential an outgoing mode as well. Consider a general metric perturbation $\dot g$ satisfying the linearized Einstein equation~\eqref{EqMSLinEq} and of the form
\begin{equation}
\label{EqMSScal1}
  \dot g=\begin{pmatrix} \wt f\scal \\ -\tfrac{r}{k}f\otimes\sld\scal \\ 2 r^2 H_L\scal\slg \end{pmatrix}
\end{equation}
with $\scal\in\scal_1$. While $H_T$ is no longer defined in this situation, we use the expressions from the $l\geq 2$ discussion with $H_T$ set to $0$, thus
\begin{equation}
\label{EqMSSc1Quant}
  \bfX:=\tfrac{r}{k}f,\quad \wt F:=\wt f+2\wh\delta^*\bfX,\quad J:=H_L-r^{-1}\iota_\varpi\bfX.
\end{equation}
Pure gauge solutions $\updelta\dot g=\delta_g^*\omega$, $\omega=(2 T\scal,-\frac{2 r}{k}L\sld\scal)$, as in~\eqref{EqMSSc2Pert} change $\wt f,f,H_L$ as in~\eqref{EqMSSc2Pert2}, and therefore $\updelta\bfX=-T+\tfrac{r^2}{k}\wh d r^{-1}L$. Let us first choose $L=0$, $T=\bfX$, and replace $\dot g$ by
\begin{equation}
\label{EqMSSc1GaugePot}
  \dot g + \delta_g^*\omega,\quad \omega=(2\bfX\scal,0),
\end{equation}
which for $\sigma=0$ lies in the same weighted space as $\dot g$ itself, and for $\sigma\neq 0$ is still outgoing; note though that $\omega$ has an extra factor of $r$ relative to $\dot g$. This replacement implements the partial gauge $\bfX=0$ (thus $f=0$), \emph{in which we shall work from now on}. The remaining gauge freedom is the following: given any aspherical function $L$, we can add
\begin{equation}
\label{EqMSSc1Gauge1}
   \delta_g^*\omega,\quad \omega=\left(2 T(L)\scal,-\tfrac{2 r}{k}L\sld\scal\right), \qquad T(L):=\tfrac{r^2}{k}\wh d r^{-1}L,
\end{equation}
which ensures that $\updelta\bfX=0$. The change in the quantities $\wt F$ and $J$ upon addition of such a pure gauge term is
\begin{equation}
\label{EqMSSc1InvChange}
  \updelta\wt F=\tfrac{2}{k}\wh\delta^* r^2\wh d r^{-1}L,\quad
  \updelta J=\tfrac{k}{2 r}L-\tfrac{r}{k}\iota_\varpi\wh d r^{-1}L.
\end{equation}

\emph{In the gauge} $\bfX=0$ and for any fixed choice of $L$, we have $(\wt F,J)=(\wt f,H_L)$ as before, and the linearized Einstein equation is again given by the system~\eqref{EqMSSc2wtf}--\eqref{EqMSSc2HL} (with the equation~\eqref{EqMSSc2HT} for $H_T^E$ absent since we are considering scalar $l=1$ modes). Moreover, if $\wt F=J=0$, we have $\wt f=H_L=0$ (and $f=0$), hence $\dot g\equiv 0$, which means our original perturbation $\dot g$ is pure gauge,
\begin{equation}
\label{EqMSSc1Gauge2}
  \dot g=\delta_g^*\omega(L),\quad \omega(L)=\bigl(2 T(L)\scal,-\tfrac{2 r}{k}L\sld\scal\bigr).
\end{equation}

We shall choose $L$ so as to simplify the structure of the linearized Einstein equations further. Namely, let us \emph{define} $H_T^E:=-\frac{k^2}{2 r^2}\wh\tr\wt F$, cf.\ \eqref{EqMSSc2HT}, and note that by~\eqref{EqMSSc1InvChange},
\[
  \updelta H_T^E = k r^{-2}\wh\delta r^2\wh d r^{-1}L = k\Box_g(r^{-1}L).
\]
We shall demonstrate how to arrange $H_T^E+\updelta H_T^E=0$ by choosing $L$ appropriately.

\emph{Let us first consider non-stationary mode solutions, $\Im\sigma\geq 0$, $\sigma\neq 0$.} In this case, $\omega$ in~\eqref{EqMSSc1Gauge1} is an outgoing mode, and
\begin{equation}
\label{EqMSSc1L}
  \Box_g(r^{-1}L)=-k^{-1}H_T^E=\tfrac{k}{2 r^2}\wh\tr\wt F
\end{equation}
has a spherically symmetric and outgoing mode solution $r^{-1}L$, since the right hand side is outgoing and $\wh{\Box_g}(\sigma)$ is invertible on the relevant function spaces by Theorem~\ref{Thm0}. The linearized Einstein equation for $\dot g':=\dot g+\delta_g^*\omega(L)$ now takes the form~\eqref{EqMSSc2fCoord}--\eqref{EqMSSc2Err}, with $\wt F,J$ and $X,Y,Z$ defined relative to $\dot g'$; one can then follow the arguments of the previous section to deduce that the quantities $\wt F,J$ vanish, and therefore $\dot g'\equiv 0$ (since $\dot g'$ is in the gauge $\bfX=0$), so $\dot g=-\delta_g^*\omega(L)$ is pure gauge.

\emph{Next, we consider stationary mode solutions, $\sigma=0$}, with $\dot g\in\Hbext^{\infty,\ell}$, where we only assume $\ell<-\half$ (as in part~\eqref{ItMSSc1} of Theorem~\ref{ThmMS}). Since $\wh{\delta^*}$ acts on the stationary $\bfX\in\Hbext^{\infty,\ell-1}$ as an element of $\rho\Diffb^1$, we have $\wh{\delta^*}\bfX\in\Hbext^{\infty,\ell}$; $\wt f$ lies in the same space. The right hand side of~\eqref{EqMSSc1L} thus lies in $\Hbext^{\infty,\ell+2}$. Since $\wh{\Box_g}(0)\colon\Hbext^{\infty,\ell}\to\Hbext^{\infty,\ell+2}$ is surjective for any $\ell<-\half$, $\ell\notin-\half-\N$, we can solve~\eqref{EqMSSc1L} with $r^{-1}L\in\Hbext^{\infty,\ell}$, so $L\in\Hbext^{\infty,\ell-1}$; moreover, $L$ is spherically symmetric since $\wt F$ is. Therefore, $\omega(L)\in\Hbext^{\infty,\ell-1}$. Letting
\begin{equation}
\label{EqMSSc1GaugePot2b}
  \dot g'=\dot g+\delta_g^*\omega(L),
\end{equation}
the quantities $\wt F$ and $J$ for $\dot g'$ change, relative to those for $\dot g$, by terms in $\Hbext^{\infty,\ell}$ in view of~\eqref{EqMSSc1InvChange}, hence $\dot g'\in\Hbext^{\infty,\ell}$ lies in the same weighted space as $\dot g$; and we now have $\wh\tr\wt F=0$.

We will apply a modification of the arguments of the scalar $l\geq 2$, $\sigma=0$ discussion in~\S\ref{SssMSSc2} to $\dot g'$. We still have $X,Y\in\Hbext^{\infty,\ell}$, but the quantities in~\eqref{EqMSSc2Master0} (at $\sigma=0$) are now $P_X=9 x(x-2)$, $P_Y=-9 x(3 x-2)$ (which have $r^{-1}$ more decay at infinity than in the case $\ell\geq 2$) and $\wt H=24\bhm^2 r^{-3}$ (likewise), hence
\[
  C_X=\tfrac{r(\bhm-r)}{2\bhm},\ C_Y=\tfrac{r(r-3\bhm)}{2\bhm},\quad
  V=2\bhm r^{-3}(1-\tfrac{2\bhm}{r});
\]
the master quantity $\Phi$ is thus
\[
  \Phi = C_X X+C_Y Y = C_X(X+Y) + (C_Y-C_X)Y = C_X(X+Y) + \tfrac{r^2}{\bhm}\mu Y \in \Hbext^{\infty,\ell-2}.
\]
As discussed around~\eqref{EqMSSc2VF}, $X+Y=4 J$ is smooth at $r=2\bhm$, and so is $\mu Y=-\mu^2(\wt F_{r r}+2 J\hat g_{r r})$, hence so is $\Phi$. The ODE~\eqref{EqMSSc2MasterEq0} has two linearly independent solutions
\[
  \Phi_1=r,\quad \Phi_2=\tfrac{r}{2\bhm}\log(1-\tfrac{2\bhm}{r}),
\]
in $r>2\bhm$. The smoothness of $\Phi$ at $r=2\bhm$ implies that $\Phi=c\Phi_1$, $c\in\R$. By~\eqref{EqMSSc2MasterRec} with $\sigma=0$ and $k^2=2$, this implies $X=-c$, $Y=-c$, hence $Z=0$ by~\eqref{EqMSSc2Etr}, and therefore $\wt F=0$ and $J=-\half c$ in~\eqref{EqMSSc2XYZ}. But since $J\in\Hbext^{\infty,\ell}$, we must have $c=0$ if $\ell>-\tfrac32$, hence $\Phi=0$ and $X=Y=0$, and we are done. If on the other hand $\ell<-\tfrac32$, then $c$ may be non-zero; but by~\eqref{EqMSSc1Quant}, and since $\bfX=0$, the metric perturbation then equals
\[
  \dot g = -c r^2\scal\slg = \delta_g^*\omega,\quad
  \omega=\begin{pmatrix} 0 \\ -\tfrac{2 c r^2}{k^2}\sld\scal \end{pmatrix}
\]
by~\eqref{EqMSSc2Pert} (or~\eqref{EqMSSc1InvChange} with $L=c\tfrac{r}{k}$, $T=T(L)=0$), hence is pure gauge with gauge potential $\omega\in\rho^{-1}\CI\subset\Hbext^{\infty,\ell-1}$, as desired.

%%%%%%%%%%%%%%%%%%%%%%%%%%%%%%
\subsubsection{Spherically symmetric modes (\texorpdfstring{$l=0$}{l equal to 0})}
\label{SssMSSc0}

We follow the linearization of the argument in~\cite{SchleichWittBirkhoff} as in \cite[\S7.2]{HintzVasyKdSStability} and \cite[\S5.5]{HintzKNdSStability}. Thus, we work with the coordinates $(t_0,r)$ in which $g=\mu\,d t_0^2-2 d t_0\,d r-r^2\slg$. Rather than using the form~\eqref{EqYMSym2} of spherically symmetric (scalar type $l=0$) metric perturbations, we write $\dot g$ in the form
\[
  \dot g=\dot\mu\,d t_0^2-2\dot X\,d t_0\,d r+\dot Z\,d r^2 - 2 r^2\dot Y\slg,
\]
with coefficients which are smooth functions of $(t_0,r)$ in $0<r<\infty$. To describe pure gauge perturbations, we calculate that as a map between sections of $\la\pa_{t_0}\ra\oplus\la\pa_r\ra$ and $\la d t_0^2\ra\oplus\la 2\,d t_0\,d r\ra\oplus\la d r^2\ra\oplus\la\slg\ra$, so in particular $\dot g=(\dot\mu,-\dot X,\dot Z,-2 r^2\dot Y)$, we have
\[
  \cL_{(-)} g=
  \begin{pmatrix}
    2\mu\pa_{t_0} & -2\pa_{t_0}+\mu' \\
    \mu\pa_r-\pa_{t_0} & -\pa_r \\
    -2\pa_r & 0 \\
    0 & -2 r
  \end{pmatrix}.
\]
Therefore, putting
\begin{equation}
\label{EqMSSc0Z}
  Z_1:=\half\int_{3\bhm}^r \dot Z\,d r, \qquad
  \omega:=Z_1\pa_{t_0}-r\dot Y\pa_r = \begin{pmatrix} Z_1 \\ -r\dot Y \end{pmatrix},
\end{equation}
we have
\[
  \updelta\dot g:=\cL_\omega g
  =\begin{pmatrix}
     2\mu\pa_{t_0}Z_1 +2 r\pa_{t_0}\dot Y-r\mu'\dot Y \\
     \half\mu\dot Z-\pa_{t_0}Z_1+\pa_r(r\dot Y) \\
     -\dot Z \\
     2 r^2\dot Y
   \end{pmatrix}.
\]

Consider the case that $\dot g$ is a mode. If $\sigma=0$, so $\dot g\in\Hbext^{\infty,\ell}$ and $\omega$ are stationary, then
\[
  Z_1,\ \omega\in\Hbext^{\infty,\ell-1},\quad \updelta\dot g\in\Hbext^{\infty,\ell}.
\]
When $\sigma\neq 0$, we integrate in~\eqref{EqMSSc0Z} along level sets of $t_*$, thus $\omega$, $\updelta\dot g$ are outgoing modes as well. If $\dot g\in\Poly(t_*)^k\Hbext^{\infty,\ell}$ is a generalized zero mode as in part~\eqref{ItMSSc0} of Theorem~\ref{ThmMS}, then integration gives $\omega\in\Poly(t_*)^k\Hbext^{\infty,\ell'}$ for some $\ell'\leq\ell-1$, and therefore $\updelta\dot g\in\Poly(t_*)^k\Hbext^{\infty,\ell'-1}$. Replacing $\dot g$ by $\dot g+\updelta\dot g$, we have reduced to the case
\[
  \dot g=\dot\mu\,d t_0^2-2\dot X\,d t_0\,d r,
\]
without increasing the highest power of $t_*$ in the generalized zero mode setting.

Next, the Einstein equation $\Ric(\check g)=0$ for metrics $\check g=\check\mu\,d t_0^2-2\check X\,d t_0\,d r-r^2\slg$, with $\check\mu,\check X$ functions of $(t_0,r)$, has $d r^2$-component $\pa_r\check X/(2 r\check X)=0$, thus implies $\pa_r\check X=0$ when $\check X\neq 0$. For $\check g=g$, we have $\check X=1$; the linearized Einstein equation for $\dot g$ thus implies $\pa_r\dot X=0$, so
\begin{equation}
\label{EqMSSc0dotX}
  \dot X=\tilde X(t_0)
\end{equation}
for some function $\tilde X$; since $\dot X$ is a mode, we have $\tilde X(t_0)=\tilde X(0)e^{-i\sigma t_0}$. Thus, for $\sigma\neq 0$, $\dot X$ violates the outgoing condition unless $\tilde X=0$; for $\sigma=0$, the membership $\dot g\in\Hbext^{\infty,\ell}\subset\cA^{\ell+3/2}$ (so $\dot g=o(1)$, thus $\dot X=o(1)$) likewise forces $\dot X\equiv 0$. Therefore, when $\dot g$ is a mode, we have
\begin{equation}
\label{EqMSSc0Final}
  \dot g=\dot\mu\,d t_0^2.
\end{equation}
The spherical component of the linearized Einstein equation then reads $-\pa_r(r\dot\mu)=0$, therefore $\dot\mu=-\frac{2\dot\bhm}{r}$, where $\dot\bhm=\dot\bhm(t_0)$ is a constant of integration. The $d t_0^2$ component of the linearized Einstein equation however implies $\pa_{t_0}\dot\bhm(t_0)=0$, so $\dot\bhm$ is in fact constant. This shows that $\dot g$ in~\eqref{EqMSSc0Final} is necessarily stationary (in particular, if $\dot g$ is a mode solution with non-zero frequency, it must vanish), and in fact equal to the metric perturbation arising by an infinitesimal change of the Schwarzschild mass: $\dot g=\dot g_{(\bhm,0)}^0(\dot\bhm,0)$.

In the generalized zero mode case, we can only conclude $\dot X\in\Poly(t_0)^k$ from~\eqref{EqMSSc0dotX}; however, we cannot conclude that $\dot X=0$ since $\dot X$ does not necessarily decay at infinity. Instead, write $\dot X=\pa_{t_0}f$ with $f=f(t_0)\in\Poly(t_0)^{k+1}$; note that this has an extra power of $t_0$. Then
\[
  \updelta\dot g:=\cL_{2 f\pa_{t_0}}g = (4\mu\dot X,\,-2\dot X,\,0,\,0),
\]
and therefore $\dot g+\updelta\dot g$ is now of the form~\eqref{EqMSSc0Final}, hence a linearized Schwarzschild metric as before.

%%%%%%%%%%%%%%%%%%%%%%%%%%%%%%%%%%%%%%%%%%%%%%%%%%
\subsection{Vector type perturbations}
\label{SsMSVec}

We discuss $l\geq 2$ and $l=1$ modes separately. We denote by $\vect\in\vect_l$ (in particular $\sldelta\vect=0$) a spherical harmonic 1-form with eigenvalue $k^2$, $k=(l(l+1)-1)^{1/2}$.

%%%%%%%%%%%%%%%%%%%%%%%%%%%%%%
\subsubsection{Modes with \texorpdfstring{$l\geq 2$}{l at least 2}}
\label{SsMSVec2}

Here, $k^2\geq 5$. To study metric perturbations of the form
\begin{equation}
\label{EqMSVect2}
  \dot g=\begin{pmatrix} 0 \\ r f\otimes\vect \\ -\tfrac{2}{k} r^2 H_T\sldelta^*\vect \end{pmatrix},
\end{equation}
we first compute the form of pure gauge solutions of the same type, to wit
\[
  \updelta\dot g=\delta_g^*\omega=\begin{pmatrix} 0 \\ r^2\wh d r^{-1} L\otimes\vect \\ 2 r L\sldelta^*\vect \end{pmatrix}, \qquad
  \omega=\begin{pmatrix} 0 \\ 2 r L\vect \end{pmatrix}.
\]
The change in the parameters $f,H_T$ of $\dot g$ upon adding $\updelta\dot g$ is thus
\[
  \updelta f=r\wh d r^{-1}L, \quad
  \updelta H_T=-\tfrac{k}{r}L,
\]
which implies that the quantity
\[
  J := f+\tfrac{r}{k}\wh d H_T
\]
is gauge-invariant. If $J=0$, then $\dot g$ is a mode solution; indeed,
\[
  \dot g=\delta_g^*\bigl(0,-\tfrac{2 r^2}{k}H_T\vect\bigr).
\]
The gauge potential here has an extra factor of $r$ relative to $\dot g$.

Expressing the vector type tensor $D_g\Ric(\dot g)$ similarly to~\eqref{EqMSVect2} in terms of $f^E$ (aspherical 1-form) and $H_T^E$ (aspherical function), the linearized second Bianchi identity gives
\[
  r^{-2}\wh\delta r^3 f^E+\tfrac{k^2-1}{k}H_T^E = 0.
\]
Thus, $f^E=0$ automatically gives $H_T^E=0$. The linearized Einstein equation for $\dot g$ is thus equivalent to $f^E=0$; using the calculations starting with~\eqref{EqMSExpr}, see also \cite[Equation~(5.73)]{HintzKNdSStability}, this takes the form
\begin{equation}
\label{EqMSVec2LinEin}
  r^{-2}\wh\delta r^4\wh d r^{-1}J - (k^2-1)r^{-1}J = 0.
\end{equation}
Now, on an orientable, signature $(p,q)$ pseudo-Riemannian manifold of dimension $n=p+q$, one has $\star^2=(-1)^{k(n-k)+q}$ and $\delta=(-1)^{k(n-k+1)+q}\star d\star$ on $k$-forms. Therefore,~\eqref{EqMSVec2LinEin} implies
\[
  \wh\star r J=\wh d(r\Phi),\quad
  \Phi:=\tfrac{1}{k^2-1}r^3\wh\delta r^{-1}\wh\star J.
\]
Applying $r\wh\delta r^{-2}=r\wh\star\wh d\,\wh\star r^{-2}$ from the left, we obtain the equation $\frac{k^2-1}{r^2}\Phi=r\wh\delta r^{-2}\wh d r\Phi$, so
\begin{equation}
\label{EqMSVec2Master}
  (\wh\Box-V)\Phi=0,\quad V=r^{-2}\Bigl(k^2+1-\frac{6\bhm}{r}\Bigr).
\end{equation}
In static coordinates $(t,r)$, this reads
\begin{equation}
\label{EqMSVec2Master2}
  (D_t^2-(\mu D_r)^2-\mu V)\Phi = 0.
\end{equation}
Note that $V>0$ in $r>2\bhm$, and its asymptotic behavior is $V=(k^2+1)r^{-2}+\cO(r^{-3})$. To show that $\dot g$ is pure gauge, it suffices to show that $\Phi$ must vanish. In $\Im\sigma>0$, this follows from the positivity of $V$ and an integration by parts argument as in~\S\ref{SssMSSc2}; note that $\Phi$ is outgoing at $r=2\bhm$ and rapidly decaying as $r\to\infty$. When $\sigma\in\R$ is non-zero, $\Phi$ is outgoing on the real line $\R_{r_*}$, as follows from its definition; therefore, a boundary pairing or Wronskian argument implies the vanishing of $\Phi$ in this case as well.

For $\sigma=0$ finally, note that $J\in\Hbext^{\infty,\ell}$, therefore $\Phi\in\Hbext^{\infty,\ell-1}$, in particular $\Phi=o(r)$. In view of~\eqref{EqMSVec2Master}, the asymptotic behavior of $\Phi$ as $r\to\infty$ is governed by the indicial roots of $r^2\pa_r^2-(k^2+1)$, which are $\lambda_\pm=\half(1\pm\sqrt{4 k^2+5})$; using $k^2\geq 5$, the solution $r^{\lambda_+}$ corresponding to $\lambda_+\geq 3$ is excluded, hence $\Phi=\cO(r^{\lambda_-})=\cO(r^{-2})$. Thus, one can pair~\eqref{EqMSVec2Master2} with $\Phi$ and integrate by parts, obtaining again $\Phi=0$.

%%%%%%%%%%%%%%%%%%%%%%%%%%%%%%
\subsubsection{Modes with \texorpdfstring{$l=1$}{l equal to 1}}
\label{SsMSVec1}

Finally, we consider a metric perturbation $\dot g=(0,r f\otimes\vect,0)$ with $\vect\in\vect_1$. Pure gauge solutions of the same type are
\[
  \updelta\dot g=\delta_g^*\omega=\begin{pmatrix} 0 \\ r^2\wh d r^{-1}L\otimes\vect \\ 0 \end{pmatrix},\quad \omega=\begin{pmatrix} 0 \\ 2 r L\vect \end{pmatrix},
\]
and adding them to  $\dot g$ changes $f$ by $\updelta f=r\wh d r^{-1}L$. Therefore, the quantity $\wh d(r^{-1}f)$ is gauge-invariant; note that this is a (top degree) differential 2-form on $\wh X$.

If $\wh d(r^{-1}f)$ vanishes, then $\dot g$ is pure gauge: indeed, since $\wh X$ is contractible, we can write $r^{-1}f=\wh d(r^{-1}L)$, and then
\begin{equation}
\label{EqMSVec1GaugePota}
  \dot g=\delta_g^*\omega,\quad \omega=(0,2 r L\vect).
\end{equation}
Let us determine the size of $L$: in the splitting $T^*\wh X=\la d t_*\ra \oplus \la d r\ra$, and using the basis $d t_*\wedge d r$ for 2-forms, the exterior differential $\wh d$ acting on modes of frequency $\sigma$, that is, the operator $\wh d(\sigma):=e^{i\sigma t_*}\wh d e^{-i\sigma t_*}$, is given by
\[
  \wh d(\sigma) = \begin{pmatrix} -i\sigma \\ \pa_r \end{pmatrix}\ \ \text{on functions}, \qquad
  \wh d(\sigma) = (-\pa_r,\,-i\sigma)\ \ \text{on 1-forms}.
\]
If $\sigma\neq 0$, write $L=e^{-i\sigma t_*}L'$, $f=e^{-i\sigma t_*}f'$. Then $\wh d(r^{-1}f)=0$ is equivalent to $\pa_r(r^{-1}f'_{t_*})+i\sigma r^{-1}f'_r=0$. Thus, for $L'=\frac{i}{\sigma}f'_{t_*}$, we have $\wh d(\sigma)(r^{-1}L')=r^{-1}f'$; hence, we can take $L = \tfrac{i}{\sigma}f_{t_*}$, which is outgoing, in~\eqref{EqMSVec1GaugePota}. In the stationary case $\sigma=0$, we have $\pa_r(r^{-1}f_{t_*})=0$, so $f_{t_*}=c r$, $c\in\C$; since $f\in\Hbext^{\infty,\ell}$ is $o(1)$ as $r\to\infty$, we must have $c=0$, so $f_{t_*}\equiv 0$. Therefore, we have
\[
  r^{-1}f=r^{-1}f_r\,d r=\wh d(r^{-1}L),\quad L:=r\int_\infty^r r^{-1}f_r\,d r,
\]
with $L\in\Hbext^{\infty,\ell-1}$, thus also $\omega\in\Hbext^{\infty,\ell-1}$ in~\eqref{EqMSVec1GaugePota}.

Returning to the study of $\dot g$, the linearized Einstein equation is equivalent to the single equation $r^{-2}\wh\delta r^4\wh d r^{-1}f=0$ for $f$, thus
\[
  \wh d(\wh\star r^4\wh d r^{-1}f)=0,
\]
i.e.\ the \emph{function} $\wh\star r^4\wh d r^{-1}f$ is constant. Now, differentiating the Kerr family in the angular momentum parameter, we recall from~\eqref{EqKaLin2} that
\[
  \dot g_{(\bhm,0)}^0(0,\bha) = 2 r f_\bhm\otimes_s\vect, \quad
  f_\bhm=2\bhm r^{-2}\,d t_0+r^{-1}\,d r,\ \vect=\sin^2\theta\,d\varphi=(\pa_\varphi)^\flat.
\]
In particular, $\wh\star r^4\wh d r^{-1}f_\bhm=6\bhm$; this equality is independent of the particular presentation of the Kerr family since $\wh d r^{-1}f$ is gauge-independent. Therefore, replacing $\dot g$ by $\dot g-\dot g_{(\bhm,0)}(0,\dot\bha)$ for a suitable linearized angular momentum $\dot\bha\in\R$, we may assume that $\wh\star r^4\wh d r^{-1}f\equiv 0$, hence $\wh d r^{-1}f=0$, so $\dot g$ is pure gauge.

This finishes the proof of Theorem~\ref{ThmMS}.

%%%%%%%%%%%%%%%%%%%%%%%%%%%%%%%%%%%%%%%%%%%%%%%%%%%%%%%%%%%%%%%%%%%%%%
\section{Modes of the unmodified linearized gauge-fixed Einstein operator}
\label{SL}

We now combine the results of the previous sections to study the linearized unmodified gauge-fixed Einstein operator
\[
  L_g = 2(D_g\Ric + \delta_g^*\delta_g\sfG_g)
\]
from~\eqref{EqOpLinOp} on the \emph{Schwarzschild spacetime} $g=g_{b_0}$. In~\S\ref{SsL0}, we prove the absence of non-zero modes of this operator in the closed upper half plane $\Im\sigma\geq 0$, and compute the space of zero modes. In~\S\ref{SsL0Lin}, we find all generalized zero modes on Schwarzschild spacetimes which grow at most linearly in $t_*$; these include asymptotic Lorentz boosts. We briefly discuss (at least) quadratically growing generalized modes of $L_{g_{b_0}}$ in~\S\ref{SsL0Qu}; they do exist, but do \emph{not} satisfy the linearized Einstein equation. We will eliminate such pathological modes by means of constraint damping in~\S\ref{SCD}.

In~\S\S\ref{SsL0}--\ref{SsL0Lin}, we will at the same time construct spaces of zero modes and linearly growing zero modes of $L_{g_b}$, where $b=(\bhm,\bha)$ denotes Kerr black hole parameters close to $b_0$, which have the same dimension as the corresponding spaces for $L_{g_{b_0}}$. However, we stress that, prior to controlling the resolvent of $L_{g_{b_0}}$ on a Schwarzschild spacetime in a neighborhood of $\sigma=0$ in a suitably non-degenerate manner, cf.\ the toy model at the end of~\S\ref{SssIIK}, we cannot even prove the absence of small non-zero modes for $L_{g_b}$. Thus, the results in~\S\S\ref{SsL0}--\ref{SsL0Lin} are merely \emph{existence} results for generalized zero modes of $L_{g_b}$ when $\bha\neq 0$. The constraint damping modification $L_{g_b,\gamma}$, $\gamma\neq 0$, discussed in~\S\S\ref{SsOpCD} and \ref{SCD}, does have a non-degenerate (and rather explicit) resolvent, as we show in~\S\ref{SR}.

%%%%%%%%%%%%%%%%%%%%%%%%%%%%%%%%%%%%%%%%%%%%%%%%%%
\subsection{Modes in \texorpdfstring{$\Im\sigma\geq 0$}{the upper half plane}}
\label{SsL0}

We now prove the analogue of Theorems~\ref{Thm0} and \ref{Thm1}; we use the notation for 1-forms from Theorem~\ref{Thm1} and Propositions~\ref{Prop10Grow} and \ref{Prop10Genv1}, and define the spectral family $\wh{L_{g_b}}(\sigma)$ for $b=(\bhm,\bha)$ using the function $t_{\bhm,*}$ as in~\eqref{EqOpLinFT}.

\begin{prop}
\label{PropL0}
  The spectral family of $L_{g_{b_0}}$ on the Schwarzschild spacetime has the following properties:
  \begin{enumerate}
  \item For $\Im\sigma\geq 0$, $\sigma\neq 0$, the operator
    \begin{align*}
      \wh{L_{g_{b_0}}}(\sigma)\colon&\{\omega\in \Hbext^{s,\ell}(X;S^2\,\wt\Tsc{}^*X)\colon\wh{L_{g_{b_0}}}(\sigma)\omega\in\Hbext^{s,\ell+1}(X;S^2\,\wt\Tsc{}^*X)\} \\
        &\qquad \to\Hbext^{s,\ell+1}(X;S^2\,\wt\Tsc{}^*X)
    \end{align*}
    is invertible when $s>\tfrac52$, $\ell<-\half$, $s+\ell>-\half$.
  \item For $s>\tfrac52$ and $\ell\in(-\tfrac32,-\half)$, the zero energy operator
    \begin{equation}
    \label{EqL0Op}
    \begin{split}
      \wh{L_{g_{b_0}}}(0)\colon&\{\omega\in\Hbext^{s,\ell}(X;S^2\,\wt\Tsc{}^*X)\colon\wh{L_{g_{b_0}}}(0)\omega\in\Hbext^{s-1,\ell+2}(X;S^2\,\wt\Tsc{}^*X)\} \\
        &\qquad \to\Hbext^{s-1,\ell+2}(X;S^2\,\wt\Tsc{}^*X)
    \end{split}
    \end{equation}
    has 7-dimensional kernel and cokernel.
  \end{enumerate}
  
  The second statement holds also for $\wh{L_{g_b}}(0)$ with $b=(\bhm,\bha)$ near $(\bhm_0,0)$; concretely,
  \begin{subequations}
  \begin{alignat}{3}
  \label{EqL0Ker}
    &\ker\wh{L_{g_b}}(0) \cap \Hbext^{\infty,-1/2-}& &=&\ \la h_{b,s 0}\ra \oplus \{ h_{b,v 1}(\vect)\colon\vect\in\vect_1 \} \oplus \{ h_{b,s 1}(\scal)\colon \scal\in\scal_1 \}, \\
  \label{EqL0KerAdj}
    &\ker\wh{L_{g_b}}(0)^* \cap \Hbsupp^{-\infty,-1/2-}& &=& \la h_{b,s 0}^*\ra \oplus \{ h_{b,v 1}^*(\vect) \colon \vect\in\vect_1 \} \oplus \{ h_{b,s 1}^*(\scal)\colon \scal\in\scal_1 \},
  \end{alignat}
  \end{subequations}
  where the $h_{b,\bullet\bullet}^{(*)}$ depend continuously on $b$, with
  \begin{subequations}
  \begin{alignat}{3}
  \label{EqL0s0}
    h_{b,s 0} &= \delta_{g_b}^*\omega_{b,s 0}, & h_{b,s 0}^* &= \sfG_{g_b}\delta_{g_b}^*\omega_{b,s 0}^*, \\
  \label{EqL0s1}
    h_{b,s 1}(\scal) &= \delta_{g_b}^*\omega_{b,s 1}(\scal),\quad & h_{b,s 1}^*(\scal) &= \sfG_{g_b}\delta_{g_b}^*\omega_{b,s 1}^*(\scal), \\
  \label{EqL0v1}
    h_{b,v 1}(\vect)&=\dot g_b(\dot b)+\delta_{g_b}^*\omega,\quad & h_{b,v 1}^*(\vect) &= \sfG_{g_b}\delta_{g_b}^*\omega_{b,v 1}^*(\vect),
  \end{alignat}
  \end{subequations}
  where $\dot b$, $\omega\in\Hbext^{\infty,-3/2-}$ depend on $b,\vect$; here $\scal\in\scal_1$, $\vect\in\vect_1$. At $b=b_0$, we have $h_{b_0,v 1}(\vect)=2\omega_{b_0,s 0}\otimes_s\vect$. The dual states are supported in $r\geq r_b$, $\CI$ in $r>r_b$, conormal at $\pa_+X$ with the stated weight, and lie in $H^{-3/2-}$ near the event horizon.

  Furthermore, all zero modes are solutions of the linearized Einstein equation and satisfy the linearized gauge condition; that is, $\ker\wh{L_{g_b}}(0)\cap\Hbext^{\infty,-1/2-}\subset\ker D_{g_b}\Ric\cap\ker D_{g_b}\Ups(-;g_b)$ in the notation of Definition~\ref{DefOpGauge}.
\end{prop}

The zero energy states are linear combinations of linearized Kerr metrics and pure gauge tensors. The dual states all turn out to be \emph{dual-pure-gauge} solutions: recall from~\eqref{EqOpCDAdjoint} (with $E=0$) that $L_g^*$ annihilates $\sfG_g\delta_g^*\omega^*$ if $0=\delta_g(\sfG_g\delta_g^*\omega^*)=\half\Box_{g,1}\omega^*$, which is the origin of the expressions on the right in~\eqref{EqL0s0}--\eqref{EqL0v1}. We also note, as in Remark~\ref{Rmk10AdjDecay}, that elements of $\ker\wh{L_{g_b}}(0)^*$ which lie in the dual space $\Hbsupp^{-s+1,-\ell-2}$ of the range of~\eqref{EqL0Op} automatically have the decay rate $-\half-$ by normal operator arguments as in Proposition~\ref{PropOpNull}.

The vector $l=1$ modes $h_{b_0,v 1}(\vect)$ on Schwarzschild spacetimes are linearizations of the Kerr family (plus a suitably chosen pure gauge term) in the angular momentum, see the arguments following~\eqref{EqL0V1} below. On the other hand, the linearized Schwarzschild family, i.e.\ the linearization of $g_{(\bhm,0)}$ in $\bhm$, does not appear here due to our choice of gauge (see also Remark~\ref{RmkCDUGauge}); it shows up only as a generalized zero mode in~\S\ref{SsL0Lin}.

\begin{proof}[Proof of Proposition~\ref{PropL0}]
  \emph{We first consider the Schwarzschild case $b=b_0=(\bhm_0,0)$, and write $g=g_{b_0}$.} Consider a non-zero mode solution $\wh{L_g}(\sigma)h=0$, $\Im\sigma\geq 0$. The linearized second Bianchi identity implies
  \[
    \delta_g\sfG_g\delta_g^*(\delta_g\sfG_g h)=0.
  \]
  If $\sigma\neq 0$, then $\delta_g\sfG_g h$ is an outgoing mode; if $\sigma=0$, then $\delta_g\sfG_g h\in\Hbext^{\infty,1/2-}$ by~\eqref{EqKStDivSymmGrad}. In both cases, Theorem~\ref{Thm1} and the fact that the generator $\omega_{b_0,s 0}$ of the kernel in~\eqref{Eq1s0Ker} does \emph{not} lie in $\Hbext^{\infty,1/2-}$ imply
  \begin{equation}
  \label{EqL0LinGauge}
    \delta_g\sfG_g h = 0
  \end{equation}
  and thus also
  \begin{equation}
  \label{EqL0LinEin}
    D_g\Ric(h) = 0.
  \end{equation}

  Next, we apply the mode stability result, Theorem~\ref{ThmMS}. Consider first the case $\sigma\neq 0$; then $h=\delta_g^*\omega$ with $\omega$ an outgoing mode; plugging this into~\eqref{EqL0LinGauge}, we obtain $\Box_{g,1}\omega=0$ and hence $\omega=0$ by Theorem~\ref{Thm1}, thus $h=0$. This proves the injectivity of $\wh{L_g}(\sigma)$ for non-zero $\sigma$ with $\Im\sigma\geq 0$, hence its invertibility by Theorem~\ref{ThmOp}.

  Suppose $\sigma=0$; without loss, we can assume $h$ to be of pure type. If $h$ is of scalar or vector type $l\geq 2$, then $h=\delta_g^*\omega$ with $\omega\in\Hbext^{\infty,-3/2-}$, and again $\omega\in\ker\wh{\Box_g}(0)$ by~\eqref{EqL0LinGauge}. By Proposition~\ref{Prop10Grow}, there are no non-trivial such scalar or vector type $l\geq 2$ 1-forms, giving $\omega=0$ and thus $h=0$ again. We study the remaining three types of zero modes separately.
  
  \pfstep{Scalar type $l=0$ modes.} If $h$ is of scalar type $l=0$, then
  \[
    h = \dot g_{b_0}^0(\dot\bhm,0) + \delta_g^*\omega,\quad
    \omega\in\Hbext^{\infty,-3/2-},
  \]
  and~\eqref{EqL0LinGauge} gives the additional condition
  \[
    \wh{\Box_g}(0)\omega = -2\delta_g\sfG_g\dot g_{b_0}^0(\dot\bhm,0) = \tfrac{4\dot\bhm}{r^2}d t_0.
  \]
  Integrating this against $\omega_{b_0,s 0}^*$, see~\eqref{Eq1s0}, which annihilates the range of $\wh{\Box_g}(0)$ on $\Hbext^{\infty,\ell}$ for any $\ell\in\R$, the left hand side gives $0$, while the right hand side equals $-8\dot\bhm$ in view of
  \begin{equation}
  \label{EqL0SchwGauge}
    \la\delta_g\sfG_g\dot g_{b_0}^0(1,0), \omega_{b_0,s 0}^*\ra = 4.
  \end{equation}
  Therefore, $\dot\bhm=0$, and $h=\delta_g^*\omega$ is pure gauge, with $\omega$ of scalar type $l=0$ and of size $o(r)$; by Proposition~\ref{Prop10Grow}, we therefore have $\omega=c_1\omega_{b_0,s 0}+c_2\omega_{b_0,s 0}^{(0)}$ for some $c_1,c_2\in\C$. In view of~\eqref{Eq10Grows0}, this gives $h=\delta_g^*\omega=c_2\delta_g^*\omega_{b_0,s 0}=c_2 h_{b,s 0}$, as advertised in~\eqref{EqL0s0}. Conversely, $h_{b_0,s 0}\in\ker\wh{L_g}(0)$; in fact, $h_{b_0,s 0}$ satisfies both equation~\eqref{EqL0LinGauge} (by construction in terms of the zero mode $\omega_{b,s 0}$ of $\delta_g\sfG_g\delta_g^*$) and equation~\eqref{EqL0LinEin} (since it is pure gauge).

  \pfstep{Scalar type $l=1$ modes.} If $h$ is of scalar type $l=1$, then $h=\delta_g^*\omega$, with $\omega\in\Hbext^{\infty,-3/2-}\cap\ker\wh{\Box_g}(0)$ of scalar type $l=1$, which means $\omega=c\omega_{b_0,s 1}(\scal)$ for some $\scal\in\scal_1$ by Proposition~\ref{Prop10Grow}, and thus $h=c\delta_g^*\omega_{b_0,s 1}(\scal)=c h_{b_0,s 1}(\scal)$ indeed, giving~\eqref{EqL0s1}. Conversely, $h_{b_0,s 1}(\scal)$ satisfies equations~\eqref{EqL0LinGauge}--\eqref{EqL0LinEin}, hence is, a fortiori, a zero mode of $L_g$.

  \pfstep{Vector type $l=1$ modes.} For $h$ of vector type $l=1$, we have
  \[
    h = \dot g_{b_0}(0,\dot\bha) + \delta_g^*\omega,\quad \omega\in\Hbext^{\infty,-3/2-},
  \]
  with $\dot\bha\in\R^3$, and $\omega$ moreover satisfies $\wh{\Box_g}(0)\omega=-2\delta_g\sfG_g\dot g_{b_0}(0,\dot\bha)\in\Hbext^{\infty,1/2-}$. Since the vector type $l=1$ kernel of $\wh{\Box_g}(0)$ on $\Hbext^{\infty,-3/2-}$ is trivial, there is at most one such $\omega$ for fixed $\dot\bha$. Conversely, given $\dot\bha\in\R^3$, the existence of $\omega$ is guaranteed by duality since $\ker\wh{\Box_g}(0)^*\cap\Hbsupp^{-\infty,-1/2+}$ does not contain non-zero vector $l=1$ solutions by Theorem~\ref{Thm1}. Thus, the vector $l=1$ nullspace of $\wh{L_g}(0)$ is $3$-dimensional and spanned by such $h$.
  
  We can easily find $h$ explicitly by writing it as
  \begin{equation}
  \label{EqL0V1}
    h = \dot g_{b_0}^0(0,\dot\bha) + \delta_g^*\omega^0
  \end{equation}
  and verifying that it has the required decay at infinity. Here, rescaling to $|\dot\bha|=1$, $\omega^0$ solves
  \[
    \wh{\Box_g}(0)\omega^0 = -2\delta_g\sfG_g\dot g_{b_0}^0(0,\dot\bha) = -4 r^{-1}\vect,
  \]
  where $\vect=\sin^2\theta\,d\varphi$ in polar coordinates adapted to $\dot\bha$. This has the explicit solution $\omega^0=2(\bhm_0+r)\vect$, hence $\delta_g^*\omega^0=-(1+\tfrac{2\bhm_0}{r})d r\otimes_s\vect$ and therefore indeed
  \[
    h = 4\bhm r^{-1}(d t_0-d r)\otimes_s\vect = 4\bhm\omega_{b_0,s 0}\otimes_s\vect,
  \]
  finishing the calculation of zero energy states on Schwarzschild spacetimes.

  \pfstep{Dual states for Schwarzschild and Kerr spacetimes.} From our calculations thus far, and noting that $L_{g_{b_0}}$, restricted to tensors of pure type, is Fredholm of index $0$, we conclude that the dimensions of the spaces of dual zero energy states of pure type are $1$ (scalar type $l=0$), $3$ (scalar type $l=1$), and $3$ (vector type $l=1$). We can find bases for them by a dual-pure-gauge ansatz, as explained above, using the calculations of 1-forms zero energy dual states in~\S\ref{S1}. In fact, the elements on the right in~\eqref{EqL0s0}--\eqref{EqL0v1} span a 7-dimensional space of dual states also on slowly rotating \emph{Kerr} spacetimes. We briefly check the required decay at infinity: for $h_{b,s 0}^*$, which is a differentiated $\delta$-distribution at the event horizon $r=r_b$, this is clear; for $h_{b,s 1}^*$, it follows from $\wh{\delta_{g_b}^*}(0)\in\rho\Diffb^1$ and $\omega_{b,s 1}^*(\scal)\in\Hbsupp^{-\infty,-3/2-}$; and for $h_{b,v 1}^*$, it was proved in Proposition~\ref{Prop10Genv1}.

  \pfstep{Persistence for slowly rotating Kerr spacetimes.} The argument at the end of~\S\ref{Ss10} shows that the dimensions of $\ker\wh{L_{g_b}}(0)\cap\Hbext^{\infty,-1/2-}$ and $\ker\wh{L_{g_b}}(0)^*\cap\Hbsupp^{-\infty,-1/2-}$ are upper semicontinuous at $b=b_0$, where they are equal to $7$; the explicit construction of a $7$-dimensional space of dual states for $L_{g_b}$ shows that they are at least $7$-dimensional for $b$ near $b_0$. Thus, they equal $7$; the continuous dependence of the nullspace on $b$ then follows by a general functional analytic argument (cf.\ the proof of Lemma~\ref{LemmaCD0Dual}).

  Getting a precise description of the space zero energy states of $L_{g_b}$ requires a direct argument. Now, we certainly have $h_{b,s 0}$, $h_{b,s 1}(\scal)\in\ker\wh{L_{g_b}}(0)$; it remains to construct a continuous family (in $b$) of elements of $\ker\wh{L_{g_b}}(0)\cap\Hbext^{\infty,-1/2-}$ extending $h_{b_0,v 1}(\vect)$. For $\vect\in\vect_1$ which is (dual to) the rotation around the axis $\dot\bha\in\R^3$ (with $\vect$ having angular speed $|\dot\bha|$), we make the ansatz
  \begin{equation}
  \label{EqL0v1Ansatz}
    h_{b,v 1}(\vect) = \dot g_b(\lambda_b(\dot\bha),\dot\bha) + \delta_{g_b}^*\omega,
  \end{equation}
  with $\lambda_b\in(\R^3)^*$, $\lambda_{b_0}=0$, and $\omega\in\Hbext^{\infty,-3/2-}$ to be found. The equation $\wh{L_{g_b}}(0)h_{b,v 1}(\vect)=0$ is then satisfied provided
  \begin{equation}
  \label{EqL0v1Kerr}
    \wh{\Box_{g_b}}(0)\omega = -2\delta_{g_b}\sfG_{g_b}\dot g_b(\lambda_b(\dot\bha),\dot\bha) \in \Hbext^{\infty,1/2-}.
  \end{equation}
  In view of Theorem~\ref{Thm1}, the obstruction for solvability of this is the cokernel $\ker\wh{\Box_{g_b}}(0)^*\cap\Hbsupp^{-\infty,-1/2+}=\la\omega_{b,s 0}^*\ra$. That is, we need to choose $\lambda_b(\dot\bha)$ so that the right hand side of~\eqref{EqL0v1Kerr} is orthogonal to $\omega_{b,s 0}^*$. This gives
  \[
    \lambda_b(\dot\bha) = -\frac{\la\delta_{g_b}\sfG_{g_b}\dot g_b(0,\dot\bha),\omega_{b,s 0}^*\ra}{\la\delta_{g_b}\sfG_{g_b}\dot g_b(1,0),\omega_{b,s 0}^*\ra};
  \]
  note here that the denominator is non-zero $b$ near $b_0$ by continuity from the calculation~\eqref{EqL0SchwGauge} and the orthogonality $\la\delta_{g_{b_0}}\sfG_{g_{b_0}}(\delta_{g_{b_0}}^*\omega),\omega_{b_0,s 0}^*\ra=0$ for any $\omega$. The proof is complete.
\end{proof}

\begin{rmk}
\label{RmkL0RotVF}
  The `asymptotic rotations' $\omega_{b,v 1}(\vect)$ of Proposition~\ref{Prop10Genv1} were not used here, even though they give rise to zero energy states $h_b(\vect):=\delta_{g_b}^*\omega_{b,v 1}(\vect)\in\ker\wh{L_{g_b}}(0)\cap\Hbext^{\infty,-1/2-}$. To explain why they are, in fact, already captured by Proposition~\ref{PropL0}, note first that when $b=(\bhm,0)$ describes a \emph{Schwarzschild} black hole, then $\omega_{b,v 1}(\vect)=r^2\vect$ is dual to a rotation, thus Killing, vector field, hence $h_b(\vect)\equiv 0$. On the other hand, when $b=(\bhm,\bha)$ with $\bha\neq 0$, consider the orthogonal splitting $\vect_1=\la\pa_\varphi^\flat\ra \oplus \vect^\perp$, where $\pa_\varphi$ is unit speed rotation around the axis $\bha/|\bha|$; the latter is a Killing vector field for the metric $g_b$, and thus $h_b(\pa_\varphi^\flat)=0$. On the other hand, $\vect^\perp\ni\vect\mapsto h_b(\vect)$ is now injective; that this does \emph{not} give rise to new (i.e.\ not captured by Proposition~\ref{PropL0}) zero energy states is due to the fact that for such  $b$, the parametrization of the linearized Kerr family $\dot b\mapsto\dot g_b(\dot b)$ is no longer injective when quotienting out by pure gauge solutions, but rather has a 2-dimensional kernel. Indeed, if $\dot\bha\perp\bha$, then $\dot g_{(\bhm,\bha)}(0,\dot\bha)$ is pure gauge: it merely describes the same Kerr black hole with rotation axis rotated infinitesimally, i.e.\ is precisely of the form $h_b(\vect)$ for $\vect\in\vect^\perp$ (plus an extra pure gauge term depending on the presentation of the Kerr family, as described in~\S\ref{SsKa}). In summary then,
  \[
    h_{b,v 1}(\vect_1) + h_b(\vect_1) = h_{b,v 1}(\vect_1),\quad b=(\bhm,\bha),
  \]
  is 3-dimensional for $\bha=0$ as well as for $\bha\neq 0$.
\end{rmk}

%%%%%%%%%%%%%%%%%%%%%%%%%%%%%%%%%%%%%%%%%%%%%%%%%%
\subsection{Generalized stationary modes: linear growth}
\label{SsL0Lin}

The zero energy behavior of the linearized gauge-fixed Einstein operator in the vector $l=1$ sector is non-degenerate, similarly to the wave operator on 1-forms in Lemma~\ref{Lemma10NoLin}:

\begin{lemma}
\label{LemmaL0Linv1}
  Let $d\geq 1$. There does not exist $h=\sum_{j=0}^d t_*^j h_j$ with $h_j\in\Hbext^{\infty,-1/2-}$ of vector type $l=1$ and $h_d\neq 0$ for which $L_{g_{b_0}}h=0$.
\end{lemma}
\begin{proof}
  It suffices to prove this for $d=1$; indeed, in the case $d\geq 2$, given $h$ as in the statement, $\pa_{t_*}h$ has one degree less growth in $t_*$ and still lies in $\ker L_{g_{b_0}}$, hence $\pa_{t_*}h=0$ by induction, and in particular the leading term of $h$ vanishes, contrary to the assumption.

  Now, for $h$ as in the statement of the lemma, we necessarily have $h_1=h_{b_0,v 1}(\vect)$ for some $0\neq\vect\in\vect_1$. We wish to show that there is no $h_0\in\Hbext^{\infty,-1/2-}$ such that
  \[
    L_{g_{b_0}}h_0 = -[L_{g_{b_0}},t_*]h_{b_0,v 1}(\vect).
  \]
  It suffices to check that the pairing of the right hand side with $h_{b_0,v 1}^*(\vect)$ is non-zero; as in the proof of Lemma~\ref{Lemma10NoLin}, the pairing is unchanged upon replacing $t_*$ by $t_0$. For $\vect=(\sin^2\theta)\pa_\varphi$ (which can be arranged in adapted polar coordinates by rescaling $h$), we calculate
  \begin{equation}
  \label{EqL0Linv1Nondeg}
  \begin{split}
    &\la [L_{g_{b_0}},t_0]h_{b_0,v 1}(\vect), h_{b_0,v 1}^*(\vect')\ra \\
    &\qquad =
      \big\la
        2 r^{-2}\bigl(-(1+\tfrac{\bhm_0}{r})d t_0+d r\bigr)\otimes_s\vect,
        4\bhm_0^2\delta(r-2\bhm_0)d r\otimes_s \vect'
      \big\ra \\
    &\qquad = -2(\vol(\Sph^2))^{-1}\la\vect,\vect'\ra_{L^2(\Sph^2;T^*\Sph^2)},
  \end{split}
  \end{equation}
  which is non-zero for $\vect'=\vect$, as desired.
\end{proof}

More generally, there are no non-trivial polynomially bounded solutions which are of a pure type restricted to which $\wh{L_{g_{b_0}}}(0)$ has trivial kernel within $\Hbext^{\infty,-1/2-}$; this applies to scalar type $l\geq 2$ and vector type $l\geq 2$ modes. On the other hand, there \emph{do exist} generalized scalar type $l=0$ and $l=1$ zero modes with linear growth:

\begin{prop}
\label{PropL0Lin}
  For $b_0=(\bhm_0,0)$, define the spaces of generalized zero modes
  \begin{align*}
    \wh\cK_{b_0,s 0} := \left\{ h \in \ker L_{g_{b_0}} \cap \Poly^1(t_*)\Hbext^{\infty,-1/2-} \colon h\ \text{is of scalar type $l=0$} \right\}, \\
    \wh\cK_{b_0,s 1} := \left\{ h \in \ker L_{g_{b_0}} \cap \Poly^1(t_*)\Hbext^{\infty,-1/2-} \colon h\ \text{is of scalar type $l=1$} \right\}.
  \end{align*}
  Then $\wh\cK_{b_0,s 0}=\mathspan\{h_{b_0,s 0},\hat h_{b_0,s 0}\}$, where $h_{b_0,s 0}$ is defined in~\eqref{EqL0s0} and
  \begin{equation}
  \label{EqL0Lins0}
    \hat h_{b_0,s 0} = \dot g^0_{b_0}(-\tfrac14,0) + \delta_{g_{b_0}}^*(t_0\omega_{b_0,s 0}+\breve\omega_{b_0,s 0}^0),\ \ 
    \breve\omega_{b_0,s 0}^0=-\tfrac{r\log(r/2\bhm_0)}{r-2\bhm_0}\pa_t^\flat+\half(1+\tfrac{2\bhm_0}{r})d r.
  \end{equation}
  Furthermore, $\wh\cK_{b_0,s 1}=\mathspan\{ h_{b_0,s 1}(\scal),\,\hat h_{b_0,s 1}(\scal) \colon \scal\in\scal_1\}$ where $h_{b_0,s 1}(\scal)$ is defined~\eqref{EqL0s1}, and $\hat h_{b_0,s 1}=\delta_{g_{b_0}}^*\hat\omega_{b_0,s 1}$, see~\eqref{Eq10GenSc1}.
  
  The generalized modes $\hat h_{b_0,s 0}$, $\hat h_{b_0,s 1}(\scal)$ extend to continuous families
  \begin{align*}
    b&\mapsto\hat h_{b,s 0} = \dot g_b(\dot b)+\delta_{g_b}^*\omega \in \ker L_{g_b} \cap \Poly^1(t_*)\Hbext^{\infty,-1/2-}, \\
    b&\mapsto\hat h_{b,s 1}(\scal):=\delta_{g_b}^*\hat\omega_{b,s 1}(\scal)\in \ker L_{g_b} \cap \Poly^1(t_*)\Hbext^{\infty,-1/2-},\qquad \scal\in\scal_1,
  \end{align*}
  for suitable $\dot b\in\R^4$, $\omega\in\Poly^1(t_*)\Hbext^{\infty,-3/2-}$; they furthermore satisfy $\hat h_{b,s 0}$, $\hat h_{b,s 1}(\scal_1)\in\ker D_{g_b}\Ric\cap\ker D_{g_b}\Ups(-;g_b)$. Similarly, in the notation of Lemma~\ref{Prop10GenSc}, we have
  \begin{equation}
  \label{EqL0LinDual}
  \begin{split}
    b&\mapsto\hat h_{b,s 0}^*=\sfG_{g_b}\delta_{g_b}^*\hat\omega_{b,s 0}^* \in \ker L_{g_b}^* \cap \Poly^1(t_*)\Hbsupp^{-\infty,-1/2-}, \\
    b&\mapsto\hat h_{b,s 1}^*(\scal)=\sfG_{g_b}\delta_{g_b}^*\hat\omega_{b,s 1}^*(\scal)\in \ker L_{g_b}^* \cap \Poly^1(t_*)\Hbsupp^{-\infty,-1/2-},\qquad \scal\in\scal_1.
  \end{split}
  \end{equation}
\end{prop}
\begin{proof}
  Given $\wh\cK_{b_0,s j}\ni h=t_* h_1+h_0$ for $j=0$ or $j=1$, we have $L_{g_{b_0}}h_1=0$. Thus, either $h_1=0$, in which case $h_0$ is a scalar multiple of $h_{b_0,s j}$ by Proposition~\ref{PropL0}; or, after rescaling by a non-zero constant, $h_1=h_{b_0,s j}=\delta_{g_{b_0}}^*\omega_{b_0,s j}$ (depending on $\scal\in\scal_1$ when $j=1$). The second Bianchi identity and $\delta_{g_{b_0}}\sfG_{g_{b_0}}h_{b_0,s j}=0$ then imply, by~\eqref{EqKStDivSymmGrad} and \eqref{Eq10SymGradDecay},
  \begin{align*}
    \ker \Box_{g_{b_0}} \ni \delta_{g_{b_0}}\sfG_{g_{b_0}} h &= [\delta_{g_{b_0}}\sfG_{g_{b_0}},t_*]h_{b_0,s j} + \delta_{g_{b_0}}\sfG_{g_{b_0}} h_0 \\
    & \in \Hbext^{\infty,1/2-} + \Hbext^{\infty,1/2-} = \Hbext^{\infty,1/2-}.
  \end{align*}
  By Theorem~\ref{Thm1}, this implies $\delta_{g_{b_0}}\sfG_{g_{b_0}} h=0$, thus $D_{g_{b_0}}\Ric(h)=0$; thus, writing
  \begin{equation}
  \label{EqL0LinRewrite}
    h = \delta_{g_{b_0}}^*(t_*\omega_{b_0,s j}) + h_0',\quad h_0'=h_0 - [\delta_{g_{b_0}}^*,t_*]\omega_{b_0,s j},
  \end{equation}
  we obtain the two equations
  \begin{subequations}
  \begin{gather}
  \label{EqL0LinEqn1}
    D_{g_{b_0}}\Ric(h_0')=0, \\
  \label{EqL0LinEqn2}
    [\Box_{g_{b_0}},t_*]\omega_{b_0,s j} + 2\delta_{g_{b_0}}\sfG_{g_{b_0}} h_0'=0.
  \end{gather}
  \end{subequations}

  \pfstep{Generalized scalar $l=0$ states.} Here, $j=0$. Then $h_0'\in\Hbext^{\infty,-1/2-}$ by~\eqref{Eq1s0Ker} and \eqref{EqKStDivSymmGrad}; by mode stability, Theorem~\ref{ThmMS} and Remark~\ref{RmkMSWhichLin}, we have $h_0'=\dot g_{b_0}^0(\dot\bhm,0)+\delta_{g_{b_0}}^*\omega_1'$ for some $\dot\bhm\in\R$ and $\omega_1'\in\Hbext^{\infty,-3/2-}$. Equation~\eqref{EqL0LinEqn2} thus reads
  \begin{equation}
  \label{EqL0LinEqn3}
    \wh{\Box_{g_{b_0}}}(0)\omega_1' + [\Box_{g_{b_0}},t_*]\omega_{b_0,s 0} + 2\dot\bhm \delta_{g_{b_0}}\sfG_{g_{b_0}}\dot g^0_{b_0}(1,0) = 0.
  \end{equation}
  Integrating this against $\omega_{b_0,s 0}^*$, the first term gives $0$, the second term gives $2$ by~\eqref{Eq10NoLinPair}, and the last term gives $8\dot\bhm$ by~\eqref{EqL0SchwGauge}; thus $2+8\dot\bhm=0$. Let us hence fix $\dot\bhm=-\tfrac14$. Since the second and third term in~\eqref{EqL0LinEqn3} lie in $\Hbext^{\infty,1/2-}$, we can then solve for $\omega_1'\in\Hbext^{\infty,-3/2-}$; moreover, in this space, $\omega_1'$ is unique modulo multiples of $\la\omega_{b,s 0}\ra$. Thus, $h$ is necessarily of the form
  \[
    h = \hat h_{b_0,s 0} := \dot g^0_{(\bhm_0,0)}(-\tfrac14,0) + \delta_{g_{b_0}}^*(t_*\omega_{b_0,s 0} + \omega_1').
  \]
  This proves that $\wh\cK_{b_0,s 0}$ is 2-dimensional. We can easily make this explicit by replacing $t_*$ by $t_0$ and $\omega_1'$ by $\omega_1^0:=\omega_1'+(t_*-t_0)\omega_{b_0,s 0}$, the latter solving
  \[
    \wh{\Box_{g_{b_0}}}(0)\omega_1^0 = -[\Box_{g_{b_0}},t_0]\omega_{b_0,s 0} + \delta_{g_{b_0}}\sfG_{g_{b_0}}\dot g^0_{(\bhm_0,0)}(\half,0) = -r^{-2}(1-\tfrac{2\bhm_0}{r})d t_0,
  \]
  which has the solution $\omega_1^0=\breve\omega_{b_0,s 0}^0$, see~\eqref{EqL0Lins0}.

  We next construct a generalized zero mode $\hat h_{b,s 0}$ for $b$ near $b_0$. We change the point of view and make $\omega_{b,s 0}$ into the main term of an ansatz, while the linearized Kerr family will give a correction term similarly to above. Thus, starting with
  \begin{equation}
  \label{EqL0Lins0Ansatz}
    \hat h_{b,s 0} = \delta_{g_b}^*(t_*\omega_{b,s 0} + \breve\omega_{b,s 0}) + \dot g_b(\dot\bhm(b),0),
  \end{equation}
  we shall determine $\dot\bhm(b)$, with $\dot\bhm(b_0)=-\tfrac14$, such that the equation $L_{g_b}\hat h_{b,s 0}=0$ can be solved for $\breve\omega_{b,s 0}\in\Hbext^{\infty,-3/2-}$. But this can be done provided we arrange the orthogonality
  \[
    \big\la [\Box_{g_b},t_*]\omega_{b,s 0} + 2\delta_{g_b}\sfG_{g_b}\dot g_b(\dot\bhm(b),0), \omega_{b,s 0}^* \big\ra = 0.
  \]
  But as in the proof of Proposition~\ref{PropL0}, this holds for a unique $\dot\bhm(b)$ because of the non-degeneracy~\eqref{EqL0SchwGauge}, which persists for $b$ near $b_0$.

  \pfstep{Generalized scalar $l=1$ states.} Now, $j=1$ in~\eqref{EqL0LinRewrite}, where $\omega_{b_0,s 1}=\omega_{b_0,s 1}(\scal)$ for some $\scal\in\scal_1$; and $j=1$ in \eqref{EqL0LinEqn2}, so $h_0'\in\Hbext^{\infty,-3/2-}$ since $\omega_{b_0,s 1}(\scal)$ is only of size $\cO(1)$. By Theorem~\ref{ThmMS}, equation~\eqref{EqL0LinEqn1} now implies $h_0'=\delta_{g_{b_0}}^*\omega_1'$ for $\omega_1'\in\Hbext^{\infty,-5/2-}$, and therefore
  \begin{equation}
  \label{EqL0Lins1Ansatz}
    h = \delta_{g_{b_0}}^*(t_*\omega_{b_0,s 1}(\scal)+\omega_1'),
  \end{equation}
  with $\omega_1'$ satisfying $\wh{\Box_{g_{b_0}}}(0)\omega_1'=-[\Box_{g_{b_0}},t_*]\omega_{b_0,s 1}(\scal)\in\Hbext^{\infty,-1/2-}$. By Lemma~\ref{Lemma10Gens12}, $\omega_1'$ is unique modulo $\{\omega_{b_0,s 1}^{(1)}(\scal)\colon\scal\in\scal_1\}$ in the scalar type $l=1$ sector. Expanding
  \[
    h=t_*h_{b_0,s 1} + ([\delta_{g_{b_0}}^*,t_*]\omega_{b_0,s 1}(\scal)+\delta_{g_{b_0}}^*\omega_1'),
  \]
  recall that the membership $h\in\wh\cK_{b_0,s 1}$ requires the term in parentheses to lie in $\Hbext^{\infty,-1/2-}$, thus to be of size $o(1)$. In view of~\eqref{Eq10Gens12SymGrad}, this implies that $\omega_1'$ is \emph{unique}. Thus, there exists at most one generalized scalar $l=1$ mode of $L_{g_{b_0}}$ with leading term $t_* h_{b_0,s 1}(\scal)$.

  On the other hand, \emph{existence} of $h$ of this form, and its extension to a continuous family of generalized modes for $L_{g_b}$, $b$ near $b_0$, follows immediately from Proposition~\ref{Prop10GenSc} by setting
  \begin{equation}
  \label{EqL0Lins1Def}
    \hat h_{b,s 1}(\scal)=\delta_{g_b}^*\hat\omega_{b,s 1}(\scal).
  \end{equation}

  \pfstep{Generalized scalar $l=0$ and $l=1$ dual states.} Certainly, the expressions in~\eqref{EqL0LinDual} produce elements $\ker L_{g_b}^*$ of the desired form.
\end{proof}

By Proposition~\ref{PropOpNull}, all zero energy modes are polyhomogeneous at $r=\infty$; this is also true for the coefficients of $t_*^0$ and $t_*^1$ of the generalized zero modes constructed above. For later use, we determine their leading order behavior more precisely. Note first that the construction of $\hat h_{b,s 0}$ in~\eqref{EqL0Lins0Ansatz} shows that its $t_*$-coefficient is $h_{b,s 0}$. Consider similarly the definition~\eqref{EqL0Lins1Def} of $\hat h_{b,s 1}(\scal)$ in terms of $\hat\omega_{b,s 1}$, which is constructed in~\eqref{EqL0Ansatz}; upon re-defining $\omega_{b,s 1}(\scal)$ as $\omega_{b,s 1}(\scal)+c_b\omega_{b,s 0}$ (which is still a continuous family in $b$, linear in $\scal$, and agrees with $\omega_{b_0,s 1}(\scal)$ for $b=b_0$ since $c_{b_0}=0$ by~\eqref{EqL0cb0}), and then letting $h_{b,s 1}(\scal)=\delta_{g_b}^*\omega_{b,s 1}(\scal)$, we ensure that the $t_*$-coefficient of $\hat h_{b,s 1}(\scal)$ is equal to $h_{b,s 1}(\scal)$. Analogous arguments apply to $\hat h_{b,s 0}^*$ and $\hat h_{b,s 1}^*$. Choosing $h_{b,s 0}$ etc.\ in this manner, we now make the following definition:

\begin{definition}
\label{DefL0Breve}
  For $t_*=t_{\bhm,*}$, $\scal\in\scal_1$, $\vect\in\vect_1$, we set
  \begin{alignat*}{3}
    \breve h_{b,s 0}&:=\hat h_{b,s 0}-t_* h_{b,s 0}, &
    \breve h^*_{b,s 0}&:=\hat h^*_{b,s 0}-t_* h^*_{b,s 0}, \\
    \breve h_{b,s 1}(\scal)&:=\hat h_{b,s 1}(\scal)-t_* h_{b,s 1}(\scal), &\qquad
    \breve h^*_{b,s 1}(\scal)&:=\hat h^*_{b,s 1}(\scal)-t_* h^*_{b,s 1}(\scal).
  \end{alignat*}
  These are stationary ($t_*$-independent), and lie in $\Hbext^{\infty,-1/2-}$ and $\Hbsupp^{-\infty,-1/2-}$, respectively.
\end{definition}

\begin{lemma}
\label{LemmaL0Lead}
  For $\scal\in\scal_1$ and $\vect\in\vect_1$, we have
  \begin{equation}
  \label{EqL0Lead}
  \begin{split}
    h_{b,s 0},\ h_{b,s 1}(\scal),\ h_{b,v 1}(\vect) & \in \Hbext^{\infty,1/2-}, \\
    h_{b,s 1}^*(\scal) &\in\Hbsupp^{-\infty,1/2-}, \\
    h_{b,v 1}^*(\vect) &\in\rho\CI+\Hbsupp^{-\infty,1/2-},
  \end{split}
  \end{equation}
  where the $\rho\CI$ term has support in $r\geq 4\bhm_0$; moreover, $h_{b,s 0}^*$ has compact support. In the notation of Definition~\ref{DefL0Breve}, we have
  \begin{equation}
  \label{EqL0LeadBreve}
    \breve h_{b,s 0},\ \breve h_{b,s 1}(\scal) \in \rho\CI+\Hbext^{\infty,1/2-}, \quad
    \breve h_{b,s 0}^*,\ \breve h_{b,s 1}^*(\scal) \in \rho\CI+\Hbsupp^{-\infty,1/2-};
  \end{equation}
  the remainder terms of $\breve h_{b,s 0}^*$ and $\breve h_{b,s 1}^*(\scal)$ have support in $r\geq r_b$, are smooth in $r>r_b$, conormal at $\pa_+X$ with the stated weight, and lie in $H^{-3/2-}$ near the event horizon $r=r_b$.\footnote{Put differently, we have $\breve h_{b,s 0}^*\in\rho\CI+\chi\Hb^{\infty,1/2-}+(1-\chi)\Hsupp^{-3/2-}$, where $\chi\equiv 1$ for $r>4\bhm_0$ and $\chi\equiv 0$ for $r<3\bhm_0$; likewise for $\breve h_{b,s 1}^*(\scal)$.}
\end{lemma}
\begin{proof}
  The statement for $h_{b,v 1}$ was proved in Proposition~\ref{Prop10Genv1}; the statement for $h_{b,v 1}^*(\vect)$ follows directly from Proposition~\ref{PropOpNull}, using the fact that zero frequency solutions of the wave operator on symmetric 2-tensors on Minkowski space, which in the standard coordinate splitting is a $10\times 10$ matrix of scalar wave operators, have $r^{-1}$ asymptotics provided they belong to $\Hb^{\infty,-1/2-}$ near infinity.
  
  The statements about $h_{b,s 0}^*$ and $h_{b,s 0}$ follow immediately from~\eqref{Eq1s0} and \eqref{EqL0s0}. For $h_{b,s 0}$, we record a more precise statement: since $\omega_{b,s 0}\in\rho\CI+\Hbext^{\infty,1/2-}$ (either by inspection, see Remark~\ref{Rmk10KerrExpl}, or by similar normal operator arguments), we have $h_{b,s 0}\in\rho^2\CI+\Hbext^{\infty,3/2-}$. Since elements of the kernel of the spectral family of the scalar wave operator on Minkowski space at zero energy with decay $r^{-2}$ have a leading order term in $r^{-2}\scal_1$, we in fact deduce from $h_{b,s 0}\in\ker\wh{L_{g_b}}(0)$ that
  \[
    h_{b,s 0} \in r^{-2}\Omega_1 + \Hbext^{\infty,3/2-},
  \]
  where
  \begin{equation}
  \label{EqL0LeadOmega1}
    \Omega_1 = \mathspan\{ \scal\,d t^2,\, \scal\,d t\otimes_s d x^i,\, \scal\,d x^i\otimes_s d x^j \colon \scal\in\scal_1,\ 1\leq i,j\leq 3 \}.
  \end{equation}
  
  We proved the statement~\eqref{EqL0Lead} for $h_{b,s 1}(\scal)$ in equation~\eqref{Eq10SymGradDecay}; the argument given there also applies to $h_{b,s 1}^*(\scal)$. We again record a more precise statement: dropping the argument $\scal\in\scal_1$ from the notation, write $b=(\bhm,\bha)$ and denote by $b_1=(\bhm,0)$ the Schwarzschild parameters with the same mass. Then $\omega_{b,s 1}=\omega_{b_1,s 1}+\omega'$, where $\omega'\in\Hbext^{\infty,-1/2-}$ solves
  \begin{equation}
  \label{EqL0LeadBoxOmegap}
    \wh{\Box_{g_b}}(0)\omega' = -\wh{\Box_{g_b}}(0)\omega_{b_1,s 1} = \bigl(\wh{\Box_{g_{b_1}}}(0)-\wh{\Box_{g_b}}(0)\bigr)\omega_{b_1,s 1} \in \Hbext^{\infty,5/2-};
  \end{equation}
  here, we used Lemma~\ref{LemmaKStNormal}. Normal operator arguments give $\omega'\in\rho\CI+\Hbext^{\infty,1/2-}$, hence
  \begin{align*}
    h_{b,s 1}=\delta_{g_b}^*\omega_{b,s 1} &= \delta_{g_{b_1}}^*\omega_{b_1,s 1} + (\delta_{g_b}^*-\delta_{g_{b_1}}^*)\omega_{b_1,s 1} + \delta_{g_b}^*\omega' \\
      &\in \delta_{g_{b_1}}^*\omega_{b_1,s 1} + \rho^2\CI + (\rho^2\CI+\Hbext^{\infty,3/2-}) \\
      &\subset \rho^2\CI + \Hbext^{\infty,3/2-}.
  \end{align*}
  A normal operator argument implies that the leading order term has to lie in $\Omega_1$, so
  \[
    h_{b,s 1}(\scal) \in r^{-2}\Omega_1 + \Hbext^{\infty,3/2-}.
  \]
  We argue similarly for $h^*_{b,s 1}$: by~\eqref{Eq00Grows1}, \eqref{Eq10Grows1}, it lies in $\rho^2\CI+\Hb^{\infty,3/2-}$ near infinity, hence
  \begin{equation}
  \label{EqL0Leads1Dual}
    h^*_{b,s 1}\in \chi r^{-2}\Omega_1 + \Hbsupp^{-\infty,3/2-},
  \end{equation}
  where $\chi$ is a cutoff, identically $0$ for $r\leq 3\bhm_0$ and identically $1$ for $r\geq 4\bhm_0$.

  To prove the statements~\eqref{EqL0LeadBreve}, recall that $\breve h_{b,s 0}$ solves
  \begin{equation}
  \label{EqL0LeadEq}
    \wh{L_{g_b}}(0)\breve h_{b,s 0} = -[L_{g_b},t_*]h_{b,s 0} \in \rho^3\CI+\Hbext^{\infty,5/2-}.
  \end{equation}
  This can be solved by first solving away the leading order term via inversion of the normal operator of $\wh{L_{g_b}}(0)$ at infinity; since the latter has $-i$ in its boundary spectrum, this may a priori produce logarithmic terms $r^{-1}\log r$ in addition to $r^{-1}$ terms. It then remains to find a correction term that solves away, globally on $X$, an error term lying in $\Hbext^{\infty,5/2-}$; this can certainly be done in the space $\Hbext^{\infty,-1/2-}$ (since we already know that a solution, $\breve h_{b,s 0}$, to the full equation~\eqref{EqL0LeadEq} exists), and by the usual normal operator argument this correction automatically lies in $\rho\CI+\Hbext^{\infty,-1/2-}$.
  
  Thus, it suffices to show that the \emph{leading term} of $[L_{g_b},t_*]h_{b,s 0}$ can be solved away without a logarithmic term of size $r^{-1}\log r$. But this only requires a normal operator calculation; in particular, in view of Lemma~\ref{LemmaOpLinFT}, we can replace $[L_{g_b},t_*]$ by $2 i\rho(\rho D_\rho+i)$ where $\rho=r^{-1}$, which on $\rho^2\CI/\Hbext^{\infty,3/2-}$ is simply multiplication by $2$; so the task is to solve
  \begin{equation}
  \label{EqL0LeadBreveExpl}
    \Box_{\ubar g,2}\breve h \in -2 r^{-1}\cdot r^{-2}\Omega_1 = r^{-3}\Omega_1.
  \end{equation}
  The space of (generalized) resonant states of $\wh{\Box_{\ubar g,2}}(0)^*$ at $-i$, i.e.\ the space of tensors which are (quasi)homogeneous of degree $-1$ and annihilated by $\wh{\Box_{\ubar g,2}}(0)^*$, is spanned by $r^{-1}$ times
  \begin{equation}
  \label{EqL0LeadTensor}
    d t^2,\ d t\otimes_s d x^i,\ d x^i\otimes_s d x^j\ \ (1\leq i,j\leq 3).
  \end{equation}
  Thus, $\breve h$ has no logarithmic terms provided each of these are orthogonal to $\Omega_1$ when integrated over the sphere $\Sph^2$ at infinity; and this is indeed the case, due to the fact that $\scal\in\scal_1$ integrates to $0$ over $\{\rho=0\}\cong\Sph^2$. The same argument proves the result for $\breve h_{b,s 1}$.

  The statement for $\breve h_{b,s 0}^*$, which solves $\wh{L_{g_b}}(0)^*\breve h_{b,s 0}^*=-[L_{g_b},t_*]h_{b,s 0}^*$ \emph{with compactly supported right hand side}, is clear. (Note also that the right hand side is one derivative less regular than $h_{b,s 0}^*$ at the event horizon; the solution operator $(\wh{L_{g_b}}(0)^*)^{-1}$ gains one derivative there, hence $\breve h_{b,s 0}^*$ has (at least) the same regularity as $h_{b,s 0}^*$ itself.) The claim for $\breve h_{b,s 1}^*$ is proved like that for $\breve h_{b,s 1}$ in view of~\eqref{EqL0Leads1Dual}.
\end{proof}

%%%%%%%%%%%%%%%%%%%%%%%%%%%%%%%%%%%%%%%%%%%%%%%%%%
\subsection{Generalized stationary modes: quadratic growth}
\label{SsL0Qu}

We next study whether the linearization $L_{g_{b_0}}$ at the Schwarzschild metric admits \emph{quadratically} growing solutions of scalar type $l=0$ or $l=1$. (For all other types, this possibility has been excluded already.)

\begin{lemma}
\label{LemmaL0Qus1}
  Let $d\geq 2$. There does not exist $h=\sum_{j=0}^d t_*^j h_j$ with $h_j\in\Hbext^{\infty,-1/2-}$ of scalar type $l=1$ and $h_d\neq 0$ for which $L_{g_{b_0}}h=0$.
\end{lemma}
\begin{proof}
  By the argument at the beginning of the proof of Lemma~\ref{LemmaL0Linv1}, it suffices to prove this for $d=2$. Then $h_2=h_{b_0,s 1}(\scal)\in\ker L_{g_{b_0}}$ for some $0\neq\scal\in\scal_1$, and
  \begin{equation}
  \label{EqL0Qus1}
    0 = L_{g_{b_0}}h = t_*\bigl(2[L_{g_{b_0}},t_*]h_2 + L_{g_{b_0}}h_1\bigr) + \bigl([[L_{g_{b_0}},t_*],t_*]h_2 + [L_{g_{b_0}},t_*]h_1 + L_{g_{b_0}}h_0\bigr).
  \end{equation}
  The vanishing of the linear (in $t_*$) term is equivalent to $L_{g_{b_0}}(2 t_* h_2+h_1)=0$, thus $2 t_* h_2+h_1=2\hat h_{b_0,s 1}(\scal)+c h_{b_0,s 1}(\scal')$, $c\in\R$, $\scal'\in\scal_1$. By subtracting from $h$ the 2-tensor $c\hat h_{b_0,s 1}(\scal')\in\ker L_{g_{b_0}}$, we can set $c=0$, thus $h_1=2(\hat h_{b_0,s 1}-t_* h_{b_0,s 1})$. It thus remains to determine whether $h_0$ can be chosen to make the constant term in~\eqref{EqL0Qus1} vanish. This is equivalent to vanishing of the pairing of this term with $h_{b_0,s 1}^*$; but, dropping $\scal$ from the notation, a lengthy calculation shows that, for $\scal=\cos\theta$ (which can be arranged by choosing suitable polar coordinates, and rescaling by a non-zero complex number),
  \begin{equation}
  \label{EqL0Qus1Pair}
    \big\la [[L_{g_{b_0}},t_*],t_*]h_{b_0,s 1} + 2[L_{g_{b_0}},t_*](\hat h_{b_0,s 1}-t_* h_{b_0,s 1}), h_{b_0,s 1}^* \big\ra = -4\bhm_0 \neq 0.
  \end{equation}
  (One can show that this calculation is unaffected when one replaces $t_*$ by $t_0$; but then $t_0$ being null implies that $[[L_{g_{b_0}},t_0],t_0]\equiv 0$, which simplifies the calculation.)
\end{proof}

By continuity, the non-degeneracy~\eqref{EqL0Qus1Pair} remains valid for $b_0$ replaced by nearby $b$.

For scalar type $l=0$ modes on the other hand, one can verify that the pairing
\begin{equation}
\label{EqL0Qus0Pair}
    \big\la [[L_{g_{b_0}},t_*],t_*]h_{b_0,s 0} + 2[L_{g_{b_0}},t_*](\hat h_{b_0,s 0}-t_* h_{b_0,s 0}), h_{b_0,s 0}^* \big\ra
\end{equation}
\emph{vanishes}, which by the arguments following~\eqref{EqL0Qus1} implies the existence of a quadratically growing generalized mode solution. We do not use this degeneracy in the sequel, but do point out that \emph{this is the reason for implementing constraint damping}, as we discuss momentarily. The key observation is that these quadratically growing solutions are pathological in that they \emph{cannot} satisfy the linearized gauge condition:

\begin{lemma}
\label{LemmaL0Qus0}
  Suppose $h=t_*^2 h_2+t_* h_1+h_0\in\Poly^2(t_*)\Hbext^{\infty,-1/2-}$ is of scalar type $l=0$ and solves $D_{g_{b_0}}\Ric(h)=0$ and $\delta_{g_{b_0}}\sfG_{g_{b_0}}h=0$, then necessarily $h_2=0$. In particular, if $L_{g_{b_0}}h=0$ but $h_2\neq 0$, then $\delta_{g_{b_0}}\sfG_{g_{b_0}}h\neq 0$.
\end{lemma}
\begin{proof}
  The assumptions on $h$ imply that $L_{g_{b_0}}h=0$. Suppose we can find a generalized mode $h$ with $h_2\neq 0$. As in the proof of the previous lemma, after multiplying by a non-zero scalar and subtracting from $h$ a multiple of $\hat h_{b_0,s 0}$, we must have $h_2=h_{b_0,s 0}$ and $h_1=2(\hat h_{b_0,s 0}-t_* h_{b_0,s 0})$. Thus, by Proposition~\ref{PropL0Lin}, and writing $\breve\omega_{b_0,s 0}=\breve\omega_{b_0,s 0}^0+(t_0-t_*)\omega_{b_0,s 0}$,
  \begin{align}
  \label{EqL0Qus0h}
    h &= t_*^2 h_{b_0,s 0} + 2 t_*(\hat h_{b_0,s 0}-t_* h_{b_0,s 0}) + h_0 \\
      &= t_*^2 \delta_{g_{b_0}}^*\omega_{b_0,s 0} + 2 t_*\bigl(\delta_{g_{b_0}}^*(t_*\omega_{b_0,s 0}+\breve\omega_{b_0,s 0})-t_*\delta_{g_{b_0}}^*\omega_{b_0,s 0} + \dot g^0_{b_0}(-\tfrac14,0)\bigr) + h_0 \nonumber\\
      &= \delta_{g_{b_0}}^*(t_*^2\omega_{b_0,s 0} + 2 t_*\breve\omega_{b_0,s 0}) + 2 t_*\dot g^0_{b_0}(-\tfrac14,0) + h_0', \nonumber
  \end{align}
  where $h_0'=-2[\delta_{g_{b_0}}^*,t_*]\breve\omega_{b_0,s 0}+h_0$. Since $D_{g_{b_0}}\Ric(h)=0$, we have $D_{g_{b_0}}\Ric(t_*\dot g^0_{b_0}(1,0)-2 h'_0)=0$; but then Theorem~\ref{ThmMS}\eqref{ItMSSc0} implies that there exist scalar type $l=0$ 1-forms $\omega_0,\omega_1,\omega_2\in\Hbext^{\infty,\ell'}$ for some $\ell'\in\R$ such that
  \begin{equation}
  \label{EqL0Qus0PureGauge}
    t_*\dot g^0_{b_0}(1,0)-2 h'_0 = \delta_{g_{b_0}}^*(t_*^2\omega_2+t_*\omega_1+\omega_0).
  \end{equation}
  Expanding the right hand side into powers of $t_*$, the coefficient of $t_*^2$ vanishes, hence $\delta_{g_{b_0}}^*\omega_2=0$, thus $\omega_2=c\pa_t^\flat$, $c\in\C$. The linear term of~\eqref{EqL0Qus0PureGauge} then reads
  \begin{equation}
  \label{EqL0Qus0Schw}
    \dot g^0_{b_0}(1,0) = 2 [\delta_{g_{b_0}}^*,t_*]\omega_2 + \delta_{g_{b_0}}^*\omega_1 = \delta_{g_{b_0}}^*(2 c t_*\pa_t^\flat + \omega_1);
  \end{equation}
  that is, the linearized Schwarzschild metric is pure gauge, which is not the case.\footnote{This is easy to check explicitly. Suppose equation~\eqref{EqL0Qus0Schw} held; more generally, let $\omega=p(t_0,r)d t_0+q(t_0,r)d r$ and suppose $S:=\delta_{g_{b_0}}^*\omega-\dot g_{b_0}^0(1,0)=0$. Since $S_{r r}=\pa_r q$, we have $q=q(t_0)$. Then, $S_{\theta\theta}=r p+(r-2\bhm_0)q=0$ implies $p=-(1-\tfrac{2\bhm_0}{r})q$, and then $S_{t_0 r}=\half q'=0$ implies that $q(t_0)$ is a constant; therefore, $\omega=q(-(1-\tfrac{2\bhm_0}{r})d t_0+d r)=-q\pa_{t_0}^\flat$; but then we have $S=-\dot g_{b_0}^0(1,0)\neq 0$.}
\end{proof}

According to this lemma, in order to exclude quadratically (and faster polynomial) growing generalized mode solutions in the scalar $l=0$ sector on Schwarzschild spacetimes, it suffices to ensure that the linearized gauge condition necessarily holds for $h\in\ker L_{g_{b_0}}\cap\Poly^2(t_*)\Hbext^{\infty,-1/2-}$. We thus proceed to explain why this may (and indeed does) fail for the \emph{unmodified} linearized gauge-fixed Einstein operator: consider again $h$ as in~\eqref{EqL0Qus0h}; then by the linearized second Bianchi identity, and writing $\breve h_{b_0,s 0}=\hat h_{b_0,s 0}-t_* h_{b_0,s 0}\in\Hbext^{\infty,-1/2-}$,
\begin{align*}
  \ker \Box_{g_{b_0}} \ni \delta_{g_{b_0}}\sfG_{g_{b_0}}h
   &= 2 t_*\bigl([\delta_{g_{b_0}}\sfG_{g_{b_0}},t_*]h_{b_0,s 0} + \delta_{g_{b_0}}\sfG_{g_{b_0}}\breve h_{b_0,s 0}\bigr) \\
   &\qquad\qquad + \bigl(2[\delta_{g_{b_0}}\sfG_{g_{b_0}},t_*]\breve h_{b_0,s 0} + \delta_{g_{b_0}}\sfG_{g_{b_0}}h_0\bigr).
\end{align*}
The coefficient of the linear (in $t_*$) term thus lies in $\ker\wh{\Box_{g_{b_0}}}(0)\cap\Hbext^{\infty,1/2-}$, hence vanishes by Theorem~\ref{Thm1}. Thus, the (stationary) second line lies in
\begin{equation}
\label{EqL0Qus0Fail}
  \ker \bigl(\delta_{g_{b_0}}\sfG_{g_{b_0}}\circ\delta_{g_{b_0}}^*\bigr)\ftrans(0) \cap \Hbext^{\infty,-1/2-}.
\end{equation}
\emph{But the latter space is non-trivial}, allowing for $\delta_{g_{b_0}}\sfG_{g_{b_0}}h=c\omega_{b_0,s 0}$ for $c\neq 0$. (This calculation also implies that necessarily $D_{g_{b_0}}\Ric(h)\neq 0$: indeed, we otherwise would also have $0=\delta_{g_{b_0}}^*\delta_{g_{b_0}}\sfG_{g_{b_0}}h=c h_{b_0,s 0}$, forcing $c=0$, and thus $h$ would only be linearly growing by Lemma~\ref{LemmaL0Qus0}---a contradiction.)

\emph{This is therefore the place where constraint damping (CD) becomes crucial.} Namely, replacing $\delta_{g_{b_0}}^*$ in the definition of the linearized gauge-fixed Einstein operator, and thus in~\eqref{EqL0Qus0Fail}, by a lower order (and spherically symmetric) modification of the form described in Definition~\ref{DefOpCD}---let us simply denote the resulting operator by $\wt\delta^*$ here---one can ensure that the zero energy nullspace of $\delta_{g_{b_0}}\sfG_{g_{b_0}}\circ\wt\delta^*$ on $\Hbext^{\infty,-1/2-}$ is trivial. For putative quadratically growing scalar $l=0$ zero modes $h$, with non-vanishing $t_*^2$ term, of the corresponding linearized \emph{modified} gauge-fixed Einstein operator, we can then conclude $\delta_{g_{b_0}}\sfG_{g_{b_0}}h=0$, which is a contradiction by Lemma~\ref{LemmaL0Qus0}.

%%%%%%%%%%%%%%%%%%%%%%%%%%%%%%%%%%%%%%%%%%%%%%%%%%%%%%%%%%%%%%%%%%%%%%
\section{Constraint damping (CD)}
\label{SCD}

We proceed to describe constraint damping modifications, as motivated in~\S\ref{SsL0Qu}. We will show in~\S\ref{SsCD0} that replacing $\delta_g^*$ by $\wt\delta_{g,E}^*$, see Definition~\ref{DefOpCD}, with $E=E(g;\cd,\gamma_1,\gamma_2)$ (see~\eqref{EqOpCD}) being a modification, with a suitably chosen (and in fact compactly supported) $\cd$, and with $\gamma_1,\gamma_2$ small, eliminates the zero energy nullspace of the constraint propagation operator $\delta_g\sfG_g\wt\delta_{g,E}^*$ on $\Hbext^{\infty,-1/2-}$ for the Schwarzschild metric $g=g_{b_0}$, and thus for slowly rotating Kerr metrics $g=g_b$. This information suffices to get a complete description of the generalized zero energy nullspace of $L_{g_b}$, see Theorem~\ref{ThmCD0Modes}. In~\S\ref{SsCDU} we show that one can ensure that this modification not only eliminates the zero energy nullspace, but also preserves the absence of non-zero resonances $\sigma\in\C$, $\Im\sigma\geq 0$, both for the constraint propagation operator and for the linearized modified gauge-fixed Einstein operator.

\begin{rmk}
\label{RmkCDSmall}
  For comparison with the large (i.e.\ taking $\gamma_1,\gamma_2\gg 1$) CD used in \cite{HintzVasyKdSStability}, note that here, the only problematic behavior which we aim to eliminate by means of CD concerns quadratically (or more) growing generalized \emph{zero} energy modes, whereas on Kerr--de~Sitter we needed to eliminate non-pure-gauge modes in the \emph{open upper half plane} (see e.g.\ \cite[Appendix~C.2]{HintzVasyKdSStability} for the explicit calculations on static de~Sitter spacetimes), which clearly cannot be done by \emph{perturbative} methods.
\end{rmk}

\begin{rmk}
\label{RmkCDSmall2}
  On the other hand, in \cite{HintzVasyMink4}, we used small CD which however is asymptotically (at $\scri^+$) non-trivial (roughly, in the reference we took $\cd=r^{-1}d t$ near $\scri^+$). Such CD modifications affect (albeit only mildly so for small $\gamma_1,\gamma_2$) the asymptotic behavior of (mode) solutions of the modified linearized gauge-fixed operator. Recall however that CD at $\scri^+$ in~\cite{HintzVasyMink4} was implemented only to ensure better decay properties of certain metric components at $\scri^+$ in a \emph{nonlinear} iteration scheme; hence, for the present \emph{linear} stability problem, there is no need for such asymptotically non-trivial CD. With an eye towards a possible proof of the nonlinear stability of the Kerr family, we do remark however that this type of CD can be implemented in this paper as well (as an additional small perturbation on top of an already working compactly supported CD); the changes in the behavior of the resolvent are minor, as discussed in a general setting in~\cite{VasyLowEnergyLag}. The details will be discussed elsewhere.
\end{rmk}

Let $\chi\in\CIc((r_-,3\bhm_0))$ be a localizer, identically $1$ near $2\bhm_0$; let further
\begin{equation}
\label{EqCD1form}
  \cd' := d t_0 - \fv\,d r,\quad
  \cd := \chi(r)\cd',
\end{equation}
with $\fv\in\R$ to be chosen later. (For $\fv>0$, this is a future timelike 1-form on the Schwarzschild spacetime $(M^\circ,g_{b_0})$.) We then let
\[
  E = E(g;\cd,\gamma,\gamma),\quad
  \wt\delta_{g,\gamma}^* := \wt\delta_{g,E}^*.
\]
We study the linearized modified gauge-fixed Einstein operator
\[
  L_{g,\gamma}:=L_{g,E},
\]
see~\eqref{EqOpCDEinLin}, in detail~\S\ref{SR}; here, we focus on the operator arising via the second Bianchi identity, $\delta_g\sfG_g L_{g,\gamma}:=2\delta_g\sfG_g\wt\delta_{g,\gamma}^*\circ\delta_g\sfG_g$, and draw a few simple conclusions for $L_{g,\gamma}$. Define thus the (modified) \emph{gauge propagation operator}
\begin{equation}
\label{EqCD}
  \cP_{g,\gamma} := 2\delta_g\sfG_g\wt\delta_{g,\gamma}^*,\quad
  \cP_{b,\gamma} := \cP_{g_b,\gamma}.
\end{equation}
So far, we worked with $L_0=L_{g_b}$, thus $\cP_{b,0}=\Box_{g_b,1}$ is the tensor wave operator.

%%%%%%%%%%%%%%%%%%%%%%%%%%%%%%%%%%%%%%%%%%%%%%%%%%
\subsection{Zero frequency improvements}
\label{SsCD0}

We show that the kernel $\ker\wh{\cP_{b,\gamma}}(0)\cap\Hbext^{\infty,-1/2-}=\la\omega_{b,s 0}\ra$ for $\gamma=0$ becomes trivial for small $\gamma\neq 0$ upon choosing $\fv$ in~\eqref{EqCD1form} suitably:

\begin{prop}
\label{PropCD0}
  Let $\fv\neq 1$. There exists $\gamma_0>0$ such that for fixed $\gamma$ with $0<|\gamma|<\gamma_0$, the following holds: for $b$ sufficiently close to $b_0$,
  \begin{equation}
  \label{EqCD0}
    \ker\wh{\cP_{b,\gamma}}(0)\cap\Hbext^{\infty,-1/2-} = 0,\qquad
    \ker\wh{\cP_{b,\gamma}}(0)^*\cap\Hbsupp^{-\infty,-1/2-} = 0.
  \end{equation}
\end{prop}
\begin{proof}
  Fix $\ell\in(-\tfrac32,-\half)$, $s>\tfrac32$, and put
  \[
    \cX_b^{s,\ell} = \bigl\{ u\in\Hbext^{s,\ell} \colon \wh{\cP_{b,\gamma}}(0)u\in \Hbext^{s-1,\ell+2} \bigr\}.
  \]
  Recall from (the proof of) Theorem~\ref{ThmOp} that $\wh{\cP_{b,\gamma}}(0)\colon\cX_b^{s,\ell}\to\Hbext^{s-1,\ell+2}$ is Fredholm of index $0$ when $\gamma$ is small. Since $\wh{\cP_{b,\gamma}}(0)-\wh{\cP_{b,0}}(0)$ is a compactly supported, first order operator, the space $\cX_b^{s,\ell}$ does not depend on $\gamma$.
  
  We first consider the Schwarzschild case $b=b_0$. We split domain and target by writing
  \begin{alignat*}{4}
    \cX_{b_0}^{s,\ell} &= \cK^\perp \oplus \cK,&\quad& \cK = \ker_{\cX_{b_0}^{s,\ell}}\wh{\cP_{b_0,0}}(0) &\ =\ & \la\omega_{b_0,s 0}\ra, \\
    \Hbext^{s-1,\ell+2} &= \cR \oplus \cR^\perp,&\quad& \cR = \ran_{\cX_{b_0}^{s,\ell}}\wh{\cP_{b_0,0}}(0) &\ =\ & \ann\omega_{b_0,s 0}^*,
  \end{alignat*}
  where $\ann$ denotes the annihilator, $\cK^\perp$ is a complementary subspace to $\cK$, and $\cR^\perp=\la\eta\ra$ is complementary to $\cR$ inside of $\Hbext^{s-1,\ell+2}$; here we may choose $\eta\in\CIdot(X;\wt{\Tsc^*}X)$, and we may arrange $\la\eta,\omega_{b_0,s 0}^*\ra=1$. The operator $\wh{\cP_{b_0,\gamma}}(0)$ takes the form
  \begin{equation}
  \label{EqCD0PhatMtx}
    \wh{\cP_{b_0,\gamma}}(0) = \begin{pmatrix} \cP_{0 0}+\gamma\cP_{0 0}^\flat & \gamma\cP_{0 1}^\flat \\ \gamma\cP_{1 0}^\flat & \gamma\cP_{1 1}^\flat \end{pmatrix}.
  \end{equation}
  If we identify $\C\cong\cK$ via $c\mapsto c\omega_{b_0,s 0}$, and further $\C\cong\cR^\perp$ via $c\mapsto c\eta$ (thus $\cR^\perp\to\C$ is given by $\eta'\mapsto\la\eta',\omega_{b_0,s 0}^*\ra$), then $\cP_{1 1}^\flat$ is simply a \emph{number}: indeed, one computes
  \begin{equation}
  \label{EqCD0Nondeg}
  \begin{split}
    \cP_{1 1}^\flat &= \big\la\gamma^{-1}(\wh{\cP_{b_0,\gamma}}(0)-\wh{\cP_{b_0,0}}(0))\omega_{b_0,s 0}, \omega_{b_0,s 0}^*\big\ra \\
      &= \big\la 2\delta_g\sfG_g(2\cd\otimes_s\omega_{b_0,s 0}-G(\cd,\omega_{b_0,s 0})g), \omega_{b_0,s 0}^*\big\ra \\
      &= 4(\fv-1).
  \end{split}
  \end{equation}
  Suppose now $\ker\wh{\cP_{b_0,\gamma}}(0)\ni(\omega_0,\omega_1)\in\cK^\perp\oplus\cK$; then $\omega_0=-\gamma(\cP_{0 0}+\gamma\cP_{0 0}^\flat)^{-1}\cP_{0 1}^\flat \omega_1$, so
  \[
    \bigl(\cP_{1 1}^\flat - \gamma\cP_{1 0}^\flat(\cP_{0 0}+\gamma\cP_{0 0}^\flat)^{-1}\cP_{0 1}^\flat\bigr)\omega_1 = 0.
  \]
  For $\fv\neq 1$ and small $\gamma$, this forces $\omega_1=0$, thus $\omega_0=0$, proving the injectivity, and hence invertibility, of $\wh{\cP_{b_0,\gamma}}(0)$; it also implies that the adjoint has trivial kernel on $\Hbsupp^{-\infty,-3/2+}$.

  Fixing such small non-zero $\gamma$, the invertibility of $\wh{\cP_{b_0,\gamma}}(0)$ implies that of $\wh{\cP_{b,\gamma}}(0)$ by simple perturbation arguments as in~\cite[\S2.7]{VasyMicroKerrdS}.
\end{proof}

This is sufficient to exclude quadratically growing zero modes of the operator $L_{g_{b_0},\gamma}$; in fact, we can now give a full description of the generalized zero energy nullspace of $L_{g_b,\gamma}$ for $b$ near $b_0$:

\begin{thm}
\label{ThmCD0Modes}
  Let $s>\tfrac52$, $\ell\in(-\tfrac32,-\half)$, and fix $\gamma$ as in Proposition~\ref{PropCD0}. Then there exists $C_0>0$ such that for Kerr parameters $b\in\R^4$, $|b-b_0|<C_0$, the operator $L_{g_b,\gamma}$ has the following properties:
  \begin{enumerate}
  \item\label{ItCD0Modesker} the kernel of $\wh{L_{g_b,\gamma}}(0)$ is $7$-dimensional,
    \begin{equation}
    \label{EqCD0Modes0}
      \cK_b := \ker_{\Hbext^{s,\ell}} \wh{L_{g_b,\gamma}}(0) = \C h_{b,s 0} \oplus h_{b,s 1}(\scal_1) \oplus h_{b,v 1}(\vect_1),
    \end{equation}
    where we use the notation of Proposition~\ref{PropL0};
  \item\label{ItCD0Modes0gen} the generalized zero energy nullspace
    \begin{equation}
    \label{EqCD0Modes0gen}
      \wh\cK_b :=\{ h\in\Poly(t_*)\Hbext^{s,\ell} \colon L_{g_b,\gamma}h=0 \}
    \end{equation}
    is $11$-dimensional. The quotient $\hat\cK_b/\cK_b$ is spanned by (the image, in the quotient space, of) $\C\hat h_{b,s 0} \oplus \hat h_{b,s 1}(\scal_1)$, where we use the notation of Proposition~\ref{PropL0Lin}.
  \end{enumerate}
\end{thm}

Part~\eqref{ItCD0Modesker} is proved like Proposition~\ref{PropL0}; the arguments there are in fact slightly simplified since the 1-form operator $\wh{\cP_{b,\gamma}}(0)$, which controls whether a gauge potential has symmetric gradient satisfying the linearized gauge condition, is \emph{injective}. The proof of part~\eqref{ItCD0Modes0gen} is more subtle. In view of the role played by dual pairings such as~\eqref{EqL0Linv1Nondeg}, we first show that the zero energy dual states can be chosen to be continuous in the parameters $b,\gamma$:

\begin{lemma}
\label{LemmaCD0Dual}
  In the notation of Theorem~\ref{ThmCD0Modes}, the $7$-dimensional\footnote{This is a consequence of Theorem~\ref{ThmCD0Modes}\eqref{ItCD0Modesker} and the fact that $\wh{L_{g_b,\gamma}}(0)$ has index $0$.} space
  \[
    \cK_b^{\gamma*} := \ker \wh{L_{g_b,\gamma}}(0)^* \cap \Hbsupp^{-\infty,-1/2-}
  \]
  depends continuously on $(b,\gamma)$ near $(b_0,0)$: there exist continuous (in $b,\gamma$) families
  \[
    h_{b,s 0}^{\gamma*},\ h_{b,s 1}^{\gamma*}(\scal),\ h_{b,v 1}^{\gamma*}(\vect) \in \cK_b^{\gamma*},\quad \scal\in\scal_1,\ \vect\in\vect_1,
  \]
  linear in $\scal$ and $\vect$, which satisfy $h_{b,s 0}^{0*}=h_{b,s 0}^*$, $h_{b,s 1}^{0*}(\scal)=h_{b,s 1}^*(\scal)$, and $h_{b,v 1}^{0*}(\vect)=h_{b,v 1}^*(\vect)$.
\end{lemma}

\begin{rmk}
  Paralleling Remark~\ref{Rmk00Reg}, we note that the dual states are still necessarily supported in $r\geq r_b$ and smooth/conormal in $r>r_b$. Their Sobolev regularity at the event horizon $r=r_b$ is $-\tfrac32+\cO(|\gamma|)$, as the constraint damping modification is non-trivial there and may thus shift the threshold regularity; see \cite[\S9.2]{HintzVasyKdSStability}.
\end{rmk}

\begin{proof}[Proof of Lemma~\ref{LemmaCD0Dual}]
  We use an argument by contradiction similar to that at the end of~\S\ref{Ss10}. Namely, fix $s>\tfrac52$, $\ell\in(-\tfrac32,-\half)$, and suppose there exists $h^*_j\in\cK_{b_j}^{\gamma_j*}$ with $\|h_j^*\|_{\Hbsupp^{s',\ell'}}=1$, where $s'=1-s<-\tfrac32$, $\ell'=-2-\ell\in(-\tfrac32,-\half)$, and where $(b_j,\gamma_j)\to(b_0,0)$ as $j\to\infty$, but so that $h_j^*$ stays a fixed distance away from $\cK_{b_0}^{0*}$; that is, there exists $h\in\CIdot$ such that $h\in\ann\cK_{b_0}^{0*}$, but $|\la h,h_j^*\ra|\geq c>0$ for all $j$. Recall that we have \emph{uniform} estimates
  \[
    \|h^*\|_{\Hbsupp^{s',\ell'}} \leq C\left(\bigl\|\wh{L_{g_b,\gamma}}(0)^*h^*\bigr\|_{\Hbsupp^{s'-1,\ell'+2}} + \|h^*\|_{\Hbsupp^{s_0,\ell_0}}\right)
  \]
  for fixed $s_0<s'$, $\ell_0<\ell'$, and for all $(b,\gamma)$ in a fixed small neighborhood of $(b_0,0)$. Applying this to $h_j^*$ at $(b_j,\gamma_j)$ shows that $\|h_j^*\|_{\Hbsupp^{s_0,\ell_0}}$ is bounded from below by a positive constant; hence a weak limit $h_j^*\weakto h^*\in\Hbsupp^{s',\ell'}$, which automatically lies in $\cK_{b_0}^{0*}$, is necessarily non-zero. Since it also satisfies $|\la h,h^*\ra|\geq c>0$, this is a contradiction. This argument works equally well when $(b_j,\gamma_j)\to(b,\gamma)$ for some $(b,\gamma)$ close to $(b_0,0)$, finishing the proof of the continuity of $\cK_b^{\gamma*}$.

  In order to construct $h_{b,s 0}^{\gamma*}$, fix elements $h_{b,1},\ldots,h_{b,7}\in\CIdot$, which induce linear forms $\ell_{b,j}=\la h_{b,j},-\ra$ on $\cK_b^{\gamma*}$; we may arrange that $\{h_{b,s 0}^{0*}\}=\cK_b^{0*}\cap\ell_{b,1}^{-1}(1)\cap\bigcap_{j=2}^7\ell_{b,j}^{-1}(0)$, and that the $h_{b,j},\ell_{b,j}$ depend continuously on $b$. We can then define $h_{b,s 0}^{\gamma*}$ by $\{h_{b,s 0}^{\gamma*}\}=\cK_b^{\gamma*}\cap\ell_{b,1}^{-1}(1)\cap\bigcap_{j=2}^7\ell_{b,j}^{-1}(0)$. One can similarly construct $h_{b,s 1}^{\gamma*}(\scal)$ when $\scal$ is an element of a fixed basis of $\scal_1$, and define $h_{b,s 1}^{\gamma*}(\scal)$ for general $\scal$ as linear combination; similarly for $h_{b,v 1}^{\gamma*}(\vect)$.
\end{proof}

\begin{rmk}
\label{RmkCD0Dual}
  An alternative proof of Lemma~\ref{LemmaCD0Dual} proceeds by constructing $h_{b,s 0}^{\gamma*}$ etc.\ directly by adapting the normal operator arguments from~\S\S\ref{S0}, \ref{S1} and \ref{SL}; see the proof of Lemma~\ref{LemmaCD0GenDual} for details.
\end{rmk}

\begin{proof}[Proof of Theorem~\ref{ThmCD0Modes}\eqref{ItCD0Modes0gen}]
  Recall that the choice of symmetric gradient does not affect the existence of (generalized) mode solutions of the linearized gauge-fixed Einstein equation which are solutions of the linearized Einstein equation and also satisfy the linearized gauge condition; thus, for all these, the CD modification encoded by $\gamma$ is irrelevant.
  
  We first study the case $b=b_0$ and prove that the space $\wh\cK_{b_0}$ is exactly $11$-dimensional. The arguments in the proof of Lemma~\ref{LemmaL0Linv1} imply that there are no growing scalar or vector type $l\geq 2$ zero energy modes for $L_{g_{b_0},\gamma}$. Moreover, the non-existence of linearly growing vector type $l=1$ modes follows from the non-degeneracy of the pairing~\eqref{EqL0Linv1Nondeg} for $\gamma=0$; this argument extends to nearby $b$ and small $\gamma$ by replacing $L_{g_{b_0}}$, $h_{b_0,v 1}$, $h_{b_0,v 1}^*$ there by $L_{g_{b_0},\gamma}$, $h_{b_0,v 1}$, $h_{b_0,v 1}^{\gamma*}$, respectively. The non-existence of quadratically growing scalar $l=1$ modes of $L_{g_{b_0}}$ follows from the non-degeneracy of the pairing~\eqref{EqL0Qus1Pair}, which persists by similar arguments.
  
  Lastly, the existence of quadratically growing scalar $l=0$ modes can be excluded for small $\gamma\neq 0$ as follows: following the arguments around~\eqref{EqL0Qus0Fail} with $\delta_{g_{b_0}}^*$ replaced by $\wt\delta_{g,\gamma}^*$, \emph{and using constraint damping} (which requires $\gamma$ to be non-zero) in the form of Proposition~\ref{PropCD0}, we conclude that a quadratically growing scalar $l=0$ mode $h$ with non-zero quadratic (in $t_*$) term satisfies the linearized gauge condition $\delta_{g_{b_0}}\sfG_{g_{b_0}} h=0$ and is thus, by Lemma~\ref{LemmaL0Qus0}, in fact only linearly growing---a contradiction. In terms of pairings, this means that
  \begin{equation}
  \label{EqRModes0s0Pair}
    \big\la [[L_{g_{b_0},\gamma},t_*],t_*]h_{b_0,s 0} + 2[L_{g_{b_0},\gamma},t_*](\hat h_{b_0,s 0}-t_* h_{b_0,s 0}), h_{b_0,s 0}^{\gamma*} \big\ra,
  \end{equation}
  which reduces to~\eqref{EqL0Qus0Pair} for $\gamma=0$, \emph{does not vanish} for small non-zero $\gamma$. This completes the argument for $\wh\cK_{b_0}$.
  
  Fixing such $\gamma$, the pairing~\eqref{EqRModes0s0Pair} remains non-degenerate for black hole parameters $b$ sufficiently close to $b_0$, likewise for the analogues of the pairings~\eqref{EqL0Linv1Nondeg} and~\eqref{EqL0Qus1Pair} for $L_{g_b,\gamma}$. Thus, the arguments used around~\eqref{EqL0Linv1Nondeg} and \eqref{EqL0Qus1Pair} show that the subspace of those elements of $\cK_b$ which arise as leading order terms of at least linearly growing zero modes is $4$-dimensional, and the subspace consisting of those elements which are leading order terms of at least quadratically growing zero modes is trivial.
\end{proof}

We record the analogue of Lemmas~\ref{LemmaL0Lead} and \ref{LemmaCD0Dual} for generalized zero energy dual states.
\begin{lemma}
\label{LemmaCD0GenDual}
  In the notation of Theorem~\ref{ThmCD0Modes}, the 11-dimensional space
  \[
    \wh\cK_b^{\gamma*} := \ker\wh{L_{g_b,\gamma}}(0)^*\cap\Poly(t_*)\Hbsupp^{-\infty,-1/2-}
  \]
  depends continuously on $(b,\gamma)$, $\gamma\neq 0$, near $(b_0,0)$; moreover, this space is continuous down to $\gamma=0$. Furthermore:
  \begin{enumerate}
  \item The quotient space $\wh\cK_b^{\gamma*}/\cK_b^{\gamma*}$ is spanned by (the images in the quotient of) continuous families of generalized zero modes
    \[
      \hat h_{b,s 0}^{\gamma*},\ \hat h_{b,s 1}^{\gamma*}(\scal) \in \wh\cK_b^{\gamma*},
    \]
    linear in $\scal\in\scal_1$, which satisfy $\hat h_{b,s 0}^{0*}=\hat h_{b,s 0}^*$ and $\hat h_{b,s 1}^{0*}(\scal)=\hat h_{b,s 1}^*(\scal)$.
  \item\label{ItCD0GenDual2} We may choose the continuous families $h_{b,s 0}^{\gamma*}$, $h_{b,s 1}^{\gamma*}(\scal)$ in Lemma~\ref{LemmaCD0Dual} so that $h_{b,s 0}^{\gamma*}$, $h_{b,s 1}^{\gamma*}(\scal)$ are the $t_*$-coefficients of $\hat h_{b,s 0}^{\gamma*}$, $\hat h_{b,s 1}^{\gamma*}(\scal)$, respectively.
  \item\label{ItCD0GenDual3} Putting $\breve h_{b,s 0}^{\gamma*}:=\hat h_{b,s 0}^{\gamma*}-t_* h_{b,s 0}^{\gamma*}$ and $\breve h_{b,s 1}^{\gamma*}(\scal):=\hat h_{b,s 1}^{\gamma*}(\scal)-t_* h_{b,s 1}^{\gamma*}(\scal)$, we have
    \begin{subequations}
    \begin{align}
      h_{b,s 0}^{\gamma*},\ h_{b,s 1}^{\gamma*}(\scal),\ h_{b,v 1}^{\gamma*}(\vect) &\in \Hbsupp^{-\infty,1/2-}, \\
    \label{EqCD0GenDualBreve}
      \breve h_{b,s 0}^{\gamma*},\ \breve h_{b,s 1}^{\gamma*}(\scal) &\in \rho\CI+\Hbsupp^{-\infty,1/2-},
    \end{align}
    \end{subequations}
    where the $\rho\CI$ terms are supported in $r\geq 4\bhm_0$.
  \end{enumerate}
\end{lemma}
\begin{proof}
  The first statement is proved similarly to Lemma~\ref{LemmaCD0Dual} and uses the non-degeneracies exploited in the proof of Theorem~\ref{ThmCD0Modes}.

  The (generalized) zero energy dual states for $L_{g_b,\gamma}$ are again dual-pure-gauge states with suitable 1-forms as potentials, as we demonstrate momentarily. The analysis for $\gamma=0$ in the previous sections was simplified by the fact that some of the 1-form potentials themselves were differentials of scalar generalized modes; since for $\gamma\neq 0$, the exterior derivative $d$ and $\cP_{b,\gamma}^*$ do not satisfy a useful commutation relation, we need to argue directly on the level of 1-forms and 2-tensors in order to get more precise information on the dual-pure-gauge potentials, as needed for the control of $\breve h_{b,s 1}^*$ etc.\ as in the proof of Lemma~\ref{LemmaL0Lead}. We recall the space $\Omega_1$ from~\eqref{EqL0LeadOmega1}, and fix a cutoff $\chi\in\CI(\R)$, $\chi\equiv 0$ on $(-\infty,3\bhm_0]$, $\chi\equiv 1$ on $[4\bhm_0,\infty)$.

  %%%%%%%%%%%%%%%%%%%%%%%%%%%%%%
  \pfstep{Construction of $h_{b,s 0}^{\gamma*}$.} We wish to set $h_{b,s 0}^{\gamma*} = \sfG_{g_b}\delta_{g_b}^*\omega_{b,s 0}^{\gamma*}$ with $\omega_{b,s 0}^{\gamma*}\in\ker\wh{\cP_{b,\gamma}}(0)^*$, but the kernel here is trivial when intersected with $\Hbsupp^{-\infty,-1/2-}$ when $\gamma\neq 0$. The key is that we can now extend the 1-form $\pa_t^\flat$ from $r\gg 1$ to an element $\wt\omega_{b,s 0}^{\gamma*}\in\ker\wh{\cP_{b,\gamma}}(0)^*\cap\Hbsupp^{-\infty,-3/2-}$, since the obstruction~\eqref{Eq10NoDualdt} coming from the non-triviality of $\ker\wh{\Box_{g_b}}(0)\cap\Hbext^{\infty,-3/2+}$, disappears for small non-zero $\gamma$: thus, we can set $\omega_{b,s 0}^{\gamma*}=\chi\pa_t^\flat+\omega'_0$, where $\omega'_0\in\Hbsupp^{-\infty,-1/2-}$ solves $\wh{\cP_{b,\gamma}}(0)^*\omega'_0=-[\wh{\cP_{b,\gamma}}(0),\chi]\pa_t^\flat\in\CIc(X^\circ)$. Therefore, we can put (i.e.\ re-define), for now,
    \begin{subequations}
  \begin{equation}
  \label{EqCD0GenDualStrs0}
    h_{b,s 0}^{\gamma*} := \sfG_{g_b}\delta_{g_b}^*\omega_{b,s 0}^{\gamma*} \in \chi r^{-2}\Omega_1 + \Hbsupp^{-\infty,3/2-},
  \end{equation}
  where the structure of the leading order term again follows from the a priori membership $h_{b,s 0}^{\gamma*}\in\Hbsupp^{-\infty,1/2-}$ by normal operator considerations.
  
  %%%%%%%%%%%%%%%%%%%%%%%%%%%%%%
  \pfstep{Construction of $h_{b,s 1}^{\gamma*}(\scal)$.} Write $b=(\bhm,\bha)$ and $b_1=(\bhm,0)$. Set $\omega_{b,s 1}^{\gamma*}(\scal)=\chi\omega_{b_1,s 1}^*(\scal)+\omega'_1$, where $\omega'_1\in\Hbsupp^{-\infty,-1/2-}$ satisfies $\wh{\cP_{b,\gamma}}(0)^*\omega'_1=-\wh{\cP_{b,\gamma}}(0)^*(\chi\omega_{b_1,s 1}^*(\scal))\in\Hbsupp^{-\infty,5/2-}$, cf.\ \eqref{EqL0LeadBoxOmegap}, hence lies in $\rho\CI+\Hbsupp^{-\infty,1/2-}$. Thus, we can re-define, for now,
  \begin{equation}
  \label{EqCD0GenDualStrs1}
    h_{b,s 1}^{\gamma*}(\scal) := \sfG_{g_b}\delta_{g_b}^*\omega_{b,s 1}^{\gamma*}(\scal) \in \chi r^{-2}\Omega_1 + \Hbsupp^{-\infty,3/2-}.
  \end{equation}

  %%%%%%%%%%%%%%%%%%%%%%%%%%%%%%
  \pfstep{Construction of $h_{b,v 1}^{\gamma*}(\vect)$.} In order for the ansatz $\omega_{b,v 1}^{\gamma*}(\vect)=\chi r^2\vect+\omega'$ to produce an element of $\ker\wh{\cP_{b,\gamma}}(0)^*$, we need $\wh{\cP_{b,\gamma}}(0)^*\omega'=-\wh{\cP_{b,\gamma}}(0)^*(\chi r^2\vect)\in\Hbsupp^{-\infty,3/2-}$, which can be solved for $\omega'\in\Hbsupp^{-\infty,-1/2-}$ in view of Proposition~\ref{PropCD0}. (Note that this is stronger than what we proved in Proposition~\ref{Prop10Genv1}.) Thus,
  \begin{equation}
  \label{EqCD0GenDualStrv1}
    h_{b,v 1}^{\gamma*}(\vect) := \sfG_{g_b}\delta_{g_b}^*\omega_{b,v 1}^{\gamma*}(\vect) \in \chi r^{-2}\Omega_1 + \Hbsupp^{-\infty,3/2-}
  \end{equation}
  \end{subequations}
  as a consequence of the a priori $\Hbsupp^{-\infty,1/2-}$ membership.
  
  %%%%%%%%%%%%%%%%%%%%%%%%%%%%%%
  \pfstep{Construction of generalized dual states.} We begin by constructing $\hat h_{b,s 0}^{\gamma*}$ using the ansatz
  \[
    \hat h_{b,s 0}^{\gamma*} = t_*\bigl(h_{b,s 0}^{\gamma*} + h_{b,v 1}^{\gamma*}(\vect)\bigr) + \breve h_{b,s 0}^{\gamma*},
  \]
  which lies in $\ker L_{g_b,\gamma}^*$ provided
  \begin{equation}
  \label{EqCD0GenDuals0}
    \wh{L_{g_b,\gamma}}(0)^*\breve h_{b,s 0}^{\gamma*} = -[L_{g_b,\gamma}^*,t_*]h_{b,s 0}^{\gamma*} - [L_{g_b,\gamma}^*,t_*]h_{b,v 1}^{\gamma*}(\vect).
  \end{equation}
  This can be solved if and only if the pairing of the right hand side with $\cK_b$ is trivial. The pairing with $h_{b,s 0}$ and $h_{b,s 1}(\scal)$ automatically vanishes, since the latter are leading order terms of linearly growing generalized zero modes of $L_{g_b,\gamma}$. On the other hand, the pairing
  \[
    \vect_1\times\vect_1\ni (\vect,\vect') \mapsto \la[L_{g_b,\gamma}^*,t_*]h_{b,v 1}^{\gamma*}(\vect),h_{b,v 1}(\vect')\ra = -\la h_{b,v 1}^{\gamma*}(\vect), [L_{g_b,\gamma},t_*]h_{b,v 1}(\vect')\ra
  \]
  is non-degenerate due to~\eqref{EqL0Linv1Nondeg} and by continuity in $(b,\gamma)$. Thus, we can choose $\vect\in\vect_1$ such that~\eqref{EqCD0GenDuals0} has a solution $\breve h_{b,s 0}^{\gamma*}\in\Hbsupp^{-\infty,-1/2-}$. But then the arguments used in the proof of Lemma~\ref{LemmaL0Lead} apply in view of~\eqref{EqCD0GenDualStrs0} and \eqref{EqCD0GenDualStrv1} and imply that, in fact, $\breve h_{b,s 0}^{\gamma*}\in\chi\rho\CI+\Hbsupp^{-\infty,1/2-}$. Replacing $h_{b,s 0}^{\gamma*}$ by $h_{b,s 0}^{\gamma*}+h_{b,v 1}^{\gamma*}(\vect)$ for this choice of $\vect$ accomplishes parts~\eqref{ItCD0GenDual2} and \eqref{ItCD0GenDual3} of the lemma.

  The arguments for $\hat h_{b,s 1}^{\gamma*}$ and $\breve h_{b,s 1}^{\gamma*}$ are analogous, now using~\eqref{EqCD0GenDualStrs1} and \eqref{EqCD0GenDualStrv1}.

  We have thus constructed an explicit basis of $\wh\cK_b^{\gamma*}/\cK_b^{\gamma*}$. By the already known continuity of $\wh\cK_b^{\gamma*}$ in $(b,\gamma)$, we can then re-define $\hat h_{b,s 0}^{\gamma*}$ and $\hat h_{b,s 1}^{\gamma*}(\scal)$ to be suitable linear combinations of this basis and elements of $\cK_b^{\gamma*}$ to ensure the continuity in $(b,\gamma)$, in particular at $\gamma=0$.
\end{proof}

%%%%%%%%%%%%%%%%%%%%%%%%%%%%%%%%%%%%%%%%%%%%%%%%%%
\subsection{Mode stability of the gauge propagation operator in \texorpdfstring{$\Im\sigma\geq 0$}{the closed upper half plane}}
\label{SsCDU}

Without further on restrictions $\fv,\gamma$, it may happen that the zero energy state of $\cP_{b,0}$ is perturbed into a resonance of $\cP_{b,\gamma}$ in the upper half plane. We now show that for $\fv,\gamma$ with suitable signs, this does \emph{not} happen. We do this in two steps:
\begin{enumerate}
  \item we show in Lemma~\ref{LemmaCDU} that for $\fv>1$ and for \emph{all} sufficiently small $\gamma>0$, $\cP_{b_0,\gamma}$ has no modes in a \emph{fixed} neighborhood of $\sigma=0$ in the closed upper half plane.
  \item In Proposition~\ref{PropCDU}, we combine this with high energy estimates, as well as with perturbative (in $\gamma$) arguments in compact subsets of $\{\Im\sigma\geq 0,\ \sigma\neq 0\}$, to prove the mode stability of $\cP_{b_0,\gamma}$ for sufficiently small $\gamma>0$. \emph{Fixing} such $\gamma>0$, simple perturbation arguments then imply the mode stability of $\cP_{b,\gamma}$ for $b$ close to $b_0$.
\end{enumerate}

Let $s>\tfrac32$ and $\ell\in(-\tfrac32,-\half)$. The domains
\begin{equation}
\label{EqCDUFn}
  \cX_b^{s,\ell}(\sigma) = \bigl\{ \omega\in\Hbext^{s,\ell}(X;\wt\Tsc{}^*X) \colon \wh{\cP_{b,\gamma}}(\sigma)\omega\in\Hbext^{s-1,\ell+2}(X;\wt\Tsc{}^*X) \bigr\},
\end{equation}
of the operators $\wh{\cP_{b,\gamma}}(\sigma)\colon\cX_b^{s,\ell}(\sigma)\to\Hbext^{s-1,\ell+2}$ depend in a serious manner on $\sigma$ (but are independent of $\gamma$). Thus, the first step of the perturbation argument is to pass to operators with fixed domain and target spaces. This relies on:

\begin{lemma}
\label{LemmaCDUInv}
  Fix $\fv\neq 1$, and fix a value $\gamma_1\neq 0$ for which~\eqref{EqCD0} (with $\gamma$ replaced by $\gamma_1$) holds. Then $\wh{\cP_{b_0,\gamma_1}}(\sigma)\colon\cX_{b_0}^{s,\ell}(\sigma)\to\Hbext^{s-1,\ell+2}$ is invertible for $\sigma\in\C$ with $\Im\sigma\geq 0$ and $|\sigma|$ small. Moreover, $\wh{\cP_{b_0,\gamma_1}}(\sigma)^{-1}$ is continuous (in $\sigma$) with values in $\cL_{\rm weak}(\Hbext^{s-1,\ell+2},\Hbext^{s,\ell})$ (the space of bounded operators equipped with the weak operator topology), and is continuous with values in $\cL_{\rm op}(\Hbext^{s-1+\eps,\ell+2+\eps},\Hbext^{s-\eps,\ell-\eps})$ (norm topology) for any $\eps>0$.
\end{lemma}
\begin{proof}
  Let us write $\wh\cP(\sigma)\equiv\wh{\cP_{b_0,\gamma_1}}(\sigma)$ in this proof. The first statement follows from having \emph{uniform} semi-Fredholm estimates
  \begin{equation}
  \label{EqCDUInvEst}
    \|\omega\|_{\Hbext^{s,\ell}} \leq C\bigl(\|\wh\cP(\sigma)\omega\|_{\Hbext^{s-1,\ell+2}} + \|\omega\|_{\Hbext^{s_0,\ell_0}}\bigr),\quad \omega\in\cX_{b_0}^{s,\ell}(\sigma)
  \end{equation}
  for $\sigma\in\C$, $\Im\sigma\geq 0$, with $|\sigma|$ small, together with the invertibility of $\wh\cP(0)$: this implies by a standard contradiction argument, see e.g.\ \cite[Proof of Theorem~1.1]{VasyLowEnergyLag}, that the second, error, term on the right in~\eqref{EqCDUInvEst} can be dropped for $\sigma$ near $0$ (upon increasing the constant $C$ if necessary). This gives the injectivity of $\wh\cP(\sigma)$; since this operator is Fredholm of index $0$, its invertibility is an immediate consequence.

  We prove the continuity of $\wh\cP(\sigma)^{-1}$ following the line of reasoning of~\cite[\S2.7]{VasyMicroKerrdS}. Suppose $\sigma_j\to\sigma$, and suppose $f_j\in\Hbext^{s-1,\ell+2}$ is a sequence converging to $f\in\Hbext^{s-1,\ell+2}$; put $\omega_j=\wh\cP(\sigma_j)^{-1}f_j$, which is bounded in $\Hbext^{s,\ell}$. Consider a subsequential limit $\omega_{j_k}\weakto\omega\in\Hbext^{s,\ell}$; then $\wh\cP(\sigma)\omega$ is necessarily equal to the weak limit $\lim_{k\to\infty} \wh\cP(\sigma_{j_k})\omega_{j_k}=f$, independently of the subsequence. Therefore, the entire sequence converges weakly, $\omega_j\weakto\omega$, proving continuity in the weak operator topology.

  Suppose the continuity in the operator norm topology failed: then we could find $\delta>0$, a sequence $\sigma_j\to\sigma$, and a bounded sequence $f_j\in\Hbext^{s-1+\eps,\ell+2+\eps}$ such that for $\omega_j=\wh\cP(\sigma_j)^{-1}f_j$ and $\omega'_j=\wh\cP(\sigma)^{-1}f_j$, we have
  \begin{equation}
  \label{EqCDUInvLower}
    \|\omega_j-\omega'_j\|_{\Hbext^{s-\eps,\ell-\eps}}\geq\delta.
  \end{equation}
  Passing to a subsequence, we can assume $f_j\weakto f\in\Hbext^{s-1+\eps,\ell+2+\eps}$, with norm convergence in $\Hbext^{s-1,\ell+2}$; in particular, $\omega'_j\to\wh\cP(\sigma)^{-1}f=:\omega$ in $\Hbext^{s,\ell}$. On the other hand, by continuity in the weak operator topology, we have $\omega_j\weakto\omega$ in $\Hbext^{s,\ell}$, hence $\omega_j\to\omega$ in $\Hbext^{s-\eps,\ell-\eps}$. But this implies $\omega_j-\omega'_j\to 0$ in $\Hbext^{s-\eps,\ell-\eps}$, contradicting~\eqref{EqCDUInvLower} and finishing the proof.
\end{proof}

The improvement of Proposition~\ref{PropCD0} on the Schwarzschild spacetime is:

\begin{lemma}
\label{LemmaCDU}
  Let $s>\tfrac32$, $\ell\in(-\tfrac32,-\half)$, and $\fv>1$. Then there exist $\gamma'_0>0$, $C_0>0$ such that for all $0<\gamma<\gamma'_0$ and $\sigma\in\C$, $\Im\sigma\geq 0$, $|\sigma|<C_0$, the operator $\wh{\cP_{b_0,\gamma}}(\sigma)\colon\cX_{b_0,\gamma}(\sigma)\to\Hbext^{s-1,\ell+2}$ is invertible.
\end{lemma}
\begin{proof}
  We abbreviate $\cP_\gamma:=\cP_{b_0,\gamma}$, $\cX(\sigma):=\cX_{b_0}^{s,\ell}(\sigma)$.
  
  \pfstep{Formal argument.} We first give a non-rigorous argument showcasing the relevant calculation. Namely, split
  \begin{alignat*}{3}
    \cX(0)&=\cK^\perp\oplus\cK, &\quad \cK&=\ker_{\cX(0)}\wh{\cP_0}(0)=\la\omega_{b_0,s 0}\ra, \\
    \Hbext^{s-1,\ell+2}&=\cR\oplus\cR^\perp,&\quad \cR&=\ran_{\cX(0)}\wh{\cP_0}(0),
  \end{alignat*}
  where $\cK^\perp\subset\cX(0)$ and $\cR^\perp\subset\Hbext^{s-1,\ell+2}$ are arbitrary but fixed complementary subspaces; we can identify $\cR^\perp\cong\C$ via $\eta'\mapsto\la\eta',\omega_{b_0,s 0}^*\ra$. We then write
  \[
    \wh{\cP_\gamma}(\sigma) = \begin{pmatrix} \cP_{0 0}(\gamma,\sigma) & \cP_{0 1}(\gamma,\sigma) \\ \cP_{1 0}(\gamma,\sigma) & \cP_{1 1}(\gamma,\sigma) \end{pmatrix},
  \]
  where $\cP_{0 1}(0,0)=\cP_{1 0}(0,0)=\cP_{1 1}(0,0)=0$, and $\cP_{0 0}(0,0)$ is invertible. Suppose $\wh{\cP_\gamma}(\sigma)\omega=0$, and write $\omega=(\omega_0,\omega_1)$. Then $\omega_0=-\cP_{0 0}(\gamma,\sigma)^{-1}\cP_{0 1}(\gamma,\sigma)\omega_1$ and
  \begin{equation}
  \label{EqCDUInj11}
    \bigl(\cP_{1 1}(\gamma,\sigma) - \cP_{1 0}(\gamma,\sigma)\cP_{0 0}(\gamma,\sigma)^{-1}\cP_{0 1}(\gamma,\sigma)^{-1}\bigr)\omega_1 = 0;
  \end{equation}
  we want to show that $\omega_1=0$ (and thus $\omega_0=0$) for $|\sigma|+\gamma$ small, $\gamma>0$. Now, the second summand in parentheses is of size $\cO((|\gamma|+|\sigma|)^2)$ since $\cP_{1 0}(\gamma,\sigma),\cP_{0 1}(\gamma,\sigma)=\cO(\gamma+|\sigma|)$; it thus suffices to compute $\cP_{1 1}(\gamma,\sigma)$ modulo $\cO(\gamma^2+|\sigma|^2)$:
  \begin{equation}
  \label{EqCDU11Exp}
    \cP_{1 1}(\gamma,\sigma) = \gamma\cP_{1 1}^\flat+\sigma\cP_{1 1}^\sharp + \cO(\gamma^2+|\sigma|^2),\quad \cP_{1 1}^\flat=\pa_\gamma\cP_{1 1}(0,0),\ \cP_{1 1}^\sharp=\pa_\sigma\cP_{1 1}(0,0).
  \end{equation}
  The calculations~\eqref{Eq10NoLinComm}, \eqref{Eq10NoLinPair}, and~\eqref{EqCD0Nondeg} give
  \begin{equation}
  \label{EqCDU11Exp2}
    \cP_{1 1}^\flat = 4(\fv-1),\quad
    \cP_{1 1}^\sharp = \la -i[\cP_0,t_*]\omega_{b_0,s 0},\omega_{b_0,s 0}^*\ra = -2 i,
  \end{equation}
  and therefore
  \begin{equation}
  \label{EqCDUCalc}
    \cP_{1 1}(\gamma,\sigma)=4(\fv-1)\gamma - 2 i\sigma + \cO(\gamma^2+|\sigma|^2).
  \end{equation}
  Fixing $\fv>1$, this is non-zero in $\Im\sigma\geq 0$ for $\gamma+|\sigma|\leq C_0$, $\gamma>0$, as desired.

  \pfstep{Passage to fixed function spaces.} The above formal argument is not rigorous since the splitting of $\cX(0)$ does not give a splitting of $\cX(\sigma)$ for $\sigma\neq 0$. To remedy this, we use a standard trick in scattering theory and consider, for $\gamma_1\neq 0$ as in Lemma~\ref{LemmaCDUInv}, the operator
  \[
    \wh{\cP_\gamma}(\sigma)\wh{\cP_{\gamma_1}}(\sigma)^{-1} \colon \Hbext^{s-1,\ell+2} \to \Hbext^{s-1,\ell+2},
  \]
  which thus acts on a \emph{fixed} space. We split the target space as $\cR\oplus\cR^\perp$ as above, and the domain as $\wt\cK^\perp\oplus\wt\cK$, where
  \[
    \wt\cK = \wh{\cP_{\gamma_1}}(0)\cK = \la \wt\omega \ra,\quad \wt\omega:=\wh{\cP_{\gamma_1}}(0)\omega_{b_0,s 0}=(\wh{\cP_{\gamma_1}}(0)-\wh{\cP_0}(0))\omega_{b_0,s 0}\in\CIc(X^\circ),
  \]
  and $\wt\cK^\perp$ is any complement of $\wt\cK$ in $\Hbext^{s-1,\ell+2}$. In these splittings, we write
  \begin{equation}
  \label{EqCDUFixed}
    \wh{\cP_\gamma}(\sigma)\wh{\cP_{\gamma_1}}(\sigma)^{-1} = \begin{pmatrix} \wt\cP_{0 0}(\gamma,\sigma) & \wt\cP_{0 1}(\gamma,\sigma) \\ \wt\cP_{1 0}(\gamma,\sigma) & \wt\cP_{1 1}(\gamma,\sigma) \end{pmatrix},
  \end{equation}
  which is Fredholm of index $0$; moreover, $\wt\cP_{0 1}(0,0)=\wt\cP_{1 0}(0,0)=\wt\cP_{1 1}(0)=0$, and $\wt\cP_{0 0}(0,0)$ is invertible, and Fredholm of index $0$ for small $(\gamma,\sigma)$, $\Im\sigma\geq 0$, since the $(1,2)$, $(2,1)$, $(2,2)$ entries of~\eqref{EqCDUFixed} have rank $\leq 1$, hence are compact operators.
  
  \pfstep{Invertibility of $\wt\cP_{0 0}$.} We show that for $|\gamma|+|\sigma|$ small, there exists a uniform bound
  \begin{equation}
  \label{EqCDU00}
    \|\wt\cP_{0 0}(\gamma,\sigma)^{-1}\|_{\cR\to\wt\cK^\perp} \leq C.
  \end{equation}
  This is proved similarly to Lemma~\ref{LemmaCDUInv}, and uses that the bound~\eqref{EqCDUInvEst}, with $\cP_\gamma$ in place of $\cP$, is uniform for small $|\gamma|+|\sigma|$. Concretely, denote the projection onto $\cR$ along $\cR^\perp$ by $\Pi$; assuming that~\eqref{EqCDU00} fails, we then find sequences $(\gamma_j,\sigma_j)\to(0,0)$ and $f_j\in\wt\cK^\perp$, $\|f_j\|_{\Hbext^{s-1,\ell+2}}=1$ so that for $\omega_j=\wh{\cP_{\gamma_1}}(\sigma_j)^{-1}f_j$ (which is bounded in $\Hbext^{s,\ell}$ by Lemma~\ref{LemmaCDUInv}), we have $\Pi\wh{\cP_{\gamma_j}}(\sigma_j)\omega_j\to 0$ in $\Hbext^{s-1,\ell+2}$. Therefore,
  \begin{equation}
  \label{EqCDU00IM}
    \|\omega_j\|_{\Hbext^{s,\ell}} \leq C\left(o(1)+\|(I-\Pi)\wh{\cP_{\gamma_j}}(\sigma_j)\omega_j\|_{\Hbext^{s-1,\ell+2}} + \|\omega_j\|_{\Hbext^{s_0,\ell_0}}\right),
  \end{equation}
  where $s_0<s$, $\ell_0<\ell$. The second term on the right is bounded by a uniform constant times
  \begin{equation}
  \label{EqCDU00Conv}
    |\la\wh{\cP_{\gamma_j}}(\sigma_j)\omega_j,\omega_{b_0,s 0}^*\ra|= \bigl|\big\la\omega_j,\bigl(\wh{\cP_{\gamma_j}}(\sigma_j)-\wh{\cP_0}(0)\bigr)\omega_{b_0,s 0}^*\big\ra\bigr|\to 0,\quad j\to\infty;
  \end{equation}
  since $(\wh{\cP_{\gamma_j}}(\sigma_j)-\wh{\cP_0}(0))\omega_{b_0,s 0}^*$ converges to $0$ in norm in $\Hbsupp^{-s,-\ell}$. We thus conclude from equation~\eqref{EqCDU00IM} that a weakly convergent subsequence $\omega_j\weakto\omega$ has a \emph{non-zero} limit $\omega\in\Hbext^{s,\ell}$. Since $\wh{\cP_{\gamma_j}}(\sigma_j)\omega_j\to\wh{\cP_0}(0)\omega$ in distributions, and since by assumption and using~\eqref{EqCDU00Conv} this limit is $0$, we have $\omega\in\cK$. Now, taking a weakly convergent subsequence $f_j\weakto f$ in $\Hbext^{s-1,\ell+2}$, thus $f_j\to f$ in $\Hbext^{s-1-\eps,\ell+2-\eps}$, we have $\omega_j\weakto\wh{\cP_{\gamma_1}}(0)^{-1}f$ in $\Hbext^{s-\eps,\ell-\eps}$, and therefore $f=\wh{\cP_{\gamma_1}}(0)\omega\in\wt\cK$, and $f\neq 0$ because of $\omega\neq 0$ and since $\wh{\cP_{\gamma_1}}(0)$ is injective. But this contradicts $f_j\in\wt\cK^\perp$: indeed, choosing $f^*\in\Hbsupp^{-s+1,-\ell-2}$ such that $\wt\cK^\perp=\ker f^*$, we have $0=f^*(f_j)\to f^*(f)\neq 0$. This proves~\eqref{EqCDU00}.

  \pfstep{Differentiability of $\wt\cP_{0 1}$ and $\wt\cP_{1 1}$.} The next step is to show that the rank $1$ operator
  \begin{equation}
  \label{EqCDU1}
    \wt\cP_1(\gamma,\sigma):=\wt\cP_{0 1}(\gamma,\sigma)\oplus\wt\cP_{1 1}(\gamma,\sigma)\colon\wt\cK\to\Hbext^{s-1,\ell+2},
  \end{equation}
  is once differentiable at $(0,0)$. Writing
  \begin{equation}
  \label{EqCDU1Rewrite}
    \wh{\cP_\gamma}(\sigma)\wh{\cP_{\gamma_1}}(\sigma)^{-1} = I + \bigl(\wh{\cP_\gamma}(\sigma)-\wh{\cP_{\gamma_1}}(\sigma)\bigr)\wh{\cP_{\gamma_1}}(\sigma)^{-1},
  \end{equation}
  this amounts to proving the differentiability of the second summand at $(0,0)$. We first prove its continuity: since $\wt\omega\in\CIc(X^\circ)$, the 1-form $\wh{\cP_{\gamma_1}}(\sigma)^{-1}\wt\omega\in\Hbext^{s',\ell'}$ is continuous in $\sigma$ for any $s'\in\R$, $\ell'\in(-\tfrac32,-\half)$. But since $\wh{\cP_\gamma}(\sigma)-\wh{\cP_{\gamma_1}}(\sigma)\in\Diff^1$ has coefficients in $\CIc(X^\circ)$ and depends smoothly on $(\gamma,\sigma)$, the continuity of $\wh{\cP_\gamma}(\sigma)\wh{\cP_{\gamma_1}}(\sigma)^{-1}\wt\omega\in\CIc(X^\circ)$ follows.

  We also deduce that for proving differentiability at $(0,0)$, it suffices to prove the differentiability at $\sigma=0$ of $\wh{\cP_{\gamma_1}}(\sigma)^{-1}\wt\omega$. To this end, we write (formally at this point)
  \begin{equation}
  \label{EqCDU1diff}
    \bigl(\wh{\cP_{\gamma_1}}(\sigma)^{-1}-\wh{\cP_{\gamma_1}}(0)^{-1}\bigr)\wt\omega = \wh{\cP_{\gamma_1}}(\sigma)^{-1}\bigl(\wh{\cP_{\gamma_1}}(0)-\wh{\cP_{\gamma_1}}(\sigma)\bigr)\wh{\cP_{\gamma_1}}(0)^{-1}\wt\omega.
  \end{equation}
  The right hand side is well-defined since $\wh{\cP_{\gamma_1}}(0)^{-1}\wt\omega=\omega_{b_0,s 0}\in\rho\CI+\Hbext^{\infty,1/2-}$ is annihilated, modulo $\Hbext^{\infty,3/2-}$, by the normal operator $-2\sigma\rho(\rho D_\rho+i)$ (cf.\ \eqref{EqOpLinFT2}) of $\wh{\cP_{\gamma_1}}(0)-\wh{\cP_{\gamma_1}}(\sigma)$; thus
  \[
    \wt\omega(\sigma):=\bigl(\wh{\cP_{\gamma_1}}(0)-\wh{\cP_{\gamma_1}}(\sigma)\bigr)\wh{\cP_{\gamma_1}}(0)^{-1}\wt\omega \in \Hbext^{\infty,3/2-},
  \]
  with smooth dependence on $\sigma$, to which one can indeed apply $\wh{\cP_{\gamma_1}}(\sigma)^{-1}$. With both sides of~\eqref{EqCDU1diff} well-defined in $\Hbext^{\infty,-1/2-}$, the equality can be justified by a regularization argument as in \cite[\S4]{VasyLowEnergyLag}. More precisely, let us write
  \[
    \wt\omega(\sigma) = \sigma\wt\omega'(0) + \cO(|\sigma|^2),\quad \wt\omega'(0)=-\pa_\sigma\wh{\cP_{\gamma_1}}(0)\omega_{b_0,s 0},
  \]
  then Lemma~\ref{LemmaCDUInv} gives
  \begin{equation}
  \label{EqCDU1Taylor1}
    \bigl(\wh{\cP_{\gamma_1}}(\sigma)^{-1}-\wh{\cP_{\gamma_1}}(0)^{-1}\bigr)\wt\omega = \sigma\wh{\cP_{\gamma_1}}(0)^{-1}\wt\omega'(0) + o_{\Hbext^{s',\ell'}}(|\sigma|),\quad \sigma\to 0,
  \end{equation}
  that is, the remainder has $\Hbext^{s',\ell'}$-norm $o(|\sigma|)$; here, $s'\in\R$, $\ell'\in(-\tfrac32,-\half)$ are arbitrary. Since $\wt\cK$ is finite-dimensional, this proves the differentiability of~\eqref{EqCDU1} with error term of the Taylor expansion of size $o(|\sigma|)$ in $\cL_{\rm op}(\wt\cK,\Hbext^{s-1,\ell+2})$.
  
  In view of in the formal argument above, we need to compute the Taylor expansion of $\wt\cP_{1 1}(\gamma,\sigma)$. To do this, write $\wh\sC(\sigma)=\pa_\gamma\wh{\cP_\gamma}(\sigma)$; this is independent of $\gamma$, and we have $\wh{\cP_{\gamma_1}}(\sigma)-\wh{\cP_{\gamma_2}}(\sigma)=(\gamma_1-\gamma_2)\wh\sC(\sigma)$ for all $\gamma_1,\gamma_2\in\R$. Therefore, by~\eqref{EqCDU1Rewrite} and~\eqref{EqCDU1Taylor1},
  \begin{align*}
    &\wh{\cP_\gamma}(\sigma)\wh{\cP_{\gamma_1}}(\sigma)^{-1}\wt\omega \\
    &\quad = \wh{\cP_{\gamma_1}}(0)\omega_{b_0,s 0} + (\gamma-\gamma_1)\wh\sC(\sigma)\bigl(\omega_{b_0,s 0}-\sigma\wh{\cP_{\gamma_1}}(0)^{-1}\pa_\sigma\wh{\cP_{\gamma_1}}(0)\omega_{b_0,s 0} + o(|\sigma|)\bigr) \\
    &\quad = (\wh{\cP_{\gamma_1}}(0)-\wh{\cP_0}(0))\omega_{b_0,s 0} - \gamma_1\wh\sC(\sigma)\omega_{b_0,s 0} + \gamma\wh\sC(0)\omega_{b_0,s 0} \\
    &\quad\hspace{5em} + \sigma\gamma_1\bigl(\wh\sC(0)\wh{\cP_{\gamma_1}}(0)^{-1}\pa_\sigma\wh{\cP_{\gamma_1}}(0)\omega_{b_0,s 0} - \pa_\sigma\wh\sC(0)\omega_{b_0,s 0}\bigr) + \cO(|\gamma\sigma|) + o(|\sigma|) \\
    &\quad = \gamma\wh\sC(0)\omega_{b_0,s 0} + \sigma\Bigl(\bigl(\wh{\cP_{\gamma_1}}(0)-\wh{\cP_0}(0)\bigr)\wh{\cP_{\gamma_1}}(0)^{-1}\bigl(\pa_\sigma\wh{\cP_0}(0)+\gamma_1\pa_\sigma\wh\sC(0)\bigr)\omega_{b_0,s 0} \\
    &\quad\hspace{12em} - \gamma_1\pa_\sigma\wh\sC(0)\omega_{b_0,s 0}\Bigr) + \cO(|\gamma\sigma|) + o(|\sigma|) \\
    &\quad = \gamma\wh\sC(0)\omega_{b_0,s 0} + \sigma\pa_\sigma\wh{\cP_0}(0)\omega_{b_0,s 0} - \sigma\wh{\cP_0}(0)\bigl(\wh{\cP_{\gamma_1}}(0)^{-1}\pa_\sigma\wh{\cP_{\gamma_1}}(0)\omega_{b_0,s 0}\bigr) + \cO(|\gamma\sigma|) + o(|\sigma|),
  \end{align*}
  with the error terms measured in $\Hbext^{s',\ell'}$ for any $s',\ell'\in\R$. Paired with $\omega_{b_0,s 0}^*$ to get the component in $\cR^\perp$, the third summand gives $0$, and hence we obtain, using the calculation leading to~\eqref{EqCDUCalc},
  \begin{equation}
  \label{EqCDU11}
    \wt\cP_{1 1}(\gamma,\sigma) = \gamma\cP_{1 1}^\flat + \sigma\cP_{1 1}^\sharp + o(|\gamma|+|\sigma|) = 4(\fv-1)\gamma-2 i\sigma + o(|\gamma|+|\sigma|).
  \end{equation}

  \pfstep{Continuity of $\wt\cP_{1 0}$.} Next, we prove that $\wt\cP_{1 0}(\gamma,\sigma)\colon\wt\cK^\perp\to\cR^\perp$ is continuous at $(\gamma,\sigma)=(0,0)$, which follows a fortiori from the continuity of $(\wt\cP_{1 0},\wt\cP_{1 1})\colon\Hbext^{s-1,\ell+2}\to\cR^\perp$, thus from that of the complex-valued map
  \[
    \Hbext^{s-1,\ell+2}\ni\omega\mapsto\la\wh{\cP_\gamma}(\sigma)\wh{\cP_{\gamma_1}}(\sigma)^{-1}\omega,\omega_{b_0,s 0}^*\ra = \la \wh{\cP_{\gamma_1}}(\sigma)^{-1}\omega, \wh{\cP_\gamma}(\sigma)^*\omega_{b_0,s 0}^*\ra,
  \]
  in $(\gamma,\sigma)$ at $(0,0)$. In view of~\eqref{EqCDU00}, its operator norm is bounded by a uniform constant times $\|\wh{\cP_\gamma}(\sigma)^*\omega_{b_0,s 0}^*\|_{\Hbsupp^{-s,-\ell}}$; the latter however is clearly continuous in $(\gamma,\sigma)$, and equal to $0$ at $(\gamma,\sigma)=(0,0)$. (In fact, one can prove the differentiability of $(\wt\cP_{1 0},\wt\cP_{1 1})$ at $(0,0)$ by means of arguments similar to those used above.)

  \pfstep{Conclusion of the proof.} We now combine all the pieces: in view of~\eqref{EqCDU00}, the proof that $\wh{\cP_\gamma}(\sigma)\wh{\cP_{\gamma_1}}(\sigma)^{-1}$ is injective for $\gamma>0$, $\Im\sigma\geq 0$, with $|\sigma|+|\gamma|$ small, reduces, as in~\eqref{EqCDUInj11}, to the proof that the operator
  \begin{equation}
  \label{EqCDU11inv}
    \wt\cP_{1 1}(\gamma,\sigma) - \wt\cP_{1 0}(\gamma,\sigma)\wt\cP_{0 0}(\gamma,\sigma)^{-1}\wt\cP_{0 1}(\gamma,\sigma) \colon \wt\cK \to \cR^\perp
  \end{equation}
  is injective. But this follows from~\eqref{EqCDU11}, the vanishing $\wt\cP_{0 1}(\gamma,\sigma)=\cO(|\sigma|+\gamma)$, as well as the vanishing $\wt\cP_{1 0}(\gamma,\sigma)=o(1)$ as $(\gamma,\sigma)\to (0,0)$: indeed, the operator~\eqref{EqCDU11inv} is equal to $\wt\cP_{1 1}(\gamma,\sigma)+o(|\gamma|+|\sigma|)$. Surjectivity of $\wh{\cP_\gamma}(\sigma)$, which has index $0$, is an immediate consequence.
\end{proof}

We now show that robust high energy estimates and perturbation theory in compact subsets of the closed upper half plane imply full mode stability, also on slowly rotating Kerr spacetimes:

\begin{prop}
\label{PropCDU}
  Let $\fv>1$, $s>\tfrac32$, $\ell\in(-\tfrac32,-\half)$. Then there exists $\gamma_0>0$ such that for $0<\gamma<\gamma_0$, there exists a constant $C(\gamma)>0$ such that the following holds: if $|b-b_0|<C(\gamma)$, then the operator $\wh{\cP_{b,\gamma}}(\sigma)\colon\cX_b^{s,\ell}(\sigma)\to\Hbext^{s-1,\ell+2}$ is invertible for all $\sigma\in\C$, $\Im\sigma\geq 0$.
\end{prop}
\begin{proof}
  Consider first $b=b_0$. Then Lemma~\ref{LemmaCDU} provides us with $\gamma'_0>0$ and $C_0>0$ such that the conclusion holds when $0<\gamma<\gamma'_0$ and $|\sigma|<C_0$. For these $\gamma$, we have uniform high energy estimates as in the proof of Theorem~\ref{ThmOp}; in particular, there exists $C_1>1$ such that for all $\gamma\in(0,\gamma'_0)$, the operator $\wh{\cP_{b_0,\gamma}}(\sigma)$ is invertible when $\Im\sigma\geq 0$, $|\sigma|\geq C_1$.

  Suppose now $\sigma_0\in\C$, $\Im\sigma_0\geq 0$, and $C_0\leq|\sigma_0|\leq C_1$. Then $\wh{\cP_{b_0,0}}(\sigma_0)$ is invertible by Theorem~\ref{Thm1}; a simple perturbation theory argument as in~\cite[\S2.7]{VasyMicroKerrdS} implies the invertibility of $\wh{\cP_{b_0,\gamma}}(\sigma)$ for $(\gamma,\sigma)$ in an open set around $(0,\sigma_0)$. A compactness argument implies the existence of $0<\gamma''_0$ such that $\wh{\cP_{b_0,\gamma}}(\sigma)$ is invertible for $C_0\leq|\sigma|\leq C_1$, $\Im\sigma\geq 0$, and $|\gamma|<\gamma''_0$. The proposition thus holds for $\gamma_0:=\min(\gamma'_0,\gamma''_0)$ and $b=b_0$.

  Perturbation arguments as in the proof of Theorem~\ref{Thm0} imply, for any \emph{fixed} choice of $\gamma\in(0,\gamma_0)$, the invertibility of $\wh{\cP_{b,\gamma}}(\sigma)$ in $\Im\sigma\geq 0$ for $b$ sufficiently close to $b_0$. The proof is complete.
\end{proof}

\begin{cor}
\label{CorCDUModes}
  Let $s>\tfrac52$, $\ell\in(-\tfrac32,-\half)$, and fix $\gamma\in(0,\gamma_0)$ as in Proposition~\ref{PropCDU}. Then mode stability holds for the operator $L_{g_{b_0},\gamma}$ on the Schwarzschild spacetime: the operator $\wh{L_{g_{b_0},\gamma}}(\sigma)\colon\cX_{b_0}^{s,\ell}(\sigma)\to\Hbext^{s-1,\ell+2}$ is invertible for $\sigma\in\C$, $\Im\sigma\geq 0$, and $\sigma\neq 0$.
\end{cor}
\begin{proof}
  It suffices to prove injectivity. But any $h\in\ker\wh{L_{g_{b_0},\gamma}}(\sigma)\cap\Hbext^{s,\ell}$ satisfies
  \[
    \delta_{g_{b_0}}\sfG_{g_{b_0}} h\in\ker\wh{\cP_{b_0,\gamma}}(\sigma)\cap\Hbext^{s-1,\ell}=\{0\}
  \]
  by Proposition~\ref{PropCDU}, hence $D_{g_{b_0}}\Ric(h)=0$. The rest of the proof is the same as the proof of the corresponding statement for $\wh{L_{g_{b_0},0}}(\sigma)$ in Proposition~\ref{PropL0}.
\end{proof}

In Proposition~\ref{PropRExist} below, we prove this for Kerr parameters $b\neq b_0$, the most delicate input being a rather explicit description of the resolvent near zero energy; the latter will rely on the non-degenerate structures used in the proof of Theorem~\ref{ThmCD0Modes}.

\begin{rmk}
\label{RmkCDUGauge}
  Recall that besides the modified constraint propagation operator $\cP_{g,\gamma}=2\delta_g\sfG_g\circ\wt\delta_{g,\gamma}^*$, which controls the properties (in particular: absence of zero energy states) of the gauge potentials of putative (generalized) zero modes of $L_{g,\gamma}$, there is another 1-form wave operator, $2 D_g\Ups\circ\delta_g^*$, which we called the \emph{gauge potential wave operator}, which controls what gauge potentials satisfy the linearized gauge condition, and also, more generally, how to add to a given linearized Kerr solution a pure gauge term so as to obtain a solution of the gauge-fixed Einstein equation. A modification of the gauge condition thus affects the latter wave operator, but not the former. Concretely then, we may modified the linearized gauge condition to be
  \[
    \wt\delta_{g,\gamma}\sfG_g h = 0,\quad \wt\delta_{g,\gamma}:=(\wt\delta_{g,\gamma}^*)^*.
  \]
  Correspondingly, the modified gauge potential wave operator is
  \[
    2\wt\delta_{g,\gamma}\sfG_g\circ \delta_g^* = \cP_{g,\gamma}^*,
  \]
  that is, gauge and constraint damping modifications are formally dual to one another (but not on the level of function spaces, as we need to work with extendible spaces for both when studying modes of $L_{g_b,\gamma}$). Using the same type of calculation as in~\S\ref{SsCDU}, one can then show that $2\wt\delta_{g,\gamma}\sfG_g\delta_g^*$ does not have any modes in $\Im\sigma\geq 0$ (in particular, the zero energy nullspace is trivial) when $\fv>2$ and $\gamma>0$ is small. With the thus modified gauge, the space of generalized scalar $l=0$ zero modes becomes 1-dimensional, spanned by a linearized Schwarzschild solution plus a pure gauge term (noting that the obstruction in~\eqref{EqL0SchwGauge} is now absent). Ultimately, this leads to a reduction of the space of pure gauge solutions arising in Theorem~\ref{ThmIBaby}: it only consists of Lie derivatives along asymptotic translations, rotations, and Lorentz boosts. In this sense, it is generated entirely by the asymptotic symmetries of the spacetime at null infinity.
\end{rmk}

%%%%%%%%%%%%%%%%%%%%%%%%%%%%%%%%%%%%%%%%%%%%%%%%%%%%%%%%%%%%%%%%%%%%%%
\section{Structure of the resolvent of the linearized modified gauge-fixed Einstein operator}
\label{SR}

We now use Theorem~\ref{ThmCD0Modes} (and the non-degenerate structure of $L_{g_b,\gamma}$ going into its proof) to show that its resolvent $\wh{L_{g_b,\gamma}}(\sigma)^{-1}$ exists for $b$ close to $b_0$ and $\sigma\in\C$, $\Im\sigma\geq 0$, $\sigma\neq 0$, see~\S\ref{SsRE}, in particular Proposition~\ref{PropRExist}. This utilizes arguments similar to (but more intricate, due to the more complicated generalized null space structure, than) those used in the proof of Lemma~\ref{LemmaCDU}. In~\S\ref{SsRS}, we give a precise description of the resolvent near $\sigma=0$ as the sum of a finite rank operator which is meromorphic in $\sigma$ with a double pole at $0$, and the `regular part' which is continuous down to $\sigma=0$. Subsequent sections refine this further by establishing higher regularity of the regular part.

We remark that the results in this section \emph{imply} the mode stability of slowly rotating Kerr black holes under metric perturbations; we stress that this is due to our embedding of the linear stability problem into an \emph{analytically non-degenerate framework}, obviating the need for arguments based separations of variables in the non-Schwarzschild black hole case.

We stress that from this point onwards, we only use \emph{structural} information on $L_{g_b,\gamma}$ and its zero energy behavior from the previous sections (rather than explicit expressions of (generalized) zero energy (dual) states): this is all one needs when using the general perturbation stable Fredholm framework developed by Vasy \cite{VasyLowEnergyLag,VasyLAPLag} for the study of resolvents on asymptotically conic spaces.

Define $\cd$ by~\eqref{EqCD1form} with $\fv>1$. Let
\begin{subequations}
\begin{equation}
\label{EqRsell}
  s>\tfrac52,\quad \ell\in(-\tfrac32,-\half).
\end{equation}
With $\gamma_0>0$ as in Proposition~\ref{PropCDU}, \emph{let us henceforth fix}
\begin{equation}
\label{EqRgamma}
  \gamma\in(0,\gamma_0).
\end{equation}
We then put
\begin{equation}
\label{EqRLb}
  L_b := L_{g_b,E},
  \quad E=E(g;\cd,\gamma,\gamma),
\end{equation}
\end{subequations}
where $g=g_b$, $|b-b_0|<C(\gamma)$ is a slowly rotating Kerr metric. We re-define
\[
  \cX_b^{s,\ell}(\sigma) := \bigl\{ h\in\Hbext^{s,\ell}(X;S^2\,\wt\Tsc{}^*X) \colon \wh{L_b}(\sigma)h\in\Hbext^{s-1,\ell+2}(X;S^2\,\wt\Tsc{}^*X) \bigr\}.
\]
Moreover, we recall the (generalized) zero modes $h_{b,s 0}$, $h_{b,s 1}(\scal)$, $h_{b,v 1}(\scal)$, $\hat h_{b,s 0}$, $\hat h_{b,s 1}(\scal)$ from Propositions~\ref{PropL0} and \ref{PropL0Lin}, and denote the (generalized) zero energy dual states of $L_b$ using the notation of Lemma~\ref{LemmaCD0Dual} by
\[
  h_{b,s 0}^*:=h_{b,s 0}^{\gamma*},\quad
  h_{b,s 1}^*(\scal):=h_{b,s 1}^{\gamma*}(\scal),\quad
  h_{b,v 1}^*(\vect):=h_{b,v 1}^{\gamma*}(\vect);
\]
we also write $\hat h_{b,s 0}^*:=\hat h_{b,s 0}^{\gamma*}$ etc.\ in the notation of Lemma~\ref{LemmaCD0GenDual}, and
\[
  \breve h_{b,s 0}:=\hat h_{b,s 0}-t_*h_{b,s 0},\quad
  \breve h_{b,s 0}^*:=\hat h_{b,s 0}^*-t_* h_{b,s 0}^*,
\]
etc., which is the `same' as in Definition~\ref{DefL0Breve}, except here we use the dual states for the \emph{modified} operator. This constitutes an abuse of notation, as these dual states are \emph{not} equal to the dual states of Proposition~\ref{PropL0}; however, \emph{we henceforth only work with the modified operator $L_b$}, thus there is no ambiguity in meaning.

%%%%%%%%%%%%%%%%%%%%%%%%%%%%%%%%%%%%%%%%%%%%%%%%%%
\subsection{Existence of the resolvent; rough description near zero energy}
\label{SsRE}

The determination of the structure of the resolvent relies on perturbation arguments. As in~\S\ref{SsCDU}, we first perturb $L_b$ to a `reference operator' which is invertible near $\sigma=0$:

\begin{lemma}
\label{LemmaRPert0}
  There exist $V\in\Psi^{-\infty}(X^\circ;S^2 T^*X^\circ)$, with compactly supported Schwartz kernel, and a constant $C_1>0$ such that
  \[
    \check L_b(\sigma) := \wh{L_b}(\sigma) + V \colon \cX_b^{s,\ell}(\sigma)\to\Hbext^{s-1,\ell+2}
  \]
  is invertible for $\sigma\in\C$, $\Im\sigma\geq 0$, $|\sigma|<C_1$, and $|b-b_0|\leq C_1$. Moreover, $\check L_b(\sigma)^{-1}$ is continuous in $\sigma$ with values in $\cL_{\rm weak}(\Hbext^{s-1,\ell+2},\Hbext^{s,\ell})\cap\cL_{\rm op}(\Hbext^{s-1+\eps,\ell+2+\eps},\Hbext^{s-\eps,\ell-\eps})$, $\eps>0$.
\end{lemma}
\begin{proof}
  It suffices to prove the invertibility for $b=b_0$, $\sigma=0$, since arguments as in the proof of Lemma~\ref{LemmaCDUInv} then imply the invertibility and continuous dependence for $(b,\sigma)$ close to $(b_0,0)$. In the splittings $\cX_{b_0}^{s,\ell}(0)=\cK^\perp\oplus\cK$ with $\cK=\ker\wh{L_{b_0}}(0)$ and $\Hbext^{s-1,l+2}=\cR\oplus\cR^\perp$ with $\cR=\ran\wh{L_{b_0}}(0)$, write
  \[
    \wh{L_{b_0}}(0) = \begin{pmatrix} L_{0 0} & 0 \\ 0 & 0 \end{pmatrix}.
  \]
  The main input for the perturbation theory is the fact that $\cK$ and $\cR^\perp$ have the same (finite) dimension, namely $7$. Identifying $\cK\cong\C^7$ by choosing a basis $h_1,\ldots,h_7$, and identifying $\cR^\perp\cong\C^7$ via $f\mapsto(\la f,h^*_j\ra)_{j=1,\ldots,7}$, where $h_1^*,\ldots,h_7^*$ is a basis of $\ker\wh{L_{b_0}}(0)^*\cap\Hbsupp^{-\infty,-3/2+}$, it suffices (by the same arguments as in the proof of Proposition~\ref{PropCD0}) to construct $V$ so that the $7\times 7$ matrix
  \begin{equation}
  \label{EqRPert0Mtx}
    (\la Vh_i,h_j^*\ra)_{1\leq i,j\leq 7}
  \end{equation}
  is invertible. To do this, select $h_1^\sharp,\ldots,h_7^\sharp\in\CIc(X^\circ;S^2 T^*X^\circ)$ such that $(\la h_i^\sharp,h_j^*\ra)=\delta_{i j}$; this is possible since the $h_i^*$ are linearly independent distributions. Likewise, the $h_i$ are linearly independent; thus, we can select $h_1^\flat,\ldots,h_7^\flat\in\CIc(X^\circ;S^2 T^*X^\circ)$ with $\la h_i,h_j^\flat\ra=\delta_{i j}$. We then set $V=\sum_{k=1}^7 h_k^\sharp\la-,h_k^\flat\ra$.
\end{proof}

To set up the low energy spectral theory, we define the spaces
\begin{equation}
\label{EqRZeroModes}
\begin{aligned}
  \cK_{b,s 0} &:= \C h_{b,s 0}, && \quad & \cK_{b,s 0}^* & := \C h_{b,s 0}^*, \\
  \cK_{b,s 1} &:= h_{b,s 1}(\scal_1), && \quad & \cK_{b,s 1}^* & := h_{b,s 0}^*(\scal_1), \\
  \cK_{b,s} &:= \cK_{b,s 0}\oplus\cK_{b,s 1},&& \quad & \cK_{b,s}^* &:= \cK_{b,s 0}^* \oplus \cK_{b,s 1}^*, \\
  \cK_{b,v} &:= h_{b,v 1}(\vect_1), & && \cK_{b,v}^* & := h_{b,v 1}^*(\vect_1), \\
  \cK_b &:= \cK_{b,s}\oplus\cK_{b,v}, & && \cK_b^*&:=\cK_{b,s}^*\oplus\cK_{b,v}^*.
\end{aligned}
\end{equation}
The reason for combining the scalar $l=0$ and scalar $l=1$ spaces is that $\wh{L_b}(\sigma)$ will be more singular on both of them due to the existence of linearly growing solutions with leading terms in $\cK_{b,s}$. For $\check L_b(\sigma)$ as in Lemma~\ref{LemmaRPert0}, we set
\begin{equation}
\label{EqRwtcK}
\begin{split}
  \wt\cK_{b,s j} &:= \check L_b(0)\cK_{b,s j}\ (j=0,1), \\
  \wt\cK_{b,s} &:= \check L_b(0)\cK_{b,s}, \\
  \wt\cK_{b,v} &:= \check L_b(0)\cK_{b,v}, \\
  \wt\cK_b &:= \wt\cK_{b,s}\oplus\wt\cK_{b,v}.
\end{split}
\end{equation}
By definition of $\check L_b(0)$, these are subspaces of $\CIc(X^\circ;S^2 T^*X^\circ)$ which depend continuously on $b$. We fix a complementary subspace $\wt\cK^\perp\subset\Hbext^{s-1,\ell+2}$ of $\wt\cK$.

We decompose the target space $\Hbext^{s-1,\ell+2}$ into the range of $\wh{L_b}(0)$ and a complement. To do this in a continuous (in $b$) manner, we prove a slight generalization of the procedure used in the proof of Lemma~\ref{LemmaRPert0}:
\begin{lemma}
\label{LemmaRProj}
  There exists a linear projection map $\Pi_b^\perp\colon\Hbext^{s-1,\ell+2}\to\Hbext^{s-1,\ell+2}$ which is of rank $7$, depends continuously on $b$ near $b_0$ in the norm topology, and satisfies
  \[
    \la(I-\Pi_b^\perp)f,h^*\ra=0\quad\forall\,h^*\in\cK_b^*.
  \]
  The Schwartz kernel of $\Pi_b^\perp$ can be chosen to be independent of $s,\ell$ satisfying~\eqref{EqRsell}.
\end{lemma}
\begin{proof}
  Let $h_{b,1}^*,\ldots,h_{b,7}^*\in\Hbsupp^{1-s,-1/2-}$ denote a basis of $\cK_b^*$ which depends continuously on $b$, and fix $h_1^\sharp,\ldots,h_7^\sharp\in\CIc(X^\circ)$ for which the matrix $A_b=(A_{b,i j})=(\la h_i^\sharp,h_{b,j}^*\ra)$ is invertible for $b=b_0$, hence for nearby $b$. Therefore, there exists $(p_b^{i j})$, continuous in $b$, such that $\sum_{i=1}^7 p_b^{i j}A_{b,i k}=\delta_{j k}$. We then put $\Pi_b^\perp=p^{i j}h_i^\sharp\la-,h_{b,j}^*\ra$, which satisfies all requirements.
\end{proof}

Defining the complementary projection
\[
  \Pi_b := I-\Pi_b^\perp \colon \Hbext^{s-1,\ell+2} \to \ran\Pi_b=\ann\cK_b^*=\ran_{\cX_b^{s,\ell}(0)}\wh{L_b}(0),
\]
we then split domain and target according to
\begin{equation}
\begin{aligned}
\label{EqRSplit}
  &\text{domain:}\ && \Hbext^{s-1,\ell+2} \cong \wt\cK^\perp \oplus \wt\cK_{b,s} \oplus \wt\cK_{b,v}, \\
  &\text{target:}\ && \Hbext^{s-1,\ell+2} \cong \ran\Pi_b \oplus \cR_s^\perp \oplus \cR_v^\perp,
\end{aligned}
\end{equation}
where $\cR_s^\perp$, resp.\ $\cR_v^\perp$, is a space of dimension $\dim\cK_{b_0,s}=4$, resp.\ $\dim\cK_{b_0,v}=3$, chosen such that the $L^2$-pairing $\cR_s^\perp\times\cK_{b,s}^*\to\C$, resp.\ $\cR_v^\perp\times\cK_{b,v}^*\to\C$, is non-degenerate. (We can choose $\cR^\perp_{s/v}$ to be a subspace of $\CIc(X^\circ;S^2 T^*X^\circ)$.) Via these pairings, we can identify
\[
  \cR_s^\perp \cong (\cK_{b,s}^*)^*,\quad
  \cR_v^\perp \cong (\cK_{b,v}^*)^*;
\]
we shall use these identifications implicitly below.

We now prove that the resolvent at $\sigma\neq 0$, $\Im\sigma\geq 0$, exists; in the course of the proof, we will obtain a rough description of its structure near $\sigma=0$, which we will successively improve later on.

\begin{prop}
\label{PropRExist}
  Fix $s,\ell,\gamma$ as in~\eqref{EqRsell}--\eqref{EqRgamma}. For small $C_0>0$ and Kerr parameters $b\in\R^4$, $|b-b_0|<C_0$, the operator $\wh{L_b}(\sigma)\colon\cX^{s,\ell}_b(\sigma)\to\Hbext^{s-1,\ell+2}$ is invertible for $\sigma\in\C$, $\Im\sigma\geq 0$, $\sigma\neq 0$.
\end{prop}
\begin{proof}
  This is the content of Corollary~\ref{CorCDUModes} when $b=b_0$. The key facts we will use in the proof for $b$ near $b_0$ when $\sigma$ is near $0$ are:
  \begin{enumerate}
  \item the zero energy nullspace $\cK_b$ is $7$-dimensional, and the generalized zero energy nullspace $\wh\cK_b$ is $11$-dimensional, with both depending continuously on $b$ in the $\Poly(t_*)\Hbext^{\infty,-1/2-}$ topology, see Theorem~\ref{ThmCD0Modes}; similarly for the spaces $\cK_b^*$ and $\wh\cK_b^*$ of (generalized) dual states at zero energy in view of Lemmas~\ref{LemmaCD0Dual} and \ref{LemmaCD0GenDual};
  \item suitable pairings, such as~\eqref{EqL0Linv1Nondeg}, between (generalized) zero energy states and dual states are non-degenerate for $b=b_0$, which persists for $b$ near $b_0$; this was already exploited in the proof of Theorem~\ref{ThmCD0Modes}.
  \end{enumerate}
  We recall that such pairings are closely related to properties of the Taylor expansion of $\wh{L_b}(\sigma)$ at $\sigma=0$, as was already exploited in the proof of mode stability for the modified constraint propagation operator in Proposition~\ref{PropCDU}, see in particular equations \eqref{EqCDU11Exp}--\eqref{EqCDU11Exp2}.

  Concretely then, we shall determine the structure of the operator
  \[
    \wh{L_b}(\sigma)\check L_b(\sigma)^{-1}\colon\Hbext^{s-1,\ell+2}\to\Hbext^{s-1,\ell+2}
  \]
  when $(b,\sigma)$, $\Im\sigma\geq 0$, lies in a neighborhood of $(b_0,0)$, and prove its invertibility for non-zero $\sigma$. Thus, there is a \emph{fixed} constant $\delta>0$ such that \emph{for all} Kerr parameters $b$ close to $b_0$, $\wh{L_b}(\sigma)$ is invertible for $0<|\sigma|<\delta$, $\Im\sigma\geq 0$. Given such $\delta>0$, and using high energy estimates as well as perturbative arguments (starting with the mode stability of $L_{b_0}$) in compact subsets of $\{\Im\sigma\geq 0,\,\sigma\neq 0\}$ as in the proof of Proposition~\ref{PropCDU}, the operator $\wh{L_b}(\sigma)$ is invertible for $|\sigma|\geq\delta>0$ when $|b-b_0|$ is sufficiently small, proving the proposition.

  In the splittings~\eqref{EqRSplit} of $\Hbext^{s-1,\ell+2}$, we write
  \begin{equation}
  \label{EqRMtx}
    \wh{L_b}(\sigma)\check L_b(\sigma)^{-1} = \begin{pmatrix} L_{0 0} & L_{0 1} & L_{0 2} \\ L_{1 0} & L_{1 1} & L_{1 2} \\ L_{2 0} & L_{2 1} & L_{2 2} \end{pmatrix},\quad
    L_{i j}=L_{i j}(b,\sigma);
  \end{equation}
  for instance, $L_{0 0}(b,\sigma):=\Pi_b\wh{L_b}(\sigma)\check L_b(\sigma)|_{\wt\cK^\perp}$. Since the range of $\wh{L_b}(\sigma)$ is annihilated by $\cK_b^*$, we have $L_{1 0}(b,0)=L_{2 0}(b,0)=0$; likewise, $\wh{L_b}(\sigma)|_{\cK_b}\equiv 0$ implies that
  \[
    L_{\ker} := \begin{pmatrix} L_{0 1} & L_{0 2} \\ L_{1 1} & L_{1 2} \\ L_{2 1} & L_{2 2} \end{pmatrix}
  \]
  satisfies $L_{\ker}(b,0)=0$. Lastly, $L_{0 0}(b_0,0)$ is invertible.

  \pfstep{Uniform invertibility of $(0,0)$ entry.} The first step of our analysis of $\wh{L_b}(\sigma)\check L_b(\sigma)^{-1}$ is the analogue of~\eqref{EqCDU00} in the present setting (and proved in the same manner): there exists a uniform constant $C<\infty$ such that for $(b,\sigma)$ near $(b_0,0)$,
  \begin{equation}
  \label{EqR00}
    \|L_{0 0}(b,\sigma)^{-1} \|_{\cR_b\to\wt\cK^\perp} \leq C.
  \end{equation}

  \pfstep{Differentiability of $L_{i j}$, $(i,j)\neq(0,0)$.} The next step is the analogue of the differentiability of~\eqref{EqCDU1}, namely the differentiability of $L_{\ker}(b,\sigma)\wt h$ at $\sigma=0$ for $\wt h=\check L_b(0)h$, $h\in\cK_b$, with uniform control of the error term of the Taylor expansion. The key is that
  \begin{equation}
  \label{EqRReParen}
    \bigl(\check L_b(\sigma)^{-1}-\check L_b(0)^{-1}\bigr)\wt h = \check L_b(\sigma)^{-1}\bigl(\check L_b(0)-\check L_b(\sigma)\bigr)h,
  \end{equation}
  as follows from a regularization argument as in~\cite[\S4]{VasyLowEnergyLag}. But
  \begin{equation}
  \label{EqRDiff}
    \check L_b(0)-\check L_b(\sigma) = \wh{L_b}(0)-\wh{L_b}(\sigma) = -\sigma\pa_\sigma\wh{L_b}(0) - \half\sigma^2\pa_\sigma^2\wh{L_b}(0),
  \end{equation}
  with $\pa_\sigma^2\wh{L_b}(0)\in\rho^2\CI(X;\End(S^2\,\wt{\Tsc^*}X))$ decaying quadratically in $r$. Thus, using the continuity properties of $\check L_b(\sigma)^{-1}$ proved in Lemma~\ref{LemmaRPert0}, we conclude that
  \begin{align}
    \wh{L_b}(\sigma)\check L_b(\sigma)^{-1}(\check L_b(0)h) &=  (I-V\check L_b(\sigma)^{-1})(\check L_b(0)h) \nonumber\\
  \label{EqRExpKer}
      &=\wh{L_b}(0)h - V\check L_b(\sigma)^{-1}(\wh{L_b}(0)-\wh{L_b}(\sigma))h \\
      &= \sigma V\check L_b(\sigma)^{-1}\pa_\sigma\wh{L_b}(0)h + \cO(|\sigma|^2)h, \nonumber
  \end{align}
  where the error term (which in fact maps into $\CIc$) is measured using the operator norm on $\cL(\cK_b,\Hbext^{s-1,\ell+2})$. Now note that $\pa_\sigma\wh{L_b}(0)h\in\Hbext^{\infty,3/2-}$ since $\pa_\sigma\wh{L_b}(0)\in\rho\Diffb^1$, and $h\in\cK_b\subset\Hbext^{\infty,1/2-}$ is quadratically decaying. Therefore,
  \begin{equation}
  \label{EqRKerCont}
    \check L_b(\sigma)^{-1}\pa_\sigma\wh{L_b}(0)h\in\Hbext^{\infty,-1/2-}
  \end{equation}
  is continuous \emph{down to $\sigma=0$}, and hence
  \begin{equation}
  \label{EqRKerCont2}
    \wh{L_b}(\sigma)\check L_b(\sigma)^{-1}(\check L_b(0)h) = \sigma V\check L_b(0)^{-1}\pa_\sigma\wh{L_b}(0)h + o(|\sigma|)h,
  \end{equation}
  proving the desired differentiability. (The $o(|\sigma|)$ remainder here, as well as in subsequent calculations, is \emph{uniform} in $b$, and is in fact uniformly (in $b$) bounded in norm by $C|\sigma|^{1+\alpha}$ for some $\alpha>0$ and for some uniform constant $C>0$; this follows from the fact that $\check L_b(\sigma)$ is in fact H\"older-$\alpha$ regular at $\sigma=0$, with uniform H\"older constant, when one strengthens the domain or relaxes the target space, see Proposition~\ref{PropRegCheckL} below.)
  
  A crucial observation for subsequent arguments is that the continuity of~\eqref{EqRKerCont} holds provided merely $h\in\rho\CI+\Hbext^{\infty,1/2-}$, i.e.\ an $r^{-1}$ leading term is acceptable too; this relies on the fact that the normal operator of $\pa_\sigma\wh{L_b}(0)$ is $-2\rho(\rho D_\rho+i)$ by Lemma~\ref{LemmaOpLinFT}, which maps $\rho\CI\to\rho^3\CI$, i.e.\ gains one more order of decay than a priori expected from an element of $\rho\Diffb$. (This was already exploited in the proof of Lemma~\ref{LemmaCDU} around equation~\eqref{EqCDU1diff}.)

  Similar arguments give the differentiability of $L_{1 0}\oplus L_{2 0}$ at $\sigma=0$. Indeed, for $f\in\wt\cK^\perp$ and $h^*\in\cK_b^*$, we need to compute
  \[
    \la \wh{L_b}(\sigma)\check L_b(\sigma)^{-1}f, h^*\ra = \la f, (\check L_b(\sigma)^{-1})^*\wh{L_b}(\sigma)^*h^* \ra = \la f, h^*\ra - \la f, (\check L_b(\sigma)^{-1})^*V^* h^*\ra,
  \]
  where we can rewrite the second term by means of
  \begin{equation}
  \label{EqRKerAdj}
    (\check L_b(\sigma)^{-1})^*V^* = (\check L_b(0)^{-1})^*V^* + \bigl((\check L_b(\sigma)^{-1})^*-(\check L_b(0)^{-1})^*\bigr)\check L_b(0)^*\circ(\check L_b(0)^{-1})^*V^*.
  \end{equation}
  Since $V^*$ has compactly supported (in $(X^\circ)^2$) smooth Schwartz kernel, we have
  \begin{equation}
  \label{EqRLbVstar}
    (\check L_b(0)^{-1})^*V^*=(\check L_b(0)^*)^{-1}V^*\colon \sD'(X^\circ) \to \rho\CI+\Hbsupp^{-\infty,1/2-}(X).
  \end{equation}
  Therefore, we can rewrite the second term in~\eqref{EqRKerAdj} as
  \[
    (\check L_b(\sigma)^{-1})^*\bigl(\check L_b(0)^*-\check L_b(\sigma)^*\bigr) \circ (\check L_b(0)^{-1})^*V^*;
  \]
  by Taylor expansion as in~\eqref{EqRDiff} we conclude that
  \[
    (\check L_b(\sigma)^{-1})^*V^*h^* = (\check L_b(0)^{-1})^*V^*h^* - \sigma(\check L_b(0)^{-1})^*\pa_\sigma\wh{L_b}(0)^*(\check L_b(0)^{-1})^*V^*h^*+o(|\sigma|)h^*.
  \]
  The error term here is measured in the norm topology on $\cL(\cK_b^*,\Hbsupp^{-s+1,-\ell-2})$.

  \pfstep{Coefficients of $\sigma$.} We next compute the leading coefficient of some of the $L_{i j}$. For $j=1,2$, we rewrite~\eqref{EqRKerCont2} for $h\in\cK_b$ using $V=\check L_b(0)-\wh{L_b}(0)$ as
  \[
    \wh{L_b}(\sigma)\check L_b(\sigma)^{-1}\check L_b(0)h = \sigma(I-\wh{L_b}(0)\check L_b(0)^{-1})\pa_\sigma\wh{L_b}(0) h + o(|\sigma|).
  \]
  Pairing this against an element $h^*\in\cK_b^*$, the coefficient of $\sigma$ is
  \begin{equation}
  \label{EqRLowerCoeff}
    \la\pa_\sigma\wh{L_b}(0)h,h^*\ra - \la\check L_b(0)^{-1}\pa_\sigma\wh{L_b}(0)h,\wh{L_b}(0)^*h^*\ra.
  \end{equation}

  Consider this first for $h=h_{b,v 1}(\vect)\in\cK_{b,v}$, $h^*=h_{b,v 1}^*(\vect')\in\cK_{b,v}^*$, $\vect,\vect'\in\vect_1$. Then the second summand in~\eqref{EqRLowerCoeff} vanishes, and the first summand gives a non-degenerate pairing on $\cK_{b,v}\times\cK_{b,v}^*\cong\vect_1\times\vect_1$ by continuity from~\eqref{EqL0Linv1Nondeg}. Thus, $\pa_\sigma L_{2 2}(b,0)$ is invertible, and so
  \begin{equation}
  \label{EqR22Inv}
    \sigma^{-1}L_{2 2}(b,\sigma)\ \text{is invertible for $(b,\sigma)$ near $(b_0,0)$}.
  \end{equation}

  Taking $h\in\cK_{b,v}$ still, but now $h^*=h_{b,s 0}^*\in\cK_{b,s}^*$, the second summand in~\eqref{EqRLowerCoeff} still vanishes, and now the first summand does, too; likewise for $h^*=h_{b,s 1}^*(\scal)$. Indeed,
  \begin{equation}
  \label{EqRpasigma}
    \pa_\sigma\wh{L_b}(0)^*h_{b,s 0}^* = -i[L_b,t_*]h_{b,s 0}^* = -i L_b(\hat h_{b,s 0}^*-\breve h_{b,s 0}^*) = i\wh{L_b}(0)^*\breve h_{b,s 0}^*,
  \end{equation}
  which gives
  \begin{equation}
  \label{EqR12Vanish}
    \la\pa_\sigma\wh{L_b}(0)h,h^*\ra = \la h,\pa_\sigma\wh{L_b}(0)^*h^*\ra = -i\la h,\wh{L_b}(0)^*\breve h_{b,s 0}^*\ra = -i\la\wh{L_b}(0)h,\breve h_{b,s 0}^*\ra = 0.
  \end{equation}
  
  For $h\in\cK_{b,s}$ and $h^*\in\cK_b^*$, the pairing~\eqref{EqRLowerCoeff} vanishes, too, since
  \[
    \la\pa_\sigma\wh{L_b}(0)h_{b,s 0},h^*\ra = i \la\wh{L_b}(0)\breve h_{b,s 0},h^*\ra = i\la\breve h_{b,s 0},\wh{L_b}(0)^*h^*\ra = 0.
  \]
  We have thus proved that
  \begin{equation}
  \label{EqR111221Deg}
    \pa_\sigma L_{1 1}(b,0) = \pa_\sigma L_{1 2}(b,0) = \pa_\sigma L_{2 1}(b,0) = 0.
  \end{equation}

  For $h\in\cK_{b,s}$, the conclusions for $L_{1 1}$ and $L_{2 1}$ imply $\Pi_b^\perp\pa_\sigma|_{\sigma=0}(\wh{L_b}(\sigma)\check L_b(\sigma)^{-1}\check L_b(0)h)=0$; therefore, on $\cK_{b,s}$, we have $\pa_\sigma L_{0 1}(b,0)=\pa_\sigma(\wh{L_b}(\sigma)\check L_b(\sigma)^{-1})|_{\sigma=0}$, so
  \begin{equation}
  \label{EqR01Lead}
  \begin{split}
    \pa_\sigma L_{0 1}(b,0)(\check L_b(0)h)|_{\sigma=0} &= i V\check L_b(0)^{-1}\wh{L_b}(0)\breve h \\
     &= i(V-V\check L_b(0)^{-1}V)\breve h,\quad \text{where}\ L_b(t_*h+\breve h)=0.
  \end{split}
  \end{equation}
  (Note that the right hand side only depends on $\breve h\bmod\cK_b$.) Similarly, using
  \begin{equation}
  \label{EqRcheckLVstar}
    (\check L_b(0)^{-1})^*V^*h^*=h^*-(\check L_b(0)^{-1})^*\wh{L_b}(0)^*h^*=h^*,\quad h^*\in\cK_b^*,
  \end{equation}
  we have, for $f\in\wt\cK^\perp$ and $h^*\in\cK_{b,s}^*$,
  \begin{equation}
  \label{EqR10Lead}
  \begin{split}
    \la\pa_\sigma L_{1 0}(b,0)f,h^*\ra|_{\sigma=0} &= \la f,(\check L_b(0)^{-1})^*\pa_\sigma\wh{L_b}(0)^*(\check L_b(0)^{-1})^*V^*h^*\ra \\
      &= -i\la f,(I-(\check L_b(0)^{-1})^*V^*)\breve h^*\ra,\quad \text{where}\ L_b^*(t_*h^*+\breve h^*)=0.
  \end{split}
  \end{equation}
  (The right hand side only depends on $\breve h^*\bmod\cK_b^*$ by~\eqref{EqRcheckLVstar}.)
  
  \pfstep{Leading order term of $L_{1 1}$.} In order to capture $\wh{L_b}(\sigma)\check L_b(\sigma)^{-1}$ in a non-degenerate manner, we must compute more terms in the Taylor expansion of $L_{1 1}$ at $\sigma=0$. Consider thus $h\in\cK_{b,s}$ and $h^*\in\cK_{b,s}^*$, then, in view of~\eqref{EqRExpKer}, \eqref{EqR111221Deg}, and~\eqref{EqRcheckLVstar},
  \begin{align}
    &\la \wh{L_b}(\sigma)\check L_b(\sigma)^{-1}(\check L_b(0)h),h^*\ra \nonumber\\
  \label{EqRbottomCalc}
    &\quad = \la V\check L_b(\sigma)^{-1}(\sigma\pa_\sigma\wh{L_b}(0)+\tfrac{\sigma^2}{2}\pa_\sigma^2\wh{L_b}(0))h,h^*\ra \\
    &\quad = \la V\check L_b(\sigma)^{-1}\tfrac{\sigma^2}{2}\pa_\sigma^2\wh{L_b}(0)h, h^*\ra + \la V(\check L_b(\sigma)^{-1}-\check L_b(0)^{-1})\sigma\pa_\sigma\wh{L_b}(0)h,h^*\ra \nonumber\\
    &\quad= \sigma^2\Bigl(\half\la V\check L_b(\sigma)^{-1}\pa_\sigma^2\wh{L_b}(0)h,h^*\ra 
      \nonumber\\
    &\quad\qquad\qquad -\big\la\pa_\sigma\wh{L_b}(0)h,(\check L_b(\sigma)^{-1})^*(\pa_\sigma\wh{L_b}(0)^*+\tfrac{\sigma}{2}\pa_\sigma^2\wh{L_b}(0)^*)(\check L_b(0)^{-1})^*V^*h^*\big\ra\Bigr) \nonumber\\
  \label{EqR11Calc0}
  \begin{split}
    &\quad = \sigma^2\Bigl(\half\la\pa_\sigma^2\wh{L_b}(0)h,(\check L_b(\sigma)^{-1})^*V^*h^*\ra \\
    &\quad\qquad\qquad - \big\la\pa_\sigma\wh{L_b}(0)h,(\check L_b(\sigma)^{-1})^*(\pa_\sigma\wh{L_b}(0)^*+\tfrac{\sigma}{2}\pa_\sigma^2\wh{L_b}(0)^*)h^*\big\ra \Bigr).
  \end{split}
  \end{align}
  This is equal to $\sigma^2$ times a constant depending bilinearly on $(h,h^*)$, plus a $o(|\sigma|^2)$ remainder. To evaluate the constant, we introduce the pairing
  \begin{equation}
  \label{EqR11NondegPre}
    \ell_{b,s}(h,h^*) := -\half\la[[L_b,t_*],t_*]h+2[L_b,t_*]\breve h,h^*\ra,\quad (h,h^*)\in\cK_{b,s}\times\cK_{b,s}^*.
  \end{equation}
  Recall from~\eqref{EqL0Qus1Pair} that $\ell_{b,s}(h,-)=0\in(\cK_{b,s}^*)^*$ only if there exists a quadratically growing generalized zero mode with leading coefficient $h$; thus, this is a \emph{non-degenerate} pairing in view of the absence of such modes, see Theorem~\ref{ThmCD0Modes}. Then, with $\breve h,\breve h^*$ as in~\eqref{EqR01Lead}, \eqref{EqR10Lead}, and using~\eqref{EqRcheckLVstar} again to simplify the first pairing in~\eqref{EqR11Calc0} for $\sigma=0$, we have
  \begin{equation}
  \label{EqR11Calc}
  \begin{split}
    \la\half\pa_\sigma^2 L_{1 1}(b,0)h,h^*\ra &= \la\half\pa_\sigma^2\wh{L_b}(0)h - \pa_\sigma\wh{L_b}(0)\check L_b(0)^{-1}\pa_\sigma\wh{L_b}(0)h, h^*\ra \\
      &= \la-\half[[L_b,t_*],t_*]h - i\pa_\sigma\wh{L_b}(0)\check L_b(0)^{-1}\wh{L_b}(0)\breve h,h^*\ra \\
      &= -\half\la[[L_b,t_*],t_*]h+2[L_b,t_*]\breve h,h^*\ra + \la i\pa_\sigma\wh{L_b}(0)\check L_b(0)^{-1}V\breve h,h^*\ra \\
      &= \ell_{b,s}(h,h^*) + \la(V-V\check L_b(0)^{-1}V)\breve h,\breve h^*\ra.
  \end{split}
  \end{equation}
  Now, if we were only considering the top left $2\times 2$ minor of~\eqref{EqRMtx},\footnote{This in particular fully captures, on Schwarzschild spacetimes, the action of the operator $\wh{L_b}(\sigma)\check L_b(\sigma)^{-1}$ on symmetric 2-tensors which do not have a vector $l=1$ component, i.e.\ its invertibility for $b=b_0$ is necessary for the invertibility of the full operator.}
  \[
    \begin{pmatrix} L_{0 0} & L_{0 1} \\ L_{1 0} & L_{1 1} \end{pmatrix} = \begin{pmatrix} \cO(1) & \cO(\sigma) \\ \cO(\sigma) & \cO(\sigma^2) \end{pmatrix},
  \]
  its invertibility near $\sigma=0$ would be guaranteed provided $\sigma^{-2}(L_{1 1}-L_{1 0}L_{0 0}^{-1}L_{0 1})=\half\pa_\sigma^2 L_{1 1}-\pa_\sigma L_{1 0}\circ L_{0 0}^{-1}\circ\pa_\sigma L_{0 1}$ induces a non-degenerate pairing on $\wt\cK_{b,s}\times\cK_{b,s}^*$; but the calculations~\eqref{EqR01Lead} and \eqref{EqR10Lead} imply
  \begin{align*}
    \la \pa_\sigma L_{1 0}(b,0)L_{0 0}(b,0)^{-1}\pa_\sigma L_{0 1}(b,0)\check L_b(0)h,h^*\ra
      &= -i\la \pa_\sigma L_{0 1}(b,0)\check L_b(0)h,\breve h^*\ra \\
      &= \la(V-V\check L_b(0)^{-1}V)\breve h,\breve h^*\ra.
  \end{align*}
  In view of the non-degeneracy of~\eqref{EqR11NondegPre}, this implies that
  \begin{equation}
  \label{EqR11Nondeg}
  \begin{split}
    &\wt L_{1 1}^\sharp(b,\sigma):=\sigma^{-2}\bigl(L_{1 1}(b,\sigma) - L_{1 0}(b,\sigma)L_{0 0}(b,\sigma)^{-1}L_{0 1}(b,\sigma)\bigr) \\
    &\qquad \text{is invertible for $(b,\sigma)$ near $(b_0,0)$.}
  \end{split}
  \end{equation}

  We also note that the calculation~\eqref{EqR11Calc0} works also for $h\in\cK_{b,s}$, $h^*\in\cK_{b,v}^*$, as well as for $h\in\cK_{b,v}$, $h^*\in\cK_{b,s}^*$, implying that
  \begin{equation}
  \label{EqR1221Quadr}
    L_{1 2}(b,\sigma),\ L_{2 1}(b,\sigma) = \cO(|\sigma|^2);
  \end{equation}
  in fact, they are equal to $\sigma^2\wt L_{1 2}(b,\sigma)$, $\sigma^2\wt L_{2 1}(b,\sigma)$, with $\wt L_{1 2}$ and $\wt L_{2 1}$ continuous at $\sigma=0$.
  
  Furthermore, for $h\in\cK_{b,v}$ and $h^*\in\cK_{b,v}^*$, the calculation~\eqref{EqR11Calc0} is valid upon adding the linear (in $\sigma$) term $\la V\check L_b(0)^{-1}\sigma\pa_\sigma\wh{L_b}(0)h,h^*\ra$ (which we recall was zero when one of $h,h^*$ lied in $\cK_{b,s}^{(*)}$) in each line after~\eqref{EqRbottomCalc}. This proves the following strengthening of~\eqref{EqR22Inv}:
  \begin{equation}
  \label{EqR22InvErr}
  \begin{split}
    &\sigma^{-1}L_{2 2}(b,\sigma) = \ell_{b,v} + \cO(|\sigma|), \quad \ell_{b,v}(h,h^*) := \la-i[L_b,t_*]h,h^*\ra, \\
    &\qquad\text{as bilinear forms}\ \cK_{b,v}\times\cK_{b,v}^*\to\C.
  \end{split}
  \end{equation}

  \pfstep{The inverse of $\wh{L_b}(\sigma)\check L_b(\sigma)^{-1}$.} Equipped with the information from~\eqref{EqR00}, \eqref{EqR22Inv}, \eqref{EqR111221Deg}, \eqref{EqR11Nondeg}, and~\eqref{EqR1221Quadr}, we can now write
  \begin{equation}
  \label{EqROp}
    \wh{L_b}(\sigma)\check L_b(\sigma)^{-1}
     =\begin{pmatrix}
        L_{0 0} & \sigma\wt L_{0 1} & \sigma\wt L_{0 2} \\
        \sigma\wt L_{1 0} & \sigma^2\wt L_{1 1} & \sigma^2\wt L_{1 2} \\
        \sigma\wt L_{2 0} & \sigma^2\wt L_{2 1} & \sigma\wt L_{2 2}
      \end{pmatrix},
  \end{equation}
  with tacit dependence on $(b,\sigma)$, and solve
  \[
    \wh{L_b}(\sigma)\check L_b(\sigma)^{-1}\wt h=f,\qquad
    \wt h=(\wt h_0,\wt h_1,\wt h_2),\ \ f=(f_0,f_1,f_2)
  \]
  for small $\sigma\neq 0$ by first solving the first component of this equation for $\wt h_0$ (using~\eqref{EqR00}), then the third component for $\wt h_2$ (using~\eqref{EqR22Inv}), and then the second component for $\wt h_1$ (using~\eqref{EqR11Nondeg}. This gives
  \begin{equation}
  \label{EqRInv}
    \wt R_b(\sigma) := \bigl(\wh{L_b}(\sigma)\check L_b(\sigma)^{-1}\bigr)^{-1}
      =\begin{pmatrix}
         \wt R_{0 0} & \sigma^{-1}\wt R_{0 1} & \wt R_{0 2} \\
         \sigma^{-1}\wt R_{1 0} & \sigma^{-2}\wt R_{1 1} & \sigma^{-1}\wt R_{1 2} \\
         \wt R_{2 0} & \sigma^{-1}\wt R_{2 1} & \sigma^{-1}\wt R_{2 2}
       \end{pmatrix}
  \end{equation}
  in the splittings~\eqref{EqRSplit}, where the $\wt R_{i j}=\wt R_{i j}(b,\sigma)$ are continuous in $\sigma$. Explicitly, set
  \begin{equation}
  \label{EqRInv1}
    \wt L_{i j}^\sharp := \wt L_{i j} - \wt L_{i 0}L_{0 0}^{-1}\wt L_{0 j}, \quad
    \wt L_{i j}^\flat := \wt L_{i j} - \sigma\wt L_{i 0}L_{0 0}^{-1}\wt L_{0 j}, \quad
    \wt L_{1 1}^\natural=\wt L_{1 1}^\sharp-\sigma\wt L_{1 2}^\sharp(\wt L_{2 2}^\flat)^{-1}\wt L_{2 1}^\sharp,
  \end{equation}
  and recall from~\eqref{EqR00}, \eqref{EqR11NondegPre}, and \eqref{EqR22Inv} that $L_{0 0}$, $\wt L_{1 1}^\sharp$ (hence $\wt L_{1 1}^\natural$ for small $\sigma$), and $\wt L_{2 2}$ (hence $\wt L_{2 2}^\flat$ for small $\sigma$) are invertible. The singular terms in~\eqref{EqRInv} are then given by
  \begin{equation}
  \label{EqRInv2}
  \begin{gathered}
    \begin{aligned}
    \wt R_{0 1} &= -\bigl(L_{0 0}^{-1}\wt L_{0 1}-\sigma L_{0 0}^{-1}\wt L_{0 2}(\wt L_{2 2}^\flat)^{-1}\wt L_{2 1}^\sharp\bigr)\wt R_{1 1},  \\
    \wt R_{1 0} &= -\wt R_{1 1}\bigl(\wt L_{1 0}L_{0 0}^{-1}-\sigma\wt L_{1 2}^\sharp(\wt L_{2 2}^\flat)^{-1}\wt L_{2 0}L_{0 0}^{-1}\bigr),
    \end{aligned} \\
    \begin{aligned}
      \wt R_{1 1} &= (\wt L_{1 1}^\natural)^{-1}, &
      \wt R_{1 2} &= -\wt R_{1 1}\wt L_{1 2}^\sharp(\wt L_{2 2}^\flat)^{-1}, \\
      \wt R_{2 1} &= -(\wt L_{2 2}^\flat)^{-1}\wt L_{2 1}^\sharp \wt R_{1 1}, & \qquad
      \wt R_{2 2} &= (\wt L_{2 2}^\flat)^{-1} - \sigma(\wt L_{2 2}^\flat)^{-1}\wt L_{2 1}^\sharp \wt R_{1 2}.
    \end{aligned}
  \end{gathered}
  \end{equation}
  Since $\check L_b(\sigma)$ is invertible, the expression \eqref{EqRInv}, together with
  \begin{equation}
  \label{EqRInv3}
    \wh{L_b}(\sigma)^{-1} = \check L_b(\sigma)^{-1}\wt R_b(\sigma),
  \end{equation}
  explicitly demonstrates the invertibility of $\wh{L_b}(\sigma)$ for $0<|\sigma|<C_0$, $\Im\sigma\geq 0$, and $|b-b_0|<C_0$, when $C_0>0$ is sufficiently small.
\end{proof}

%%%%%%%%%%%%%%%%%%%%%%%%%%%%%%%%%%%%%%%%%%%%%%%%%%
\subsection{Precise structure of the resolvent near zero energy}
\label{SsRS}

\emph{We continue using the notation of the proof of Proposition~\ref{PropRExist}}. The formula~\eqref{EqRInv3} for the resolvent in terms of~\eqref{EqRInv} is not yet satisfactory for the purpose of solving the wave equation $L_b h=f$ by means of the inverse Fourier transform, $h(\sigma)=\wh{L_b}(\sigma)^{-1}\hat f(\sigma)$. Recall that the inverse Fourier transform of a term $\sigma^{-\alpha}$, read as $(\sigma+i 0)^{-\alpha}$, has asymptotic behavior $t_*^{\alpha-1}$; the $(1,0)$, $(1,2)$, $(2,1)$, and $(2,2)$ entries are thus already acceptable since they produce stationary terms in $\cK_b$. This is not the case for the $(0,1)$ term, as it produces a stationary term lying in an as of yet uncontrolled subspace of $\Hbext^{s,\ell}$. Note also that the $(1,1)$ term is only controlled modulo $o(|\sigma|^{-2})$ (or really $\cO(|\sigma|^{-2+\alpha})$ for some $\alpha>0$), which is not precise enough to allow for a useful description of the asymptotic behavior it produces (control of the time dependence being the issue); some degree of conormal regularity would be sufficient to prove that it produces a pure gauge solution modulo a decaying tail.\footnote{As an illustration, note that $t_*^{1-\alpha}h_{b,s 0}-\delta_{g_b}^*(t_*^{1-\alpha}\omega_{b,s 0})=\cO(t_*^{-\alpha})$.}

Rather than fixing these issues minimalistically for the purpose of obtaining a rather weak linear stability result (as far as decay is concerned), we proceed to obtain a \emph{complete} description of the singular part of the resolvent. The two main ingredients are:
\begin{enumerate}
\item a more careful choice of $\check L_b(\sigma)$ and of the splittings~\eqref{EqRSplit} ensures that the $(1,0)$ and $(0,1)$ components of $\wt R_b(\sigma)$ are regular at $\sigma=0$. This is set up in Lemma~\ref{LemmaRImprovedV} and explained in the first step of the proof of Theorem~\ref{ThmR};
\item the $r^{-1}$ leading order behavior not only of zero energy states but also of the stationary parts of generalized (dual) zero energy states, see Lemmas~\ref{LemmaL0Lead} and \ref{LemmaCD0GenDual}, enables us to Taylor expand certain components of $\wh{L_b}(\sigma)\check L_b(\sigma)^{-1}$ to high order. (The relevance of having such leading order terms was already indicated in the second paragraph after equation~\eqref{EqRKerCont2}.)
\end{enumerate}

\begin{definition}
\label{DefRBreve}
  Recall the spaces $\cK_b=\cK_{b,s}\oplus\cK_{b,v}$ and $\cK_b^*$ from~\eqref{EqRZeroModes}. For $h_s=h_{b,s j}\in\cK_{b,s}$, set $\breve h_s=\breve h_{b,s j}$ in the notation of Definition~\ref{DefL0Breve}, so that $L_b(t_* h_s+\breve h_s)=0$. We then define the non-degenerate (for $b=b_0$ and thus for nearby $b$) sesquilinear pairing
  \begin{equation}
  \label{EqRBreve}
  \begin{split}
    &k_b \colon \cK_b \times \cK_b^* \to \C \\
    &\qquad k_b((h_s,h_v),h^*) := \big\la\half\bigl([[L_b,t_*],t_*]h_s+2[L_b,t_*]\breve h_s\bigr) + [L_b,t_*]h_v, h^*\big\ra.
  \end{split}
  \end{equation}
\end{definition}

\begin{thm}
\label{ThmR}
  For $(b,\sigma)$, $\Im\sigma\geq 0$, in a small neighborhood of $(b_0,0)$ we have
  \[
    \wh{L_b}(\sigma)^{-1} = P_b(\sigma) + L^-_b(\sigma) \colon \Hbext^{s-1,\ell+2}(X;S^2\,\wt\Tsc{}^*X) \to \Hbext^{s,\ell}(X;S^2\,\wt\Tsc{}^*X).
  \]
  Here, the regular part $L_b^-(\sigma)$ has uniformly bounded operator norm, and is continuous with values in $\cL_{\rm weak}(\Hbext^{s-1,\ell+2},\Hbext^{s,\ell})\cap\cL_{\rm op}(\Hbext^{s-1+\eps,\ell+2+\eps},\Hbext^{s-\eps,\ell-\eps})$, $\eps>0$. The principal part $P_b(\sigma)$ is a quadratic polynomial in $\sigma^{-1}$ with finite rank coefficients; explicitly,\footnote{We take these signs and factors of $i$ because of $\cF^{-1}(i(\sigma+i 0)^{-1})=H(t_*)$, $\cF^{-1}(-(\sigma+i 0)^{-2})=t_* H(t_*)$.}
  \begin{equation}
  \label{EqRPrincipal}
    P_b(\sigma)f = (-\sigma^{-2}h_s+i\sigma^{-1}\breve h_s) + i\sigma^{-1}(h'_s + h_v),
  \end{equation}
  where $h_s,h'_s\in\cK_{b,s}$ and $h_v\in\cK_{b,v}$ are uniquely determined by the conditions
  \begin{subequations}
  \begin{alignat}{3}
  \label{EqRSing1}
    k_b((h_s,h_v),h^*) &= \la f,h^*\ra &\quad \text{for all}\ &h^*\in\cK_b^*, \\
  \label{EqRSing2}
    k_b(h'_s,h_s^*) &= -\la\half[[L_b,t_*],t_*](\breve h_s+h_v)+[L_b,t_*]\bar h,h_s^*\ra &\quad\text{for all}\ &h_s^*\in\cK_{b,s}^*,
  \end{alignat}
  where $\bar h\in\Hbext^{s,\ell}$ is a stationary solution of
  \begin{equation}
  \label{EqRBarh}
    L_b\bar h=f-\half\bigl([[L_b,t_*],t_*]h_s+2[L_b,t_*](\breve h_s+h_v)\bigr).
  \end{equation}
  \end{subequations}
\end{thm}

Given~\eqref{EqRSing1}, the solution of equation~\eqref{EqRBarh} is unique only modulo $\cK_b$; but condition~\eqref{EqRSing2} only depends on the image of $\bar h$ in $\Hbext^{s,\ell}/\cK_b$.

\begin{rmk}
  If we ask only that the tensor $\breve h_s$ be \emph{any} stationary solution of $L_b(t_*h_s+\breve h_s)=0$, then we can re-define $\breve h_s$ as $\breve h_s+h_v$, and subsequently set $h_v=0$. Keeping $h_v=0$, the term $\breve h_s$ is then only well-defined modulo $\cK_{b,s}$; it is easy to check that changing $\breve h_s$ to $\breve h_s+h''_s$, $h''_s\in\cK_{b,s}$, changes $h'_s$ to $h'_s-h''_s$ according to~\eqref{EqRSing1}--\eqref{EqRBarh}. Therefore, the description~\eqref{EqRPrincipal} of the principal part is well-defined independently of the choice of $\breve h_s$.
\end{rmk}

We first show that that the $(0,1)$ and $(1,0)$ components of $\wh{L_b}(\sigma)\check L_b(\sigma)^{-1}$ in~\eqref{EqROp} can be made to vanish quadratically at $\sigma=0$ by a more careful choice of the operator $\check L_b(\sigma)$ and the space $\wt\cK^\perp$ in~\eqref{EqRSplit}; this uses the explicit form~\eqref{EqR01Lead} and \eqref{EqR10Lead} of $\pa_\sigma L_{1 0},\pa_\sigma L_{0 1}$ at $\sigma=0$. To this end, we refine Lemma~\ref{LemmaRPert0} as follows:

\begin{lemma}
\label{LemmaRImprovedV}
  There exists $V_b\in\Psi^{-\infty}(X^\circ;S^2 T^*X^\circ)$ which is continuous in $b$ with uniformly compactly supported Schwartz kernel, such that
  \[
    \check L_b(\sigma)=\wh{L_b}(\sigma)+V_b\colon\cX_b^{s,\ell}(\sigma)\to\Hbext^{s-1,\ell+2}
  \]
  satisfies the conclusions of Lemma~\ref{LemmaRPert0}, and so that moreover for a suitably chosen (continuous in $b$) complementary subspace $\wt\cK_b^\perp$ of $\wt\cK_b=\check L_b(0)\cK_b=V_b(\cK_b)$, we have
  \begin{alignat}{2}
  \label{EqRImprovedV}
    V_b\breve h&=0,\quad \breve h&\,\in\,&\breve\cK_b:=\C\breve h_{b,s 0}\oplus\breve h_{b,s 1}(\scal_1), \\
  \label{EqRImprovedVAdj}
    \big\la f,\bigl(I-(\breve L_b(0)^{-1})^*V_b^*\bigr)\breve h^*\big\ra&=0,\quad
    \breve h^*&\,\in\,&\breve\cK_b^*:=\C\breve h_{b,s 0}^*\oplus\breve h_{b,s 1}^*(\scal_1),\ \ 
    f\in\wt\cK_b^\perp.
  \end{alignat}
\end{lemma}
\begin{proof}
  As argued around equation~\eqref{EqRPert0Mtx}, the conclusions of Lemma~\ref{LemmaRPert0} are satisfied provided $V_b\colon\sD'(X^\circ)\to\CIc(X^\circ)$ induces an injective map
  \begin{subequations}
  \begin{equation}
  \label{EqRImprovedV1}
    V_b|_{\cK_b}\colon\cK_b\hra(\cK_b^*)^*
  \end{equation}
  Moreover, one can find $\wt\cK_b^\perp$ such that~\eqref{EqRImprovedVAdj} holds iff $(I-(\check L_b(0)^{-1})^*V_b^*)\breve\cK_b^*\cap\ann(\wt\cK_b)=0$, which upon applying the invertible map $\check L_b(0)^*$ is equivalent to
  \begin{equation}
  \label{EqRImprovedV2}
    \wh{L_b}(0)^*\breve\cK_b^* \cap \check L_b(0)^*\ann(V_b(\cK_b)) = 0.
  \end{equation}
  \end{subequations}

  We proceed to arrange~\eqref{EqRImprovedV}, \eqref{EqRImprovedV1}, and \eqref{EqRImprovedV2}. Fix bases $\{h_{b,1},\ldots,h_{b,7}\}$ of $\cK_b$ and $\{h_{b,1}^*,\ldots,h_{b,7}^*\}$ of $\cK_b^*$ which depend continuously on $b$. Fix moreover $h_{b,i}^\sharp,h_{b,i}^\flat\in\CIc(X^\circ)$ satisfying $\la h_{b,i}^\sharp,h_{b,j}^*\ra=\delta_{i j}$ and $\la h_{b,i},h_{b,j}^\flat\ra=\delta_{i j}$; we make the ansatz
  \[
    V_b = V_{b,1} + V_{b,2}, \quad V_{b,1} = \sum_{i=1}^7 h_{b,i}^\sharp\la-,h_{b,i}^\flat\ra;
  \]
  if we choose $V_{b,2}$ such that $\cK_b\subset\ker V_{b,2}$, then~\eqref{EqRImprovedV1} holds.
  
  Next,~\eqref{EqRImprovedV} holds if $\breve\cK_b\subset\ker V_{b,2}$ and if we choose the $h_{b,j}^\flat$ to also satisfy $h_{b,j}^\flat\perp\breve\cK_b$ for $j=1,\ldots,7$. The latter can be arranged iff
  \begin{equation}
  \label{EqRImprovedVCap}
    \breve\cK_b\cap\cK_b=0.
  \end{equation}
  This holds true for $b=b_0$ by inspection of the expressions in Proposition~\ref{PropL0}. Indeed, elements of $\cK_{b_0}$ are of size $\cO(r^{-2})$ with non-zero $r^{-2}$ coefficients, which implies that the elements of $\breve\cK_b$ must have a non-vanishing $r^{-1}$ leading term; see the proof of Lemma~\ref{LemmaL0Lead}. A better (in that it does not rely on any explicit calculations) perspective on~\eqref{EqRImprovedVCap} is the following: if $0\neq h\in\cK_{b,s}$, there exists a function $F\in\CIc(X^\circ)$ such that $L_b(F h)\neq 0$; indeed, the space $\CIc(X^\circ)h$ is infinite-dimensional, while $\ker L_b\cap\Hbext^{\infty,-1/2-}$ is finite-dimensional. But
  \[
    L_b((t_*+F)h+(\breve h-F h)) = 0;
  \]
  this means that if we work with $t_*+F$ instead of $t_*$ to define $\breve h$ and the spectral family $\wh{L_b}(\sigma)$ (note that $\wh{L_b}(0)$ is unaffected by such a change), then $\breve h$ changes by $-F h\notin\ker\wh{L_b}(0)$. Therefore, we can always arrange~\eqref{EqRImprovedVCap} upon changing $t_*$ in a compact subset of $X^\circ$ (by an arbitrarily small amount). For later use, we also note that
  \begin{equation}
  \label{EqRImprovedVCap2}
    \breve\cK_b^* \cap \cK_b^* = 0;
  \end{equation}
  this can again be either checked explicitly for $b=b_0$ and the unmodified operator $L_{g_{b_0},0}$, and thus holds by continuity for nearby $b$; or it can be arranged by slightly modifying $t_*$.

  It remains to arrange~\eqref{EqRImprovedV2} and the extra condition $\cK_b\oplus\breve\cK_b\subset V_{b,2}$. Assuming the latter, we have $V_b(\cK_b)=V_{b,1}(\cK_b)=\ran V_{b,1}$, and~\eqref{EqRImprovedV2} is then equivalent to
  \[
    \wh{L_b}(0)^*\breve\cK_b^* \cap (\wh{L_b}(0)+V_{b,2})^*(\ker V_{b,1}^*) = 0.
  \]
  We arrange the stronger condition in which $\ker V_{b,1}^*\subset\Hbsupp^{-s+1,-\ell+2}$ is replaced by the full space $\Hbsupp^{-s+1,\ell+2}$; this is then equivalent to the requirement that
  \[
    \wh{L_b}(0)^*\breve\cK_b^* \to \bigl(\ker(\wh{L_b}(0)+V_{b,2})\bigr)^*
  \]
  be injective. (Note that in view of~\eqref{EqRImprovedVCap2}, the space on the left is 4-dimensional and depends continuously on $b$.) To arrange this, choose a continuous (in $b$) basis $\{\breve h_{b,1}^*,\ldots,\breve h_{b,4}^*\}$ of $\breve\cK_b^*$, and continuously select $h'_{b,1},\ldots,h'_{b,4}\in\CIc$ with the property that $H'_b:=\mathspan\{h'_{b,1},\ldots,h'_{b,4}\}$ and $\cK_b\oplus\breve\cK_b$ have trivial intersection, and such that $\la\wh{L_b}(0)^*\breve h_{b,i}^*,h'_{b,j}\ra=\delta_{i j}$. We then want to define $V_{b,2}$ so that $h'_{b,j}\in\ker(\wh{L_b}(0)+V_{b,2})$ for $j=1,\ldots,4$; this holds provided we define $V_{b,2}$ as a rank $4$ operator which assigns $h'_{b,j}\mapsto-\wh{L_b}(0)h'_{b,j}$, while on a complement of $H'_b$ depending continuously on $b$ and containing $\cK_b\oplus\breve\cK_b$, we let $V_{b,2}\equiv 0$.
\end{proof}

\begin{proof}[Proof of Theorem~\ref{ThmR}]
  We continue where the proof of Proposition~\ref{PropRExist} ended, and in particular use the notation~\eqref{EqROp}--\eqref{EqRInv2}; however, now we use $\check L_b(0)=\wh{L_b}(0)+V_b$ as well as the complement $\wt\cK_b^\perp$ of $\wt\cK_b$ defined by Lemma~\ref{LemmaRImprovedV}. Note that the arguments in the proof of Proposition~\ref{PropRExist} are unaffected by our more careful choice of $V_b$ and $\wt\cK_b$, $\wt\cK_b^\perp$.

  \pfstep{Quadratic vanishing of $L_{0 1}$ and $L_{1 0}$.} Lemma~\ref{LemmaRImprovedV} now gives
  \[
    \sigma^{-1}L_{0 1}(b,\sigma),\ \sigma^{-1}L_{1 0}(b,\sigma) = o(1),\quad \sigma\to 0.
  \]
  This can be strengthened further: for $h^*\in\cK_{b,s}^*$, we compute, using~\eqref{EqRpasigma},
  \begin{align}
    (\check L_b(\sigma)^{-1})^*V_b^*h^* &= h^* + \bigl((\check L_b(\sigma)^{-1})^*-(\check L_b(0)^{-1})^*\bigr)\check L_b(0)^*h^* \nonumber\\
  \label{EqRcheckLVstarImpr0}
      &= h^* - (\check L_b(\sigma)^{-1})^*\bigl(\sigma\pa_\sigma\wh{L_b}(0)^*+\tfrac{\sigma^2}{2}\pa_\sigma^2\wh{L_b}(0)^*\bigr)h^* \\
      &= h^* - i\sigma(\check L_b(\sigma)^{-1})^*\wh{L_b}(0)^*\breve h^* - \tfrac{\sigma^2}{2}(\check L_b(\sigma)^{-1})^*\pa_\sigma^2\wh{L_b}(0)^*h^* \nonumber\\
      &= h^* - i\sigma\breve h^* + i\sigma(\check L_b(\sigma)^{-1})^*V_b^*\breve h^* + i\sigma(\check L_b(\sigma)^{-1})^*(\wh{L_b}(\sigma)^*-\wh{L_b}(0)^*)\breve h^* \nonumber\\
      &\quad - \tfrac{\sigma^2}{2}\pa_\sigma^2\wh{L_b}(0)^*h^* \nonumber\\
  \label{EqRcheckLVstarImpr}
  \begin{split}
      &= h^*-i\sigma(I-(\check L_b(0)^{-1})^*V_b^*)\breve h^* - \sigma^2\Bigl(\half\pa_\sigma^2\wh{L_b}(0)^*h^* \\
      &\quad - i(\check L_b(\sigma)^{-1})^*(\pa_\sigma\wh{L_b}(0)^*+\tfrac{\sigma}{2}\pa_\sigma^2\wh{L_b}(0)^*)(I-(\check L_b(0)^{-1})^*V_b^*)\breve h^*\Bigr);
  \end{split}
  \end{align}
  here, we used Lemma~\ref{LemmaCD0GenDual} to justify the rewriting of the fourth term in the penultimate line when passing to the last line. Thus, for $f\in\wt\cK_b^\perp$,
  \begin{equation}
  \label{EqRcheckL10}
    \la L_{1 0}(b,\sigma)f,h^*\ra = \big\la f,\bigl(I-(\check L_b(\sigma)^{-1})^*V_b^*\bigr)h^*\big\ra
  \end{equation}
  has a Taylor expansion in $\sigma$ up to quadratic terms, with $o(|\sigma|^2)$ remainder; in fact, the calculation~\eqref{EqRcheckLVstarImpr} and Lemma~\ref{LemmaRImprovedV} show that $L_{1 0}(b,\sigma)=\cO(|\sigma|^2)$ in operator norm.
  
  Similarly, for $h\in\cK_{b,s}$, we have
  \begin{align}
    &\wh{L_b}(\sigma)\check L_b(\sigma)^{-1}\check L_b(0)h \nonumber\\
      &\quad= \check L_b(0)h - V_b\check L_b(\sigma)^{-1}\check L_b(0)h \nonumber\\
   \label{EqR10exp0}
      &\quad= V_b\check L_b(\sigma)^{-1}(\sigma\pa_\sigma\wh{L_b}(0)+\tfrac{\sigma^2}{2}\pa_\sigma^2\wh{L_b}(0))h \\
      &\quad= i\sigma V_b\check L_b(\sigma)^{-1}\wh{L_b}(0)\breve h + \tfrac{\sigma^2}{2}V_b\check L_b(\sigma)^{-1}\pa_\sigma^2\wh{L_b}(0)h \nonumber\\
      &\quad= i\sigma(V_b-V_b\check L_b(0)^{-1}V_b)\breve h \nonumber\\
      &\quad\qquad + \sigma^2 V_b\check L_b(\sigma)^{-1}\bigl(\half\pa_\sigma^2\wh{L_b}(0)h - i(\pa_\sigma\wh{L_b}(0)+\tfrac{\sigma}{2}\pa_\sigma^2\wh{L_b}(0))\breve h\bigr) \nonumber\\
  \label{EqRL10exp}
      &\quad= \sigma^2 V_b\check L_b(\sigma)^{-1}\bigl(\half\pa_\sigma^2\wh{L_b}(0)h - i(\pa_\sigma\wh{L_b}(0)+\tfrac{\sigma}{2}\pa_\sigma^2\wh{L_b}(0))\breve h\bigr),
  \end{align}
  where we used Lemma~\ref{LemmaL0Lead} to justify the penultimate equality, and Lemma~\ref{LemmaRImprovedV} for the final one. Recall now that $L_{1 1}$ and $L_{2 1}$ vanish quadratically at $\sigma=0$, and in fact have a Taylor expansion with $o(|\sigma|^2)$ error term, by~\eqref{EqR1221Quadr}. We thus conclude that
  \[
    L_{0 1}(b,\sigma)\check L_b(0)h=\wh{L_b}(\sigma)^{-1}\check L_b(\sigma)^{-1}\check L_b(0)h-(L_{1 1}(b,\sigma)+L_{2 1}(b,\sigma))\check L_b(0)h
  \]
  has a Taylor expansion modulo $o(|\sigma|^2)$, with vanishing linear (in $\sigma$) term.

  In summary, in the expression~\eqref{EqROp} for $\wh{L_b}(\sigma)\check L_b(\sigma)^{-1}$, we now have
  \begin{equation}
  \label{EqRL0110Impr}
    \wt L_{0 1}=\sigma\wt L_{0 1}',\quad
    \wt L_{1 0}=\sigma\wt L_{1 0}',
  \end{equation}
  with $\wt L_{0 1}'$ and $\wt L_{1 0}'$ continuous at $\sigma=0$ and of size $\cO(1)$ in operator norm. In view of~\eqref{EqRInv2}, this immediately implies that the entries $\sigma^{-1}\wt R_{0 1}$ and $\sigma^{-1}\wt R_{1 0}$ of the inverse $\wt R_b(\sigma)$ in~\eqref{EqRInv1} are in fact regular (bounded and continuous) at $\sigma=0$.

  For the remainder of the proof, to simplify notation, we shall denote by `$\cO^k$' operators which are $\sigma^k$ times a $\sigma$-dependent family of operators which is continuous at $\sigma=0$.

  \pfstep{Control of $\wt R_{1 1}\bmod\cO^2$.} This is the most delicate calculation; it requires calculating $\wt L_{1 1}^\natural\bmod\cO^2=\wt L_{1 1}^\sharp-\sigma\wt L_{1 2}^\sharp\wt L_{2 2}^{-1}\wt L_{2 1}^\sharp$. The factor $\wt L_{2 2}^{-1}$ is already controlled modulo $\cO^1$ by~\eqref{EqR22InvErr}; it thus remains to control
  \[
    \wt L_{1 1}^\sharp\equiv\wt L_{1 1}-\sigma^2\wt L'_{1 0}L_{0 0}^{-1}\wt L'_{0 1}\equiv \wt L_{1 1} \bmod\cO^2,
  \]
  as well as
  \[
    \wt L_{1 2}^\sharp = \wt L_{1 2}-\sigma\wt L_{1 0}'L_{0 0}^{-1}\wt L_{0 2} \equiv \wt L_{1 2}\bmod\cO^1, \quad
    \wt L_{2 1}^\sharp = \wt L_{2 1}-\sigma\wt L_{2 0}L_{0 0}^{-1}\wt L_{0 1}' \equiv \wt L_{2 1}\bmod\cO^1.
  \]

   \pfsubstep{Control of $\wt L_{1 1}$ modulo $\cO^2$.} We wish to expand $\wt L_{1 1}$ \emph{two orders} further than before, requiring a \emph{fourth order} Taylor expansion of the pairing
   \begin{equation}
   \label{EqR11deep0}
     \la \wh{L_b}(\sigma)\check L_b(\sigma)^{-1}(\check L_b(0)h),h^*\ra,\quad h\in\cK_{b,s},\ h^*\in\cK_{b,s}^*,
   \end{equation}
   see~\eqref{EqRbottomCalc}. As above, the relevant structure will be that $\pa_\sigma\wh{L_b}(0)h=i\wh{L_b}(0)\breve h$, with $\breve h$ having a $r^{-1}$ leading term by Lemma~\ref{LemmaL0Lead}; recall that such a leading term (rather than, say, $r^{-1}\log r$) is key for a rewriting~\eqref{EqRReParen} (with $\breve h$ taking the role of $\wt h$ there) to be valid. This allows the Taylor expansion to be taken one order further; using the same structure for dual states, i.e.\ $\pa_\sigma\wh{L_b}(0)^*h^*=i\wh{L_b}(0)^*\breve h^*$ plus Lemma~\ref{LemmaCD0GenDual}, gives another order.
    
  Concretely then, starting with~\eqref{EqRbottomCalc}, we wish to calculate
  \begin{equation}
  \label{EqR11deep}
    \la V_b\check L_b(\sigma)^{-1}\pa_\sigma\wh{L_b}(0)h,h^*\ra + \tfrac{\sigma}{2}\la V_b\check L_b(\sigma)^{-1}\pa_\sigma^2\wh{L_b}(0)h,h^*\ra\quad\bmod \cO^3.
  \end{equation}
  Let us begin with the second term, which, for later use, we expand further than necessary at this point. Write $f_2=\pa_\sigma^2\wh{L_b}(0)h\in\Hbext^{\infty,5/2-}$ (which gains two orders of decay relative to $h\in\Hbext^{\infty,1/2-}$ by Lemma~\ref{LemmaOpLinFT}) and $h_2=\check L_b(0)^{-1}f_2\in\rho\CI+\Hbext^{\infty,1/2-}$, then $2\sigma^{-1}$ times the second term in~\eqref{EqR11deep} is
  \begin{equation}
  \label{EqR11deepErr}
  \begin{split}
    \la V_b\check L_b(\sigma)^{-1}f_2,h^*\ra &= \la V_b\check L_b(0)^{-1}f_2,h^*\ra - \la V_b\check L_b(\sigma)^{-1}\sigma\pa_\sigma\wh{L_b}(0)h_2,h^*\ra \\
     &\qquad - \la V_b\check L_b(\sigma)^{-1}\tfrac{\sigma^2}{2}\pa_\sigma^2\wh{L_b}(0)h_2,h^*\ra.
  \end{split}
  \end{equation}
  In the last term, we can integrate by parts and expand $(\check L_b(\sigma)^{-1})^*V_b^*h^*$ up to $\cO^2$ errors using~\eqref{EqRcheckLVstarImpr}, contributing a $\sigma^2$ leading order term and a $\cO^4$ remainder to~\eqref{EqR11deepErr}. The second term of~\eqref{EqR11deepErr} times $\sigma^{-1}$ on the other hand is, modulo $\cO^2$, equal to (using~\eqref{EqRcheckLVstar})
  \begin{align*}
    &-\la\pa_\sigma\wh{L_b}(0)h_2,(\check L_b(\sigma)^{-1})^*V_b^* h^*\ra \\
    &\quad = -\la\pa_\sigma\wh{L_b}(0)h_2,h^*\ra + \la\pa_\sigma\wh{L_b}(0)h_2,(\check L_b(\sigma)^{-1})^*\sigma\pa_\sigma\wh{L_b}(0)^*h^*\ra,
  \end{align*}
  with the second term equal to
  \begin{align*}
    &-i\sigma\la\pa_\sigma\wh{L_b}(0)h_2,(\check L_b(\sigma)^{-1})^*\wh{L_b}(0)^*\breve h^*\ra \\
    &\qquad = i\sigma\la\pa_\sigma\wh{L_b}(0)h_2,(\check L_b(\sigma)^{-1})^*V_b^*\breve h^*\ra - i\sigma\la\pa_\sigma\wh{L_b}(0)h_2,\breve h^*\ra \\
    &\qquad\qquad - i\sigma\la\pa_\sigma\wh{L_b}(0)h_2,((\check L_b(\sigma)^{-1})^*-(\check L_b(0)^{-1})^*)\check L_b(0)^*\breve h^*\ra.
  \end{align*}
  Using Lemma~\ref{LemmaCD0GenDual}, the last summand can be re-parenthesized to
  \[
    i\sigma\la\pa_\sigma\wh{L_b}(0)h_2,(\check L_b(\sigma)^{-1})^*(\wh{L_b}(\sigma)-\wh{L_b}(0))\breve h^*\ra = \cO^2.
  \]
  Altogether, we have shown that~\eqref{EqR11deepErr} has a Taylor expansion at $\sigma=0$ up to a $\cO^3$ error. Thus, the second term in~\eqref{EqR11deep} is controlled modulo $\cO^4$.

  The analysis of the first term of~\eqref{EqR11deep} is similar: it equals
  \[
    i\la V_b\check L_b(\sigma)^{-1}\wh{L_b}(0)\breve h,h^*\ra = -i\la V_b\check L_b(\sigma)^{-1}(V_b\breve h),h^*\ra + i\la V_b\check L_b(\sigma)^{-1}\check L_b(0)\breve h,h^*\ra;
  \]
  the first term has an expansion modulo $\cO^3$ just like~\eqref{EqR11deepErr}, while the second one has a leading term $i\la V_b\breve h,h^*\ra$ plus an error (again using Lemma~\ref{LemmaL0Lead})
  \[
    i\la V_b\check L_b(\sigma)^{-1}(\sigma\pa_\sigma\wh{L_b}(0)+\tfrac{\sigma^2}{2}\pa_\sigma^2\wh{L_b}(0))\breve h,h^*\ra = i\la(\sigma\pa_\sigma\wh{L_b}(0)+\tfrac{\sigma^2}{2}\pa_\sigma^2\wh{L_b}(0))\breve h,(\check L_b(\sigma)^{-1})^*V_b^*h^*\ra.
  \]
  Using~\eqref{EqRcheckLVstarImpr}, this has an expansion modulo $\cO^3$.

  \pfsubstep{Control of $\wt L_{1 2}$ and $\wt L_{2 1}$ modulo $\cO^1$.} For $\wt L_{1 2}$, we need to compute, modulo $\cO^3$, the pairing $\la\wh{L_b}(\sigma)\check L_b(\sigma)^{-1}\check L_b(0)h,h^*\ra$ for $h\in\cK_{b,s}$ and $h^*\in\cK_{b,v}^*$. Using~\eqref{EqRL10exp}, we need
  \[
    \half\sigma^2\la V_b\check L_b(\sigma)^{-1}\pa_\sigma^2\wh{L_b}(0)h,h^*\ra - i\sigma^2\la V_b\check L_b(\sigma)^{-1}\pa_\sigma\wh{L_b}(0)\breve h,h^*\ra \bmod \cO^3.
  \]
  Integrating by parts in each term, and using the equality~\eqref{EqRcheckLVstarImpr0} (which is valid also for $h^*\in\cK_{b,v}^*$) gives a $\sigma^2$ leading term plus a $\cO^3$ remainder. The argument for $\wt L_{2 1}$ is analogous, now using the full strength of~\eqref{EqRL10exp} but only the expansion~\eqref{EqR10exp0} for $h\in\cK_{b,v}$.

  \pfstep{Control of $\wt R_{1 2}$, $\wt R_{2 1}$, $\wt R_{2 2}$ modulo $\cO^1$.} We now use the explicit formulas~\eqref{EqRInv2}. The expansion for $\wt R_{2 2}$ follows from that of $\wt L_{2 2}^\flat\equiv\wt L_{2 2}\bmod\cO^1$, which was already proved in~\eqref{EqR22InvErr} (where an inspection of the argument shows that the error term is indeed $\sigma$ times a family of operators which is continuous at $\sigma=0$). For $\wt R_{1 2}$ and $\wt R_{2 1}$ on the other hand, we use the expansion modulo $\cO^1$ of $\wt L_{2 2}^\flat$ and $\wt L_{1 2}$, resp.\ $\wt L_{2 1}$, together with that of $\wt R_{1 1}$.

  We have thus established
  \begin{equation}
  \label{EqwtRb}
    \wt R_b(\sigma) = \bigl(\wh{L_b}(\sigma)\check L_b(\sigma)^{-1}\bigr)^{-1}
    =\begin{pmatrix}
       \wt R_{0 0} & \wt R_{0 1}' & \wt R_{0 2} \\
       \wt R_{1 0}' & \sigma^{-2}\wt P_{1 1} + \wt R_{1 1}' & \sigma^{-1}\wt P_{1 2} + \wt R_{1 2}' \\
       \wt R_{2 0} & \sigma^{-1}\wt P_{2 1} + \wt R_{2 1}' & \sigma^{-1}\wt P_{2 2} + \wt R_{2 2}',
     \end{pmatrix}
  \end{equation}
  where the $\wt R_{i j}=\wt R_{i j}(b,\sigma)$ and $\wt R_{i j}'=\wt R_{i j}'(b,\sigma)$ have uniformly bounded operator norm for $(b,\sigma)$ near $(b_0,0)$ (and are in fact continuous at $\sigma=0$), while $\wt P_{1 1}=\wt P_{1 1}(b,\sigma)$ is linear in $\sigma$, and for $(i,j)=(1,2),(2,1),(2,2)$, $\wt P_{i j}=\wt P_{i j}(b)$ is $\sigma$-independent.

  \pfstep{Expansion with $\cO^0$ errors for $\wh{L_b}(\sigma)^{-1}$.} We now compute $\wh{L_b}(\sigma)^{-1} = \check L_b(\sigma)^{-1}\wt R_b(\sigma)$. To get the desired expansion, we merely need to show that $\check L_b(\sigma)^{-1}|_{\wt\cK_{b,s}}$, resp.\ $\check L_b(\sigma)^{-1}|_{\wt\cK_{b,v}}$ has a Taylor expansion modulo $\cO^2$ remainder, resp.\ $\cO^1$. But for $h\in\cK_b$,
  \begin{equation}
  \label{EqRCheckLSing1}
    \check L_b(\sigma)^{-1}\check L_b(0)h = h - \sigma\check L_b(\sigma)^{-1}\pa_\sigma\wh{L_b}(0)h + \cO^2 = h + \cO^1.
  \end{equation}
  For $h\in\cK_{b,s}$, we can expand further: using Lemma~\ref{LemmaRImprovedV}, we have
  \begin{equation}
  \label{EqCheckLSing2}
    \check L_b(\sigma)^{-1}\check L_b(0)h = h - i\sigma\breve h + i\sigma\check L_b(\sigma)^{-1}\bigl(\wh{L_b}(\sigma)-\wh{L_b}(0)\bigr)\breve h = h - i\sigma\breve h + \cO^2.
  \end{equation}

  \pfstep{Explicit form of the singular part of $\wh{L_b}(\sigma)^{-1}$.} The calculations~\eqref{EqRCheckLSing1}--\eqref{EqCheckLSing2} imply that the range of the singular coefficients of $\wh{L_b}(\sigma)^{-1}$ is contained in $\cK_b+\breve\cK_b$. We proceed to determine their full structure. This can be done by keeping track of the terms in the above Taylor expansions, or using a simple matching argument which we proceed to explain. Thus, let $f\in\Hbext^{s-1,\ell+2}$, and consider
  \[
    h(\sigma) = \wh{L_b}(\sigma)^{-1}f =: -\sigma^{-2}h_{-2} + i\sigma^{-1}h_{-1} + h_0 + \wt h(\sigma),
  \]
  where $\wt h(\sigma)=o(1)$ in $\Hbext^{s-\eps,\ell-\eps}$, $\eps>0$, as $\sigma\to 0$. Then
  \begin{equation}
  \label{EqREqn}
  \begin{split}
    f = \wh{L_b}(\sigma)h(\sigma) &= -\sigma^{-2}\wh{L_b}(0)h_{-2} \\
      &\quad + \sigma^{-1}\bigl(-\pa_\sigma\wh{L_b}(0)h_{-2} + i\wh{L_b}(0)h_{-1}\bigr) \\
      &\quad + \bigl(-\half\pa_\sigma^2\wh{L_b}(0)h_{-2} + i\pa_\sigma\wh{L_b}(0)h_{-1} + \wh{L_b}(0)h_0\bigr) \\
      &\quad + \wh{L_b}(0)\wt h(\sigma) + \sigma\bigl(\tfrac{i}{2}\pa_\sigma^2\wh{L_b}(0)h_{-1} + \pa_\sigma\wh{L_b}(0)h_0\bigr) \\
      &\quad + \bigl(\wh{L_b}(\sigma)-\wh{L_b}(0)\bigr)\wt h(\sigma) + \tfrac{\sigma^2}{2}\pa_\sigma^2\wh{L_b}(0)h_0.
  \end{split}
  \end{equation}
  Therefore, $h_{-2}\in\cK_b$ and $\pa_\sigma\wh{L_b}(0)h_{-2}=i\wh{L_b}(0)h_{-1}$; this implies
  \[
    h_{-2}=h_s\in\cK_{b,s}, \quad
    h_{-1}=\breve h_s+h'_s+h_v,
  \]
  for some $h'_s\in\cK_{b,s}$, $h_v\in\cK_{b,v}$, which proceed to determine.
  
  Consider the equality of constant coefficients,
  \begin{equation}
  \label{EqRh0}
  \begin{split}
    \wh{L_b}(0)h_0 &= f + \half\pa_\sigma^2\wh{L_b}(0)h_{-2} - i\pa_\sigma\wh{L_b}(0)h_{-1} \\
      &=f + \wh{L_b}(0)\breve h'_s + \half\pa_\sigma^2\wh{L_b}(0)h_s - i\pa_\sigma\wh{L_b}(0)\breve h_s - i\pa_\sigma\wh{L_b}(0)h_v,
  \end{split}
  \end{equation}
  where $\breve h'_s$ is defined relative to $h'_s$ as in Definition~\ref{DefRBreve}. Pairing this with $h^*\in\cK_b^*$ (which annihilates $\ran\wh{L_b}(0)$), we get
  \begin{equation}
  \label{EqRPair}
    \la f,h^*\ra = \big\la \half\bigl([[L_b,t_*],t_*]h_s+2[L_b,t_*]\breve h_s\bigr) + [L_b,t_*]h_v, h^*\big\ra = k_b((h_s,h_v),h^*)
  \end{equation}
  using the pairing $k_b\colon\cK_b\times\cK_b^*\to\C$ from Definition~\ref{DefRBreve}. Since this is non-degenerate, we can uniquely determine $h_s$ (thus $\breve h_s$) and $h_v$ from~\eqref{EqRPair}.

  Consider again~\eqref{EqRh0}; on the right hand side, the only term not yet determined by $f$ is $\breve h_s'$ in the second term. Let us drop this term and consider the PDE
  \begin{equation}
  \label{EqRPair2}
    \wh{L_b}(0)\bar h_0=f+\half\pa_\sigma^2\wh{L_b}(0)h_s-i\pa_\sigma\wh{L_b}(0)\breve h_s-i\pa_\sigma\wh{L_b}(0)h_v;
  \end{equation}
  it can be solved for $\bar h_0\in\Hbext^{s,\ell}$ because of~\eqref{EqRPair}; since the right hand side is uniquely determined by $f$, so is $\bar h_0\bmod\cK_b$. By~\eqref{EqRh0}, we have
  \[
    h'':=\bar h_0-(h_0-\breve h_s')\in\cK_b.
  \]
  Considering then the fourth line of~\eqref{EqREqn} and pairing with $h_s^*\in\cK_{b,s}^*$, the term involving $\wt h(\sigma)$ does not contribute; we obtain, upon multiplication by $-i$ and writing $h_0=\bar h_0+\breve h_s'-h''$,
  \begin{align*}
    0 &= -i\big\la\tfrac{i}{2}\pa_\sigma^2\wh{L_b}(0)h_{-1} + \pa_\sigma\wh{L_b}(0)h_0, h_s^*\big\ra \\
      &= \big\la \tfrac{1}{2}\pa_\sigma^2\wh{L_b}(0)(\breve h_s+h'_s+h_v) - i\pa_\sigma\wh{L_b}(0)\breve h'_s - i\pa_\sigma\wh{L_b}(0)\bar h_0 + i\pa_\sigma\wh{L_b}(0)h'', h_s^* \big\ra.
  \end{align*}
  Integration by parts of the fourth term, using $\pa_\sigma\wh{L_b}(0)^*h_s^*=i\wh{L_b}(0)^*\breve h_s^*$, and integrating by parts again, one gets zero since $\wh{L_b}(0)h''=0$. Therefore,
  \[
    \big\la\half\bigl([[L_b,t_*],t_*]h'_s+2[L_b,t_*]\breve h'_s\bigr),h_s^*\big\ra = -\big\la\half[[L_b,t_*],t_*](\breve h_s+h_v)+[L_b,t_*]\bar h_0,h_s^*\big\ra,
  \]
  which uniquely determines $h'_s$. This completes the proof of Theorem~\ref{ThmR}.
\end{proof}

%%%%%%%%%%%%%%%%%%%%%%%%%%%%%%%%%%%%%%%%%%%%%%%%%%%%%%%%%%%%%%%%%%%%%%
\section{Regularity of the resolvent in the spectral parameter}
\label{SReg}

\emph{We continue using the notation from~\S\ref{SR}, so $L_b=L_{g_b,E}$ is given by~\eqref{EqRLb}, with $\gamma$ fixed as in~\eqref{EqRgamma}.} We now study in detail the regular part $L_b^-(\sigma)$ of the resolvent $\wh{L_b}(\sigma)$ defined in Theorem~\ref{ThmR}. In~\S\ref{SsRegLow}, we prove that its derivatives of order one, resp.\ two are bounded by small inverse powers of $|\sigma|$ when acting between slightly relaxed function spaces; see Theorem~\ref{ThmRegL} and Corollary~\ref{CorRegL}. Away from $\sigma=0$, the regular part is smooth, as we show in~\S\ref{SsRegNon0}, together with quantitative high energy estimates for its derivatives. In~\S\ref{SsRegCon} finally, we show that $L_b^-(\sigma)$ is conormal at $\sigma=0$, i.e.\ satisfies the same bounds as $L_b^-(\sigma)$ (and its up to second derivatives) after any number of applications of $\sigma\pa_\sigma$.

%%%%%%%%%%%%%%%%%%%%%%%%%%%%%%%%%%%%%%%%%%%%%%%%%%
\subsection{Regularity at low frequencies}
\label{SsRegLow}

By Lemma~\ref{LemmaOpLinFT}, and using that our constraint damping is compactly supported, we have (omitting the bundle $S^2\,\wt{\Tsc^*}X$ from the notation)
\begin{equation}
\label{EqLStruct}
\begin{split}
  \wh{L_b}(\sigma) &= 2\sigma\rho(\rho D_\rho+i) + \wh{L_b}(0) + \cR_b(\sigma), \\
  &\quad \cR_b(\sigma)\in \sigma\rho^3\Diffb^1(X) + \sigma\rho^2\CI(X) + \sigma^2\rho^2\CI(X), \quad
  \wh{L_b}(0)\in\rho^2\Diffb^2(X).
\end{split}
\end{equation}
The operator $\check L_b(\sigma)=\wh{L_b}(\sigma)+V_b$ differs from this by the $\sigma$-independent smoothing operator $V_b$ with Schwartz kernel of compact support in $X^\circ\times X^\circ$, see Lemma~\ref{LemmaRImprovedV}

\begin{thm}
\label{ThmRegL}
  Let $\ell\in(-\tfrac32,-\half)$, $\eps\in(0,1)$, $\ell+\eps\in(-\half,\half)$, and $s-\eps>\tfrac72$.
  \begin{enumerate}
  \item For $\sigma\neq 0$, $\Im\sigma\geq 0$, the operator $\pa_{\sigma}L^-_b(\sigma)$ maps $\Hbext^{s-1,\ell+2}\to \Hbext^{s-\eps,\ell+\eps-1}$. For $0<|\sigma|\leq\sigma_0$, $\Im\sigma\geq 0$, and $\sigma_0$ sufficiently small, we have a uniform estimate
    \begin{equation}
    \label{EqFDP}
      \|\pa_{\sigma}L^-_b(\sigma)\|_{\Hbext^{s-1,\ell+2}\to \Hbext^{s-\eps,\ell+\eps-1}}\lesssim |\sigma|^{-\eps}.
    \end{equation}
  \item For $\sigma\neq 0$, $\Im\sigma\geq 0$, the operator $\pa_{\sigma}^2 L^-_b(\sigma)$ maps $\Hbext^{s-1,\ell+2}\to \Hbext^{s-1-\eps,\ell+\eps-1}$. For $0<|\sigma|\leq\sigma_0$, $\Im\sigma\geq 0$, and $\sigma_0$ sufficiently small, we have a uniform estimate
    \begin{equation}
    \label{EqSDP}
      \|\pa_{\sigma}^2 L^-_b(\sigma)\|_{\Hbext^{s-1,\ell+2}\to \Hbext^{s-1-\eps,\ell+\eps-1}}\lesssim |\sigma|^{-1-\eps}.
    \end{equation}
  \end{enumerate}
\end{thm}

\begin{cor}
\label{CorRegL}
  Let $\ell,\eps,s$ be as in Theorem~\ref{ThmRegL}. Then
  \begin{equation}
  \label{EqInterpolRegL}
    L^-_b\in H^{3/2-\eps}\bigl((-\sigma_0,\sigma_0);\cL(\Hbext^{s-1,\ell+2}, \Hbext^{s-\max(\eps,1/2),\ell+\eps-1})\bigr).
  \end{equation}
\end{cor}

The proof of the corollary uses an interpolation argument, a translation of~\eqref{EqFDP}--\eqref{EqSDP} into memberships of $L_b^-$ in weighted b-Sobolev spaces on the half-lines $[0,\infty)$ and $(-\infty,0]$, and the relationship of these spaces with standard Sobolev spaces on the real line. Let us denote by $\Hb^k([0,\sigma_0))$ the completion of $\CIc((0,\sigma_0])$ with respect to the norm
\[
  \|u\|_{\Hb^k}^2:=\sum_{j=0}^k \|(\sigma\pa_\sigma)^j u\|_{L^2(|d\sigma|)}^2.
\]
Hardy's inequality gives $\sigma^k\Hb^k([0,\sigma_0))=\dot H^k([0,\sigma_0))$, the latter space consisting of all restrictions to $[0,\sigma_0)$ of elements of $H^k(\R)$ with support in $[0,\infty)$. Interpolation thus gives
\begin{equation}
\label{EqRCInter}
  \sigma^\alpha\Hb^{\alpha'}([0,\sigma_0)) \subset \dot H^{\min(\alpha,\alpha')}([0,\sigma_0)),\quad \alpha,\alpha'\geq 0.
\end{equation}
  
\begin{proof}[Proof of Corollary~\ref{CorRegL}, assuming Theorem~\ref{ThmRegL}]
  Integration of the estimate~\eqref{EqFDP} for the operator $\pa_\sigma L_b^-(\sigma)$ from $\sigma=\sigma_0$ to $0$ implies
  \begin{subequations}
  \begin{equation}
  \label{EqRCInter1}
    L_b^-(\sigma)-L_b^-(0) \in |\sigma|^{3/2-\eps}\Hb^1\bigl([0,\sigma_0);\cL(\Hbext^{s-1,\ell+2},\Hbext^{s-\eps,\ell-(1-\eps)})\bigr).
  \end{equation}
  Similarly, the estimate~\eqref{EqSDP} (together with~\eqref{EqFDP}) implies
  \[
    (\sigma\pa_\sigma)^2 L_b^-(\sigma) \in |\sigma|^{3/2-\eps}\Hb^0\bigl([0,\sigma_0);\cL(\Hbext^{s-1,\ell+2},\Hbext^{s-1-\eps,\ell-(1-\eps)})\bigr).
  \]
  Since $\sigma\pa_\sigma L_b^-(\sigma)$ vanishes at $\sigma=0$, this implies $\sigma\pa_\sigma L_b^-(\sigma)\in|\sigma|^{3/2-\eps}\Hb^1$, and therefore
  \begin{equation}
  \label{EqRCInter2}
    L_b^-(\sigma) - L_b^-(0) \in |\sigma|^{3/2-\eps}\Hb^2\bigl([0,\sigma_0);\cL(\Hbext^{s-1,\ell+2},\Hbext^{s-1-\eps,\ell-(1-\eps)})\bigr).
  \end{equation}
  \end{subequations}

  The advantage of working with weighted b-Sobolev spaces in the spectral parameter $\sigma$ is that interpolation is immediately applicable. Thus, if $\eps\in(0,\half]$, we interpolate between~\eqref{EqRCInter1} and \eqref{EqRCInter2}, giving
  \[
    L_b^-(\sigma)-L_b^-(0) \in |\sigma|^{3/2-\eps}\Hb^{1+\theta}\bigl([0,\sigma_0);\cL(\Hbext^{s-1,\ell+2},\Hbext^{s-\theta-\eps,\ell-(1-\eps)})\bigr);
  \]
  in the context of~\eqref{EqRCInter}, this is optimal for $\theta=\half-\eps$, giving $\sigma$-regularity $\tfrac32-\eps$, and b-regularity $s-\half$ in the range. We now recall the following general fact: given two functions $u_\pm\in H^s(\R_\pm)$, $s\geq 0$, $s\notin\half+\N_0$, the function $u(x)$ given by $u_\pm(x)$ for $\pm x>0$ lies in $H^s(\R)$ provided that $\pa_x^j u_+(0)=\pa_x^j u_-(0)$ for all $j\in\N_0$, $j<s-\half$. (This holds for $s\in[0,\half)$ by \cite[\S4, Proposition~5.3]{TaylorPDE1}, and follows inductively for larger $s$.) Since $L_b^-(\sigma)\to L_b^-(0)$ as $\sigma\to\pm 0$, this implies~\eqref{EqInterpolRegL}.

  If $\eps\in[\half,1)$, the membership~\eqref{EqRCInter2} does not give any information over~\eqref{EqRCInter1} as far as the final $\sigma$-regularity obtained from~\eqref{EqRCInter} is concerned. Thus, in this case we simply obtain~\eqref{EqInterpolRegL} by applying~\eqref{EqRCInter} to~\eqref{EqRCInter1}.
\end{proof}

The proof of Theorem~\ref{ThmRegL} will occupy the rest of this section. To simplify the bookkeeping of powers of $|\sigma|$ below, we introduce the following notation:
\begin{definition}
\label{DefBdd}
  Let $X,Y$ denote two normed spaces. Suppose $A(\sigma)\colon X\to Y$ is a family of operators which depends on a parameter $0\neq\sigma\in\C$. Let $\alpha,\beta\in\R$. Then we write
  \[
    A(\sigma) \colon |\sigma|^\alpha X \to |\sigma|^\beta Y
  \]
  if and only if there exists a constant $C>0$ such that $\|A(\sigma)\|_{X\to Y}\leq C|\sigma|^{\beta-\alpha}$ for all $\sigma$.
\end{definition}

We start by studying the reference operator $\check L_b(\sigma)$, which is invertible near $\sigma=0$ with uniformly bounded operator norm by Lemma~\ref{LemmaRPert0}.

\begin{prop}
\label{PropRegCheckL}
  Let $\ell\in(-\tfrac32,-\half)$, $\eps\in(0,1)$, $\ell+\eps\in(-\half,\half)$, $s>\tfrac72$. For $0<|\sigma|\leq\sigma_0$, $\Im\sigma\geq 0$, with $\sigma_0>0$ small, we have
  \begin{subequations}
  \begin{align}
  \label{EqRegCheckL1}
    \pa_\sigma\check L_b(\sigma)^{-1} &\colon \Hbext^{s-1,\ell+2}\to |\sigma|^{-\eps}\Hbext^{s-\eps,\ell+\eps-1}, \\
  \label{EqRegCheckL2}
    \pa_\sigma^2\check L_b(\sigma)^{-1} &\colon \Hbext^{s-1,\ell+2}\to |\sigma|^{-1-\eps}\Hbext^{s-1-\eps,\ell+\eps-1}.
  \end{align}
  \end{subequations}
\end{prop}

Formally, we have $\pa_\sigma\check L_b(\sigma)^{-1}=-\check L_b(\sigma)^{-1}(\pa_\sigma\check L_b(\sigma))\check L_b(\sigma)^{-1}$, where
\begin{equation}
\label{EqRefDeriv}
  \pa_\sigma\check L_b(\sigma)\in 2\rho(\rho D_{\rho}+i)+\rho^3\Diffb^1+\rho^2\CI+\sigma\rho^2\CI.
\end{equation}
We thus in particular need to study the composition of $\check L_b(\sigma)^{-1}$ with the b-normal operator $2\rho(\rho D_\rho+i)$. Since the $r^{-1}$ leading order terms produced by $\check L_b(\sigma)^{-1}$ are annihilated by $\rho D_\rho+i$, we have an improvement in the decay rate produced by the composition:

\begin{lemma}
\label{LemmaRefDiff1}
  Let $\ell\in(-\tfrac32,-\half)$, $0\leq\eps\leq 1$, $s>\tfrac52$. Then $|\sigma|\leq\sigma_0$, $\sigma_0>0$ small,
  \begin{subequations}
  \begin{alignat}{3}
  \label{EqRefDiff1a}
    \rho(\rho D_\rho+i)\check L_b(\sigma)^{-1} &\colon \Hbext^{s-1,\ell+2} &\ \to\ & |\sigma|^{-\eps}\Hbext^{s-1-\eps,\ell+1+\eps}, \\
  \label{EqRefDiff1b}
    \check L_b(\sigma)^{-1}\rho(\rho D_\rho+i) & \colon \Hbext^{s+\eps,\ell+1-\eps} &\ \to\ & |\sigma|^{-\eps}\Hbext^{s,\ell}.
  \end{alignat}
  \end{subequations}
\end{lemma}
\begin{proof}
  We have $\rho(\rho D_{\rho}+i)\check L_b(\sigma)^{-1}\colon \Hbext^{s-1,\ell+2}\to \Hbext^{s-1,\ell+1}$. On the other hand, we can write
  \[
    2\sigma\rho(\rho D_\rho+i)=\check L_b(\sigma)-\check L_b(0)+\cR,\quad
    \cR\in \sigma\rho^3\Diffb^1+\sigma\rho^2\CI+\sigma^2\rho^2\CI;
  \]
  since $\check L_b(0)$, $\cR\colon\Hbext^{s,\ell}\to\Hbext^{s-2,\ell+2}$, this gives
  \[
    \rho(\rho D_\rho+i)\check L_b(\sigma)^{-1} = \sigma^{-1}\bigl(I+(-\check L_b(0)+\cR)\check L_b(\sigma)^{-1}\bigr) \colon \Hbext^{s-1,\ell+2} \to |\sigma|^{-1}\Hbext^{s-2,\ell+2}.
  \]
  Interpolation gives~\eqref{EqRefDiff1a}. To get~\eqref{EqRefDiff1b}, one interpolates between $\check L_b(\sigma)^{-1}\rho(\rho D_\rho+i) \colon \Hbext^{s,\ell+1} \to \Hbext^{s,\ell}$ and $\sigma^{-1}\bigl(I+\check L_b(\sigma)^{-1}(-\check L_b(0)+\cR)\bigr) \colon \Hbext^{s+1,\ell} \to |\sigma|^{-1}\Hbext^{s,\ell}$.
\end{proof}

\begin{proof}[Proof of Proposition~\ref{PropRegCheckL}]
  Let $'$ denote a derivative with respect to $\sigma$, so
  \[
    (\check L_b(\sigma)^{-1})'=-\check L_b(\sigma)^{-1}\check L_b'(\sigma)\check L_b(\sigma)^{-1}.
  \]
  We first consider the term $\check L_b(\sigma)^{-1}\rho(\rho D_{\rho}+i)\check L_b(\sigma)^{-1}$. The hypotheses of the Proposition allow us to apply $\check L_b(\sigma)^{-1}$ on the left in~\eqref{EqRefDiff1a}, giving
  \[
    \check L_b(\sigma)^{-1} \colon |\sigma|^{-\eps}\Hbext^{s-1-\eps,\ell+\eps+1}\to |\sigma|^{-\eps}\Hbext^{s-\eps,\ell+\eps-1},
  \]
  with uniformly bounded operator norm. The error terms in~\eqref{EqRefDeriv} lie in $\rho^2\Diffb^1$, and
  \begin{equation}
  \label{EqRegCheckLErr}
    \Hbext^{s-1,\ell+2} \xra{\check L_b(\sigma)^{-1}} \Hbext^{s,\ell} \xra{\rho^2\Diffb^1} \Hbext^{s-1,\ell+2} \xra{\check L_b(\sigma)^{-1}} \Hbext^{s,\ell}
  \end{equation}
  is uniformly bounded. This proves~\eqref{EqRegCheckL1}. We record for further use below that
  \begin{equation}
  \label{L'L}
    \check L_b'(\sigma)\check L_b(\sigma)^{-1} \colon \Hbext^{s-1,\ell+2}\to |\sigma|^{-\eps}\Hbext^{s-1-\eps,\ell+\eps+1}.
  \end{equation}
  
  We next compute the second derivative:
  \begin{align*}
    (\check L_b(\sigma)^{-1})''&=2\check L_b(\sigma)^{-1}\check L_b'(\sigma)\check L_b(\sigma)^{-1}\check L_b'(\sigma)\check L_b(\sigma)^{-1}-\check L_b(\sigma)^{-1}\check L_b''(\sigma)\check L_b(\sigma)^{-1}\\
    &=:I_1+I_2. 
  \end{align*}
  Since $\check L_b''(\sigma)\in\rho^2\CI$, the term $I_2$ is bounded just like~\eqref{EqRegCheckLErr}. In order to study $L_1$, we put $\delta:=\frac{1+\eps}{2}\in(\half,1)$. Using~\eqref{L'L}, we find, using $s-\delta>\tfrac52$ and $\ell-1+\delta\in(-\tfrac32,-\half)$,
  \begin{align*}
    \Hbext^{s-1,\ell+2} &\xra{\check L_b'(\sigma)\check L_b(\sigma)^{-1}} |\sigma|^{-\delta}\Hbext^{s-1-\delta,\ell+\delta+1} \\
      &\xra{\check L_b'(\sigma)\check L_b(\sigma)^{-1}} |\sigma|^{-2\delta}\Hbext^{s-1-2\delta,\ell+2\delta} = |\sigma|^{-\eps}\Hbext^{s-1-\eps,\ell+\eps} \\
      &\xra{\check L_b(\sigma)^{-1}} |\sigma|^{-\eps}\Hbext^{s-1-\eps, \ell+\eps-1},
  \end{align*}
  proving~\eqref{EqRegCheckL2}.
\end{proof}

The proof of Theorem~\ref{ThmRegL} requires a precise analysis of the regular part of the inverse $\wt R_b(\sigma)$ defined in equations~\eqref{EqRInv} and \eqref{EqwtRb}. We begin by studying the entries of the normalized spectral family
\[
  \wh{L_b}(\sigma)\check L_b(\sigma)^{-1}
     =\begin{pmatrix}
        L_{0 0} & \sigma^2\wt L'_{0 1} & \sigma\wt L_{0 2} \\
        \sigma^2\wt L'_{1 0} & \sigma^2\wt L_{1 1} & \sigma^2\wt L_{1 2} \\
        \sigma\wt L_{2 0} & \sigma^2\wt L_{2 1} & \sigma\wt L_{2 2}
      \end{pmatrix},
\]
see~\eqref{EqROp} and \eqref{EqRL0110Impr}. For simpler bookkeeping, we define, in the notation of~\eqref{EqRSplit},
\begin{equation}
\label{EqcKcRSplit}
\begin{aligned}
  \wt\cK_0&=\wt\cK_b^\perp, &\quad \wt\cK_1&=\wt\cK_{b,s}, &\quad \wt\cK_2&=\wt\cK_{b,v}, \\
  \cR_0&=\ran\Pi_b, &\quad \cR_1&=\cR_s^\perp, &\quad \cR_2&=\cR_v^\perp.
\end{aligned}
\end{equation}

\begin{definition}
\label{DefEpsReg}
  \begin{enumerate}
  \item We say that an operator $L(\sigma)$ is \emph{$\eps$-regular at zero} if for $\sigma\in\C$, $\Im\sigma\geq 0$, $|\sigma|$ small, it satisfies uniform estimates
    \begin{subequations}
    \begin{align}
    \label{EqEpsReg0}
      \|L(\sigma)\|_{\wt\cK_j\to  \cR_i} &\lesssim 1 ,\\
    \label{EqEpsReg1}
      \|\pa_{\sigma} L(\sigma)\|_{\wt\cK_j\to  \cR_i} &\lesssim |\sigma|^{-\eps},\\
    \label{EqEpsReg2}
      \|\pa_{\sigma}^2 L(\sigma)\|_{\wt\cK_j\to  \cR_i} &\lesssim |\sigma|^{-1-\eps}.
    \end{align}
    \end{subequations}
  \item We say that $\wt L_{i j}$ has an $\eps$-regular expansion at zero in $\sigma$ up to order one if
    \begin{equation}
    \label{12.3.0.0}
      \wt L_{i j}=\wt L_{i j}^0+\sigma \wt L_{i j}^e(\sigma)
    \end{equation}
    where $\wt L_{i j}^0\colon\wt\cK_j\to\cR_i$ is $\sigma$-independent and $\wt L_{i j}^e(\sigma)$ is $\eps$-regular.
  \item We say that $\wt L_{i j}$ has an $\eps$-regular expansion at zero in $\sigma$ up to order two if
    \begin{equation}
    \label{12.3.0.1}
      \wt L_{i j}=\wt L_{i j}^0+\sigma \wt L_{i j}^1+\sigma^{2}\wt L^e_{i j}(\sigma), 
    \end{equation}
    where $\wt L_{i j}^0$, $\wt L_{i j}^1\colon\wt\cK_j\to\cR_i$ are $\sigma$-independent and $\wt L_{i j}^e(\sigma)$ is $\eps$-regular.
  \end{enumerate}
\end{definition}

\begin{prop}
\label{PropEpsReg}
  The entries of $\wh{L_b}(\sigma)\check L_b(\sigma)^{-1}$ have the following regularity:
  \begin{enumerate}
  \item $L_{0 0}$, $\wt L'_{0 1}$, $\wt L'_{1 0}$, $\wt L_{0 2}$, $\wt L_{2 0}$ are $\eps$-regular;
  \item $\wt L_{1 2}$, $\wt L_{2 1}$, $\wt L_{2 2}$ (and $\wt L_{0 1}$, $\wt L_{1 0}$) have an $\eps$-regular expansion up to order one;
  \item $\wt L_{1 1}$ has an $\eps$-regular expansion up to order two.
  \end{enumerate}
\end{prop}
\begin{proof}
  We make fully explicit some of the calculations already present in the proof of Theorem~\ref{ThmR}; there, we showed that various entries have a Taylor expansion in $\sigma$ to a certain order by using the resolvent identity multiple times, the point being that all terms arising in this manner are either polynomials in $\sigma$, or involve $\check L_b(\sigma)^{-1}$ acting on an element of $\Hbext^{s',3/2-}$ with $s'>\tfrac32$, which gives an $\eps$-regular expression by Proposition~\ref{PropRegCheckL}. We demonstrate this in detail for a number of entries.

%%%%%%%%%%%%%%%%%%%%%%%%%%%%%% 
  \pfstep{Analysis of $\wt L'_{1 0}$.} Using \eqref{EqRcheckLVstarImpr}--\eqref{EqRcheckL10} and~\eqref{EqRImprovedVAdj}, we have for $f\in\wt\cK_b^\perp$, $h^*\in\cK^*_{b,s}$:
  \begin{align*}
    \la\wt L'_{1 0}f,h^*\ra &= \la f,\half\pa_{\sigma}^2\wh{L_b}(0)^*h^*\ra \\
      &\qquad - i\Big\la \check L_b(\sigma)^{-1}f,\bigl(\pa_\sigma\wh{L_b}(0)^*+\tfrac{\sigma}{2}\pa_\sigma^2\wh{L_b}(0)^*\bigr)\bigl(I-(\check L_b(0)^{-1})^*V_b^*\bigr)\breve h^*\Big\ra.
  \end{align*}
  Recall that $(\check L_b(0)^*)^{-1}\colon\CIc(X^\circ)\to\rho\CI+\Hbsupp^{-3/2-\nu-,1/2-}$, where the first term is supported in $r\geq 4\bhm_0$, and where $\nu=\nu(|b-b_0|,\gamma)$ is a small constant, continuous in $(b,\gamma)$ with $\nu(0,0)=0$. (The output is singular only at the conormal bundle of the event horizon, and the presence of $\nu$ permits small deviations from the threshold regularity $-\tfrac32$ for the linearized unmodified gauge-fixed Einstein operator on Schwarzschild spacetimes, see the Proof of Theorem~\ref{ThmOp}.) Using Lemma~\ref{LemmaCD0GenDual}, we thus have
  \[
    (I-(\check L_b(0)^{-1})^*V_b^*)\breve h^* \in \rho\CI+\Hbsupp^{-3/2-\nu-,1/2-},
  \]
  which gets mapped into $\Hbsupp^{-5/2-\nu,3/2-}$ by $\pa_\sigma\wh{L_b}(0)^*$ (using that $\pa_\sigma\wh{L_b}(0)\colon\rho\CI\to\rho^3\CI$) as well as by $\pa_\sigma^2\wh{L_b}(0)^*\in\rho^2\CI$. On the other hand, Proposition~\ref{PropRegCheckL} gives
  \begin{equation}
  \label{EqEpsRegwtK}
    \check L_b(\sigma)^{-1}f\in \Hbext^{s,\ell},\qquad
    \pa^{1+j}_\sigma\check L_b(\sigma)^{-1}f\in |\sigma|^{-\eps-j-0}\Hbext^{s-j-\eps,\ell+\eps-1},\quad j=0,1.
  \end{equation}
  The pairings above are thus well-defined since $\ell+\eps>-\half$ (which holds by assumption) and $s-\eps>\tfrac72+\nu$ (which is satisfied for $s-\eps>\tfrac72$ when $(b,\gamma)$ is close to $(b_0,0)$).

%%%%%%%%%%%%%%%%%%%%%%%%%%%%%% 
  \pfstep{Analysis of $\wt L_{2 0}$.} Let $f\in \wt\cK_b^{\perp}$, $h^*\in \cK_{b,v}^*$. Then we have
  \begin{equation}
  \label{EqEpsRegL20}
    \sigma^{-1}\la \wh{L_b}(\sigma)\check L_b(\sigma)^{-1}f,h^*\ra = \big\la \check L_b(\sigma)^{-1}f,(\pa_{\sigma}\wh{L_b}(0)^*+\tfrac{\sigma}{2}\pa_{\sigma}^2\wh{L_b}(0)^*)h^*\big\ra.
  \end{equation}
  By~\eqref{EqEpsRegwtK}, and using that $\pa_\sigma\wh{L_b}(0)^*h^*$, $\pa_\sigma^2\wh{L_b}(0)^*h^*\in\Hbsupp^{-5/2-\nu,3/2-}$ (which for the former tensor follows again from the fact that the normal operator of $\pa_\sigma\wh{L_b}(0)^*$ annihilates $r^{-1}$), we see that this is $\eps$-regular.

%%%%%%%%%%%%%%%%%%%%%%%%%%%%%% 
  \pfstep{Analysis of $\wt L_{2 1}$, $\wt L_{1 2}$.} Let $h\in \cK_{b,s}$, $h^*\in \cK_{b,v}^*$. Using the calculations~\eqref{EqRL10exp} (for $h$) and~\eqref{EqRcheckLVstarImpr0} (for $h^*$), as well as using~\eqref{EqRcheckLVstar}, we find that $\wt L_{1 2}(h,h^*)$ equals
  \begin{align*}
    &\sigma^{-2}\la\wh{L_b}(\sigma)\check L_b(\sigma)^{-1}\check L_b(0)h,h^*\ra \\
    &\quad=\la V_b\check L_b(\sigma)^{-1}(\half\pa_{\sigma}^2\wh{L_b}(0)h-i(\pa_{\sigma}\wh{L_b}(0)+\tfrac{\sigma}{2}\pa_{\sigma}^2\wh{L_b}(0))\breve h,h^*\ra\\
    &\quad=\la\half\pa_{\sigma}^2\wh{L_b}(0)h-i\pa_{\sigma}\wh{L_b}(0)\breve h,(\check L_b(\sigma)^{-1})^*V_b^*h^*\ra - \tfrac{i\sigma}{2}\la V_b L_b(\sigma)^{-1}\pa_{\sigma}^2\wh{L_b}(0)\breve h,h^*\ra\\
    &\quad=\la\half\pa_{\sigma}^2\wh{L_b}(0)h-i\pa_{\sigma}\wh{L_b}(0)\breve h,h^*\ra\\
    &\quad\qquad-\sigma\big\la\check L_b(\sigma)^{-1}(\half\pa_{\sigma}^2\wh{L_b}(0)h-i\pa_{\sigma}\wh{L_b}(0))\breve h,(\pa_{\sigma}\wh{L_b}(0)^*+\tfrac{\sigma}{2}\pa_{\sigma}^2\wh{L_b}(0)^*)h^*\big\ra\\
    &\quad\qquad-\tfrac{i\sigma}{2}\la V_b \check L_b(\sigma)^{-1}\pa_{\sigma}^2\wh{L_b}(0)\breve h,h^*\ra,
  \end{align*}
  We need to show that the $\sigma$-coefficient of the last two lines is $\eps$-regular, which holds because $\check L_b(\sigma)^{-1}v$ is $\eps$-regular for any $v\in\Hbext^{\infty,3/2-}$. Note here that indeed $\pa_\sigma\wh{L_b}(0)\breve h$, $\pa_\sigma^2\wh{L_b}(0)\breve h\in\Hbext^{\infty,3/2-}$ in view of Lemma~\ref{LemmaL0Lead}. 
  
  Let now $h\in\cK_{b,v}$, $h^*\in\cK_{b,s}^*$. Then~\eqref{EqR10exp0} (for $h$) and \eqref{EqRcheckLVstarImpr} (for $h^*$) give
  \begin{align*}
    &\sigma^{-2}\la \wh{L_b}(\sigma)\check L_b(\sigma)^{-1}\check L_b(0)h,h^*\ra \\
    &\quad = \sigma^{-1}\la(\pa_\sigma\wh{L_b}(0)+\tfrac{\sigma}{2}\pa_\sigma^2\wh{L_b}(0))h,(\check L_b(\sigma)^{-1})^*V_b^*h^*\ra \\
    &\quad = \sigma^{-1}\la\pa_\sigma\wh{L_b}(0)h,h^*\ra  + \half\la\pa_\sigma^2\wh{L_b}(0)h,h^*\ra \\
    &\quad\qquad - \big\la(\pa_\sigma\wh{L_b}(0)+\tfrac{\sigma}{2}\pa_\sigma^2\wh{L_b}(0))h,i(I-(\check L_b(0)^{-1})^*V_b^*)\breve h^*+\tfrac{\sigma}{2}\pa_\sigma^2\wh{L_b}(0)^*h^*\big\ra \\
    &\quad\qquad +i\sigma\big\la\check L_b(\sigma)^{-1}(\pa_\sigma\wh{L_b}(0)+\tfrac{\sigma}{2}\pa_\sigma^2\wh{L_b}(0))h, \\
    &\quad\qquad \hspace{5em} (\pa_\sigma\wh{L_b}(0)^*+\tfrac{\sigma}{2}\pa_\sigma^2\wh{L_b}(0)^*)(I-(\check L_b(0)^{-1})^*V_b^*)\breve h^*\big\ra.
  \end{align*}
  The first summand vanishes by the calculation~\eqref{EqR12Vanish}. The second and third summand are polynomials in $\sigma$. The $\sigma$-coefficient of the final term is $\eps$-regular by the same argument as before. This shows that $\wt L_{2 1}$ has an $\eps$-regular expansion up to order $1$.
  
%%%%%%%%%%%%%%%%%%%%%%%%%%%%%% 
  \pfstep{Analysis of $\wt L_{2 2}$.} For $h\in\cK_{b,v}$, $h^*\in\cK_{b,v}^*$, we compute, using~\eqref{EqRcheckLVstarImpr0} and \eqref{EqR10exp0},
  \begin{align*}
    &\sigma^{-1}\la\wh{L_b}(\sigma)\check L_b(\sigma)^{-1}\check L_b(0)h,h^*\ra \\
    &\quad=\la(\pa_{\sigma}\wh{L_b}(0)+\tfrac{\sigma}{2}\pa_\sigma^2\wh{L_b}(0))h,(\check L_b(\sigma)^{-1})^*V_b^*h^*\ra \\
    &\quad=\la(\pa_{\sigma}\wh{L_b}(0)+\tfrac{\sigma}{2}\pa_\sigma^2\wh{L_b}(0))h,h^*\ra \\
    &\quad\qquad - \sigma\la\check L_b(\sigma)^{-1}(\pa_{\sigma}\wh{L_b}(0)+\tfrac{\sigma}{2}\pa_\sigma^2\wh{L_b}(0))h,(\pa_\sigma\wh{L_b}(0)^*+\tfrac{\sigma}{2}\pa_\sigma^2\wh{L_b}(0)^*)h^*\ra.
  \end{align*}
  The $\sigma$-coefficient of the last line is $\eps$-regular by the same arguments as before.
  
%%%%%%%%%%%%%%%%%%%%%%%%%%%%%% 
  \pfstep{Analysis of $\wt L_{1 1}$.} The calculations in the proof of Theorem~\ref{ThmR} which show that $\wt L_{1 1}$ has a Taylor expansion to order $2$ immediately imply, by means of Proposition~\ref{PropRegCheckL}, that the $\cO(\sigma^2)$ coefficient of $\wt L_{1 1}$ is $\eps$-regular; thus, $\wt L_{1 1}$ has an $\eps$-regular expansion up to order two, as desired.

%%%%%%%%%%%%%%%%%%%%%%%%%%%%%% 
  \pfstep{Analysis of $\wt L'_{0 1}$.} For $h\in\cK_{b,s}$, we have
  \[
    \wt L'_{0 1}\check L_b(0)h = \sigma^{-2}\wh{L_b}(\sigma)\check L_b(\sigma)^{-1}\check L_b(0)h - (\wt L_{1 1}+\wt L_{2 1})\check L_b(0)h.
  \]
  Using~\eqref{EqRL10exp} as well as the results on $\wt L_{1 1}$, $\wt L_{2 1}$ proved already, we see that this is $\eps$-regular. (Note here that the first operator $V_b$ in~\eqref{EqRL10exp} maps into $\CIc(X^\circ)$.)

%%%%%%%%%%%%%%%%%%%%%%%%%%%%%% 
  \pfstep{Analysis of $\wt L_{0 2}$.} In a similar manner, we write, for $h\in\cK_{b,v}$,
  \[
    \wt L_{0 2}\check L_b(0)h = \sigma^{-1}\wh{L_b}(\sigma)\check L_b(\sigma)^{-1}\check L_b(0)h - (\sigma\wt L_{1 2} + \wt L_{2 2})\check L_b(0)h.
  \]
  In view of~\eqref{EqR10exp0}, the first term is regular at $\sigma=0$, and in fact is $\eps$-regular; for the second term, $\eps$-regularity was proved above. The proof is complete.

%%%%%%%%%%%%%%%%%%%%%%%%%%%%%%
  \pfstep{Analysis of $L_{0 0}$.} For $f\in\wt\cK_b^\perp$, we have
  \begin{align*}
    L_{0 0}f &= \wh{L_b}(\sigma)\check L_b(\sigma)^{-1}f - \sigma(\wt L_{1 0}+\wt L_{2 0})f \\
      &= f - V_b\check L_b(\sigma)^{-1}f - \sigma(\wt L_{1 0}+\wt L_{2 0})f.
  \end{align*}
  The third term was treated above. Differentiating in $\sigma$, only the second term remains, and it maps into $\CIc(X^\circ)$. Using Proposition~\ref{PropRegCheckL}, we see that $L_{0 0}$ is $\eps$-regular.
\end{proof}

We next study the entries $\wt R_{i j}$ of $\wt R_b(\sigma)$ in~\eqref{EqRInv} and \eqref{EqwtRb}. As a preparation, we note:

\begin{lemma}
\label{LemmaRegRInverse}
  Let $s,\ell,\eps$ be as in Theorem~\ref{ThmRegL}. Let $j,k\in\{0,1,2\}$. Suppose that $L(\sigma)$ has an $\eps$-regular expansion up to order two of the form 
  \[
    L(\sigma)=L_0+\sigma L_1+\sigma^2 L_2(\sigma),
  \]
  where $L_0\colon\wt\cK_j\to\cR_i$ is invertible, $L_1,L_2(\sigma)\colon\wt\cK_j\to\cR_i$ are bounded uniformly in $\sigma$, and $L_2(\sigma)$ is $\eps$-regular. Then for $\sigma$ sufficiently small, $L(\sigma)$ is invertible, and its inverse $L(\sigma)^{-1}$ has an $\eps$-regular expansion up to order two, in the sense that
  \[
    L(\sigma)^{-1} = L^0 + \sigma L^1 + \sigma^2 L^2(\sigma),
  \]
  where the operators $L^0,L^1\colon\cR_i\to\wt\cK_j$ are bounded, and $L^2(\sigma)$ is $\eps$-regular in the sense that $L^2(\sigma)\colon\cR_i\to\wt\cK_j$, $\pa_\sigma L^2(\sigma)\colon\cR_i\to|\sigma|^{-\eps}\wt\cK_j$, and $\pa_\sigma^2 L^2(\sigma)\colon\cR_i\to|\sigma|^{-1-\eps}\wt\cK_j$ (using the notation of Definition~\ref{DefBdd}).
  
  Similarly, if $L(\sigma)$ only has an $\eps$-regular expansion up to order one, then so does $L(\sigma)^{-1}$. Lastly, if $L(\sigma)$ is $\eps$-regular and has uniformly bounded inverse, then $L(\sigma)^{-1}$ is $\eps$-regular.
\end{lemma}
\begin{proof}
  Write $L(\sigma)=L_0(1+\sigma L_0^{-1}(L_1+\sigma L_2(\sigma)))$. Now $\|\sigma L_0^{-1}(L_1+\sigma L_2(\sigma))\|_{\wt\cK_j\to \wt\cK_j}\leq C\sigma$, hence $L(\sigma)$ is invertible for sufficiently small $\sigma$ by means of a Neumann series:
  \begin{equation}
  \label{EqRegRInverse}
  \begin{split}
    &\bigl(L_0(1+\sigma L_0^{-1}(L_1+\sigma L_2(\sigma)))\bigr)^{-1} \\
    &\quad=L_0^{-1}-\sigma L_0^{-1}L_1 L_0^{-1}-\sigma^2 L_0^{-1}L_2(\sigma)L_0^{-1}\\
    &\quad\qquad+\sigma^2\left(\sum_{k=2}^{\infty}(-1)^k\sigma^{k-2}(L_0^{-1}(L_1+\sigma L_2(\sigma)))^k L_0^{-1}\right)\\
    &=:R_0+\sigma R_1+\sigma^2 R_2(\sigma).
  \end{split}
  \end{equation}
  Now, $R_2(\sigma)\colon\cR_i\to\wt\cK_j$ is bounded, and $L_0^{-1}L_2(\sigma)L_0^{-1}$ is $\eps$-regular since $L_2(\sigma)$ is. It remains to consider the infinite series. Differentiating it once produces terms of the form
  \[
    (L_0^{-1}(L_1+\sigma L_2(\sigma)))^{k_1} \circ L_0^{-1}(L_2(\sigma)+\sigma L_2'(\sigma)))\circ (L_0^{-1}(L_1+\sigma L_2(\sigma)))^{k_2}L_0^{-1},
  \]
  where $'$ denotes differentiation in $\sigma$. The third factor is bounded $\cR_i\to\wt\cK_j$, which the second factor further maps into $L_0^{-1}(\cR_i+|\sigma|^{1-\eps}\cR_i)\subset\wt\cK_j$, and the first factor then produces $\wt\cK_j$.

  Differentiating the series twice produces two types of terms: the first type is of the form
  \[
    (L_0^{-1}(L_1+\sigma L_2(\sigma)))^{k_1}\circ L_0^{-1}(\sigma L''_2(\sigma)+2 L_2'(\sigma)) \circ (L_0^{-1}(L_1+\sigma L_2(\sigma)))^{k_2}L_0^{-1},
  \]
  mapping $\cR_i\to\wt\cK_j$ (third factor), then into $L_0^{-1}(|\sigma|^{-\eps}\cR_i+|\sigma|^{1-\eps}\cR_i)\subset|\sigma|^{-\eps}\wt\cK_j$ (second factor), and then into $|\sigma|^{-\eps}\wt\cK_j$ (first factor). The second type is of the form
  \begin{align*}
    (L_0^{-1}(L_1+\sigma L_2(\sigma)))^{k_1}&\circ L_0^{-1}(L_2(\sigma)+\sigma L_2'(\sigma)) \circ (L_0^{-1}(L_1+\sigma L_2(\sigma)))^{k_2} \\
    &\circ L_0^{-1}(L_2(\sigma)+\sigma L_2'(\sigma)) \circ (L_0^{-1}(L_1+\sigma L_2(\sigma)))^{k_3}L_0^{-1},
  \end{align*}
  which maps $\cR_i\to\wt\cK_j$, as desired.

  The proof of the final two statements is analogous (and in fact simpler).
\end{proof}

\begin{prop}
\label{PropResReg}
  The entries of
  \[
    \wt R_b(\sigma) = \bigl(\wh{L_b}(\sigma)\check L_b(\sigma)^{-1}\bigr)^{-1}
      =\begin{pmatrix}
         \wt R_{0 0} & \wt R'_{0 1} & \wt R_{0 2} \\
         \wt R'_{1 0} & \sigma^{-2}\wt R_{1 1} & \sigma^{-1}\wt R_{1 2} \\
         \wt R_{2 0} & \sigma^{-1}\wt R_{2 1} & \sigma^{-1}\wt R_{2 2}
       \end{pmatrix}
  \]
  have the following regularity:
  \begin{enumerate}
  \item $\wt R_{0 0}$, $\wt R'_{0 1}$, $\wt R'_{1 0}$, $\wt R_{0 2}$, $\wt R_{2 0}$ are $\eps$-regular (in the sense of Lemma~\ref{LemmaRegRInverse});
  \item $\wt R_{1 2}$, $\wt R_{2 1}$, $\wt R_{2 2}$ have $\eps$-regular expansions up to order one;
  \item $\wt R_{1 1}$ has an $\eps$-regular expansion up to order two.
  \end{enumerate}
\end{prop}
\begin{proof}
  First, combining Lemma~\ref{LemmaRegRInverse} and Proposition~\ref{PropEpsReg}, we see that $L_{0 0}^{-1}$ is $\eps$-regular. Next, recall that $\wt L_{i j}^\sharp=\wt L_{i j}-\wt L_{i 0}L_{0 0}^{-1}\wt L_{0 j}$. By Proposition~\ref{PropEpsReg}, $\wt L_{1 1}^{\sharp}$ has an $\eps$-regular expansion up to order two, and $\wt L_{1 2}^\sharp$, $\wt L_{2 1}^\sharp$ have $\eps$-regular expansions up to order one. Furthermore, $\wt L_{2 2}^\flat=\wt L_{2 2}-\sigma \wt L_{2 0}L_{0 0}^{-1}\wt L_{0 2}$ has an $\eps$-regular expansion up to order one. Thus, $\wt L_{1 1}^\natural=\wt L_{1 1}^\sharp-\sigma\wt L_{1 2}^\sharp(\wt L_{2 2}^b)^{-1}\wt L_{2 1}^\sharp$ has an $\eps$-regular expansion up to order two. 
  
  The $\eps$-regularity of $\wt\cR'_{0 1}=\sigma^{-1}\wt\cR_{0 1}$ and $\wt\cR'_{1 0}=\sigma^{-1}\wt\cR_{1 0}$ follows from the explicit formulas~\eqref{EqRInv2} (recalling~\eqref{EqRL0110Impr}), Proposition~\ref{PropEpsReg}, and Lemma~\ref{LemmaRegRInverse}. Similarly, $\wt\cR_{1 2}$, $\wt R_{2 1}$, $\wt R_{2 2}$ have $\eps$-regular expansions up to order one by inspection of~\eqref{EqRInv2}, and $\wt\cR_{1 1}$ has an $\eps$-regular expansion up to order two. We calculate the remaining, regular, coefficients of $\wt R_b(\sigma)$ as
  \begin{align*}
    \wt R_{0 0}&=L_{0 0}^{-1}(I-\wt L_{0 1}\wt R_{1 0}-\sigma\wt L_{0 2}\wt R_{2 0}),\\
    \wt R_{0 2}&=L_{0 0}^{-1}(-\wt L_{0 1}\wt R_{1 2}-\wt L_{0 2}\wt R_{2 2}),\\
    \wt R_{2 0}&=-(\wt L^b_{2 2})^{-1}(\wt L_{2 1}^{\sharp}\wt R_{1 0}+\wt L_{2 0}L_{0 0}^{-1});
  \end{align*}
  they are thus $\eps$-regular as well.
\end{proof}

We can now prove the main theorem of this section.
\begin{proof}[Proof of Theorem~\ref{ThmRegL}]
  We split $\wt R_b(\sigma)$ into its regular and singular parts, to wit
  \begin{align*}
    &\wt R_b(\sigma) = \wt R_{b,\rm reg}(\sigma)+\wt R_{b,\rm sing}(\sigma), \\
    &\qquad
      \wt R_{b,\rm reg}(\sigma) =\begin{pmatrix}
       \wt R_{0 0} & \wt R_{0 1}' & \wt R_{0 2} \\
       \wt R_{1 0}' & \wt R_{1 1}' & \wt R_{1 2}' \\
       \wt R_{2 0} & \wt R_{2 1}' & \wt R_{2 2}',
     \end{pmatrix}, \quad
    \wt R_{b,\rm sing}(\sigma)= \begin{pmatrix}
         0 & 0 & 0 \\
         0 & \sigma^{-2}\wt P_{1 1} & \sigma^{-1}\wt P_{1 2} \\
         0 & \sigma^{-1}\wt P_{2 1} & \sigma^{-1}\wt P_{2 2} 
       \end{pmatrix}.
  \end{align*}
  By Theorem~\ref{ThmR}, we can write
  \begin{equation}
  \label{EqRegPfSing}
    \wh{L_b}(\sigma)^{-1}=\sum_{j=1}^2\sigma^{-j}P_{b,j}+L_b^-(\sigma), \quad
    P_{b,j}\colon\Hbext^{s-1,\ell+2}\to\Hbext^{\infty,-1/2-}.
  \end{equation}
  Combining
  \[
    \wt R_{b,\rm reg}(\sigma)+\wt R_{b,\rm sing}(\sigma) = \check L_b(\sigma)(P_b(\sigma)+L^-_b(\sigma)),
  \]
  with $\check L_b(\sigma)=\check L_b(0)+\sigma\pa_{\sigma}\check L_b(0)+\tfrac{\sigma^2}{2}\pa_{\sigma}^2\check L_b(0)$, we thus obtain
  \[
    \wt R_{b,\rm reg}(\sigma)=\check L_b(\sigma)L_b^-(\sigma) + (\pa_\sigma\check L_b(0)+\tfrac{\sigma}{2}\pa_\sigma^2\check L_b(0))P_{b,1} + \half\pa_\sigma^2\check L_b(0)P_{b,2},
  \]
  and therefore
  \begin{equation}
  \label{EqRegPf}
    L^-_b(\sigma)=\check L_b(\sigma)^{-1}\wt R_{b,\rm reg}(\sigma) - \check L_b(\sigma)^{-1}\bigl((\pa_\sigma\check L_b(0)+\tfrac{\sigma}{2}\pa_\sigma^2\check L_b(0))P_{b,1} + \half\pa_\sigma^2\check L_b(0)P_{b,2}\bigr).
  \end{equation}
  By Lemma~\ref{LemmaL0Lead}, the compositions $\pa_\sigma^k\check L_b(0)\circ P_{b,j}$ for $j,k=1,2$ are bounded from $\Hbext^{s-1,\ell+2}\to\Hbext^{\infty,3/2-}$. By Proposition~\ref{PropRegCheckL}, the second term satisfies the estimates~\eqref{EqFDP}--\eqref{EqSDP}. For the first term, we compute
  \begin{align*}
    \pa_{\sigma}(\check L_b(\sigma)^{-1}\wt R_{b,\rm reg}(\sigma))&=\pa_{\sigma}\check L_b(\sigma)^{-1}\circ\wt R_{b,\rm reg}(\sigma)+\check L_b(\sigma)^{-1}\circ\pa_{\sigma}\wt R_{b,\rm reg}(\sigma),\\
    \pa_{\sigma}^2(\check L_b(\sigma)^{-1}\wt R_{b,\rm reg}(\sigma))&=\pa^2_{\sigma}\check L_b(\sigma)^{-1}\circ\wt R_{b,\rm reg}(\sigma) \\
    &\quad\qquad +2\pa_{\sigma}\check L_b(\sigma)^{-1}\circ\pa_{\sigma}\wt R_{b,\rm reg}(\sigma)+\check L_b(\sigma)^{-1}\circ\pa^2_{\sigma}\wt R_{b,\rm reg}(\sigma).
  \end{align*}
  Using Propositions~\ref{PropResReg} and \ref{PropRegCheckL}, we obtain the estimates~\eqref{EqFDP}--\eqref{EqSDP}.
\end{proof}

%%%%%%%%%%%%%%%%%%%%%%%%%%%%%%%%%%%%%%%%%%%%%%%%%%%%%%%%%%%%%%%%%%%%%%
\subsection{Regularity at intermediate and high frequencies}
\label{SsRegNon0}

We next prove the regularity of $\wh{L_b}(\sigma)^{-1}$ at non-zero $\sigma$ in the closed upper half plane, which is significantly easier. We recall from Theorem~\ref{ThmOp}\eqref{ItOpHigh} that
\[
  \wh{L_b}(\sigma)^{-1} \colon \Hbhext^{s,\ell+1}\to\Hbhext^{s,\ell},\quad h=|\sigma|^{-1},
\]
is uniformly bounded for $\sigma$ in the closed upper half plane, away from zero. (For $\sigma$ restricted to any compact set of the punctured (at zero) closed upper half plane, this simply means uniform estimates $\wh{L_b}(\sigma)^{-1}\colon\Hb^{s,\ell+1}\to\Hb^{s,\ell}$.)

\begin{prop}
\label{PropRegHigh}
  Let $\ell<-\half$, and $s>\tfrac52$, $s+\ell>-\half+m$, $m\in\N$. Let $\sigma_0>0$. Then for $\Im\sigma\geq 0$, $|\sigma|>\sigma_0$, the operator
  \[
    \pa_\sigma^m\wh{L_b}(\sigma)^{-1} \colon \Hbhext^{s,\ell+1}\to h^{-m}\Hbhext^{s-m,\ell}
  \]
  is uniformly bounded.
\end{prop}
\begin{proof}
  We have $\pa_\sigma\wh{L_b}(\sigma)^{-1}=-\wh{L_b}(\sigma)^{-1}\circ\pa_\sigma\wh{L_b}(\sigma)\circ\wh{L_b}(\sigma)^{-1}$, which, in view of $\pa_\sigma\wh{L_b}(\sigma)\in\rho\Diffb^1+\sigma\rho^2\CI\subset h^{-1}\rho\Diffbh^1$, maps
  \[
    \Hbhext^{s,\ell+1} \xra{\wh{L_b}(\sigma)^{-1}} \Hbhext^{s,\ell} \xra{\pa_\sigma\wh{L_b}(\sigma)} h^{-1}\Hbhext^{s-1,\ell+1} \xra{\wh{L_b}(\sigma)^{-1}} h^{-1}\Hbhext^{s-1,\ell}.
  \]
  Control of $\pa_\sigma^2(\wh{L_b}(\sigma)^{-1})$ requires another application of $\wh{L_b}(\sigma)^{-1}\pa_\sigma\wh{L_b}(\sigma)$ to this, which the assumptions on $s,\ell$ for $m=2$ do permit (producing $h^{-2}\Hbhext^{s-2,\ell}$); in addition, we need to use $\pa_\sigma^2\wh{L_b}(\sigma)\in\rho^2\CI$, so
  \[
    \Hbhext^{s,\ell+1} \xra{\wh{L_b}(\sigma)^{-1}} \Hbhext^{s,\ell} \xra{\pa_\sigma^2\wh{L_b}(\sigma)} \Hbhext^{s,\ell+2}\subset\Hbhext^{s,\ell+1} \xra{\wh{L_b}(\sigma)^{-1}} \Hbhext^{s,\ell}.
  \]
  An inductive argument based on these calculations proves the proposition.
\end{proof}

%%%%%%%%%%%%%%%%%%%%%%%%%%%%%%%%%%%%%%%%%%%%%%%%%%%%%%%%%%%%%%%%%%%%%%%
\subsection{Conormal regularity of the resolvent near zero}
\label{SsRegCon}

As opposed to full $\sigma$-derivatives of the regular part $L_b^-(\sigma)$ of the resolvent, of which we can only control two in a useful manner, we can control \emph{any} number of $\sigma\pa_\sigma$-derivatives. (On the inverse Fourier transform side, this means that repeated applications of $t_* D_{t_*}$ preserve the decay rate of a wave.)

\begin{thm}
\label{ThmRegConLow}
  With $L_b^-(\sigma)$ denoting the regular part of the resolvent $\wh{L_b}(\sigma)^{-1}$ as in Theorem~\ref{ThmR}, define
  \[
    \cL^{-(m)}_b(\sigma):=(\sigma\pa_\sigma)^m L_b^-(\sigma).
  \]
  Let $\ell\in(-\tfrac32,-\half)$, $\eps\in(0,1)$, $\ell+\eps\in(-\half,\half)$, and $s-m-\eps-\tfrac72>0$, $m\in\N$. Then for $0<|\sigma|\leq\sigma_0$, $\Im\sigma\geq 0$, with $\sigma_0>0$ small, we have uniform bounds
  \begin{subequations}
  \begin{align}
  \label{EqConLowBd0}
    \cL^{-(m)}_b(\sigma) &\colon \Hbext^{s-1,\ell+2}\to \Hbext^{s-m,\ell}, \\
  \label{EqConLowBd1}
    \pa_\sigma\cL^{-(m)}_b(\sigma) &\colon \Hbext^{s-1,\ell+2}\to |\sigma|^{-\eps}\Hbext^{s-m-\eps,\ell+\eps-1}, \\
  \label{EqConLowBd2}
    \pa_\sigma^2\cL^{-(m)}_b(\sigma) &\colon \Hbext^{s-1,\ell+2}\to|\sigma|^{-1-\eps}\Hbext^{s-1-m-\eps,\ell+\eps-1}.
  \end{align}
  \end{subequations}
\end{thm}

As in~\S\ref{SsRegLow}, we first prove the analogous result for the reference operator $\check L_b(\sigma)$. Writing
\[
  \sigma\pa_\sigma\check L_b(\sigma)=\check L_b(\sigma)-\check L_b(0) - \tfrac{\sigma^2}{2}\pa_\sigma^2\check L_b(0),
\]
with the second and third terms lying in $\rho^2\Diffb^2$, we have
\[
  \sigma\pa_\sigma(\check L_b(\sigma)^{-1})=-\check L_b(\sigma)^{-1}\,\sigma\pa_\sigma\check L_b(\sigma)\,\check L_b(\sigma)^{-1} \in \check L_b(\sigma)^{-1} + \check L_b(\sigma)^{-2}\circ \rho^2\Diffb^2\circ \check L_b(\sigma)^{-1}.
\]
By induction, this gives
\begin{equation}
\begin{split}
\label{EqConCheckL}
  \check\cL^{(m)}_b(\sigma) &:= (\sigma\pa_\sigma)^m\check L_b(\sigma)^{-1} \\
    &=\sum_{k=0}^m\check L_b(\sigma)^{-1}\cR_{k,1}(\sigma)\check L_b(\sigma)^{-1}\cdots\cR_{k,k}(\sigma)\check L_b(\sigma)^{-1},
\end{split}
\end{equation}
where $\cR_{k,j}(\sigma)\in\rho^2\Diffb^2$.

\begin{prop}
  \label{PropConCheckL}
  Let $s>\tfrac72+\eps$, $\ell\in(-\tfrac32,-\half)$, $\eps\in(0,1)$, $\ell+\eps\in(-\half,\half)$, and $s+\ell>\tfrac32+m$, $m\in\N$. Then for $0<|\sigma|\leq\sigma_0$, $\Im\sigma\geq 0$, with $\sigma_0>0$ small, the conclusions of Theorem~\ref{ThmRegConLow} hold for $\check\cL_b^{(m)}(\sigma)$.
\end{prop}
\begin{proof}
  Consider first the case $m=1$. Let us write
  \[
    \sigma\pa_\sigma(\check L_b(\sigma)^{-1})=\check L_b(\sigma)^{-1} + \check L_b(\sigma)^{-1}\cR(\sigma)\check L_b(\sigma)^{-1},\quad
    \cR(\sigma)\in\rho^2\Diffb^2.
  \]
  We compute the first derivative (denoted by $'$) of the second summand in $\sigma$ (the first term being handled by Proposition~\ref{PropRegCheckL} directly). We have 
  \begin{align*}
    (\check L_b(\sigma)^{-1} \cR(\sigma)\check L_b(\sigma)^{-1})'&= (\check L_b(\sigma)^{-1} )'\cR(\sigma)\check L_b(\sigma)^{-1} \\
      &\quad\qquad +\check L_b(\sigma)^{-1}\cR'(\sigma)\check L_b(\sigma)^{-1}+\check L_b(\sigma)^{-1}\cR(\sigma)(\check L_b(\sigma)^{-1})' \\
      &=:L_1+L_2+L_3.
  \end{align*}
  Using Proposition~\ref{PropRegCheckL}, we have uniform bounds
  \begin{align*}
    L_1,\,L_2 &\colon \Hbext^{s-1,\ell+2}\xra{\check L_b(\sigma)^{-1}}\Hbext^{s,\ell}\xra{\cR(\sigma),\,\cR'(\sigma)}\Hbext^{s-2,\ell+2}\xra{(\check L_b(\sigma)^{-1})'} |\sigma|^{-\eps}\Hbext^{s-1-\eps,\ell+\eps-1}, \\
    L_3 &\colon \Hbext^{s-1,\ell+2}\xra{(\check L_b(\sigma)^{-1})'}|\sigma|^{-\eps}\Hbext^{s-\eps,\ell+\eps-1}\xra{\cR(\sigma)}|\sigma|^{-\eps}\Hbext^{s-2-\eps,\ell+1+\eps} \\
      &\hspace{18em} \xra{\check L_b(\sigma)^{-1}}|\sigma|^{-\eps}\Hbext^{s-1-\eps,\ell-1+\eps}.
  \end{align*}
  Let us now consider the second derivative. We have 
  \begin{align*}
    (\check L_b(\sigma)^{-1}\cR(\sigma)\check L_b(\sigma)^{-1})''&=(\check L_b(\sigma)^{-1})''\cR(\sigma)\check L_b(\sigma)^{-1}+2(\check L_b(\sigma)^{-1})'\cR(\sigma)(\check L_b(\sigma)^{-1})'\\
    &\qquad+2(\check L_b(\sigma)^{-1})'\cR'(\sigma)\check L_b(\sigma)^{-1}
      +2\check L_b(\sigma)^{-1}\cR'(\sigma)(\check L_b(\sigma)^{-1})'\\
    &\qquad+\check L_b(\sigma)^{-1}\cR(\sigma)(\check L_b(\sigma)^{-1})''+\check L_b(\sigma)^{-1}\cR''(\sigma)\check L_b(\sigma)^{-1}\\
    &=:M_1+M_2+M_3+M_4+M_5+M_6.
  \end{align*}
  We again can treat the different terms by applying Proposition~\ref{PropRegCheckL}. Namely,
  \begin{align*}
    M_1 &\colon \Hbext^{s-1,\ell+2}\xra{\check L_b(\sigma)^{-1}}\Hbext^{s,\ell}\xra{\cR(\sigma)}\Hbext^{s-2,\ell+2}\xra{(\check L_b(\sigma)^{-1})''}|\sigma|^{-1-\eps}\Hbext^{s-2-\eps,\ell+\eps-1}, \\
    M_2 &\colon \Hbext^{s-1,\ell+2}\xra{(\check L_b(\sigma)^{-1})'}|\sigma|^{-\eps/2}\Hbext^{s-\eps/2,\ell+\eps/2-1}\xra{\cR(\sigma)}|\sigma|^{-\eps/2}\Hbext^{s-2-\eps/2,\ell+1+\eps/2} \\
     &\hspace{20em} \xra{(\check L_b(\sigma)^{-1})'}|\sigma|^{-\eps}\Hbext^{s-1-\eps,\ell-1+\eps}, \\
    M_3 &\colon \Hbext^{s-1,\ell+2}\xra{\check L_b(\sigma)^{-1}}\Hbext^{s,\ell}\xra{\cR'(\sigma)}\Hbext^{s-2,\ell+2}\xra{(\check L_b(\sigma)^{-1})'}|\sigma|^{-\eps}\Hbext^{s+1-\eps,\ell+\eps-1}, \\
    M_4 &\colon \Hbext^{s-1,\ell+2}\xra{(\check L_b(\sigma)^{-1})'}|\sigma|^{-\eps}\Hbext^{s-\eps,\ell+\eps-1}\xra{\cR'(\sigma)}|\sigma|^{-\eps}\Hbext^{s-2-\eps,\ell+1+\eps}, \\
     &\hspace{20em} \xra{\check L_b(\sigma)^{-1}}|\sigma|^{-\eps}\Hbext^{s-1-\eps,\ell-1+\eps}, \\
    M_5 &\colon \Hbext^{s-1,\ell+2}\xra{(\check L_b(\sigma)^{-1})''}|\sigma|^{-1-\eps}\Hbext^{s-1-\eps,\ell+\eps-1}\xra{\cR(\sigma)}|\sigma|^{-1-\eps}\Hbext^{s-3-\eps,\ell+\eps+1} \\
     &\hspace{20em}\xra{\check L_b(\sigma)^{-1}}|\sigma|^{-1-\eps}\Hbext^{s-2-\eps,\ell+\eps-1}, \\
    M_6 &\colon \Hbext^{s-1,\ell+2}\xra{\check L_b(\sigma)^{-1}} \Hbext^{s,\ell} \xra{\cR''(\sigma)}\Hbext^{s-2,\ell+2}\xra{\check L_b(\sigma)^{-1}}\Hbext^{s-1,\ell}.
  \end{align*}
  
  Suppose now that the proposition holds for $m\in\N$; we want to show it for $m+1$. It is sufficient to prove the proposition for an operator of type 
  \[
    \tilde{\cL}_b^{(m+1)}(\sigma)=\tilde{\cL}^{(m)}_b(\sigma)\cR(\sigma)\check L_b(\sigma)^{-1},\quad \cR(\sigma)\in\rho^2\Diffb^2,
  \]
  where $\tilde{\cL}^{(m)}_b(\sigma)$ satisfies the same estimates as $\check\cL^{(m)}_b(\sigma)$, i.e.\ \eqref{EqConLowBd1}--\eqref{EqConLowBd2}. Let us compute the first derivative of $\tilde{\cL}_b^{(m+1)}(\sigma)$:
  \begin{align*}
    (\tilde\cL_b^{(m)}(\sigma)\cR(\sigma)\check L_b(\sigma)^{-1})'&= (\cL_b^{(m)}(\sigma))'\cR(\sigma)\check L_b(\sigma)^{-1} \\
    &\quad\qquad +\tilde\cL_b^{(m)}(\sigma)\cR'(\sigma)\check L_b(\sigma)^{-1}+\tilde\cL_b^{(m)}(\sigma)\cR(\sigma)(\check L_b(\sigma)^{-1})' \\
    &=:L_1+L_2+L_3.
  \end{align*}
  Using Proposition~\ref{PropRegCheckL} and the inductive hypothesis, we have
  \begin{align*}
    L_1 &\colon \Hbext^{s-1,\ell+2}\xra{\check L_b(\sigma)^{-1}}\Hbext^{s,\ell}\xra{\cR(\sigma)}\Hbext^{s-2,\ell+2}\xra{(\tilde\cL^{(m)}_b(\sigma))'}|\sigma|^{-\eps}\Hbext^{s-1-m-\eps,\ell+\eps-1}, \\
    L_2 &\colon \Hbext^{s-1,\ell+2}\xra{\check L_b(\sigma)^{-1}}\Hbext^{s,\ell}\xra{\cR'(\sigma)}\Hbext^{s-2,\ell+2}\xra{\tilde\cL_b^{(m)}(\sigma)}\Hbext^{s-1-m,\ell}, \\
    L_3 &\colon \Hbext^{s-1,\ell+2}\xra{(\check L_b(\sigma)^{-1})'}|\sigma|^{-\eps}\Hbext^{s-\eps,\ell+\eps-1}\xra{\cR(\sigma)}|\sigma|^{-\eps}\Hbext^{s-2-\eps,\ell+1+\eps} \\
     &\hspace{18em} \xra{\tilde\cL_b^{(m)}(\sigma)}|\sigma|^{-\eps}\Hbext^{s-m-1-\eps,\ell-1+\eps}.
  \end{align*}

  We next consider the second derivative,
  \begin{align*}
    (\tilde\cL_b^{(m)}(\sigma)\cR(\sigma)\check L_b(\sigma)^{-1})''&=\tilde\cL_b^{(m)}(\sigma)''\cR(\sigma)\check L_b(\sigma)^{-1}+2\tilde\cL_b^{(m)}(\sigma)'\cR(\sigma)(\check L_b(\sigma)^{-1})'\\
    &\qquad+2\tilde\cL_b^{m}(\sigma)'\cR'(\sigma)\check L_b(\sigma)^{-1}
      +2\tilde\cL_b^{(m)}(\sigma)\cR'(\sigma)(\check L_b(\sigma)^{-1})'\\
    &\qquad+\tilde\cL_b^{(m)}(\sigma)\cR(\sigma)(\check L_b(\sigma)^{-1})''+\tilde\cL_b^{(m)}(\sigma)\cR''(\sigma)\check L_b(\sigma)^{-1}\\
    &=:M_1+M_2+M_3+M_4+M_5+M_6.
  \end{align*}
  We again can treat the different terms by applying Proposition~\ref{PropRegCheckL} and the inductive hypothesis:
  \begin{align*}
    M_1 &\colon \Hbext^{s-1,\ell+2}\xra{\check L_b(\sigma)^{-1}}\Hbext^{s,\ell}\xra{\cR(\sigma)}\Hbext^{s-2,\ell+2}\xra{\tilde\cL_b^{(m)}(\sigma)''}|\sigma |^{-1-\eps}\Hbext^{s-2-m-\eps,\ell+\eps-1}, \\
    M_2 &\colon \Hbext^{s-1,\ell+2}\xra{(\check L_b(\sigma)^{-1})'}|\sigma|^{-\eps/2}\Hbext^{s-\eps/2,\ell+\eps/2-1}\xra{\cR(\sigma)}|\sigma|^{-\eps/2}\Hbext^{s-2-\eps/2,\ell+1+\eps/2} \\
      &\hspace{18em} \xra{\tilde\cL_b^{(m)}(\sigma)'}|\sigma|^{-\eps}\Hbext^{s-1-m-\eps,\ell-1+\eps}, \\
    M_3 &\colon \Hbext^{s-1,\ell+2}\xra{\check L_b(\sigma)^{-1}}\Hbext^{s,\ell}\xra{\cR'(\sigma)}\Hbext^{s-2,\ell+2}\xra{\tilde\cL_b^{(m)}(\sigma)'}|\sigma|^{-\eps}\Hbext^{s-1-m-\eps,\ell+\eps-1}, \\
    M_4 &\colon \Hbext^{s-1,\ell+2}\xra{(\check L_b(\sigma)^{-1})'}|\sigma|^{-\eps}\Hbext^{s-\eps,\ell+\eps-1}\xra{\cR'(\sigma)}|\sigma|^{-\eps}\Hbext^{s-\eps,\ell+\eps+1} \\
      &\hspace{18em} \xra{\tilde\cL_b^{(m)}(\sigma)}\Hbext^{s+1-m-\eps,\ell+\eps-1}, \\
    M_5 &\colon \Hbext^{s-1,\ell+2}\xra{(\check L_b(\sigma)^{-1})''}|\sigma|^{-1-\eps}\Hbext^{s-1-\eps,\ell+\eps-1}\xra{\cR(\sigma)}|\sigma|^{-1-\eps}\Hbext^{s-3-\eps,\ell+\eps+1} \\
      &\hspace{18em} \xra{\tilde\cL_b^{(m)}(\sigma)}|\sigma|^{-1-\eps}\Hbext^{s-2-m-\eps,\ell+\eps-1}, \\
    M_6 &\colon \Hbext^{s-1,\ell+2}\xra{\check L_b(\sigma)^{-1}}\Hbext^{s,\ell}\xra{\cR''(\sigma)}\Hbext^{s,\ell+2}\xra{\tilde\cL_b^{(m)}(\sigma)}\Hbext^{s-m,\ell}.
  \end{align*}
  This finishes the proof of the proposition. 
\end{proof}

Using the conormal regularity of $\check L_b(\sigma)^{-1}$, we can now proceed as in \S\ref{SReg}, but we have to keep track of the loss of regularity when commuting with $\sigma\pa_\sigma$. We claim that for $m\in\N$, $(\sigma\pa_\sigma)^m\wt L_{i j}$ has an $\eps$-regular expansion up to the same order as $\wt L_{i j}$. Since $\sigma\pa_\sigma(\sigma^k)=k\sigma^k$, $k=0,1$, the only interesting terms are commutators with the non-polynomial (in $\sigma$) terms $\wt L_{i j}^e(\sigma)$; but the $\eps$-regularity of $(\sigma\pa_\sigma)^m\wt L^e_{i j}(\sigma)$ can be proved in the same way as that of $\wt L^e_{i j}$, now using Proposition~\ref{PropConCheckL} instead of Proposition~\ref{PropRegCheckL}. For example, consider $\wt L_{2 0}$, computed in~\eqref{EqEpsRegL20}: for $f\in\wt\cK_b^\perp$, $h^*\in\cK_{b,v}^*$, we have
\[
  (\sigma\pa_\sigma)^m\wt L_{2 0}(f,h^*) = \la (\sigma\pa_\sigma)^m\check L_b(\sigma)^{-1}f,(\pa_\sigma\wh{L_b}(0)^*+\tfrac{\sigma}{2}\pa_\sigma^2\wh{L_b}(0)^*)h^*\ra + \text{similar terms},
\]
where the `similar terms' arise when some of the $\sigma\pa_\sigma$-derivatives fall on the second summand in the pairing. The $\eps$-regularity of this is thus indeed an immediate consequence of Proposition~\ref{PropConCheckL}. The argument for the other $\wt L_{i j}$ is completely analogous.

In order to analyze the entries $R_{i j}$ of $\wt R_b(\sigma)$, we need a result similar to Lemma~\ref{LemmaRegRInverse}.

\begin{lemma}
\label{LemmaConInv}
  Let $s,\ell,\eps,m$ be as in Theorem~\ref{ThmRegConLow}. Suppose that $L(\sigma)$ has an $\eps$-regular expansion up to order two of the form 
  \[
    L(\sigma)=L_0+\sigma L_1+\sigma^2 L_2(\sigma),
  \]
  where $L_0\colon\wt\cK_j\to\cR_i$ is invertible, $L_1,L_2(\sigma)\colon\wt\cK_j\to\wt\cR_i$ are bounded uniformly in $\sigma$, and $(\sigma\pa_\sigma)^k L_2(\sigma)$ is $\eps$-regular for $k=0,\ldots,m$. Then for $\sigma$ sufficiently small, $L(\sigma)$ is invertible, and $(\sigma\pa_\sigma)^m L(\sigma)^{-1}$ has an $\eps$-regular expansion up to order two, in the sense that
  \[
    (\sigma\pa_\sigma)^k L(\sigma)^{-1} = L^{0(k)} + \sigma L^{1(k)} + \sigma^2 L^{2(k)}
  \]
  for $k=0,\ldots,m$, where $L^{0(k)},L^{1(k)}\colon\cR_i\to\wt\cK_j$ are bounded, and $L^{2(k)}(\sigma)$ is $\eps$-regular (in the sense explained in Lemma~\ref{LemmaRegRInverse}).
\end{lemma}
\begin{proof}
  We write $L(\sigma)^{-1}$ as a Neumann series as in~\eqref{EqRegRInverse}, with $\sigma^2$ term $R_2(\sigma)$ satisfying
  \begin{align*}
    \sigma\pa_\sigma R_2(\sigma)&=-L_0^{-1}(\sigma\pa_\sigma L_2(\sigma))L_0^{-1} \\
     &\qquad + \sum_{k=2}^{\infty}(-1)^k\sigma^{k-2}\Biggl(\sum_{j=0}^{k-1}(L_0^{-1}(L_1+\sigma L_2(\sigma)))^j L_0^{-1}\bigl(\sigma L_2(\sigma)+\sigma(\sigma\pa_\sigma L_2(\sigma))\bigr) \\
     &\qquad\qquad \times(L_0^{-1}(L_1+\sigma L_2(\sigma)))^{k-1-j}+(k-2)(L_0^{-1}(L_1+\sigma L_2(\sigma)))^k\Biggr)L_0^{-1}. 
  \end{align*}
  Induction over $m$ gives:
  \begin{align*}
    (\sigma\pa_\sigma)^m R_2(\sigma)&=-L_0^{-1}((\sigma\pa_\sigma)^m L_2(\sigma))L_0^{-1} \\
      &\qquad + \sum_{k=2}^{\infty} (-1)^k\sigma^{k-2}\prod_{j=1}^k\alpha_j(L_0^{-1}(\beta_j L_1+\sigma L_{2,j}(\sigma)))L_0^{-1},
  \end{align*}
  where $\alpha_j$, $\beta_j$ are constants and $L_{2,j}$ satisfies the same hypotheses as $L_2(\sigma)$. The end of the proof is then the same as in Lemma~\ref{LemmaRegRInverse}.  
\end{proof}

This applies in particular to $L_{0 0}(\sigma)$. Using Lemma~\ref{LemmaConInv}, we obtain the following analogue of Proposition~\ref{PropResReg}:
\begin{itemize}
  \item $(\sigma\pa_\sigma)^m\wt R_{0 0}$, $(\sigma\pa_\sigma)^m\wt R'_{0 1}$, $(\sigma\pa_\sigma)^m\wt R'_{1 0}$, $(\sigma\pa_\sigma)^m\wt R_{0 2}$, $(\sigma\pa_\sigma)^m\wt R_{2 0}$ are $\eps$-regular;
  \item $(\sigma\pa_\sigma)^m\wt R_{1 2}$, $(\sigma\pa_\sigma)^m\wt R_{2 1}$, $(\sigma\pa_\sigma)^m\wt R_{2 2}$ have $\eps$-regular expansions up to order one;
  \item $(\sigma\pa_\sigma)^m\wt R_{1 1}$ has an $\eps$-regular expansion up to order two.
\end{itemize}
  
%%%%%%%%%%%%%%%%%%%%%%%%%%%%%%%%%%%%%%%%
\begin{proof}[Proof of Theorem~\ref{ThmRegConLow}]
  Recall the expression~\eqref{EqRegPf} for $L_b^-(\sigma)$; we thus have to consider terms of the form 
  \[
    (\sigma\pa_\sigma)^j\check L_b(\sigma)^{-1}(\sigma\pa_\sigma)^l\wt R_{b,\rm reg}-(\sigma\pa_\sigma)^q\check L_b(\sigma)^{-1}\bigl((\alpha\pa_{\sigma}\check L_b(0)+\beta\tfrac{\sigma}{2}\pa_\sigma^2\check L_b(0))P_{b,1}+\gamma\pa_\sigma^2\check L_b(0)P_{b,2}\bigr)
  \]
  for numerical constants $\alpha,\beta,\gamma$. All these terms behave like the corresponding terms for $j,l,q$ equal to zero, except for the increased requirement on the b-regularity. We then conclude as in the proof of Theorem~\ref{ThmRegL}.
\end{proof}

%%%%%%%%%%%%%%%%%%%%%%%%%%%%%%%%%%%%%%%%%%%%%%%%%%%%%%%%%%%%%%%%%%%%%%%%
\section{Decay estimates}
\label{SD}

We continue using the notation of the previous section, so $L_b$ is the linearized gauge-fixed modified Einstein operator. We shall study the decay of the solution of
\begin{equation}
\label{EqD}
  L_b h=f,\quad f\in\CIc((0,\infty)_{t_*};\Hbext^{s,\ell+2}),
\end{equation}
in $t_*$. In fact, we shall allow more general $f$ lying in \emph{spacetime} Sobolev spaces
\begin{equation}
\label{EqDFn}
  \wt H_\bop^{s,\ell},\quad
  \wt H_{\bop,\rm c}^{s,\ell,k},
\end{equation}
equal to $L^2(t_*^{-1}([0,\infty));|d g_{b_0}|)$ for $s,\ell,k=0$. The index $\ell\in\R$ is the weight in $\rho=r^{-1}$, i.e.\ $\wt H_\bop^{s,\ell}=\rho^\ell\wt H_\bop^s$, likewise for the second (conormal) space. The index $s\in\R$ measures regularity with respect to $\pa_{t_*}$ and stationary b-vector fields on $X$. The index $k\in\N_0$ measures regularity with respect to $\la t_*\ra D_{t_*}$, so $u\in\wt H_{\bop,\rm c}^{s,\ell,k}$ if and only $(\la t_*\ra D_{t_*})^j u\in\wt H_\bop^{s,\ell}$, $j=0,\ldots,k$. We stress that all elements of these spaces are supported in $t_*\geq 0$. (The value $0$ here is of course artificial; the point is that $t_*$ has a finite lower bound on the support.)

\begin{thm}
\label{ThmD}
  Let $\ell\in(-\tfrac32,-\half)$, $\eps\in(0,1)$, $\ell+\eps\in(-\half,\half)$, and $s>\tfrac72+m$, $m\in\N_0$. Let $h$ denote the solution of equation~\eqref{EqD}. Then there exists a generalized zero mode $\hat h\in\wh\cK_b$ (see Theorem~\ref{ThmCD0Modes}) such that
  \[
    h=\hat h+\tilde h,
  \]
  and so that the remainder $\tilde h$ obeys the decay
  \begin{subequations}
  \begin{equation}
  \label{EqDSobEst}
    \|\tilde h\|_{\la t_*\ra^{-3/2+\eps}\wt H_{\bop,\rm c}^{s-2,\ell+\eps-1,m}} \lesssim \|f\|_{\la t_*\ra^{-7/2+\eps}\wt H_{\bop,\rm c}^{s,\ell+2,m}}.
  \end{equation}
  If $m\geq 1$, then we also have
  \begin{equation}
  \label{EqDPwise0}
    \|(\la t_*\ra D_{t_*})^{m-1}\tilde h\|_{\la t_*\ra^{-2+\eps}L^\infty(\R_{t_*};\Hbext^{s-2-m,\ell+\eps-1})} \lesssim \|f\|_{\la t_*\ra^{-7/2+\eps}\wt H_{\bop,\rm c}^{s,\ell+2,m}}.
  \end{equation}
  In particular, for $m\geq 1$, $\N_0\ni j<s-(\tfrac92+m)$, we have pointwise decay
  \begin{equation}
  \label{EqDPwise}
    |(\la t_*\ra D_{t_*})^{m-1}V^j \tilde h| \lesssim \la t_*\ra^{-2+\eps} r^{-1/2-\ell-\eps} \|f\|_{\la t_*\ra^{-7/2+\eps}\wt H_{\bop,\rm c}^{s,\ell+2,1}},
  \end{equation}
  \end{subequations}
  where $V^j$ is any up to $j$-fold composition of $\pa_{t_*}$, $r\pa_r$, and rotation vector fields.

  In the notation of~\eqref{EqRZeroModes} and Definition~\ref{DefRBreve}, the leading order term $\hat h$ can be expressed in terms of $f$ in the following manner: we have
  \[
    \hat h = (t_*h_s + \breve h_s) + h'_s + h_v + h''_s,\quad h_s,\,h'_s,\,h''_s\in\cK_{b,s},\ h_v\in\cK_{b,v},
  \]
  where $h_s,h'_s,h_v$ are determined by~\eqref{EqRSing1}--\eqref{EqRBarh} for $\int_{\R_{t_*}} f(t_*)\,d t_*$ in place of $f$, and $h''_s$ is the unique element in $\cK_{b,s}$ for which there exists $h''_v\in\cK_{b,v}$ with $k_b((h''_s,h''_v),-)=-\int_{\R_{t_*}}t_* f(t_*)\,d t_*$ as elements of $(\cK_b^*)^*$.
\end{thm}

For any fixed $\ell\in(-\tfrac32,-\half)$, one has a range of choices $\eps\in(-\half-\ell,1)$; smaller values of $\eps$ give stronger $t_*$-decay at the expense of $r$-growth. It is thus useful to resolve the asymptotic regimes \begin{enumerate*} \item $t_*\to\infty$, $r$ bounded, \item $r\to\infty$, $t_*$ bounded. \end{enumerate*} This is accomplished by lifting to the blow-up $[\ol{\R_{t_*}}\times X;\{\infty\}\times\pa X]$; see Figure~\ref{FigDBlowup}. Its boundary hypersurfaces are:
\begin{enumerate}
\item the lift of $\ol{\R_{t_*}}\times\pa X$, called (by a mild abuse of notation) null infinity $\scri^+$, with interior parameterized by $t_*\in\R$ and the polar coordinate $\omega\in\Sph^2$;
\item the front face, denoted $\iota^+$, with interior parameterized by $r/t_*\in(0,\infty)$ and $\omega\in\Sph^2$; this is a resolution of future timelike infinity;
\item the lift of $\{\infty\}\times X$, called the `Kerr face' $\cK$, with interior parameterized by $r\in(r_-,\infty)$ and $\omega\in\Sph^2$.
\end{enumerate}
The estimate~\eqref{EqDPwise} then implies pointwise decay at the inverse polynomial rate $\tfrac52+\ell-$ at $\cK$ (taking $\eps=-\half-\ell+$), $\tfrac52+\ell-$ at $\iota^+$, and $\tfrac32+\ell-$ at $\scri^+$ (taking $\eps=1-$), so
\begin{equation}
\label{EqDPwise2}
  |(\la t_*\ra D_{t_*})^{m-1}V^j \tilde h(t_*,r,\omega)| \lesssim\la t_*\ra^{-5/2-\ell+} \left(\frac{\la t_*\ra}{\la r\ra+\la t_*\ra}\right)^{3/2+\ell-}.
\end{equation}
One can similarly condense the estimates~\eqref{EqDSobEst}--\eqref{EqDPwise0}; for instance, the former implies
\begin{equation}
\label{EqDSobEst2}
  \tilde h \in \la t_*\ra^{-2-\ell+}\left(\frac{\la t_*\ra}{\la r\ra+\la t_*\ra}\right)^{3/2+\ell-}\wt H_{\bop,\rm c}^{s-2,0,m}.
\end{equation}

\begin{figure}[!ht]
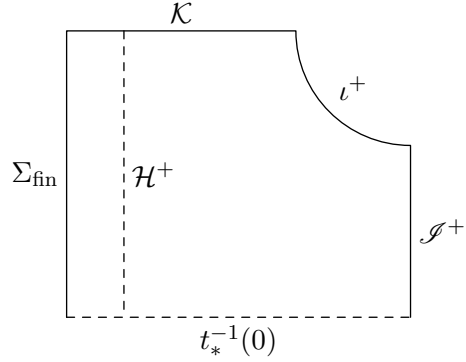

  \centering
  \inclfig{FigDBlowup}
  \caption{The part of the blowup $[\ol{\R_{t_*}}\times X;\{\infty\}\times\pa X]$ on which $t_*\geq 0$. Shown are the front face $\iota^+$, (the lift of) null infinity $\scri^+$, and the `Kerr face' $\cK$, as well as the event horizon $\cH^+$ and the final hypersurface $\Sigma_{\rm fin}$ beyond the black hole event horizon, see~\eqref{EqK0SurfFin}.}
\label{FigDBlowup}
\end{figure}

\begin{rmk}
\label{RmkDPrice}
  Even for $f$ with compact spacetime support, Theorem~\ref{ThmD} only assures pointwise $t^{-2+}$ time decay for bounded $r$. Price's law for scalar waves on the other hand asserts $t^{-3}$ decay \cite{PriceLawI,TataruDecayAsympFlat}. (For gravitational perturbations, Price's law predicts even faster decay rates \cite{PriceLawII}, which one cannot expect to hold however in the general setting of Theorem~\ref{ThmD}, as it concerns general symmetric 2-tensor valued waves, not merely those satisfying the linearized Einstein equation.) A comparison as far as the relevant low frequency behavior of the resolvent is concerned (though in a simpler setting) is \cite[Corollary~1.3]{GuillarmouHassellSikoraResIII}: it asserts $t^{-3}$ asymptotics for scalar waves on a space asymptotic to a cone with cross section the standard 2-sphere, i.e.\ to Euclidean space, while upon perturbing this cross section, one can only expect $t^{-2+\eps}$; the latter is what we prove here. The asymptotic behavior here being the better one, a proof of $t^{-3}$ decay in Theorem~\ref{ThmD} would require expanding the resolvent at $\sigma=0$ to one order more, showing that one has a term linear in $\sigma$ (which we just barely fail to capture in Theorem~\ref{ThmRegL}), and obtaining a $\cO(|\sigma|^2)$ remainder.
\end{rmk}

\begin{proof}[Proof of Theorem~\ref{ThmD}]
  As in equation~\eqref{EqIIFT}, we start with the integral representation
  \begin{equation}
  \label{EqDIntRep}
    h(t_*) = \frac{1}{2\pi}\int_{\Im\sigma=C} e^{-i\sigma t_*}\wh{L_b}(\sigma)^{-1}\hat f(\sigma)\,d\sigma
  \end{equation}
  for any $C>0$. Two important observations are: \begin{enumerate*} \item different values of $C$ produce the same result since the integrand is holomorphic in $\sigma$ with values in $\Hbext^{s,\ell}$, and with norm decaying superpolynomially as $|\Re\sigma|\to\infty$ with $\Im\sigma>0$ bounded; \item the tensor $h$ defined by~\eqref{EqDIntRep} has support in $t_*\geq -T_0$ for some ($f$-independent) $T_0$, which follows from the Paley--Wiener theorem and the fact that the large parameter (in $\sigma$) estimates of $\wh{L_b}(\sigma)^{-1}$ are uniform as $\Im\sigma\to+\infty$.\footnote{Alternatively, by varying $C>0$, one sees that~\eqref{EqDIntRep} defines a solution of $L_b h=f$ which decays superexponentially at $t_*\to-\infty$; this implies its vanishing for large negative $t_*$ by a simple energy estimate.} \end{enumerate*}
  
  We shift the contour by letting $C\to 0+$. Fixing a frequency cutoff $\chi\in\CIc(\R)$ with $\chi=1$ on $[-1,1]$, we split
  \[
    \hat f(\sigma)=\chi(\Re\sigma)\hat f(\sigma)+(1-\chi(\Re\sigma))\hat f(\sigma)=:\hat f_1(\sigma)+\hat f_2(\sigma).
  \]
  Then for $\sigma\neq 0$, and using Theorem~\ref{ThmR}, we write
  \begin{align*}
    \hat h(\sigma)&=\hat h_{\rm reg}(\sigma) + \hat h_{\rm sing}(\sigma), \\
    &\qquad \hat h_{\rm reg}(\sigma):=L_b^-(\sigma)\hat f_1(\sigma) + \wh{L_b}(\sigma)^{-1}\hat f_2(\sigma), \quad
    \hat h_{\rm sing}(\sigma):=P_b(\sigma)\hat f_1(\sigma);
  \end{align*}
  therefore, $h(t_*)=h_{\rm reg} + h_{\rm sing}$, where
  \[
    h_{\rm reg} = \frac{1}{2\pi}\int_\R e^{-i\sigma t_*}\hat h_{\rm reg}(\sigma)\,d\sigma, \qquad h_{\rm sing} = \lim_{C\to 0+} \frac{1}{2\pi}\int_{\R+i C} e^{-i\sigma t_*}\hat h_{\rm sing}(\sigma)\,d\sigma.
  \]

  %%%%%%%%%%%%%%%%%%%%%%%%%%%%%%%%%%%%%%%%
  \pfstep{Decay of the regular part.} A fortiori, we have $\hat f_1(\sigma)\in H^{3/2-\eps}(\R_\sigma;\Hbext^{s,\ell+2})$ for any $\eps\in\R$. Since $H^{3/2-\eps}(\R)$ is an algebra for $\eps<1$, we have
  \[
    L_b^-(\sigma)\hat f_1(\sigma) \in H^{3/2-\eps}\bigl((-2,2);\Hbext^{s+1-\max(\eps,1/2),\ell+\eps-1}\bigr)
  \]
  by Theorem~\ref{ThmRegL}, with norm bounded by $\|\la t_*\ra^{3/2-\eps}f\|_{L^2(\R_{t_*};\Hbext^{s,\ell+2})}$. This can be generalized by means of Theorem~\ref{ThmRegConLow}, which gives, for $k\leq m$,
  \begin{equation}
  \label{EqDReg1}
    (\sigma\pa_\sigma)^k\bigl(L_b^-(\sigma)\hat f_1(\sigma)\bigr) \in H^{3/2-\eps}\bigl((-2,2);\Hbext^{s+1-k-\max(\eps,1/2),\ell+\eps-1}\bigr),
  \end{equation}
  with norm bounded by $\sum_{j\leq k}\|(t_* D_{t_*})^j\la t_*\ra^{3/2-\eps}f\|_{L^2(\R_{t_*};\Hbext^{s-j,\ell+2})}\lesssim\|f\|_{\la t_*\ra^{-3/2+\eps}\wt H_{\bop,\rm c}^{s,\ell+2,k}}$.

  Turning to the high frequency part of $\hat h_{\rm reg}(\sigma)$, we note that Proposition~\ref{PropRegHigh} implies
  \[
    \wh{L_b}(\sigma)^{-1} \in W^{\infty,m}\bigl(\R;\cL(\bar H_{\bop,\la\sigma\ra^{-1}}^{s,\ell+2}, \la\sigma\ra^m \bar H_{\bop,\la\sigma\ra^{-1}}^{s-m,\ell})\bigr),\qquad
    \sigma\in\R\setminus[-1,1],\quad m\leq k.
  \]
  Therefore, for $\eps\in(0,1)$, and taking $m=2$, we have
  \[
    \wh{L_b}(\sigma)^{-1}\hat f_2(\sigma) \in H^{3/2-\eps}\bigl(\R;\la\sigma\ra^{-s+2}\bar H_{\bop,\la\sigma\ra^{-1}}^{s-2,\ell}\bigr)
  \]
  with norm bounded by a constant times
  \[
    \|\hat f_2\|_{H^{3/2-\eps}(\R;\la\sigma\ra^{-s}\bar H_{\bop,\la\sigma\ra^{-1}}^{s,\ell+2})} \lesssim \|f\|_{\la t_*\ra^{-3/2+\eps}\wt H_\bop^{s,\ell+2}}.
  \]
  More generally, we have
  \begin{equation}
  \label{EqDReg2}
    (\sigma\pa_\sigma)^k\bigl(\wh{L_b}(\sigma)^{-1}\hat f_2(\sigma)\bigr) \in H^{3/2-\eps}\bigl(\R;\la\sigma\ra^{-s+2+k}H_{\bop,\la\sigma\ra^{-1}}^{s-2-k,\ell}\bigr),
  \end{equation}
  with norm bounded by $\sum_{j\leq k}\|(t_* D_{t_*})^jf\|_{\la t_*\ra^{-3/2+\eps}\wt H_\bop^{s-j,\ell+2}}\lesssim\|f\|_{\la t_*\ra^{-3/2+\eps}\wt H_{\bop,\rm c}^{s,\ell+2,k}}$.

  Upon taking the inverse Fourier transform, the memberships~\eqref{EqDReg1} and \eqref{EqDReg2} imply
  \begin{equation}
  \label{EqDReg3}
    \|h_{\rm reg}\|_{\la t_*\ra^{-3/2+\eps}\wt H_{\bop,\rm c}^{s-2,\ell+\eps-1,k}} \lesssim \|f\|_{\la t_*\ra^{-3/2+\eps}\wt H_{\bop,\rm c}^{s,\ell+2,k}}.
  \end{equation}
  Note that for $m=1$, we can integrate $D_{t_*}h_{\rm reg}=t_*^{-1}(t_*D_{t_*} h_{\rm reg})$ from $t_*=\infty$ to recover $h_{\rm reg}$; this gives a $\half$ improvement of the $L^\infty$ decay rate since for $h'\in L^2(\R_{t_*})$, we have
  \[
    \left|\int_{t_*}^\infty \la s\ra^{-1}\cdot\la s\ra^{-3/2+\eps} h'(s)\,d s\right| \leq \left(\int_{t_*}^\infty \la s\ra^{-5+2\eps}\,d s\right)^{1/2} \|h'\|_{L^2} \lesssim \la t_*\ra^{-2+\eps}
  \]
  for $t_*\geq 0$; therefore, we have stronger (by $\half$) $t_*$-decay
  \begin{equation}
  \label{EqDRegLinfty}
    (\la t_*\ra D_{t_*})^{k-1}h_{\rm reg} \in \la t_*\ra^{-2+\eps}L^\infty(\R_{t_*};\Hbext^{s-2-k,\ell+\eps-1}),\quad 1\leq k\leq m,
  \end{equation}
  with norm bounded by the right hand side of~\eqref{EqDReg3}.

  %%%%%%%%%%%%%%%%%%%%%%%%%%%%%%%%%%%%%%%% 
  \pfstep{Asymptotic behavior of the singular part.} Turning to $h_{\rm sing}$, we write
  \begin{equation}
  \label{EqDSingDecomp}
    \hat f_1(\sigma) = \chi(\sigma)(f^{(0)}-i\sigma f^{(1)}) + \sigma^2\hat f_3(\sigma),
  \end{equation}
  where $f^{(0)}=\hat f_1(0)$ and $f^{(1)}=\pa_\sigma\hat f_1(0)$ satisfy, in view of $H^{3/2+}(\R)\hra\cC^1(\R)$,
  \[
    \|f^{(j)}\|_{\Hbext^{s,\ell+2}} \lesssim \|f\|_{\la t_*\ra^{-3/2-\eta}L^2(\R_{t_*};\Hbext^{s,\ell+2})}\quad j=0,1,\ \forall\,\eta>0;
  \]
  furthermore, $\hat f_3(\sigma)$ is compactly supported in $\sigma$, has the same regularity in $\sigma$ as $\hat f_1(\sigma)$ away from $\sigma=0$, while near $\sigma=0$ it loses $2$ derivatives:\footnote{This follows from the following observation: if $f\in H^\alpha(\R_\sigma)$, $\alpha>\half+k$, $k\in\N_0$, and $f(0)=\dots=f^{(k)}(0)=0$, then $\sigma^{-k}f\in H^{\alpha-k}(\R)$. This in turn follows from the elementary case $k=0$ by induction.}
  \[
    \|\hat f_3\|_{H^r(\R;\Hbext^{s,\ell+2})} \lesssim \|\hat f_1\|_{H^{r+2}(\R;\Hbext^{s,\ell+2})}\quad \forall\,r>-\half.
  \]

  Let us write $P_b(\sigma)=\sigma^{-2}P_{b,2}+\sigma^{-1}P_{b,1}$ as in~\eqref{EqRegPfSing}, with $P_{b,j}\colon\Hbext^{s-1,\ell+2}\to\Hbext^{\infty,-1/2-}$. The contribution of $\hat f_3$ to $h_{\rm sing}$ is then
  \[
    P_b(\sigma)\bigl(\sigma^2\hat f_3(\sigma)\bigr) \in H^{3/2-\eps}(\R;\Hbext^{\infty,-1/2-})
  \]
  when $\hat f_1\in H^{7/2-\eps}$. More generally, we have (cf.\ the estimate following~\eqref{EqDReg1})
  \begin{equation}
  \label{EqDSing1}
    \|(\sigma\pa_\sigma)^k\bigl(P_b(\sigma)\bigl(\sigma^2\hat f_3(\sigma)\bigr)\bigr)\|_{H^{3/2-\eps}(\R;\Hbext^{\infty,1/2-})} \lesssim \|f\|_{\la t_*\ra^{-7/2+\eps}\wt H_{\bop,\rm c}^{s,\ell+2,k}}.
  \end{equation}

  It remains to analyze the first term in~\eqref{EqDSingDecomp}. We have
  \[
    \lim_{C\to 0}\frac{1}{2\pi}\int_{\R+i C} e^{-i\sigma t_*}\chi(\Re\sigma)\sigma^{-k}\,d\sigma = \cF^{-1}\bigl(\chi(\sigma)(\sigma+i 0)^{-k}\bigr),
  \]
  which equals $\cF^{-1}((\sigma+i 0)^{-k})$ plus a term which is Schwartz as $|t_*|\to\infty$. Therefore, modulo a Schwartz function in $t_*$ with values in $\Hbext^{\infty,-1/2-}$, we have
  \[
    \cF^{-1}\left(P_b(\sigma+i 0)\bigl(\chi(\sigma)(f^{(0)}-i\sigma f^{(1)})\bigr)\right) \equiv (t_*h_s+\breve h_s) +h'_s+h_v + h''_s =: \hat h,
  \]
  where $h_s,h'_s,h_v$ are given by~\eqref{EqRSing1}--\eqref{EqRBarh} with $f^{(0)}$ in place of $f$, while $h''_s\in\cK_{b,s}$ is the unique element for which $k_b((h''_s,h''_v),-)=f^{(1)}$, as elements of $(\cK_b^*)^*$, for some $h''_v\in\cK_{b,v}$. 

  Altogether, we thus have, for any $k\in\N_0$,
  \[
    \|h_{\rm sing}-\hat h\|_{\la t_*\ra^{-3/2+\eps}\wt H_{\bop,\rm c}^{\infty,-1/2-,k}} \lesssim \|f\|_{\la t_*\ra^{-7/2+\eps}\wt H_{\bop,\rm c}^{s,\ell+2,k}}.
  \]
  Together with~\eqref{EqDReg3}, and using the argument leading to~\eqref{EqDRegLinfty} for improved pointwise (in $t_*$) decay, this proves the theorem. (For the pointwise decay estimate~\eqref{EqDPwise}, recall that $\Hbext^{s,\ell}\hra\la r\ra^{-3/2-\ell}L^\infty$ for $s>\tfrac32$ by Sobolev embedding, cf.\ \eqref{EqBHbSobEmb}.)
\end{proof}

%%%%%%%%%%%%%%%%%%%%%%%%%%%%%%%%%%%%%%%%%%%%%%%%%%%%%%%%%%%%%%%%%%%%%%
\section{Proof of linear stability}
\label{SPf}

\textit{We continue to denote the linearized modified gauge-fixed Einstein operator on the Kerr spacetime $(M^\circ,g_b)$ by $L_b$ as in~\eqref{EqRgamma}--\eqref{EqRLb}.}

Theorem~\ref{ThmIBaby} concerns the initial value problem for the linearized Einstein equation for which the Cauchy surface is equal to $t^{-1}(0)$ for large $r$. We now show, using arguments from~\cite[\S4]{HintzVasyMink4}, how to reduce this problem to Theorem~\ref{ThmD}. In fact, we shall first consider initial value problems for the operator $L_b$, see Theorem~\ref{ThmPfIVP}, and reduce the linear stability of the Kerr metric to a special case of this, see Theorem~\ref{ThmPfKerr}.

We denote by $\ft\in\CI(M^\circ)$ a function of the form $\ft=t_*+F(r)$ which satisfies $\ft=t$ for $r\geq 4\bhm_0$, and so that $d\ft$ is future timelike on $M^\circ$ with respect to $g_{b_0}$, hence for $g_b$ when $b$ is close to $b_0$. We define
\[
  \Sigma_0^\circ := \ft^{-1}(0),
\]
which is a (spacelike with respect to $g_b$) Cauchy surface. Identifying the region $r\geq 4\bhm_0$ in $\Sigma_0^\circ$ with the corresponding subset of $\R^3$, we can compactify $\Sigma_0^\circ$ at infinity to the manifold (with two boundary components) $\Sigma_0$. We denote by $\wt{\Tsc^*}\Sigma_0=\Tsc^*\Sigma_0 \oplus \R\,d\ft$ the spacetime scattering cotangent bundle, which for large $r$ is spanned over $\CI(\Sigma_0)$ by $d t$ and $d x^1,d x^2,d x^3$, with $(x^1,x^2,x^3)$ denoting standard coordinates on $\R^3$.

\begin{thm}
\label{ThmPfIVP}
  Let $\alpha\in(0,1)$ and $s>\tfrac{13}{2}+m$, $m\in\N_0$. Suppose
  \[
    h_0 \in \Hbext^{s,-1/2+\alpha}(\Sigma_0;S^2\,\wt{\Tsc^*}\Sigma_0),\quad
    h_1 \in \Hbext^{s-1,1/2+\alpha}(\Sigma_0;S^2\,\wt{\Tsc^*}\Sigma_0).
  \]
  Then the solution $h$ of the initial value problem
  \[
      L_b h = 0, \qquad
      \bigl(h|_{\Sigma_0},\,(\cL_{\pa_\ft}h)|_{\Sigma_0}\bigr) = (h_0,\,h_1).
  \]
  has the following asymptotic behavior:
  \begin{enumerate}
  \item\label{ItPfIVP1} in $t_*\geq 0$, we can write $h=\hat h+\tilde h$, where $\hat h\in\wh\cK_b$ is a generalized zero mode of $L_b$ (see Theorem~\ref{ThmCD0Modes}), and where the remainder $\tilde h$ satisfies the following decay in $t_*\geq 0$: for $\eps\in(0,1)$ with $\alpha+\eps>1$, we have
    \begin{subequations}
    \begin{equation}
    \label{EqPfIVP1Hb}
      \|\tilde h\|_{\la t_*\ra^{-3/2+\eps}\wt H_{\bop,\rm c}^{s-5,-5/2+\alpha+\eps,m}} \lesssim \|h_0\|_{\Hbext^{s,-1/2+\alpha}}+\|h_1\|_{\Hbext^{s-1,1/2+\alpha}}.
    \end{equation}
    For $m\geq 1$, we moreover have the $L^\infty$ bound
    \begin{equation}
    \label{EqPfIVP1Linfty}
      \|(\la t_*\ra D_{t_*})^{m-1}\tilde h\|_{\la t_*\ra^{-2+\eps}L^\infty(\R_{t_*};\Hbext^{s-5-m,-5/2+\alpha+\eps})} \lesssim \|h_0\|_{\Hbext^{s,-1/2+\alpha}}+\|h_1\|_{\Hbext^{s-1,1/2+\alpha}}.
    \end{equation}
    For $\N_0\ni j<s-(\tfrac{13}{2}+m)$, we have the pointwise bound
    \begin{equation}
    \label{EqPfIVP1Pt}
      |(\la t_*\ra D_{t_*})^{m-1} V^j \tilde h(t_*,r,\omega)| \lesssim \la t_*\ra^{-1-\alpha+}\left(\frac{\la t_*\ra}{\la t_*\ra+\la r\ra}\right)^{\alpha-},
    \end{equation}
    \end{subequations}
    where $V^j$ is any up to $j$-fold composition of the vector fields $\pa_{t_*}$, $r\pa_r$, and rotation vector fields.
  \item\label{ItPfIVP2} in $\ft\geq 0$, $t_*\leq 0$, we have $|\tilde h|\lesssim r^{-1}(1+|t_*|)^{-\alpha}$.
  \end{enumerate}
\end{thm}

The estimates~\eqref{EqPfIVP1Hb} and \eqref{EqPfIVP1Linfty} can be condensed by working on a blow-up of $\ol{\R_{t_*}}\times X$ at the corner $\{t_*=\infty\}\times\pa_+X$ as in~\eqref{EqDSobEst2}.

\begin{rmk}
  Since $1\in\Hbext^{\infty,-3/2-}(\Sigma_0)$ and, conversely, $\Hbext^{\infty,-3/2+}(\Sigma_0)\hra L^\infty(\Sigma_0)$, one sees that the decay for $h_0$, as well as for its $r\pa_r$ and spherical derivatives, is (almost) equivalent to pointwise $r^{-1-\alpha}$ bounds.\footnote{The `almost' is due to the small (disregarding derivative losses due to Sobolev embedding) difference of $L^\infty$ and weighted $L^2$ control.} Likewise, the assumptions on $h_1$ are essentially pointwise $r^{-2-\alpha}$ bounds.
\end{rmk}

\begin{rmk}
\label{RmkPfBetter}
  The asymptotic behavior in $t_*\leq 0$ can be described in great detail, see \cite{HintzVasyMink4} for results in the nonlinear setting; the results for linear waves here are straightforward to obtain using energy methods (see also Lemma~\ref{LemmaPfIVP} below), and the statement~\eqref{ItPfIVP2} above is merely the simplest pointwise bound one can prove. As in \cite{HintzVasyMink4}, one can moreover show that $h$ in fact has an $r^{-1}$ leading order term at null infinity $\scri^+$, viewed as a boundary hypersurface of a suitable compactification of $M^\circ$ (which in $r/t_{\chi_0}>\eps>0$ is given by ${}^\bhm M$ in the notation of the reference); we recall the argument in the proof of Lemma~\ref{LemmaPfIVP} below when restricting to the region $t_*\leq C$ for any fixed $C\in\R$.
\end{rmk}

\begin{rmk}
\label{RmkPfDecReg}
  While we work almost at the sharp level of decay, we do not strive to optimize the regularity assumptions here; the extra regularity assumed on $h_0,h_1$ comes from a loss of regularity in an argument below on the integration along approximate (radial) characteristics of $L_b$ used to get a sufficiently precise description of $h$ at $\scri^+$.
\end{rmk}

\begin{lemma}
\label{LemmaPfIVP}
  Let $\alpha,b,s,h_0,h_1,h$ be as in Theorem~\ref{ThmPfIVP}. Let $T_1<T_2$, and let $\chi\in\CI(\R)$ be such that $\supp\chi\subset[T_1,\infty)$ and $\supp(1-\chi)\subset(-\infty,T_2]$. Then $L_b(\chi h)\in\wt H_{\bop,\rm c}^{s-3-k,1/2+\alpha,k}$ for all $k\in\N_0$.
\end{lemma}
\begin{proof}
  Since $L_b(\chi h)=[L_b,\chi]h$ has support in $t_*^{-1}([T_1,T_2])$ (so regularity with respect to $\la t_*\ra D_{t_*}$ and $D_{t_*}$ is equivalent), it suffices to prove the conclusion for $k=0$. Moreover, local (in spacetime) existence and regularity theory for the wave equation imply $h\in H^s$ for $\ft\geq 0$, $t_*\leq C$, $r\leq C$ for any $C$; this implies that $[L_b,\chi]\in H^{s-1}$ in such compact sets. Therefore, it suffices to work in an arbitrarily small neighborhood $r>R_0\gg 1$ of infinity, which we shall do from now on.

  Aiming to apply certain results of \cite[\S\S3--5]{HintzVasyMink4}, write $b=(\bhm,\bha)$, and let $b_1=(\bhm,0)$; put $r_*=r+2\bhm_1\log(r-2\bhm_1)$. As in~\cite[\S2.1, \S3.1]{HintzVasyMink4}, we introduce the null coordinates
  \[
    x^0=t+r_*,\quad x^1=t-r_*,
  \]
  with respect to which we have $g_{b_1}=(1-\tfrac{2\bhm_1}{r})d x^0\,d x^1-r^2\slg$, and
  \[
    2\pa_0\equiv 2\pa_{x^0}=\pa_t+(1-\tfrac{2\bhm_1}{r})\pa_r,\quad
    2\pa_1\equiv 2\pa_{x^1}=\pa_t-(1-\tfrac{2\bhm_1}{r})\pa_r.
  \]
  Let $x^2,x^3$ denote local coordinates on $\Sph^2$, and denote spherical indices by $c,d=2,3$; let $\pa_c\equiv\pa_{x^c}$. Considering then the Kerr metric, we compute using the form~\eqref{EqKaMetric} of $g_b$ in $r\gg 1$ and on $X$ (the compactification of $t_*^{-1}(0)$)
  \begin{align*}
    &g_b(\pa_0,\pa_0),\ (g_b-g_{b_1})(\pa_0,\pa_1),\ g_b(\pa_0,r^{-1}\pa_c), \\
    &\quad g_b(\pa_1,\pa_1),\ g_b(\pa_1,r^{-1}\pa_c),\ (g_b-g_{b_1})(r^{-1}\pa_b,r^{-1}\pa_c) \in \rho^2\CI.
  \end{align*}
  Thus, in the language of \cite[Definition~3.1]{HintzVasyMink4}, $g_b$ differs from $g_{b_1}$ by a correction $g_b-g_{b_1}$ (which is denoted $h$ in the reference, but which is different from $h$ here) that has vanishing leading order terms. Therefore, by \cite[Lemma~3.8]{HintzVasyMink4}, and recalling that in the present paper we do not have constraint damping for large $r$ (thus $\gamma=0$ in the formulas in the reference), the operator $L_b$ is equal to the scalar wave operator for the Schwarzschild metric $g_{b_1}$, tensored with the identity, plus error terms,
  \[
    \rho^{-3}L_b\rho = -2\rho^{-2}\pa_0\pa_1 - \slDelta + R,
  \]
  where $R\in\rho\wt\Diff{}_\bop^2+\wt\Diff{}_\bop^1$ (acting on sections of $\wt{\Tsc^*}X$), where, roughly speaking, $\wt\Diff{}_\bop^k$ consists of up to $k$-fold products of the vector fields $x_0\pa_0\approx r(\pa_t+\pa_{r_*})$, $x_1\pa_1\approx(r_*-t)(\pa_t-\pa_{r_*})$, and rotation vector fields $\pa_a$. (Switching to $\rho^{-3}L_b\rho$ is related to the Friedlander rescaling for the scalar wave equation \cite{FriedlanderRadiation}, cf.\ \cite[\S1.1.1]{HintzVasyMink4}.)

  At this point, the conclusion of the lemma can be seen as follows: the asymptotic behavior of smooth solutions $h$ of $\rho^{-3}L_b\rho(\rho^{-1}h)=0$ at $\scri^+$ (meaning: for bounded $t_*$ and as $r\to\infty$) is given by $\rho\CI=r^{-1}\CI$, just like for scalar waves on Minkowski or Schwarzschild/Kerr spacetimes, plus terms with more decay, namely $\cO(r^{-1-\alpha})$. Now, Lemma~\ref{LemmaOpLinFT} implies that
  \begin{equation}
  \label{EqPfIVPComm}
    [L_b,\chi(t_*)]=2\rho\chi'(t_*)(\rho\pa_\rho-1)+\rho^2\Diffsc^1+\rho^2\CI\pa_{t_*};
  \end{equation}
  but this maps $\rho\CI$ into $\rho^3\CI$, the leading order term $r^{-1}$ being annihilated by the first summand in~\eqref{EqPfIVPComm}; and $\cO(r^{-1-\alpha})$ gets mapped to $\cO(r^{-2-\alpha})$, hence the stated decay rate for $L_b(\chi h)$.

  More precisely, \cite[Proposition~4.8]{HintzVasyMink4} with $a_0=\alpha$ and $a_I<0$ arbitrary (in the present setting it suffices to use the simplest form of the energy estimates there, namely one can use the vector field multiplier of \cite[Lemma~4.4]{HintzVasyMink4}) shows that near $\scri^+$, $h$ lies in $\wt H_\bop^{s,-1/2-}$ (permitting $r^{-1}$ asymptotics); rewriting the PDE for $h:=\rho u$ as $\rho^{-2}\pa_0\pa_1 u=(\slDelta-R)u\in\wt H_\bop^{s-2,-1/2-}$ and integrating along the approximate characteristics $\pa_0$ and $\pa_1$ as in \cite[\S5.1]{HintzVasyMink4} (see also the discussion around \cite[Equation~(1.16)]{HintzVasyMink4}) shows that in fact
  \[
    h = \rho H^{s-2}(\R_{t_*}\times\Sph^2) + \wt H_\bop^{s-2,-1/2+\alpha}.
  \]
  Thus, $L_b(\chi(t_*)h)\in\wt H_\bop^{s-3,1/2+\alpha}$, as claimed.
\end{proof}

\begin{proof}[Proof of Theorem~\ref{ThmPfIVP}]
  With $\chi$ as in Lemma~\ref{LemmaPfIVP}, note that
  \[
    L_b((1-\chi)h) = L_b h - L_b(\chi h) = -L_b(\chi h).
  \]
  Since $s-3>\tfrac72+m$, the estimates~\eqref{EqPfIVP1Hb} and \eqref{EqPfIVP1Linfty} now follow from Theorem~\ref{ThmD}, with $\ell+2=\half+\alpha$, so $\ell=-\tfrac32+\alpha$. The pointwise estimate~\eqref{EqPfIVP1Pt} then follows from~\eqref{EqDPwise2}.
\end{proof}

For the linear stability statement, recall that the constraint equations for a Riemannian metric $\gamma$ and a symmetric 2-tensor $k$ on $\Sigma_0^\circ$ take the form
\begin{equation}
\label{EqPfConstraints}
  \cC(\gamma,k) := \bigl(R_\gamma+(\tr_\gamma k)^2-|k|_\gamma^2,\ \delta_\gamma k+d\,\tr_\gamma k\bigr) = 0.
\end{equation}
Given a Lorentzian metric $g$ of signature $(+,-,-,-)$ on $M^\circ$, denote its initial data by
\[
  \tau(g) := (\gamma,k),
\]
where $\gamma$ is the pullback of $-g$ to $\Sigma_0^\circ$ (the minus sign making $\gamma$ into a \emph{positive} definite Riemannian metric), and $k$ is the second fundamental form of $\Sigma_0^\circ\subset(M^\circ,g)$. Let us in particular denote the initial data of the Kerr metric $g_b$, with $b$ near $b_0$, by
\[
  (\gamma_b,k_b) := \tau(g_b).
\]
Recall the gauge 1-form $\Ups_b(g):=\Ups(g;g_b)$ in the notation of Definition~\ref{DefOpGauge}.

\begin{thm}
\label{ThmPfKerr}
  Let $\alpha\in(0,1)$, and let $s>\tfrac{13}{2}+m$, $m\in\N_0$. Suppose that the tensors
  \[
    \dot\gamma \in \Hbext^{s,-1/2+\alpha}(\Sigma_0;S^2\,\Tsc^*\Sigma_0), \quad
    \dot k \in \Hbext^{s-1,1/2+\alpha}(\Sigma_0;S^2\,\Tsc^*\Sigma_0)
  \]
  satisfy the linearization of the constraint equations around the initial data $(\gamma_b,k_b)$ of a Kerr metric: $D_{(\gamma_b,k_b)}\cC(\dot\gamma,\dot k)=0$. Then there exists a solution $h$ of the initial value problem
  \[
    D_{g_b}\Ric(h) = 0,\qquad
    D_{g_b}\tau(h) = (\dot\gamma,\dot k),
  \]
  satisfying the gauge condition $D_{g_b}\Ups_b(h)=-\delta_{g_b}\sfG_{g_b}h=0$, which has the asymptotic behavior stated in Theorem~\ref{ThmPfIVP}.
\end{thm}

Theorem~\ref{ThmIBaby} is an immediate consequence of the estimate~\eqref{EqPfIVP1Pt} when we take $s=8>\tfrac{13}{2}+m$ with $m=1$, in view of the explicit description of $\hat h$ as a generalized zero mode of $L_b$ as described by Theorem~\ref{ThmCD0Modes}: $\hat h$ is a linearized Kerr metric plus a pure gauge term.

\begin{proof}[Proof of Theorem~\ref{ThmPfKerr}]
  We claim that there exist
  \begin{equation}
  \label{EqPfKerr1}
    h_0\in\Hbext^{s,-1/2+\alpha},\ h_1\in\Hbext^{s-1,1/2+\alpha},
  \end{equation}
  such that
  \begin{equation}
  \label{EqPfKerr2}
    D_{g_b}\tau(h_0+\ft h_1)=(\dot\gamma,\dot k),\quad
    D_{g_b}\Ups_b(h_0+\ft h_1)=0\qquad \text{at}\ \ft=0.
  \end{equation}
  For the solution $h$ of the initial value problem $L_{g_b}h=0$, $(h|_{\Sigma_0},\cL_{\pa_\ft}h|_{\Sigma_0})=(h_0,h_1)$, as described by Theorem~\ref{ThmPfIVP}, we then have $D_{g_b}\tau(h)=(\dot\gamma,\dot k)$ and $D_{g_b}\Ups_b(h)=0$ at $\Sigma_0$. Now, recall (e.g.\ from \cite[\S2.2]{HintzVasyKdSStability}) that the linearized constraint equations are equivalent to the vanishing of $D_{g_b}\Ein(h)(N,-)$ where $N$ is the unit normal to $\Sigma_0^\circ$ and $\Ein(g)=\Ric(g)-\half g R_g$ is the Einstein tensor. Since
  \[
    D_{g_b}\Ein(h) - \sfG_{g_b}\wt\delta_{g_b,\gamma}^*\bigl(D_{g_b}\Ups_b(h)\bigr) = 0,
  \]
  this implies that $\cL_{\pa_\ft}(D_{g_b}\Ups_b(h))=0$ at $\Sigma_0$ as well. Finally, since by the linearized second Bianchi identity, $D_{g_b}\Ups_b(h)$ satisfies the wave equation
  \[
    \cP_{b,\gamma}(D_{g_b}\Ups_b(h))=2\delta_{g_b}\sfG_{g_b}\wt\delta_{g_b,\gamma}^*(D_{g_b}\Ups_b(h))=0,
  \]
  we conclude that $D_{g_b}\Ups_b(h)=-\delta_{g_b}\sfG_{g_b}h\equiv 0$ and thus also $D_{g_b}\Ric(h)=0$, as desired.

  Thus, in order to complete the proof of linear stability, it suffices to arrange~\eqref{EqPfKerr1}--\eqref{EqPfKerr2}. This is accomplished by a minor modification of the arguments of \cite[Proposition~3.10 and Corollary~3.11]{HintzVasyKdSStability}. We present the details, following the reference, mainly in order to demonstrate that the decay rates of $(h_0,h_1)$ can be arranged to match those of $(\dot\gamma,\dot k)$. The task at hand has nothing to do with the constraint equations. Hence, it suffices to solve the following \emph{nonlinear} (but geometrically simpler) problem: given $(\gamma,k)$ close to $(\gamma_b,k_b)$ and with $(\gamma,k)-(\gamma_b,k_b)\in\Hbext^{s,-1/2+\alpha}\oplus\Hbext^{s-1,1/2+\alpha}$, find Cauchy data $(g_0,g_1)$ close to $(g_{b,0},g_{b,1}):=(g_b|_{\Sigma_0},0)$ and with $(g_0-g_{b,0},g_1-g_{b,1})\in\Hbext^{s,-1/2+\alpha}\oplus\Hbext^{s-1,1/2+\alpha}$, such that
  \begin{equation}
  \label{EqPfKerrNonlin}
    \tau(g_0+\ft g_1)=(\gamma,k),\quad
    \Ups_b(g_0+\ft g_1)=0\qquad\text{at}\ \ft=0.
  \end{equation}
  Our construction will give a smooth map $i_b\colon(\gamma,k)\mapsto(g_0,g_1)$ which takes $\tau(g_b)$ into $(g_{b,0},g_{b,1})$, thus its linearization at $\tau(g_b)$ takes $(\dot\gamma,\dot k)$ into the desired Cauchy data $(h_0,h_1)$.

  Define $\phi_b\in 1+\rho\CI(\Sigma_0)$ and $\omega_b\in\rho^2\CI(\Sigma_0;\Tsc^*\Sigma_0)$ by $g_b=:\phi_b\,d\ft^2+2\,d\ft\otimes_s\omega_b-\gamma_b$; we then define the component $g_0$ of $i_b(\gamma,k)=(g_0,g_1)$ by
  \[
    g_0=\phi_b\,d\ft^2+2\,d\ft\otimes_s\omega_b-\gamma.
  \]
  Let $N_0$, resp.\ $N_b\in T_{\Sigma_0^\circ}M^\circ$ denote the future timelike unit vector field with respect to $g_0$, resp.\ $N_b$; then $N_0-N_b\in\Hbext^{s,-1/2+\alpha}$. To make $\tau(g_0+\ft g_1)=(\gamma,k)$, we need to find $g_1$ with
  \[
    g_0\bigl((\nabla^{g_0+\ft g_1}_X-\nabla^{g_0}_X)Y,N_0) = k(X,Y)-g_0(\nabla_X^{g_0}Y,N_0),\quad X,Y\in T\Sigma_0^\circ,
  \]
  with $\nabla^g$ the Levi-Civita connection of $g$. This is equivalent to
  \[
    -(N_0\ft)g_1(X,Y) = 2\bigl(k(X,Y)-g_0(\nabla_X^{g_0}Y,N_0)\bigr).
  \]
  But $N_0\ft=1+\rho\CI+\cO(\rho^{1+\alpha})$, and $N_0\ft\neq 0$ on $\Sigma_0^\circ$ since $d\ft$ is timelike for $g_0$ (as $g_0$ is close to $g_b$ in $\cC^0$); hence this determines $g_1(X,Y)$ for $X,Y\in T\Sigma_0^\circ$, with $g_1|_{S^2\,\Tsc\Sigma_0}\in\Hbext^{s-1,1/2+\alpha}$ by assumption on $k$ and since the second term on the right loses one order of regularity but gains an order of decay, by the same arguments as given before equation~\eqref{EqKStGamma}.

  Finally, we need to arrange $\Ups_b(g_0+\ft g_1)-\Ups_b(g_0)=-\Ups_b(g_0)\in\Hbext^{s-1,1/2+\alpha}$ at $\ft=0$, so
  \[
    (\sfG_{g_0}g_1)(\nabla^{g_0}\ft,V) = -\Ups_b(g_0)(V),\quad V\in\Tsc\Sigma_0.
  \]
  Since $\nabla^{g_0}\ft\perp T\Sigma_0^\circ$ is a multiple of $N_0$, more precisely $\nabla^{g_0}\ft=(1+\rho\CI+\Hbext^{s,-1/2+\alpha})N_0$, this determines $(\sfG_{g_0}g_1)(N_0,X)=g_1(N_0,X)$ for $X\in\Tsc\Sigma_0$. But then we have $g_1(N_0,N_0)=2(\sfG_{g_0}g_1)(N_0,N_0)+(\tr_{g_0}g_1-g_1(N_0,N_0))$, with both summands on the right hand side known; this determines $g_1(N_0,N_0)$, and we have $g_1\in\Hbext^{s-1,1/2+\alpha}$. The proof is complete.
\end{proof}

\begin{rmk}
\label{RmkDecTo0}
  We explain why fast decay of $(\dot\gamma,\dot k)$ implies the decay of $h$ to zero as $t_*\to\infty$ in our chosen gauge, thus recovering the decay proved in \cite{AnderssonBackdahlBlueMaKerr} in the outgoing radiation gauge. Thus, let us assume that $\dot\gamma$ and $\dot k$ decay rapidly as $r\to\infty$ (sufficiently fast polynomial decay would suffice). Denote by $h$ the solution of the initial value problem for $L_b h=0$ constructed in Theorem~\ref{ThmPfKerr}, and denote by $\chi=\chi(\ft)$ a smooth cutoff, $\chi\equiv 1$ for $\ft\leq 1$, $\chi\equiv 0$ for $\ft\geq 2$. Then $h=\chi h+h'$, where $\ft\geq 1$ on $\supp h'$, and $L_b h'=-[L_b,\chi]h$ is supported in $\ft^{-1}([1,2])$ and decays rapidly as $r\to\infty$. We can solve this using the Fourier transform in $\ft$; the resolvent of $L_b$ with respect to $\ft$ is obtained from $\wh{L_b}(\sigma)$ via conjugation by $e^{i\sigma(\ft-t_*)}$ (which maps sufficiently fast decaying tensors into any fixed b-Sobolev space). To give a flavor of the argument, let us now pretend that the $\ft$-resolvent only has a simple pole at $\sigma=0$, with singular part given by a finite sum of terms $h_0\la\cdot,h_0^*\ra$, where $h_0\in\cK_b$, $h_0^*\in\cK_b^*$ (the full argument is only more involved algebraically); then we need to show that $\la([L_b,\chi]h)\ftrans(0),h_0^*\ra=0$. Since $D_{g_b}\Ric(h)=0$ and $D_{g_b}\Ups_b(h)=0$, and since $h_0^*=\sfG_{g_b}\delta_{g_b}^*\omega^*$ is dual-pure-gauge by Proposition~\ref{PropL0}, this is equivalent to the vanishing of the spacetime pairing (extending $h_0^*$ to spacetime by stationarity)
  \begin{align*}
    \la[L_b,\chi]h,h_0^*\ra &= \la[L_b,\chi]h,\sfG_{g_b}\delta_{g_b}^*\omega^*\ra \\
      &= \la\delta_{g_b}\sfG_{g_b}[L_b,\chi]h,\omega^*\ra \\
      &= \la\delta_{g_b}\sfG_{g_b}L_b(\chi h),\omega^*\ra \\
      &= \la -2\delta_{g_b}\sfG_{g_b}\wt\delta_{g_b,\gamma}^* D_{g_b}\Ups_b(\chi h),\omega^*\ra \\
      &= \la -\cP_{g_b,\gamma}[D_{g_b}\Ups_b,\chi]h,\omega^*\ra \\
      &= -\la[D_{g_b}\Ups_b,\chi]h,\cP_{g_b,\gamma}^*\omega^*\ra \\
      &= 0.
  \end{align*}
  Here, the fast spatial decay of $h$ is used to justify the integrations by parts in the second and the penultimate equalities.
  
  We reiterate that this argument \emph{only} applies to solutions of the linearized Einstein equation $D_{g_b}\Ric(h)=0$, $D_{g_b}\Ups_b(h)=0$, but \emph{not} to general solutions of the linearized gauge-fixed Einstein equation $L_b h=0$; indeed, even for generic smooth initial data $(h_0,h_1)$ in Theorem~\ref{ThmPfIVP} with \emph{compact} support, the solution $h$ does not decay to zero, i.e.\ the asymptotic leading order term $\hat h$ is non-zero.
\end{rmk}

%%%%%%%%%%%%%%%%%%%%%%%%%%%%%%%%%%%%%%%%%%%%%%%%%%%%%%%%%%%%%%%%%%%%%%
\bibliographystyle{alpha}
%\bibliography{/home/peter/research/bib/math,/home/peter/research/bib/mathcheck,/home/peter/research/bib/phys}

\end{document}